\documentclass[a4paper,
headsepline=off, titlepage=false, 11pt]{scrartcl}
\KOMAoptions{DIV=calc}
\setkomafont{author}{\sffamily\scshape}

\title{\vspace{-1cm}\semiLARGE
On the refined 
`Birch--Swinnerton-Dyer type'\\
conjectures of Mazur and Tate
}
\date{}
\author{Dominik Bullach \and Matthew H.\@ L.\@ Honnor} 

\usepackage[backend=bibtex, style=alphabetic, sorting=nyt, maxbibnames=9, maxalphanames=5, minalphanames=1
 ]{biblatex}
\renewbibmacro{in:}{%
  \ifentrytype{article}{}{\printtext{\bibstring{in}\intitlepunct}}}
  \DeclareFieldFormat{journaltitle}{#1\isdot}
  \DeclareFieldFormat[article, inproceedings, unpublished, phdthesis]{title}{\textit{#1}}
  
\input{preamble}

\begin{document}

\maketitle

\vspace{-1cm}
\begin{abstract}
    We prove a substantial part of conjectures of Mazur and Tate that refine the conjecture of Birch and Swinnerton-Dyer. 
    Our approach, which also leads to some results even finer than the predictions of Mazur and Tate, is via the `rank-zero component' of the relevant case of the equivariant Tamagawa Number conjecture. 
\end{abstract}

\section{Introduction}

\subsection{Refined conjectures of `Birch--Swinnerton-Dyer type'}
The conjectures of Birch and Swinnerton-Dyer, originating from \cite{BSD2} and developed into their final form by Tate in \cite{Tate95}, connect arithmetic invariants of an elliptic curve $E$ defined over $\Q$ with the order of vanishing and the leading term of its Hasse--Weil $L$-series at $s = 1$. These arithmetic invariants include the rank, as an abelian group, of the group of $\Q$-rational points $E (\Q)$ of $E$ and the cardinality of the, conjecturally finite, Tate--Shafarevich group $\sha_{E / \Q}$ of $E$. 
For a more detailed introduction, and further reading, the reader may for example consult the survey articles \cite{BurungaleSkinnerTian21, Gross11,  SwinnertonDyer, Tate95, Wiles06, Zhang14}.\\
In the 1980s Mazur and Tate formulated refinements of this conjecture that relate a certain group-ring-valued element $\theta_K^\mathrm{MT}$, which is constructed from modular symbols and interpolates the values of twisted Hasse--Weil $L$-series, to Galois-equivariant invariants of the base change of $E$ to an abelian number field $K$. In particular, the element $\theta_K^\mathrm{MT}$ is expected to encode precise information about the Galois module structure of $E ( K)$ and $\sha_{E / K}$. \\
It is convenient to subdivide the predictions of Mazur and Tate into three separate statements that we will refer to as as their `order of vanishing', `weak main conjecture', and `leading term' component.
Each of these components has been very influential and has inspired a range of similar conjectures in a variety of different contexts. This includes the analogue for Heegner points formulated by Darmon \cite{Darmon92} and extended by Bertolini and Darmon \cite{BertoliniDarmon94}, conjectures in the setting of the multiplicative group such as the `integral Gross--Stark conjecture' (from \cite{gross88}) and its refinements and generalisations due to Tate \cite{Tate04}, Darmon \cite{Darmon95}, Sano \cite{Sano14}, and Mazur--Rubin \cite{MazurRubin16}, as well as an analogue for the Rankin--Selberg convolution of two modular forms formulated by Cauchi and Lei \cite{CauchiLei}. Generalisations to modular forms of higher weight have moreover been studied by Ota \cite{Ota23}, Kim \cite{kim2023refined}, and Emerton--Pollack--Weston \cite{emerton2022explicit}.
Kurihara \cite[Conj.\@ 0.3]{Kurihara02} has also formulated a `strong main conjecture' as a strengthening of the `weak main conjecture' of Mazur and Tate, and results towards this are due to Kurihara \cite{Kurihara02}, Pollack \cite{Pollack}, and Kim--Kurihara \cite{KimKurihara}. 
\\ 
To describe the conjectures of Mazur and Tate in more detail, we set $G = G_K \coloneqq \gal{K}{\Q}$ and let $\cR$ be a ring with the property that $\cR [G]$ contains $\theta_K^\MT$. Work of Stevens \cite{Stevens89} shows that one can often take $\cR= \Z$ in practice\footnote{Recent numerical computations of Llerena-C\'ordova \cite{llerenacórdova2024} suggest that one should always choose $\cR$ big enough such that $|E (K)_\tor| \in \cR^\times$ in order for the conjectures of Mazur and Tate to be valid. 
This subtlety will however not be relevant to the present article and so we have chosen to give the original statements of the conjectures.
}, and -- even though we will not need it -- a self-contained discussion of this integrality question is given in Appendix \ref{integrality Mazur Tate appendix}.\\
Write $I_{\cR, G}$ for the augmentation ideal of $\cR [G]$ and $S (E / K )$ for the `integral Selmer group' defined by Mazur and Tate in \cite[\S\,1.7]{MT87}. If $\sha_{E / K}$ is finite, then $S ( E / K)$ is a finitely generated $\Z [G]$-module with the property that, for every prime number $p$, the `$p$-component' $S ( E / K ) \otimes_\Z \Z_p$ identifies with the Pontryagin dual $\Sel_{p, E / K}^\vee$ of the classical $p$-primary Selmer group $\Sel_{p, E / K}$ which fits into the fundamental exact sequence (cf., for example, \cite[\S\,X.4]{Silverman}) 
\begin{equation} \label{Selmer exact sequence}
    \begin{tikzcd}[column sep=small]
        0 \arrow{r} & (\sha_{E / K} [p^\infty])^\vee \arrow{r} & \Sel_{p, E / K}^\vee \arrow{r} & 
        \Hom_{\Z_p} ( E (K) \otimes_\Z \Z_p, \Z_p) \arrow{r} & 0.
    \end{tikzcd}
\end{equation}
We also write $r \coloneqq \mathrm{rk}_\Z (E (\Q))$ for the rank of $E$ and $\mathrm{sp} (m)$ for the number of primes dividing the conductor $m$ of $K$ at which $E$ has split-multiplicative reduction.

\begin{conj}[Mazur--Tate] \label{Mazur--Tate conj 1}
The following claims are valid.
\begin{liste}
    \item (`order of vanishing', \cite[Conj.\@ 4]{MT87}) $\theta_K^\MT \in I_{\cR, G}^{r + \mathrm{sp} (m)}$ 
    \item (`weak main conjecture', \cite[Conj.\@ 3]{MT87}) $\theta_K^\MT \in  \Fitt^0_{\cR [G]} (S (E / K) \otimes_\Z \cR)$
\end{liste}
\end{conj}

This conjecture holds for a fixed ring $\cR$ if and only if its `$p$-part', namely the conjecture for $\cR \otimes_\Z \Z_p$, holds for every prime number $p$. 
In this article, we will focus on the $p$-parts of Conjecture~\ref{Mazur--Tate conj 1} for prime numbers $p$ that satisfy the following mild hypothesis, which, for any given pair $(E, K)$, is known to be valid for all but finitely many prime numbers $p$ (see Remarks~\ref{big image hyp rk} and \ref{anomalous primes rk} for more details). In the statement we write $\widetilde E$ for the reduction of $E$ modulo $p$. 

\begin{hyp} \label{hyp}
    Assume $p > 3$ is a prime number with the following properties:
\begin{romanliste}
    \item The image of the Galois representation $\rho_{E, p} \: \gal{\overline{\Q}}{\Q} \to \mathrm{Aut} ( \mathrm{T}_p E) \cong \mathrm{GL}_2 (\Z_p)$ attached to $E$ contains $\mathrm{SL}_2 (\Z_p)$.
    \item At least one of the following conditions is satisfied:
    \begin{liste}
        \item $K$ contains no primitive $p$-th root of unity.
        \item $E$ has potentially good reduction at $p$ and $\widetilde{E} (\mathbb{F}_p)$ contains no point of order $p$. 
    \end{liste}
    \item If $E$ has additive reduction at $p$, then $p$ is unramified in $K$.
\end{romanliste}
\end{hyp}

In particular, $p$ is allowed to be an `anomalous' prime (in the terminology used by Mazur \cite{Mazur72}) if Hypothesis \ref{hyp}\,(ii)\,(a) is valid.
\begin{rk} \label{big image hyp rk}
    If $E$ does not have CM, then 
Serre has proved in \cite{Serre72} that $\rho_{E, p}$ is surjective for all but finitely many prime numbers $p$, and asked if in fact $\rho_{E, p}$ is always surjective when $p > 37$. It is conjectured that surjectivity is implied by $p \not \in \{ 2,3,5,11,13,17,37\}$ and the following is known in this direction.
\begin{itemize}
    \item Zywina has proved in \cite[Thm.\@ 1.10]{Zywina} that a prime that fails surjectivity is bounded from above by $\max \{ 37, N \}$, where $N$ denotes the conductor of $E$.
    \item If $E$ is a semistable elliptic curve and $p \geq 11$, then $\rho_{E, p}$ is surjective by a result of Mazur \cite[Thm.\@ 4]{Mazur78}.
    \item Zywina has proved in \cite[Thm.\@ 1.5]{Zywina} that if $\rho_{E,p}$ is not surjective (with $p > 13$) and $\ell \neq p$ is a prime at which $E$ does not have potentially good reduction, then $\ell \equiv \pm 1 \mod p$ and $p$ divides the Tamagawa number $\mathrm{Tam}_\ell$ at $\ell$. 
\end{itemize}
\end{rk}

To state our first main result, we set $r_p \coloneqq \mathrm{rk}_{\Z_p} (\Sel_{p, E / K}^\vee)$ and write $\# \: \Z_p [G] \to \Z_p [G]$ for the involution that sends $\sigma \in G$ to $\sigma^{-1}$.

\begin{thm} \label{mazur--tate main result 1}
Fix an abelian number field $K$ of conductor $m$.
    If the pair $(K, p)$ satisfies Hypothesis \ref{hyp}, then the following claims are valid.
    \begin{liste}
        \item $\theta^\mathrm{MT}_K \in I_{\Z_p, G}^{r_p + \mathrm{sp} (m) + 2 c^{(p)} (K)}$ with the integer $c^{(p)} (K) \geq 0$ defined in Remark \ref{mazur--tate main result 1 rk}\,(c) below.
        \item $\theta^{\mathrm{MT}, \#}_K \in \Fitt^0_{\Z_p [G]} ( \mathrm{Sel}_{p, E / K}^\vee)$.
     \end{liste}   
\end{thm}        

\begin{rk} \label{mazur--tate main result 1 rk}
\begin{liste}
    \item The adornment $\#$ in Theorem \ref{mazur--tate main result 1}\,(b) can often be removed. To explain this, we write $D(m) \coloneqq \mathrm{gcd} ( m , N)$ and $\delta (m) \coloneqq \mathrm{gcd} (D (m), \frac{N}{D(m)})$. If $\delta (m) = 1$, which is automatically satisfied if $E$ is semistable, then $\theta^\MT_K$ satisfies a 
    `functional equation' (stated in Remark \ref{functional equation rk}) that implies that $\theta^\MT_K$ and $\theta^{\mathrm{MT}, \#}_K$ only differ by a unit in $\Z [G]$. In this case, therefore, Theorem \ref{mazur--tate main result 1}\,(b) also shows that $\theta^{\mathrm{MT}}_K$ belongs to $\Fitt^0_{\Z_p [G]} ( \mathrm{Sel}_{p, E / K}^\vee)$.
    However, the ideals generated by $\theta^\MT_K$ and $\theta^{\mathrm{MT}, \#}_K$ may not be equal if $\delta (m) > 1$ and the authors would like to thank Juan-Pablo Llerena-C\'ordova for providing them with numerical examples where these are indeed different. 
    \item The perhaps surprising appearance of the factor 2 in Theorem \ref{mazur--tate main result 1}\,(a) can be motivated conceptually as follows. Write $\mathrm{ord} (\theta^\MT_K)$ for the largest non-negative integer $n$ such that $\theta^\MT_K$ belongs to $I_{\Z_p, G}^n$. By an observation of Mazur and Tate \cite[(1.6.4)]{MT87}, the known validity of the $p$-parity conjecture \cite{Dokchitser2} then combines with the functional equation for $\theta^\MT_K$ to imply, if $p$ is odd, that $\mathrm{ord} (\theta^\MT_K) \equiv r_p \mod 2$.
    \item To define the integer $c^{(p)} (K)$ that appears in the statement of Theorem \ref{mazur--tate main result 1}\,(a), we set $a_\ell \coloneqq \ell +1 -| \widetilde{E}(\mathbb{F}_\ell) | $ for every prime number $\ell$.
    We also define $\bm{1}_N (\ell) = 1$ if $\ell \nmid N$ and $\bm{1}_N (\ell) = 0$ otherwise. We then let $f^{(p)}_{\ell, K / \Q}$ denote prime-to-$p$ part of the residue degree of $\ell$ in $K / \Q$ and define the set (see also Lemma \ref{invertible Euler factors lemma} for an alternative characterisation)
\begin{align*}
C^{(p)}_\times (K) & \coloneqq \{ \ell : 
\ell \not \equiv \zeta ( a_\ell - \bm{1}_N (\ell) \zeta) \mod p 
    \text{ for all } \zeta \in \overline{\Q_p}^\times \text{ with } \zeta^{f^{(p)}_{\ell, K / \Q}} = 1 
    \} |
\end{align*}
as well the subsets
\begin{align*}
    C^{(p)}_2 (K) & \coloneqq \{ \ell \nmid N : \ell \neq p, \ell \in C^{(p)}_\times (K), a_\ell = 1, \ell \mid m, \ell^2 \nmid m\} \\
    C^{(p)}_0 (m) & \coloneqq \{ \ell \mid N : \ell \neq p, a_\ell = 0, \ell^2 \mid m \}
\end{align*}
as well as their cardinalities $c_2^{(p)} (K) \coloneqq | C^{(p)}_2 (K)|$ and $c_0^{(p)} (m) \coloneqq | C^{(p)}_0 (m)|$.
Given this, we may then define
\[
c^{(p)} (K) \coloneqq  c^{(p)}_2 (K) + c^{(p)}_0 (m).
\]
\end{liste}
\end{rk}

To the knowledge of the authors, Theorem \ref{mazur--tate main result 1}\,(a) is the first 
general result that combines the contributions from both the rank of $\Sel_{p, E / \Q}^\vee$ and the presence of trivial zeros. That a refinement of Conjecture \ref{Mazur--Tate conj 1}\,(a) of this shape may hold was first suggested by Ota, who proved in \cite[Thm.\@ 1.1]{Ota23} that $\theta_K^\mathrm{MT} \in I_{\Z_p, G}^{\min \{ r_p, p \}}$ if one assumes Hypothesis \ref{hyp}\,(i) and certain local conditions at $p$ and at primes $\ell$ that divide $mN$. If all split-multiplicative primes $\ell$ dividing $m$ belong to $C_\times^{(p)} (K)$ and $p$ is big enough, then this also easily implies $\theta_K^\mathrm{MT} \in I_{\Z_p, G}^{r_p + \mathrm{sp} (m) + c_2^{(p)} (K)}$ (cf.\@ \cite[Cor.~6.3]{Ota}).\\
Using entirely different methods, Bergunde and Gehrmann \cite{BerGeh17} have also proved that $\theta^\mathrm{MT}_K$ belongs to $I_{ \cR,G}^{\mathrm{sp}(m)}$ if $K$ is a real abelian field (which settles Conjecture \ref{Mazur--Tate conj 1}\,(a) in this case, if $ r = 0$).\\
As for Conjecture \ref{Mazur--Tate conj 1}\,(b), 
the following results have previously been obtained assuming Hypothesis \ref{hyp}\,(i) and a slight strengthening of \ref{hyp}\,(ii)\,(b), where $p$ is also assumed to be of good reduction for $E$.
\begin{itemize}
\item If one also assumes the vanishing of a relevant $\mu$-invariant and that $p$ does not divide the product of Tamagawa numbers $\mathrm{Tam} (E) \coloneqq \prod_{\ell \mid N} ( E (\Q_\ell) : E_0 (\Q_\ell))$, then Kurihara has proved that the image of $\theta_K^\mathrm{MT}$ in $\Z_p [G_K]$ belongs to the annihilator of a relevant Selmer group for every finite abelian $p$-extension $K /\Q$ of conductor $m$ coprime to $p N$ (see \cite[Rk.\@ 1.2.6]{Kurihara14Proceedings}) and to $\Fitt^0_{\Z_p [G]} (\Sel_{p, E / K}^\vee)$ if $E$ has good ordinary reduction at $p$, the prime $p$ is tamely ramified in $K$, and suitable additional assumptions are satisfied (see \cite[Thm.~10.]{Kurihara03}).  
    \item Results of Kataoka (see \cite[Thm.\@ 1.8]{Kataoka1} and \cite[Thm.\@ 1.6]{Kataoka2}) combine to imply the $p$-part of the `weak main conjecture' \cite[Conj.\@ 3]{MT87} holds if $E$ has good reduction at every $\ell \mid pm$ and one assumes certain further local conditions such as the vanishing of $H^0 ( \Q (\zeta_{mp^\infty}) \otimes_\Q \Q_\ell, E [p^\infty])$ for every $\ell \mid pm$.
\end{itemize}
If $K$ is a subfield of the cyclotomic $\Z_p$-extension of $\Q$, then the result of Theorem \ref{mazur--tate main result 1}\,(b) has previously been proved by Emerton--Pollack--Weston \cite{emerton2022explicit}. (We note that the main result of loc.\@ cit.\@ also involves the involution $\#$.)
\\
To state the leading term conjecture of Mazur and Tate, we denote the Tate multiplicative period of $E$ at a prime $\ell$ of split-multiplicative reduction by $q_{E, \ell} \in \Q_\ell^\times$, and write $\Tam_\ell \coloneqq \ord_\ell (q_{E, \ell})$ for the Tamagawa number of $E$ at $\ell$.
We also fix an abelian number field $L$ with Galois group $G_L \coloneqq \gal{L}{\Q}$, write $\cD^{(\ell)}_L \subseteq G_L$ for the decomposition group at $\ell$, and let $\rec_\ell \: \Q_\ell^\times \to \cD^{(\ell)}_L \subseteq G_L$ denote the associated local reciprocity map.  

\begin{conj}[`leading term', {\cite[Conj.\@ 6]{MT87}}] \label{Mazur--Tate conj 2}
    Suppose the conductor $m'$ of $L$ is a product of primes at which $E$ has split-multiplicative reduction. Then one has
    \[
    \theta_L^\MT \equiv \theta_\Q^\MT \cdot {\prod}_{\ell \mid m'} \big( \Tam_\ell^{-1} \cdot (\rec_\ell (q_{E, \ell}) - 1) \big) \mod I_{\cR, G_L}^{\mathrm{sp} (m') + 1}. 
    \]
\end{conj}

The approach of Mazur and Tate is supported by experimental evidence (see \cite[\S\,3.2]{MT87} and, more recently, the articles of Portillo-Bobadilla \cite{Portillo-Bobadilla} and Llerena-C\'ordova \cite{llerenacórdova2024}), and 
directly motivated by the earlier `$p$-adic Birch--Swinnerton-Dyer conjecture' of Mazur--Tate--Teitelbaum (from \cite{MTT84}).
Indeed, if $m' = \ell^n$ for some split-multiplicative prime $\ell$, then the `$\ell$-part' (so $\cR = \Z_\ell$) of Conjecture~\ref{Mazur--Tate conj 2} is a special case of the Mazur--Tate--Teitelbaum conjecture and therefore follows from the result of Greenberg and Stevens \cite{GreenbergStevens} if $\ell \geq 5$ (see also \cite{Kobayashi06}). 
Note that, due to its Iwasawa-theoretic nature, this approach only sees the `$\ell$-part'
of Conjecture~\ref{Mazur--Tate conj 2} (hence concerns the component on which $\ell$ is wildly ramified) and so 
misses much of the finer aspects of these congruences. The only theoretical evidence 
in an $\ell$-tamely ramified setting (still assuming $m' = \ell^n$)
that has hitherto been available in the literature is the result of de Shalit in \cite{deShalit} that applies to elliptic curves of prime conductor $\ell$ and is proved via an extension of the strategy of Greenberg and Stevens (see also recent work of Lecouturier \cite{Lecouturier} in this direction).\\
For a natural number $m$, we let $F_m$ denote the $m$-th cyclotomic field and write $\theta_m^\mathrm{MT} \coloneqq \theta_{F_m}^\mathrm{MT}$. 
We also let $\mathrm{Sp} (m)$ denote the set of prime numbers dividing $m$ at which $E$ has split-multiplicative reduction, and, for any extension $L / K$ of abelian fields, we write $\pi_{L / K} \: \Q [G_L] \to \Q [G_K]$ for the morphism induced by the restriction map $G_L \to G_K$.\\
Our second main result is a refinement of Conjecture \ref{Mazur--Tate conj 2} and reads as follows. 

\begin{thm} \label{mazur--tate main result 2}
Let $L$ be an abelian number field of conductor $m'$ and fix a prime number $p$ such that the pair $(L, p)$ satisfies Hypothesis \ref{hyp}.
Let $K$ be a subfield of $L$ and take $S' \subseteq \mathrm{Sp} (m')$ to be the subset of primes which split completely in $K$. Set $M' \coloneqq m'  \prod_{\ell \in S'} \ell^{-\ord_\ell (m')}$, and 
suppose that all primes $\ell$ dividing $M'$ satisfy $\ell \in C_\times^{(p)} (L)$.\\ Then $\theta_L^\mathrm{MT}$ belongs to $\mathfrak{A} \coloneqq (\prod_{\ell \in S'} I_{\Z_p, \cD^{(\ell)}_L}) \Z_p [G_L]$ and
one has 
    \[
    \theta_L^\mathrm{MT} \equiv \pi_{F_{M'} / K} (\theta^\mathrm{MT}_{F_{M'}}) \cdot {\prod}_{\ell \in S'} \big( \Tam_\ell^{-1} \cdot (\rec_\ell ( q_{E, \ell})  - 1) \big)
    \mod  I_{\Z_p, H} \mathfrak{A}
    \]
    with $H \coloneqq \gal{L}{K}$ and $F_{M'} \coloneqq \Q (\zeta_{M'})$ the $M'$-th cyclotomic field.
\end{thm}

\begin{rk}
\begin{liste}
    \item Using the `norm relations' in Proposition \ref{p: MT norm relns}, one can explicitly compute the element $\pi_{F_{M'} / K} (\theta^\mathrm{MT}_{F_{M'}})$ that appears in Theorem \ref{mazur--tate main result 1}\,(c) as a linear combination of the elements $\theta_F^\mathrm{MT}$ with $F$ ranging over a suitable set of subfields of $K$. If the conductor $m$ of $K$ and $M' / m$ are coprime, then one may in fact take this set to be simply $\{ K \}$. 
    \item If $E (\Q)$ has positive rank $r > 0$, then the result of Theorem \ref{mazur--tate main result 2} is trivial (by Theorem~\ref{mazur--tate main result 1}~(a)). 
    In any such case, Mazur and Tate predict in \cite[Conj.\@ 4]{MT87} a finer congruence for $\theta^\mathrm{MT}_K$ modulo $I_{\cR, G}^{r + \mathrm{sp} (m) + 1}$. If $r = 1$ and $\mathrm{sp} (m) = 0$, results towards this conjecture have recently been obtained by Burns, Kurihara, and Sano in \cite{BKS-ellipticcurves2}. We expect that their approach can be combined with ours in order to prove the `$p$-part' of the full congruence modulo $I_{\Z_p, G}^{r + \mathrm{sp} (m) + 1}$, for primes $p$ as in Theorem \ref{mazur--tate main result 1}, up to a unit in $\Z_p^\times$ if the Birch--Swinnerton-Dyer conjecture holds for $E$ and without ambiguity if $E$ validates the `generalised Perrin-Riou conjecture' from \cite{BKS-ellipticcurves2}. 
    In the generality of Hypothesis \ref{hyp}, a result of this kind would require an extension of the comparison of height pairings carried out by Burns and Mac\'ias Castillo in \cite[Thm.\@ 10.3]{BurnsMaciasCastillo}.
 \end{liste}   
\end{rk}

To end the discussion of our main results, we summarise some of the consequences towards the `$p$-parts' of the conjectures of Mazur and Tate (that is, the statements of the conjectures after extending scalars from $\cR$ to $\cR \otimes_\Z \Z_p$) in their original formulations from \cite{MT87} in the following corollary. This result is obtained by taking $K$ to be the maximal real subfield $F_{m}^+$ of the $m$-th cyclotomic field $F_{m}$ in Theorem \ref{mazur--tate main result 1} (and taking note of Remark \ref{mazur--tate main result 1 rk}\,(a)), and $K = \Q$ and $L = F_{m'}^+$ with $m'$ a product of split-multiplicative primes in Theorem \ref{mazur--tate main result 2}. Note that, in these cases, condition (ii)\,(a) in Hypothesis \ref{hyp} is always satisfied and, if $E$ is a semistable non-CM elliptic curve, then Hypothesis \ref{hyp}\,(i) is satisfied by Remark \ref{big image hyp rk} and Hypothesis \ref{hyp}\,(iii) is empty. 

\begin{cor}
If $p\geq 11$ is a prime number and $E$ is a semistable elliptic curve without CM, then the following claims are valid.
\begin{liste}
\item The `$p$-part' of the `order-of-vanishing component' of \cite[Conj.\@ 4]{MT87}, and hence also the `weak vanishing conjecture' \cite[Conj.\@ 1]{MT87}, holds. 
    \item The `$p$-part' of the `weak main conjecture' \cite[Conj.\@ 3]{MT87} holds.
    \item The `$p$-part' of the `leading term conjecture' \cite[Conj.\@ 6]{MT87} holds.
\end{liste}
\end{cor}

\subsection{Overview of proof strategy}

The approach we adopt in this article is motivated by the `Tamagawa Number Conjecture' of Bloch and Kato \cite{BlochKato}, which, when specialised appropriately, is well-known to recover the Birch--Swinnerton-Dyer conjecture (see \cite{kings-bsd, venjakob, BurungaleFlach} for details). 
Since the Mazur--Tate conjectures are themselves an equivariant refinement of the conjecture of Birch--Swinnerton-Dyer, it is natural to view them within the framework of the equivariant refinement of the Tamagawa Number Conjecture. This refinement, formulated independently by Fontaine--Perrin-Riou \cite{FontainePerrinRiou94} and Kato \cite{kato93} in the commutative setting and later generalised to motives with non-commutative coefficients by Burns and Flach \cite{BurnsFlach01}, is commonly referred to as the `equivariant Tamagawa Number Conjecture' (eTNC). While results in \cite{BurnsMaciasCastillo, BKS-ellipticcurves2} suggest there should exist a link between the eTNC and the Mazur--Tate conjectures, the full nature of this connection has hitherto remained unclear.\\
In this article, we establish a precise link between the relevant case of the eTNC and the conjectures of Mazur and Tate that sheds new light onto their relationship. With `one inclusion' in the eTNC at a prime satisfying Hypothesis \ref{hyp} recently proved in work of Burns and the first author \cite{BB} via an enhancement of the general theory of Euler systems and Kato's Euler system (from \cite{Kato04}), this new connection leads directly to Theorems~\ref{mazur--tate main result 1} and \ref{mazur--tate main result 2}.\\
To explain this strategy in a little more detail, we write $z^\mathrm{Kato} = (z^\mathrm{Kato}_m)_{m \in \N}$ for the collection of cohomology classes constructed by Kato in \cite{Kato04} (and adapted to our needs in Theorem~\ref{kato euler system}). Here each class $z_m^\mathrm{Kato}$ belongs to $H^1 (\cO_{F_m} [1 / Nmp], \mathrm{T}_p E)$ and is linked to the values of the twisted Hasse--Weil $L$-series at $s = 1$ via Kato's explicit reciprocity law. This reciprocity law can be reformulated (cf.\@ Theorem \ref{local points main result}) in such a way that there is a suitable local cohomology class $Q_m$ in $\bigoplus_{v \mid p} H^1 (F_{m, v}, \mathrm{V}_p E)$ with the property that
\[
\cP_m ( z^\mathrm{Kato}_m, Q_m) = \theta_m^\MT.
\]
Here $\cP_m (\cdot, \cdot)$ is induced by the cup product pairing on $\bigoplus_{v \mid p} H^1 (F_{m, v}, \mathrm{V}_p E)$ (see \S\,\ref{local duality section} for a precise definition).\\
The local class $Q_m$ has been given a very explicit description by Otsuki in \cite{Otsuki} (by building on earlier work of Kurihara \cite{Kurihara02}) that, amongst other things, enables us to control the denominators of $Q_m$ under condition \ref{hyp}\,(ii). More precisely, we use Honda theory to construct a suitable analogue of the `Artin--Hasse exponential' and show that $Q_m$ can be decomposed as the value of this exponential times the inverses of the relevant Euler factors $\Eul_\ell (\tilde \sigma_\ell)$, and certain elements $\nu_m^{(\ell)}$ of the group ring $\Z_p [G_m]$ of $G_m \coloneqq \gal{F_m}{\Q}$. The elements $\nu_m^{(\ell)}$ each encode the contribution of a prime $\ell \neq p$ towards Otsuki's construction and are amenable to explicit calculations. In the case of a split-multiplicative prime $\ell$, they are also given a cohomological interpretation in \S\,\ref{multiplicative group section} in the following way. For big enough $m$, one can define a canonical map
\[
\vartheta_m^{(\ell)} \: \Det_{\Z_p [G_m]} \big ( {\bigoplus}_{v \mid \ell} \mathrm{R}\Gamma (F_{m, v}, \Z_p (1)) \big)^{-1} \to \Q_p [G_m]
\]
such that $\nu^{(\ell)}_m \Eul_\ell (\tilde \sigma_\ell)^{-1}$ belongs to the image of this map. 
This approach is very much in the spirit of the eTNC since its relevant component implies that a similarly defined map,
\[
\Theta_m \: \Det_{\Z_p [G_m]} ( \mathrm{R}\Gamma (\cO_{F_m} [1 / Nmp], \mathrm{T}_p E))^{-1} \to H^1 (\cO_{F_m} [1 / Nmp], \mathrm{V}_p E),
\]
sends a $\Z_p [G_m]$-basis of $\Det_{\Z_p [G_m]} ( \mathrm{R}\Gamma (\cO_{F_m} [1 / Nmp], \mathrm{T}_p E))^{-1}$ to $z_m^\mathrm{Kato}$. In this direction, the equivariant theory of Euler systems, initiated by Burns, Sakamoto, and Sano in \cite{bss} and extended by Burns and the first author in \cite{BB}, can be used to prove that $z^\mathrm{Kato}_m$ belongs to the image of $\Theta_m$ if Hypothesis \ref{hyp} is valid.\\
To make this latter statement more explicit, one would ideally replace the \'etale cohomology complex $\mathrm{R}\Gamma (\cO_{F_m} [1 / Nmp], \mathrm{T}_p E)$ with the `finite-support cohomology' complex of Bloch and Kato, the cohomology groups of which are classical objects such as $\Sel_{p, E / K}^\vee$. This is also the strategy used in \cite{kings-bsd, venjakob, BurungaleFlach} to relate the Tamagawa Number Conjecture with the Birch--Swinnerton-Dyer conjecture. 
However, a significant obstacle arises when attempting to extend these arguments to the equivariant setting: The finite-support cohomology complex is rarely perfect as a complex of $\Z_p [G_m]$-modules, making it unavailable for equivariant computations.
 To overcome this, we carefully construct certain auxiliary Selmer complexes in \S\,\ref{approximation complexes section} that are perfect and sufficiently approximate the finite-support cohomology complex. 
By applying purely algebraic results of the nature proved by Burns and Sano in \cite[App.]{sbA} and extended in Appendix \ref{algebra appendix}, one may then prove rather generally that one has the inclusion
\begin{equation} \label{inclusion for Thm 1}
{\prod}_{\ell \mid m} ( \nu^{(\ell)}_m \Eul_\ell (\tilde \sigma_\ell))^\# \cdot \cP_m ( \Theta_m (a), Q_m) \in \Fitt^0_{\Z_p [G_m]} ( \Sel_{p, E / F_m}^\vee) 
\end{equation}
for every $a \in \Det_{\Z_p [G_m]} ( \mathrm{R}\Gamma (\cO_{F_m} [1 / Nmp], \mathrm{T}_p E))^{-1}$. This, when combined with the aforementioned result on the eTNC, then directly leads to Theorem \ref{mazur--tate main result 1}\,(b). 
\\
To prove Theorems \ref{Mazur--Tate conj 1}\,(a) and \ref{Mazur--Tate conj 2}, we `combine' the maps $\Theta_m$ and $\vartheta_m^{(\ell)}$. That is to say, we define a certain Nekov\'a\v{r}--Selmer complex $\SC^\bullet_m$ with the local condition at a split-multiplicative prime $\ell$ given by $\bigoplus_{v \mid \ell} \mathrm{R}\Gamma (F_{m, v}, \Z_p (1))$ such that the left hand side of (\ref{inclusion for Thm 1}) belongs to the image of a certain map
\[
\Det_{\Z_p [G_m]} ( \SC^\bullet_m)^{-1} \to \Z_p [G_m]. 
\]
In the context of Iwasawa theory, Nekov\v{a}\'r has previously observed in \cite[\S\,0.10]{NekovarSelmerComplexes} that a complex of this kind detects the presence of trivial zeroes.\\ 
Theorem \ref{mazur--tate main result 2} is then obtained via a calculation of `Bockstein morphisms' attached to the complex $\SC^\bullet_m$. 
These calculations are in many ways parallel to those previously performed in the context of the multiplicative group by Burns \cite[\S\,10]{burns07} and Burns--Kurihara--Sano \cite[\S\,5]{bks}.
The definition of Bockstein morphisms is directly inspired by the `algebraic height pairings' introduced by Nekov\v{a}\'r in \cite[\S\,11]{NekovarSelmerComplexes}, and we discuss their general formalism in \S\,\ref{bockstein section}.\\
It might be worth noting that, to date, the authors have found such a cohomological interpretation of $\nu_m^{(\ell)} \Eul_\ell (\tilde \sigma_\ell)^{-1}$ only for primes $\ell$ which are of split-multiplicative reduction for $E$. In the absence of a similar interpretation for the remaining primes $\ell$ dividing $Nm p$, we have to impose the condition $\ell \in C^{(p)}_\times (L)$ in Theorem \ref{mazur--tate main result 2} in order to ensure the associated Euler factors are invertible in $\Z_p [G_L]$ (cf.\@ Lemma \ref{invertible Euler factors lemma}). It would therefore be highly desirable to find a uniform cohomological intepretation of the constructions made by Otsuki in \cite{Otsuki}, likely within the framework of Kato's `local $\varepsilon$-constant conjecture' \cite{Kato-local-epsilon, FukayaKato} (see also Remark~\ref{epsilon constant rk} in this direction).

\subsection{General notation} \label{notation section}
For the convenience of the reader, we collect some general notation that we use throughout the article.
\smallskip\\
\textit{Algebra} Given an abelian group $A$, we write $A_\tor \coloneqq \bigcup_{n \in \N} A [n]$ for its torsion subgroup, and $A_\tf \coloneqq A / A_\tor$ for its torsionfree quotient. For any ring $R$, we then write $R [A] \coloneqq \bigoplus_{a \in A} R$ for the group ring of $A$ over $R$ and denote by $x^\#$ the image of an element $x \in R [A]$ under the involution $\# \: R [A] \to R [A]$ that is defined by $R$-linear extension of the rule $a \mapsto a^{-1}$ for all $a \in A$.
If $M$ is an $R [A]$-module, then $M^\#$ will denote $M$ with $A$-action given by $a \cdot m \coloneqq a^\# \cdot m$.\\
The completed group ring of $A$ is denoted as $R \llbracket A \rrbracket \coloneqq \varprojlim_{U \subseteq A} R [A / U]$ with $U$ ranging over all finite-index subgroups of $A$. Given a subgroup $U$ of $A$, we write $I_R (U) \coloneqq \ker \{ R \llbracket U \rrbracket \to R\}$ for its absolute and $I_{R, U} \coloneqq I_R (U) R \llbracket A \rrbracket = \ker \{ R \llbracket A \rrbracket \to R \llbracket A / U \rrbracket \}$ for its relative augmentation ideal. (If the coefficient ring is clear from context, then we suppress the subscript $R$.)
\\
We write $M^\vee \coloneqq \Hom_\Z (M, \Q  / \Z)$ for the Pontryagin dual of a $\Z$-module $M$. If $M$ is a $\Z [A]$-module, this is endowed with the contragredient $A$-action. Similarly, the $R$-linear dual $\Hom_R (M, R)$ of an $R [A]$-module is given the contragredient action. If $A$ is a finite group, then one has the isomorphism
\begin{equation} \label{dual sharp isomorphism}
\Hom_R (M, R)^\# \xrightarrow{\simeq} \Hom_{R [A]} (M, R [A]), \quad f \mapsto \big \{ m \mapsto {\sum}_{a \in A} f (a m) a^{-1} \big \}
\end{equation}
and we will write $M^\ast$ to mean either the $R$-linear or $R[A]$-linear dual of $M$.
Furthermore, for a fixed prime $p$ (that will always be clear from context), we let $\widehat{A} \coloneqq \varprojlim_{n \in \N} (A / p^n A)$ (or $A^\wedge$ where notationally more convenient)
denote the $p$-adic completion of $A$. The $p$-primary component of $A$ is denoted $A [p^\infty] \coloneqq \bigcup_{n \in \N} A [p^n]$.\\
For any ring commtuative ring $R$, we write $D (R)$ for the derived category of $R$-modules. The right-derived functor of a functor $\mathcal{F}$ is denoted $\mathrm{R}\mathcal{F}$ (for example, we use the notation $\RHom_R ( -, R)$). Left-derived tensor products are denoted $\otimes^\mathbb{L}$.
We will use the determinant functor $\Det_R (-)$ of Knudsen--Mumford \cite{KnudsenMumford} (see also \S\,\ref{algebra appendix section 1}).
\smallskip\\ 
\textit{Field arithmetic}  Fix an algebraic closure $\overline{\Q}$ along with an embedding $\iota \: \overline{\Q} \hookrightarrow \C$. For every natural number $m \in \N \coloneqq \Z_{> 0}$ we then set $\zeta_m \coloneqq \iota^{-1} ( e^{2 \pi i / m})$. We also write $F_m \coloneqq \Q (\zeta_m)$ for the $m$-th cyclotomic field, set $G_m \coloneqq \gal{\Q (\zeta_m)}{\Q}$, and recall that one has an isomorphism
\[
( \Z / m \Z)^\times \to G_m, \quad a \mapsto \sigma_a 
\]
with $\sigma_a$ defined by sending $\zeta_m$ to $\zeta_m^a$. For every prime number $\ell$, we also write $\Frob_\ell$ for a lift of the (arithmetic) Frobenius automorphism at $\ell$ to $\mathscr{G}_\Q \coloneqq \gal{\overline{\Q}}{\Q}$.\\
For every finite Galois extension $K$ of $\Q$ we set $G_K \coloneqq \gal{K}{\Q}$. Given another finite Galois extension $L$ of $\Q$ that contains $K$, we then write $\pi_{L / K} \: \C [G_L] \to \C [G_K]$ for the natural epimorphism induced by the restriction map $G_L \to G_K$.\\
Given a profinite group $\mathscr{G}$ and a finitely generated $\cR$-module $M$ endowed with a continuous action of $\mathscr{G}$, we write $\mathscr{C}^\bullet (\mathscr{G}, M)$ for the associated complex of continuous cochains and $\mathrm{R}\Gamma (\mathscr{G}, M)$ for the object of $D (\cR)$ defined by $\mathscr{C}^\bullet (\mathscr{G}, M)$. If $F$ is a field with absolute Galois group $\mathscr{G}_F \coloneqq \gal{\overline{F}}{F}$, then we abbreviate this to $\mathrm{R}\Gamma (F, M) \coloneqq \mathrm{R}\Gamma (\mathscr{G}_F, M)$. If $F$ is a number field and the action of $\mathscr{G}_F$ on $M$ is unramified outside a finite set $U$ of places of $F$, then $M$ is naturally a module for $\mathscr{G}_{F, U} \coloneqq \gal{F^U}{F}$ with $F^U \subseteq \overline{F}$ the maximal extension unramified outside $S$ and we write $\mathrm{R}\Gamma (\cO_{F, U}, M) \coloneqq \mathrm{R}\Gamma (\mathscr{G}_{F, U}, M)$. The asscociated cohomology groups will be denoted $H^i (F, M) \coloneqq H^i (\mathscr{G}_F, M)$ and $H^i (\cO_{F, U}, M) \coloneqq H^i (\mathscr{G}_{F, U}, M)$ for every $i \in \Z$.\smallskip \\
\textit{Elliptic curves}
Let $E$ be an an elliptic curve $E$ defined over $\Q$ and of conductor $N = N_E$, which will be fixed throughout the article. 
We also fix a global minimal Weierstra{\ss}  equation for $E$ and write $\omega_E$ for the corresponding N\'eron differential. We will often regard $E$ as an elliptic curve over a finite extension $F$ of $\Q_\ell$ for some prime number $\ell$, and write $\widehat{E}$ for the formal group of $E$ (usually with respect to $\omega_E$). The formal logarithm and formal exponential map of $\widehat{E}$ are denoted as $\log_{\widehat{E}}$ and $\exp_{\widehat{E}}$, respectively. The reduction of $E$ modulo the maximal ideal of $F$ (with respect to a Weierstra{\ss} equation that is minimal over $F$) will be written as $\widetilde E$. We also write $E_0 (F)$ and $E_1 (F)$ for the subgroups of $E (F)$ comprising points that, over the residue field $\mathbb{F}$ of $F$, reduce to a non-singular point and to the identity of $\widetilde E (\mathbb{F})$, respectively. We also use the semi-local variants $E_i (K_\ell) \coloneqq \bigoplus_{v \mid \ell} E_i (K_v)$ for $i \in \{ \emptyset, 0, 1\}$ if $K$ is a number field and $\ell$ a prime number.\\
For a prime number $p$, assumed to be odd in this article, we denote the  $p$-adic Tate modules of $E$ as
\[
\mathrm{T}_p E \coloneqq {\varprojlim}_{n \in \N} E [p^n] 
\quad \text{ and } \quad 
\mathrm{V}_p E \coloneqq \Q_p \otimes_{\Z_p} \mathrm{T}_p E.
\]
These are endowed with a natural action of $\mathscr{G}_\Q$. 
Given an abelian field $K$, we often also write $T_{K / \Q}$ for the induced representation $\mathrm{Ind}_{\mathscr{G}_\Q}^{\mathscr{G}_K} (\mathrm{T}_p E) = (\mathrm{T}_p E) \otimes_{\Z_p} \Z_p [G]$ on which $\mathscr{G}_\Q$ acts by $\sigma \cdot ( a \otimes b) \coloneqq (\sigma a) \otimes (\overline{\sigma}^{-1} b)$. Here $\overline{\sigma}$ is the image of $\sigma \in \mathscr{G}_\Q$ in $G$.\\
We let $\bm{1}_S (x)$ denote the indicator function of a set $S$. If $S$ is the set of prime numbers not dividing an integer $N$, we write simply $\bm{1}_N (x)$. 

\section{Local points and Mazur--Tate elements}

In this section we define the modular elements of Mazur and Tate, and relate them to Kato's Euler system. To do this, we follow the approach of Kurihara \cite{Kurihara02}, as further developed by Kobayashi \cite{Kobayashi03, Kobayashi06} and Otsuki \cite{Otsuki}, to construct useful local points. After some preliminaries, the main result of this section is stated as Theorem \ref{local points main result}.

\subsection{Modular symbols}
\label{s:Modular symbols and Mazur--Tate elements}

Fix a minimal Weierstrass model of $E$ over $\Z$ and write $\omega_E$ for the corresponding N\'eron differential. We also fix generators $\gamma^+$ and $\gamma^-$ of the subgroups $H_1 (E (\C), \Z)^+$ and $H_1 (E (\C), \Z)^-$ of $H_1 (E (\C), \Z)$ on which complex conjugation acts by $+1$ and $-1$, respectively. Write $c_\infty \in \{1, 2\}$ for the number of connected components of $E (\R)$, and define periods of $E$ by setting
\[
\Omega^+ = \Omega^+_{\omega, \gamma} \coloneqq  \int_{E (\R)} \mid \omega_E \mid = c_\infty \cdot \int_{\gamma^+} \omega_E 
\quad \text{ and } \quad 
\Omega^- = \Omega^-_{\omega, \gamma} \coloneqq  \int_{\gamma^-} \omega_E. 
\]
We assume $\gamma^+$ and $\gamma^-$ are chosen in such a way that $\Omega^+ > 0$ and $\Omega^- \in i \R_{> 0}$. 

\begin{rk}
Artin formalism suggests that the above normalisation of periods, which is consistent with \cite{wiersemawuthrich}, is best suited to the study of special values of $L$-series. We note, however, that our convention slightly differs from that chosen by Mazur and Tate in \cite{MT87}. To be more precise, Mazur and Tate use the periods $\frac12 \Omega^+$ and $\frac{c_\infty}{2} \Omega^-$. Since the difference only concerns a possible factor of $2$, this will be irrelevant to our main results.   
\end{rk}

By the modularity theorem \cite{BCDT01} there is a normalised newform $f$ of weight 2 associated to the isogeny class of $E$, which we use to define the map
\[
\lambda_f \: \Q \to \C, 
\quad a \mapsto 2 \pi \int_{0}^\infty f (a + i \tau) \mathrm{d}\tau.
\]
In addition, we define maps $[\cdot]^+_E \: \Q \to \R$ and $[\cdot]^-_E \: \Q \to \R$ 
by means of
\[
\lambda_f (a) = [a]^+_E \Omega^+ + [a]^-_E \Omega^-
\quad \text{ for all } a \in \Q. 
\]

\begin{definition}
For every abelian number field $K$ of conductor $m = m_K$, we define the `Mazur--Tate element' 
\[
\theta^\mathrm{MT}_K \coloneqq \pi_{F_m / K} \big ( \sum_{a \in (\Z / m \Z)^\times} 
( \msym{a}{m}^+_E + \msym{a}{m}^-_E ) \sigma_a \big)
\quad \in \R [G_K]. 
\]
If $K  = F_m$, then we abbreviate this to $\theta_m^\mathrm{MT} \coloneqq \theta^\mathrm{MT}_{F_m}$.
\end{definition}

\begin{rk}
\begin{liste}
    \item The Manin--Drinfeld theorem \cite{Manin72,Drinfeld73} implies that the modular symbols $[\cdot]^\pm$ are rational-valued, hence $\theta^\mathrm{MT}_K$ belongs to $\Q [G_K]$. In addition, the Mazur--Tate elements have nice integrality properties which are discussed in detail in Appendix~\ref{integrality Mazur Tate appendix}.
    \item Write $F_m^+$ for the maximal real subfield of $F_m$. Then $\frac12 \theta^\mathrm{MT}_{F_m^+}$ coincides with the `modular element' defined by Mazur and Tate in \cite[(1.2)]{MT87}. This follows from the fact that $\lambda ( \frac{- a}{m}) = \overline{\lambda ( \frac{a}{m})}$, (cf.\@ \cite[Lem.\@ 5]{wiersemawuthrich}).
    \item Birch's formula (see \cite[(8.6)]{MTT84}) implies that, for every primitive Dirichlet character $\chi \: G_m \to \C^\times$ of conductor $m$, the Mazur--Tate element is related to $L$-values via the interpolation property
    \[
    \chi (\theta_m^\mathrm{MT}) = \frac{G(\chi) \cdot L (E, \chi^{-1}, 1)}{\Omega^{\epsilon (\chi)}}
    \]
    with the (primitive) `Gauss sum' $G (\chi) \coloneqq \sum_{a \in (\Z / m \Z)^\times} \chi (\sigma_a) e^{2 \pi i a / m} \in \C$ and $\epsilon (\chi) \in \{ \pm \}$ is the sign of $\chi ( \sigma_{-1}) \in \{ \pm 1 \}$.
    \end{liste}
\end{rk}

The Mazur--Tate elements satisfy the following norm relations. 

\begin{prop}[Mazur--Tate]\label{p: MT norm relns}
    For every $m \in \N$ and prime number $\ell$ one has that 
    \[ 
    \pi_{m \ell / m} ( \theta^{\mathrm{MT}}_{m \ell}) =
    \begin{cases}
        (a_\ell - \bm{1}_N (\ell) \sigma_\ell^{-1} - \sigma_\ell) \cdot \theta^\mathrm{MT}_{m}
        \quad & \text{ if } \ell \nmid m, \\ 
        a_\ell \theta^\mathrm{MT}_{m} - \bm{1}_N (\ell) \NN_{F_m / F_{m / \ell}} \theta^\mathrm{MT}_{{m / \ell}} 
        & \text{ if } \ell \mid m.
    \end{cases}
    \]
\end{prop}

\begin{proof}
    This follows from \cite[(4.2)]{MTT84}, see also \cite[(1.3) on p.\@ 717]{MT87}.
\end{proof}

\begin{rk} \label{functional equation rk}
Write $D(m) \coloneqq \mathrm{gcd} ( m , N)$. If $\mathrm{gcd} (D (m), \frac{N}{D(m)}) = 1$, then the Mazur--Tate elements satisfy the `functional equation' (see \cite[(1.6.2)]{MT87})
\[
\theta_m^\mathrm{MT} = \epsilon_f \cdot \sigma_{- Q}^{-1} \cdot (\theta_m^\mathrm{MT})^\#,
\]
where $Q \coloneqq N/D(m)$ and $\epsilon_f \in \{ \pm 1 \}$. In this case, therefore, $\theta_m^\mathrm{MT}$ and $(\theta_m^\mathrm{MT})^\#$ generate the same $\Z [G_m]$-submodule of $\Q [G_m]$. The proof of this functional equation follows from the corresponding result for modular symbols (see \cite[\S\,6 Prop.]{MTT84}), and is given, in the case $m$ and $N$ are coprime, in \cite[Prop.\@ 5.16]{Ota}. The proof for general $m$ and $N$ is identical.   
\end{rk}

\subsection{Local Tate duality} \label{local duality section}

For every finite place $v$ of $K$ we write $E (K_v)^\wedge \coloneqq \varprojlim_{n \in \N} (E (K_v) / p^n E (K_v))$ for the $p$-adic completion of the group of $K_v$-rational points $E (K_v)$ of $E$. We denote the local Kummer map as $\kappa^{(v)} \: E (K_v)^\wedge \hookrightarrow H^1 (K_v, \mathrm{T}_p E)$ and, for every prime number $\ell$, denote its semilocal variant as $\kappa^{(\ell)} \: \bigoplus_{v \mid \ell} E (K_v)^{\wedge} \hookrightarrow H^1 (\Q_\ell, T_{K / \Q})$. Following Bloch and Kato \cite{BlochKato}, we define local spaces
\[
H^1_f (\Q_\ell, T_{K / \Q}) \coloneqq \im ( \kappa^{(\ell)}) 
\quad \text{ and } \quad
H^1_{/ f} (\Q_\ell, T_{K / \Q}) \coloneqq \coker (\kappa^{(\ell)}). 
\]
Recall that the Weil pairing induces a canonical isomorphism $(\mathrm{T}_p E)^\ast (1) \cong \mathrm{T}_p E$ and hence a cup product pairing 
\[
H^1 (K_v, \mathrm{T}_p E) \times H^1 (K_v, \mathrm{T}_p E) \cong 
H^1 (K_v, \mathrm{T}_p E) \times H^1 (K_v, (\mathrm{T}_p E)^\ast (1)) \xrightarrow{\cup} H^2 (K_v, \Z_p (1)) \cong \Z_p
\]
for every $p$-adic place $v$ of $K$. Taking the sum over all $p$-adic places of $K$ then gives 
a pairing
\begin{equation} \label{semilocal Tate pairing}
H^1 (\Q_p, T_{K / \Q}) \times H^1 (\Q_p, T_{K / \Q}) \to {\bigoplus}_{v \mid p} \Z_p \xrightarrow{\mathrm{Trace}} \Z_p.
\end{equation}
By \cite[Prop.\@ 3.8]{BlochKato} the space $H^1_f (\Q_p, T_{K / \Q})$ is self-orthogonal under (\ref{semilocal Tate pairing}) and so it induces a pairing
\begin{equation}\label{e: induced pairing}
    (\cdot, \cdot )_{E / K} \: H^1_{ / f} ( \Q_p, T_{K / \Q}) \times H^1_f (\Q_p, T_{K / \Q}) \to \Z_p.
\end{equation}
In particular, every $Q = (Q_v)_{v \mid p}$ in $H^1_f (\Q_p, T_{K / \Q})= \bigoplus_{v \mid p} \im (\kappa^{(v)})$ gives rise to a map
\begin{equation} \label{definition equivariant pairing map}
\cP_{K} (\cdot, Q) \: H^1 (K, \mathrm{T}_p E) \to \Z_p [G], \quad a \mapsto {\sum}_{\sigma \in G} (\mathrm{loc}^{(p)}_{/ f} (\sigma a), Q)_{E/ K} \sigma^{-1},
\end{equation}
where, as before, $G = \gal{K}{\Q} $ and $\mathrm{loc}_{/ f}^{(p)} \coloneqq \oplus_{v \mid p} \mathrm{loc}_{/ f}^{(v)}$ with 
$\mathrm{loc}_{/ f}^{(v)}$ the composite of the restriction map $\mathrm{loc}^{(v)} \: H^1 (K, \mathrm{T}_p E) \to H^1 (K_v, \mathrm{T}_p E)$ and the projection onto $H^1_{/ f} (K_v, \mathrm{T}_p E)$.

\begin{rk}\label{r: extending pairing}
Write $V_{K / \Q} \coloneqq T_{K / \Q} \otimes_{\Z_p} \Q_p$. We note that the pairing $(\cdot, \cdot)_{E / K}$ can be linearly extended to a pairing on $H^1_{ / f} ( \Q_p, V_{K / \Q}) \times H^1_f (\Q_p, V_{K / \Q})$. As a consequence, for every $Q \in H^1_f (\Q_p, V_{K / \Q})$ we also obtain a map $H^1 (K, \mathrm{V}_p E) \to \Q_p [G]$ that we denote by $\cP_K (\cdot, Q)$ as well. Since the target of (\ref{definition equivariant pairing map}) is $\Z_p$-torsion free, it factors through the natural map $H^1 (K, \mathrm{T}_p E) \to H^1 (K, \mathrm{V}_p E)$ and so we feel this slight abuse of notation is not unjustified.
\end{rk}

Write 
\begin{equation}\label{e: dual exp map Kato}
    \exp^\ast_{K_v} \: H^1_{/f} (K_v, \mathrm{V}_p E) \xrightarrow{\simeq} \mathrm{Fil}^0 D_{\dR, K_v} (\mathrm{V}_p E) 
\end{equation}
for the `dual exponential map' that is defined by Kato in \cite[Ch.\@ II, \S\,1.2.4]{kato93}
and define the composite map
\begin{align*}
 \exp^\ast_{\omega_E} \:
{\bigoplus}_{v \mid p} H^1_{/f} (K_v, \mathrm{V}_p E) \xrightarrow{\oplus_{v \mid p} \exp^\ast_{K_v}}
{\bigoplus}_{v \mid p}  \mathrm{Fil}^0 D_{\dR, K_v} (\mathrm{V}_p E) & \xrightarrow{\simeq} 
(\Q_p\otimes_\Q K) \otimes_\Q H^0 (E, \Omega^1_{E / \Q}) \\
     &  \xrightarrow{\omega_E \mapsto 1} \Q_p \otimes_\Q K,
\end{align*}
where the second arrow is the comparison isomorphism of $p$-adic Hodge theory.
The following result then
records the basic properties of the pairing $\cP (\cdot, \cdot)$.

\begin{lem} \label{pairing properties}
Fix $Q = (Q_v)_{v \mid p} \in H^1_f (\Q_p, T_{K / \Q})$ and let $a \in H^1 (\cO_{K, S}, \mathrm{T}_p E)$. 
\begin{liste}
\item One has 
\begin{equation} \label{pairing explicit formula}
\cP_K (a, Q) = {\sum}_{\sigma \in G_K} \mathrm{Tr}_{(\Q_p \otimes_\Q K) / \Q_p} \big ( (\log_{\widehat{E}} (Q_v))_{v \mid p} \cdot (\exp^\ast_{\omega_E} \circ \loc_{/f}^{(p)})(\sigma a) \big ) \sigma^{-1}.
\end{equation}
\item For every $x \in \Z_p [G]$ one has $x \cdot \cP_K (a, Q) = \cP_K (x \cdot a, Q) = \cP_{K} (a, x^\# \cdot Q)$.
\item If $L$ is a subfield of $K$ and $H \coloneqq \gal{K}{L}$, then in $\Z_p [G_L]$ one has
\[
\pi_{K / L} \big( \cP_K (a,Q) \big) = \cP_L ( \cores_{K / L}(a), \mathrm{Tr}_{K / L}(Q)).
\]
In particular, if $a$ is fixed by $H$, then 
$\cP_K (a, Q) = \cP_L (a, \mathrm{Tr}_{K / L}(Q)) \cdot \NN_H$ in $\Z_p [G_K]$.
\end{liste}
\end{lem}

\begin{proof}
In light of Remark \ref{r: extending pairing}, we may verify claim (a) after extending scalars to $\Q_p$. 
 Given this, the key point is then that, by \cite[Ex.\@ 3.11]{BlochKato}, one has the 
commutative diagram
    \[
    \begin{tikzcd}[row sep=small]
      \Q_p \otimes_{\Z_p} H^1_{ / f} ( \Q_p, T_{K / \Q}) \arrow{rr}{\exp^\ast_{\omega_E}} \arrow{d}[left]{\simeq}
      & 
      & 
      \Q_p \otimes_\Q K  \arrow{d}[left]{\simeq} \\
      (\Q_p \otimes_{\Z_p} H^1_f (\Q_p, T_{K / \Q}))^\ast \arrow{r}{\simeq} & \bigoplus_{v \mid p} ( \Q_p \otimes_{\Z_p} \widehat{E}(\cM_{K_v}) )^\ast \arrow{r}{\exp^\ast_{\widehat{E}}}   & (\Q_p \otimes_\Q K)^\ast .
    \end{tikzcd}
    \]
    Here the vertical isomorphism on the left is induced by the fact that the pairing $(\cdot, \cdot)_{E / K}$ is perfect after extending scalars to $\Q_p$ (by \cite[Prop.\@ 3.8]{BlochKato}), and the vertical isomorphism on the right is the map $\Q_p \otimes_\Q K \to (\Q_p \otimes_\Q K)^\ast$ that sends $a \mapsto \{ \alpha \mapsto \mathrm{Tr}_{(\Q_p \otimes_\Q K) / \Q_p} (a \alpha) \}$ (and so is induced by the trace pairing). Finally, the map $\exp_{\widehat{E}}^\ast$ is induced by the duals of the exponential maps $\exp_{\widehat{E}}$ of the formal groups $\widehat{E}$ of $E / K_v$ (associated to $\omega_E$).\\
    Now, the commutativity of the diagram implies $(x, \exp_{\widehat{E}} ( \alpha))_{E / K} = \mathrm{Tr}_{(\Q_p \otimes_\Q K) / \Q_p} ( \alpha \exp^\ast_{\omega_E} (x))$ for all $x \in H^1_{/ f} ( \Q_p, V_{K / \Q})$ and $\alpha \in \Q_p \otimes_\Q K$. Claim (a) then follows upon substituting $\alpha = \log_{\widehat{E}} (Q)$. \\ 
    Claims (b) and (c), in turn, are verified by means of direct calculations (cf.\@ \cite[Lem.~3.5]{Otsuki}). 
\end{proof}

\subsection{Kato's Euler system} \label{section statement local points main result}

Recall that, for every prime number $\ell$, the `Euler factor' at $\ell \neq p$ is defined as
\[
\mathrm{Eul}_\ell (X) \coloneqq {\det}_{\Q_p} ( 1 - \Frob_\ell^{-1} X \mid (\mathrm{V}_p E)^{\mathcal{I}_\ell}) \quad \in \Q_p [X],
\]
where $\mathcal{I}_\ell \subseteq \gal{\overline{\Q}}{\Q}$ denotes a choice of inertia subgroup at $\ell$.
Explicitly, one has (see, for example, \cite[Ex.~1.26]{kings-bsd})
\[
\mathrm{Eul}_\ell (X) = 1 - a_\ell \ell^{-1} X + \bm{1}_N (\ell) \ell^{-1} X^2.
\]
For any integer $n$, we let $S_n$ denote the set of prime divisors of $n$. We also write $S_{n\infty} \coloneqq S_n \cup \{\infty\}$ and set
\[
S(K) \coloneqq S_{p N m \infty}
\]
with $m \coloneqq m_K$ the conductor of $K$ and $N \coloneqq N_E$ the conductor of $E$. 

\begin{thm}[Kato] \label{kato euler system}
    Let $\mathscr{F}$ denote the set of finite abelian extensions of $\Q$.
    Then there exists a collection of elements
\[
z^\mathrm{Kato} \coloneqq (z^\mathrm{Kato}_K)_{K \in \mathscr{F}} \in {\prod}_{K \in \mathscr{F}}
H^1 ( \cO_{K, S (K)}, \mathrm{V}_p E)
\]
with the following properties.
\begin{liste}
\item If $L, K \in \mathscr{F}$ with $K \subseteq L$, then one has 
\[
\mathrm{Cores}_{L / K} ( z^\mathrm{Kato}_L) = \big( {\prod}_{\ell \in S (L) \setminus S (K)} \mathrm{Eul}_\ell (\Frob_\ell^{-1}) \big) \cdot z^\mathrm{Kato}_K,
\]
where $\mathrm{Cores}_{L / K} \: H^1 (L, \mathrm{V}_p E) \to H^1 (K, \mathrm{V}_p E)$ denotes the corestriction map. 
\item Set
\[
y_K^\mathrm{Kato} \coloneqq \big( {\prod}_{\ell \in S_N \setminus S_{p {m_K}}} \mathrm{Eul}_\ell ( \Frob_\ell^{-1}) \big)^{-1} \cdot z_K^\mathrm{Kato}.
\]
If $E [p]$ is irreducible as an $\mathbb{F}_p [\mathscr{G}_\Q]$-module
and $E (K)$ contains no point of order $p$, then $c_\infty y_K^\mathrm{Kato}$ (and hence also $c_\infty z^\mathrm{Kato}_K$) belongs to $H^1 (\cO_{K, S (K)}, \mathrm{T}_p E)$. 
\item One has the equality 
\[
({\oplus}_{v \mid p} \exp^\ast_{K_v}) ( z^\mathrm{Kato}_K) = \big( {\sum}_{\chi \in \widehat{G_K}} \frac{L_{S (K)} (E, \chi^{-1}, 1)}{\Omega^{\mathrm{sgn} (\chi)}} e_\chi \big) \otimes \omega_E 
\]
 in $\bigoplus_{v \mid p} \mathrm{Fil}^0_{\dR, K_v} (\mathrm{V}_p E) \cong \Q_p \otimes_\Q H^0 ( E, \Omega^1_{E / K}) = (\Q_p \otimes_\Q K) \otimes_\Q H^0 (E, \Omega^1_{E / \Q})$. 
\end{liste}
\end{thm}

\begin{proof}
The elements $z^\mathrm{Kato}_K$ are defined by slightly modifying the elements ${}_{c, d} z_m^{(p)} (f, 1, 1, \xi, S (K))$ constructed by Kato in \cite[(8.1.3)]{Kato04}. The integrality property in claim (b) is proved by the argument of \cite[12.6]{Kato04} (following Delbourgo \cite[App.\@ A]{Delbourgo}, see also \cite[Thm.~6.1]{Kataoka1}) and the `explicit reciprocity law' in claim (c) follows from \cite[Thm.\@ 9.7 and 12.5]{Kato04}.
\end{proof}

\begin{rk}
    Suppose the Galois representation $\rho_{E, p} \: \mathscr{G}_\Q \to \Aut (\mathrm{T}_p E) \cong \mathrm{GL}_2 (\Z_p)$ contains $\mathrm{SL}_2 (\Z_p)$. Then $E [p]$ is an irreducible $\mathbb{F}_p [\mathscr{G}_\Q]$-module and $E (K)$ has no non-trivial point of order $p$ for all finite abelian extensions $K$ of $\Q$. 
\end{rk}

We now explain the link between Kato's Euler system, as normalised in Theorem \ref{kato euler system}, and the Mazur--Tate elements $\theta^\MT_K$.\\ 
To prepare for the statement of our result in this direction we define, for $a \in \N$, an automorphism $\tilde \sigma_a$ of $\Z [\zeta_m]$ as follows. 
Set $m_2 \coloneqq \prod_{\ell \mid a} \ell^{\ord_\ell (m)}$ and $m_1 \coloneqq m / m_2$ so that we have a decomposition $m = m_1 m_2$. We then take $\tilde \sigma_a$ to be the image of $\sigma_a$ under the splitting map $G_{m_1} \hookrightarrow G_{m_1} \times G_{m_2} \cong G_{m}$ that sends $g \mapsto (g, \id_{F_{m_2}})$. In particular, $\tilde \sigma_a$ agrees with $\sigma_a$ if $a$ is coprime with $m$, and is the identity map if $m \mid a$ (more generally, if and only if $m_1 = 1$). \\
We further let $\widehat{E}$ denote the formal group of $E / \Q_p$ and, given a finite extension $F /\Q_p$, write $\cM_F$ for the maximal ideal of $F$.\\ 
The main result of this section is as follows.

\begin{thm} \label{local points main result}
 Let $m$ be a natural number coprime to $p$ and fix $n \in \Z_{\geq 0}$.
 We also write $\tau$ for the $p$-adic Teichm\"uller character (regarded as a character $G_{mp^n} \to \Z_p^\times$). Then there exists 
 \[
 \mathfrak{k}_{mp^n} \in \Q_p \otimes_{\Z_p} \widehat{E} (F_{mp^n, p})
 \coloneqq \Q_p \otimes_{\Z_p} {\bigoplus}_{v \mid p} \widehat{E} (\cM_{F_{mp^n, v}})
 \]
 and, for every prime number $\ell \mid m$, an element $\nu^{(\ell)}_{mp^n} \in \Z_p [G_{mp^n}]$ with the following properties.
\begin{liste}
\item One has
\[
\big( {\prod}_{\ell \mid m} \Eul_\ell (\tilde \sigma_\ell)^{-1} \cdot \nu_{mp^n}^{(\ell)} \big)^\# \cdot 
\cP_{F_{mp^n}} ( y^\mathrm{Kato}_{mp^n}, \mathfrak{k}_{mp^n}) = \theta_{mp^n}^\mathrm{MT}.
\]
\item The element $(1 - e_\tau)\mathfrak{k}_{mp^n}$ belongs to (the image of) $\widehat{E} (F_{mp^n, p}) + (p\Eul_p (\sigma_p))^{-1} \widehat{E} (F_{m, p})$. If $p$ is not a multiplicative prime and $a_p \not\equiv 1 \mod p$, then also $\mathfrak{k}_{mp^n}$ belongs to $\widehat{E} (F_{mp^n, p})$.
\item One has the following equalities in $\Q_p \otimes_{\Z_p} \widehat{E} (F_{mp^n, p})$.
\begin{romanliste}
\item $\mathrm{Tr}_{mp^{n + 1}/mp^n} ( \mathfrak{k}_{mp^{n + 1}} )  = 
\begin{cases}
        a_p \mathfrak{k}_{mp^{n}} - \bm{1}_N(p) \mathfrak{k}_{mp^{n-1}} & \text{ if } n \geq 1, \\
        ( a_p  - \bm{1}_N(p)   \sigma_p - \sigma_p^{-1}  ) \mathfrak{k}_m & \text{ if } n =0 .
    \end{cases}$
    \item For every prime number $\ell \neq p$, one has 
    $\mathrm{Tr}_{\ell m p^n/mp^n} ( \mathfrak{k}_{\ell mp^{n}} )  = 
    \begin{cases}
\ell \mathfrak{k}_{mp^n} & \text{ if } \ell \mid m, \\
 - \sigma_\ell^{-1} \mathfrak{k}_{mp^n} & \text{ if } \ell \nmid m.
    \end{cases}
    $
\end{romanliste}
\item Write $\cD^{(\ell)}_{mp^n} \subseteq G_{m p^n}$ for the decomposition group of $\ell$. If any of the following hold:
    \begin{itemize}
        \item $a_\ell =2$, $\ell \nmid N$ and $\ell^2 \nmid m$,
        \item $a_\ell =1$ and $\ell \mid N$,
        \item $a_\ell =0$, $\ell \mid N$ and $\ell^2 \mid m$,
    \end{itemize}  
then $\nu^{(\ell)}_{mp^n}$ belongs to $I_{\cD^{(\ell)}_{mp^n}} \coloneqq \ker \{ \Z_p [G_{m p^n}] \to \Z_p [G_{m p^n} / \cD^{(\ell)}_{mp^n}] \}$. Furthermore,  one has the following congruence modulo $I_{\cD^{(\ell)}_{mp^n}}^2$,
\[
\nu^{(\ell)}_{mp^n} \equiv \begin{cases}
    0 & \text{ if } a_\ell = 2, \ell \nmid N  \text{ and } \ell^2 \nmid m, \\
     \ell^{- (\ord_\ell (m)-1)} ( 1 - \tilde \sigma_\ell) & \text{ if } a_\ell = 1 \text{ and } \ell \mid N, \\
     0 & \text{ if } a_\ell = 0, \ell \mid N  \text{ and } \ell^2 \mid m.
\end{cases} 
\]
\item Each $\nu_{mp^n}^{(\ell)}$ belongs to the ideal of $\Z_p [G_{mp^n}]$ that is generated by $\Eul_\ell (\tilde \sigma_\ell)$ and $\NN_{\mathcal{I}_{mp^n}^{(\ell)}}$, with $\cI^{(\ell)}_{mp^n} \subseteq \cD_{mp^n}^{(\ell)}$ the inertia subgroup at $\ell$. 
\end{liste}
\end{thm}

The proof of this result is given in \S\,\ref{proof local points main result section} after a number of preparations. To end this section, we comment on the hypothesis $a_p \not \equiv 1 \mod p$ that appears in claim (b) of Theorem \ref{local points main result}. 

\begin{rk} \label{anomalous primes rk}
Primes $p$ of good reduction for which $\widetilde{E} (\mathbb{F}_p)$ contains a point of order $p$ (or, equivalently, for which $a_p \coloneqq p + 1 - | \widetilde{E} (\mathbb{F}_p)|$ is congruent to $1 \mod p$) are called `anomalous' by Mazur in \cite{Mazur72}. 
Hasse's bound implies that one has $a_p = 1$ for any anomalous prime $p \geq 7$, and this combines with a result of Serre \cite{serre81} to show that the set of anomalous primes is of density zero.
 It is not known whether an elliptic curve can have infinitely many anomalous primes (although Mazur has conjectured this to be possible and it would follow from conjectures of Hardy--Littlewood, see \cite{Qin}, and Lang--Trotter \cite{LangTrotter76}). Given any set of prime numbers $\mathcal{P}$, one can however construct an elliptic curve $E / \Q$ such that every prime in $\mathcal{P}$ is anomalous for $E$ (see \cite[Lem.\@ 8.19]{Mazur72}).\\
If $E (\Q)$ contains a torsion point, then the set of anomalous primes for $E$ is either empty, consists of  a single element, or else is contained in $\{2, 3, 5\}$ (see \cite[Lem.\@ 8.18]{Mazur72}).
For example, if $E ( \Q)$ contains a point of prime-order $\ell \neq p$ and $E$ has good reduction at $p > 5$, then $a_p \not \equiv 1 \mod p$.
\end{rk}

\subsection{Otsuki's elements}

In this section we first recall important definitions and results from Otsuki's article \cite{Otsuki}, and then further develop certain aspects of Otsuki's theory. This will be crucial to the proof of Theorem \ref{local points main result}.

\subsubsection{Review of Otsuki theory}

For every natural number $a \in \N$, we will use the $\Q$-algebra homomorphism
\[
\hat \sigma_a \: \Q [X] \to \Q [X], \quad \hat \sigma_a (X) \coloneqq X^a.
\]

\begin{lem}[{\cite[Prop.~1.3]{Otsuki}}] \label{basic otsuki properties}
    The following claims are valid for every $m \in \N$ and prime number $\ell$.
    \begin{liste}
        \item $\Eul_\ell (\hat \sigma_\ell)$ defines an invertible $\Q$-algebra endomorphism of $\Q [X]/ (X^m - 1)$. In particular, there is a unique well-defined element $\Eul_\ell (\hat \sigma_\ell)^{-1} \in \mathrm{End}_{\textnormal{$\Q$-alg}} ( \Q [X] / (X^m - 1))$.
        \item If $\ell \nmid m$, then $\Eul_\ell (\hat \sigma_\ell)$ defines an invertible $\Q$-algebra endomorphism of $\Q [X]/ (\Phi_m)$ with $\Phi_m$ the $m$-th cyclotomic polynomial. In other words, $\Eul_\ell (\hat \sigma_\ell)^{-1}$ induces a well-defined element of $\mathrm{End}_\textnormal{$\Q$-alg} ( \Q [X] / (\Phi_m))$.
    \end{liste}
\end{lem}

\begin{proof}
    To prove claim (a), it suffices to show that $\Eul_\ell (\hat \sigma_\ell)$ defines an injective endomorphism of the finite-dimensional $\Q$-vector space $\Q [X] / (X^m - 1)$ and this is verified in \cite[Prop.\@ 1.3]{Otsuki}. As for claim (b), it follows from the fact that $\hat \sigma_\ell$ preserves the ideal $(\Phi_m)$ if $\ell \nmid m$ that $\Eul_\ell (\hat \sigma_\ell)$ does the same. Since we know $\Eul_\ell (\hat \sigma_\ell)$ to be injective by claim (a), this shows that $\Eul_\ell (\hat \sigma_\ell)$ restricts to an automorphism of the $\Q$-subvector space $(\Phi_m) / (X^m - 1)$ of $\Q [X] / (X^m - 1)$. As a consequence, $\Eul_\ell (\hat \sigma_\ell)^{-1}$ preserves $(\Phi_m) / (X^m - 1)$. This shows that $\Eul_\ell (\hat \sigma_\ell)$ and $\Eul_\ell (\hat \sigma_\ell)^{-1}$ both descend to elements of $\mathrm{End}_\textnormal{$\Q$-alg} ( \Q [X] / (\Phi_m))$, as required to prove claim (b).
\end{proof}

The endomorphism $\Eul_\ell (\hat \sigma_\ell)^{-1}$ introduced above is rather inexplicit but does satisfy a certain inductive relation that will be useful in computations.
To state this relation, we shall use certain constructions, themselves based on \cite[\S\,2.2]{Kurihara02}, that are made by Otsuki in \cite{Otsuki} and which we now recall. For a prime divisor $\ell$ of $m$ and integer $i\in \Z_{\geq 0}$ we inductively define elements $c_i^{(\ell)} \in \Z[1/\ell]$ as follows. Set $c_0^{(\ell)} \coloneqq 0, \ c_1^{(\ell)} \coloneqq 1$ and, for $i \geq 2$, 
\begin{equation}\label{e: c_i relations}
c_{i+1}^{(\ell)} \coloneqq \frac{a_\ell}{\ell}c_i^{(\ell)} - \frac{\bm{1}_N(\ell)}{\ell} c_{i-1}^{(\ell)} . 
\end{equation}
For $i\in \Z_{\geq 0}$ we also define a polynomial $\widetilde F_\ell^{(i)}(X) \in \Z[1/\ell][X]$ by
\[ \widetilde F_\ell^{(i)}(X) \coloneqq c_{i+1}^{(\ell)} - \frac{\bm{1}_N(\ell)}{\ell} c_i^{(\ell)} X  . \]
By induction on $j \in \N$, one then proves the key relation
\begin{equation} \label{Otsuki relation}
\mathrm{Eul}_\ell (\hat \sigma_\ell)^{-1}  = \sum_{i=0}^{j-1} c_{i+1}^{(l)} \hat \sigma_\ell^i + \widetilde{F}_\ell^{(j)} (\hat \sigma_\ell) \mathrm{Eul}_\ell (\hat \sigma_\ell)^{-1} \hat \sigma_\ell^j ,
\end{equation}
(see \cite[Lem.\@ 2.6]{Otsuki} for details).

\subsubsection{The definition of Otsuki's elements}
\label{section definition nu lambda}

In this section we discuss the canonical local elements defined by Otsuki in \cite[Def.\@ 2.4]{Otsuki}.

\begin{definition} \label{otsuki element}
    For every natural number $m$, we define
    \[
    x_m \coloneqq \Big( \big ( {\prod}_{\ell \mid  mp} \Eul_\ell (\hat \sigma_\ell )^{-1} \big) (X) \Big) (\zeta_m)
    \in F_m.
    \]
\end{definition}

\begin{remark}\label{r: difference in x_m}
    Our definition of $x_m$ differs slightly from that of Otsuki in \cite{Otsuki}. To be precise, if we write $x_m^{\text{Otsuki}}$ for the element denoted as $x'_m$ in \cite[Def.\@ 2.4]{Otsuki}, then
    \[ x_m = \begin{cases}
        x_m^{\text{Otsuki}} & \text{ if } p \mid m, \\
        \Eul_p (\sigma_p )^{-1} x_m^{\text{Otsuki}} & \text{ if } p \nmid m.
    \end{cases} \]
\end{remark}

In the remainder of this section, we will make the element $x_m$ more explicit.
To do this, it is convenient to introduce some further notation. For every prime number $\ell$ and integers $0 \leq j \leq n$, we set $e_{n, 0}^{(\ell)} \coloneqq \ell^{-( n - 1)} \NN_{F_{\ell^n} / \Q}$ and 
$e^{(\ell)}_{n, j} \coloneqq \ell^{ -(n- j)} \NN_{F_{\ell^n} / F_{\ell^{j}}}$. If $m' \in \N$ is coprime with $\ell$, then one has
\begin{equation} \label{cyclotomic relation}
\zeta_{m^\prime \ell^j} = \omega^{(\ell)}_{n, j} \cdot \sum_{i = 0}^n \zeta_{m^\prime \ell^i}
\quad \text{ with } \quad 
\omega^{(\ell)}_{n, j} \coloneqq 
\begin{cases}
(e^{(\ell)}_{n, j} - e^{(\ell)}_{n, j - 1})  \quad & \text{ if } j \geq 2, \\
e_{n,1}^{(\ell)}  & \text{ if } j = 1,\\
    - \tilde \sigma_\ell e_{n,0}^{(\ell)}  & \text{ if } j = 0.
\end{cases}
\end{equation}
Here $\omega^{(\ell)}_{n, j}$ is considered to be an element of $\Z_p [G_{m^\prime \ell^n}]$ via the splitting $G_{m^\prime \ell^n} \cong G_{m^\prime} \times G_{\ell^n}$ induced by the relevant restriction maps.
We may then define an element of $\Z [1 / \ell] [G_{\ell^n}] [X]$ by 
\[
\lambda^{(\ell)}_{n} (X) \coloneqq 
\mathrm{Eul}_\ell (X) \Big ( \sum_{i=0}^{n-1} c_{i+1}^{(l)}  \omega^{(\ell)}_{n, n - i}
\Big) - \widetilde{F}_\ell^{(n)} (X) X e^{(\ell)}_{n,0}. 
\]

\begin{lem} \label{Otsuki calculations lemma 1}
For every prime number $\ell \neq p$ and natural number $n$ 
the polynomial $\lambda^{(\ell)}_{n} (X)$ has the property that, for every $m^\prime \in \N$ with $\ell \nmid m^\prime$, one has the following equality in $F_{\ell^n m'}$.
\[
\big( \Eul_\ell (\hat \sigma_\ell)^{-1} (X) \big) (\zeta_{m^\prime \ell^n}) = \Eul_\ell (\tilde \sigma_\ell)^{-1} \cdot \lambda_{n}^{(\ell)} (\tilde \sigma_\ell) \cdot \sum_{i = 0}^n \zeta_{m^\prime \ell^i}.
\]
\end{lem}

\begin{proof}
We first make the useful observation that one has 
\begin{equation} \label{useful observation 1}
\big( (\widetilde{F}_\ell^{(n)} (\hat \sigma_\ell) \mathrm{Eul}_\ell (\hat \sigma_\ell)^{-1} \hat \sigma_\ell^n)(X) \big) (\zeta_{m^\prime \ell^n}) = \widetilde{F}_\ell^{(n)} (\sigma_\ell) \mathrm{Eul}_\ell (\sigma_\ell)^{-1} \cdot \zeta_{m^\prime}.
\end{equation}
To justify this, we note that $\hat \sigma_\ell^n (X) = X^{\ell^n}$ belongs to the image of the map $\Q [X] / (X^{m'} - 1) \to \Q [X] / (X^{\ell^n m'} - 1)$ that is defined by sending $X \mapsto X^{\ell^n}$. From the commutative diagram
\begin{cdiagram}
    \Q [X] / (X^{m'} - 1) \arrow{r}{X \mapsto X^{\ell^n}} \arrow{d}{X \mapsto \zeta_{m'}} & \Q [X] / (X^{\ell^n m'} - 1) \arrow{d}{X \mapsto \zeta_{\ell^n m'}} \\ 
    \Q [\zeta_{m'}] \arrow{r}{\zeta_{m'} \mapsto \zeta_{m'}} & \Q [\zeta_{\ell^n m'}]
\end{cdiagram}%
we therefore see that is suffices to compute $(\widetilde{F}_\ell^{(n)} (\hat \sigma_\ell) \mathrm{Eul}_\ell (\hat \sigma_\ell)^{-1}) (X)$ evaluated at $X = \zeta_{m'}$. By Lemma \ref{basic otsuki properties}\,(b) the endomorphism $\mathrm{Eul}_\ell (\hat \sigma_\ell)^{-1}$ of $\Q [X] / (X^{m'} - 1)$ descends to $\Q [X] / (\Phi_{m'})$ (and corresponds with multiplication by $\Eul_\ell (\sigma_\ell)^{-1}$ under the 
isomorphism $\Q [X] / (\Phi_{m'}) \cong \Q [\zeta_{m'}]$ that sends $X \mapsto \zeta_{m'}$), so this computation can be done in $\Q [\zeta_{m'}]$ and leads to the claimed equality (\ref{useful observation 1}).\\
Now, (\ref{useful observation 1}) combines with the relation (\ref{Otsuki relation}) to imply that one has
\begin{align*}
\big( \Eul_\ell (\hat \sigma_\ell)^{-1} (X) \big) (\zeta_{m^\prime \ell^n})
& = 
 \big( \sum_{i=0}^{n-1} c_{i+1}^{(l)} \hat \sigma_\ell^i (X)  + (\widetilde{F}_\ell^{(n)} (\hat \sigma_\ell) \mathrm{Eul}_\ell (\hat \sigma_\ell)^{-1} \hat \sigma_\ell^n)(X) \big) (\zeta_{m^\prime \ell^n}) \\
& = \big( \sum_{i=0}^{n-1} c_{i+1}^{(l)}\zeta_{m^\prime \ell^{n - i}} \big) + \big( \widetilde{F}_\ell^{(n)} (\tilde \sigma_\ell) \mathrm{Eul}_\ell (\tilde \sigma_\ell)^{-1} \zeta_{m^\prime}\big) \\
& = \big ( \sum_{i=0}^{n-1} c_{i+1}^{(l)} \omega^{(\ell)}_{n, n - i} + \widetilde{F}_\ell^{(n)} (\tilde \sigma_\ell) \mathrm{Eul}_\ell ( \tilde \sigma_\ell)^{-1} \omega_{n, 0}^{(\ell)} \big) \cdot \sum_{i = 0}^n \zeta_{m^\prime \ell^i} \\
& = \Eul_\ell (\tilde \sigma_\ell )^{-1} \cdot \lambda^{(\ell)}_{n} (\tilde \sigma_\ell) \cdot \sum_{i = 0}^n \zeta_{m^\prime \ell^i},
\end{align*}
as claimed.
\end{proof}

For every natural number $m \in \N$ and prime divisor $\ell \neq p$ of $m$, we define
\[
\nu_{m}^{(\ell)} \coloneqq \lambda_{\ord_\ell (m)}^{(l)} (\tilde \sigma_\ell) \in \Z_p [G_{m}].
\]
The following result, in which we write $m_0 \coloneqq \prod_{\ell \mid m} \ell$ for the `squarefree radical' of an integer $m$, describes the contribution to the definition of $x_m$ of the factors $\Eul_\ell (\hat \sigma_\ell)^{-1}$ for $\ell \neq p$ in terms of the elements $\nu_m^{(\ell)}$. 

\begin{prop} \label{Prop Otsuki l-adic euler factors}
Let $m$ be a natural number coprime to $p$. For every integer $n \geq 0$ one has
\[
x_{m p^n} = \big( \prod_{\ell \mid m} \Eul_\ell (\tilde \sigma_\ell)^{-1} \cdot \nu^{(\ell)}_{mp^n} \big) \cdot \sum_{m_0 \mid d \mid m} (\Eul_p (\hat \sigma_p)^{-1} (X)) ( \zeta_{d p^n}).
\]
\end{prop}

\begin{proof}
    Factorise $m$ as $\prod_{i = 1}^s \ell_i^{n_i}$ and set $\ell_{s + 1} \coloneqq p$ and $n_{s + 1} \coloneqq n$ so that $mp^n = \prod_{i = 1}^{s + 1} \ell_i^{n_i}$. We shall prove by induction on $0 \leq j \leq s$ that one has
\begin{equation} \label{induction claim}
x_{m p^n} = \big( \prod_{i = 1}^j \Eul_{\ell_i} (\tilde \sigma_{\ell_i})^{-1} \cdot \nu^{(\ell_i)}_{mp^n} \big) \cdot 
\Big ( \sum_{m_{j, 0} \mid d \mid m_j} \big( \prod_{i = j + 1}^{s + 1} \Eul_{\ell_i} (\hat \sigma_{\ell_i})^{-1} \big) ( X) \Big) ( \zeta_{d m p^n / m_j})
\end{equation}
with $m_j \coloneqq \prod_{i = 1}^j \ell_i^{n_i}$.
Taking $j = s$, this then implies claim (a). If $j = 0$, the claim follows directly from the definition of $x_{m p^n}$.
For the inductive step we assume that $j > 0$ and that (\ref{induction claim}) is valid for $j - 1$. We introduce the notation
\[
\tilde c_{q}^{(\ell_i)} \coloneqq \begin{cases}
    c_q^{(\ell_i)} \quad & \text{ if } 0 \leq q \leq n_i - 1, \\ 
    \widetilde F_\ell^{(n_i)} (\hat \sigma_\ell) \Eul_\ell (\hat \sigma_\ell)^{-1} & \text{ if } q = n_i,
\end{cases}
\]
so that the relation (\ref{Otsuki relation}) can be compactly written as 
\[
\Eul_{\ell_i} (\hat \sigma_{\ell_i})^{-1} = \sum_{q = 0}^{n_i} \tilde c_{q + 1}^{(\ell_i)} \hat \sigma_{\ell_i}^q.
\]
Define $\mathscr{A} (j) \coloneqq  \prod_{i = j + 1}^{s + 1} \{0, \dots, n_i\}$
and, for every $a = (a_{j+1}, \dots, a_s) \in \mathscr{A} (j)$, write $|a| \coloneqq \prod_{i = j + 1}^{s + 1} a_i$. Then we have
\begin{align} \nonumber
 \big( \prod_{i = j}^{s + 1} \Eul_{\ell_i} (\hat \sigma_{\ell_i})^{-1} \big) ( X) 
& = \Big( \Eul_{\ell_j} (\hat \sigma_{\ell_j})^{-1} \cdot \prod_{i = j + 1}^{s + 1} \big( \sum_{q = 0}^{n_i} \tilde c_{q + 1}^{(\ell_i)} \hat \sigma_{\ell_i}^q \big)  \Big) (X) \\
& = \sum_{a \in \mathscr{A} (j)} ( \prod_{i = j + 1}^{s + 1} \tilde c^{(\ell_i)}_{a_i} ) \cdot \Eul_{\ell_j} (\hat \sigma_{\ell_j})^{-1} ( X^{|a|})
.
\label{calculating everything out}
\end{align}
Fix a divisor $d'$ of $m_{j - 1}$ and define $d'_j \coloneqq d' m p^n / m_{j - 1}$.
If we set
\[
C_{q}^{(\ell_i)} \coloneqq \begin{cases}
    c_q^{(\ell_i)} \quad & \text{ if } 0 \leq q \leq n_i - 1, \\ 
    \widetilde F_\ell^{(n_i)} (\sigma_\ell) \Eul_\ell (\sigma_\ell)^{-1} & \text{ if } q = n_i,
\end{cases}
\]
then (\ref{calculating everything out}) combines with (\ref{useful observation 1}) to imply that
\begin{align*}
\big( \prod_{i = j}^{s + 1} \Eul_{\ell_i} (\hat \sigma_{\ell_i})^{-1} \big) ( X) ( \zeta_{d'_j})
  = 
  \sum_{a \in \mathscr{A} (j)} ( \prod_{i = j + 1}^{s + 1} C^{(\ell_i)}_{a_i} ) \cdot \big( \Eul_{\ell_j} (\hat \sigma_{\ell_j})^{-1} ( X) \big) (\zeta_{d'_j / |a|}).
\end{align*}
Let $d$ be a divisor of $m_j$ and set $d_j \coloneqq dmp^n / m_j$, then the same computation also shows that one has
\begin{equation} \label{on the other hand}
\big( \prod_{i = j + 1}^{s + 1} \Eul_{\ell_i} (\hat \sigma_{\ell_i})^{-1} \big) ( X) ( \zeta_{d_j})
= 
\sum_{a \in \mathscr{A} (j)} ( \prod_{i = j + 1}^{s + 1} C^{(\ell_i)}_{a_i} )\cdot \zeta_{d_j / |a|}.
\end{equation}
Next we recall that
Lemma \ref{Otsuki calculations lemma 1} proves an equality
\begin{equation} \label{input from Otsuki calculations lemma 1}
\big( \Eul_{\ell_j} (\hat \sigma_{\ell_j})^{-1} ( X) \big) (\zeta_{d'_j / |a|}) = \Eul_{\ell_j} (\tilde \sigma_{\ell_j})^{-1} \cdot \lambda_{n_j}(\tilde \sigma_\ell) \cdot \sum_{q = 0}^{n_j - 1} \zeta_{ d'_j / ( \ell^q |a|)}.
\end{equation}
Now, every $m_{j, 0} \mid d \mid m_j$ is of the form $d' \ell^{n_j - q}$ for some $m_{j - 1, 0} \mid d' \mid m_{j - 1}$ and $q \in \{0, \dots, n_j - 1\}$ so that $d_j$ is of the form $d'_j / \ell^q$.
As a consequence, we may calculate that
\begin{align*}
& \sum_{m_{j - 1,0} \mid d' \mid m_{j - 1}}  \big( \prod_{i = j}^{s + 1} \Eul_{\ell_i} (\hat \sigma_{\ell_i})^{-1} \big) ( X) (\zeta_{d'_j}) \\
& \qquad \stackrel{\ (\ref{calculating everything out}) \ }{=} 
\sum_{m_{j - 1,0} \mid d' \mid m_{j - 1}}  \sum_{a \in \mathscr{A} (j)} ( \prod_{i = j + 1}^{s + 1} \tilde c^{(\ell_i)}_{a_i} ) \cdot \big(\Eul_{\ell_j} (\hat \sigma_{\ell_j})^{-1} ( X)\big) (\zeta_{d'_j / |a|})
\\
    & \qquad \stackrel{ \ref{Otsuki calculations lemma 1} }{=} \sum_{m_{j - 1,0} \mid d' \mid m_{j - 1}} \sum_{a \in \mathscr{A} (j)} ( \prod_{i = j + 1}^{s + 1} C^{(\ell_i)}_{a_i} ) \cdot \big( \Eul_{\ell_j} (\tilde \sigma_{\ell_j})^{-1} \cdot \lambda^{(\ell_j)}_{n_j}(\tilde \sigma_\ell) \cdot \sum_{q = 0}^{n_j - 1} \zeta_{ d'_j / ( \ell^q |a|)} \big) \\
    & \qquad \stackrel{\ (\ref{input from Otsuki calculations lemma 1}) \ }{=}  \Eul_{\ell_j} (\tilde \sigma_{\ell_j})^{-1} \cdot \lambda^{(\ell_j)}_{n_j}(\tilde \sigma_\ell) \cdot 
    \sum_{m_{j,0} \mid d \mid m_{j}} \sum_{a \in \mathscr{A} (j)} ( \prod_{i = j + 1}^{s + 1} C^{(\ell_i)}_{a_i} ) \cdot \zeta_{d_j / |a|} \\ 
    & \qquad \stackrel{\ (\ref{on the other hand}) \ }{=}  \Eul_{\ell_j} (\tilde \sigma_{\ell_j})^{-1} \cdot \lambda^{(\ell_j)}_{n_j}(\tilde \sigma_\ell) \cdot 
   \Big( \sum_{m_{j,0} \mid d \mid m_{j}}
\big( \prod_{i = j + 1}^{s + 1} \Eul_{\ell_i} (\hat \sigma_{\ell_i})^{-1} \big) ( X)\Big) ( \zeta_{d_j}).
\end{align*}
By the induction hypothesis, this shows the claimed equality (\ref{induction claim}) and therefore concludes the proof of the proposition.
\end{proof}

\subsubsection{Congruences modulo augmentation ideals}

In this section we investigate when the elements $\nu_{mp^n}^{(\ell)}$ belong to a relevant augmentation ideal and also compute its class modulo the square of the augmentation ideal. In particular, the  second part of Theorem \ref{local points main result}\,(d) will be a consequence of these calculations.\\
We recall that we have fixed an \textit{odd} prime $p$, which will be important in the proof of the following result.

\begin{prop}\label{l: lambda mod I^2}
    Let $m > 1$ be a natural number, let $\ell \neq p$ be a prime number, and set $n \coloneqq \ord_\ell (m)$. Then the following claims are valid.
    \begin{liste}
    \item 
Write $\cD^{(\ell)}_{m} \subseteq G_{m}$ for the decomposition group of $\ell$. The element $\nu^{(\ell)}_{m}$ belongs to $I_{\cD^{(\ell)}_{m}} \coloneqq \ker \{ \Z_p [G_{m}] \to \Z_p [G_{m} / \cD^{(\ell)}_{m}] \}$
if any of the following hold:
    \begin{itemize}
        \item $a_{\ell} =2$, $\ell \nmid N$ and $\ell^2 \nmid m$,
        \item $a_{\ell} =1$ and $\ell \mid N$,
        \item $a_{\ell} =0$, $\ell \mid N$ and $\ell^2 \mid m$.
    \end{itemize}   
    \item 
    \label{p: congruence of nu mod I^2}
    One has $\nu^{(\ell)}_{m} \equiv c_{n}^{(\ell)} ( 1 - a_{\ell}\tilde \sigma_{\ell} )  + \frac{\bm{1}_N(\ell)}{\ell} ( \ell c_{n}^{(\ell)} \tilde \sigma_{\ell}^2  + c_{n-1}^{(\ell)}\tilde\sigma_{\ell}(\ell - 1 ) ) \pmod{I_{D^{(\ell)}_{m}}^2}$.
    \end{liste}
\end{prop}

\begin{proof}
It is convenient to set $m' \coloneqq m \ell^{-n}$ in this proof.
To prove claim (a), we first note that, since $\ell$ is totally ramified in $F_{m^\prime \ell^n} / F_{m^\prime \ell^j}$, one has that $\gal{F_{m^\prime \ell^n}}{F_{m^\prime \ell^j}}$ is contained in $\cD^{(\ell)}_{m}$ for every $j \in \{0, \dots, n \}$. As a consequence, we have
\begin{equation} \label{e congruence}
e^{(\ell)}_{n,j} \equiv 1 \mod I_{\cD^{(\ell)}_{m}}
\quad \text{ if } j \in \{1, \dots, n \}.
\end{equation}
In particular, $e^{(\ell)}_{n, n - i} - e^{(\ell)}_{n, n - i - 1}$ belongs to $I_{\cD^{(\ell)}_{m}}$ for $i \in \{ 0, \dots , n-2\}$, and therefore
\begin{align}
    \nu_{m}^{(\ell)} = \lambda^{(\ell)}_{n} (\tilde \sigma_\ell)  &\equiv \mathrm{Eul}_\ell (\tilde \sigma_\ell)c_n^{(\ell)} e_{n, 1}^{(\ell)} -  \widetilde{F}_\ell^{(n)} (\tilde \sigma_\ell) \tilde \sigma_\ell e_{n, 0}^{(\ell)} \pmod{I_{\cD^{(\ell)}_{m}}} \notag \\
    \label{e: reduction of lambda mod I} 
    & \stackrel{(\ref{e congruence})}{\equiv} \mathrm{Eul}_\ell (\tilde \sigma_\ell)c_n^{(\ell)}  -  \widetilde{F}_\ell^{(n)} (\tilde \sigma_\ell) (\ell -1 ) \pmod{I_{\cD^{(\ell)}_{m}}} .
\end{align}
For ease of notation we write $\overline{\lambda}^{(\ell)}_{m^\prime, i}$ for the quantity on the right hand side of (\ref{e: reduction of lambda mod I}) with $n$ replaced by $i$. Note that, if $i=n$, then (\ref{e: reduction of lambda mod I}) gives that $\lambda^{(\ell)}_{n} (\tilde \sigma_\ell) \equiv \overline{\lambda}^{(\ell)}_{m^\prime, n} \pmod{I_{\cD^{(\ell)}_{m}}}$. We work inductively on $i$ and begin by considering $\overline{\lambda}^{(\ell)}_{m^\prime, 1}$. Since, $c_0^{(\ell)} \coloneqq 0$ and $c_1^{(\ell)} \coloneqq 1$ we can calculate
\begin{align*}
    \overline{\lambda}^{(\ell)}_{m^\prime, 1} &= \mathrm{Eul}_\ell (\tilde \sigma_\ell)  -  ( \frac{a_\ell}{\ell} - \frac{\bm{1}_N(\ell)}{\ell}\tilde \sigma_\ell ) ( \ell -1 )  \\
    &= \frac{1}{\ell} ( \ell - a_\ell \tilde \sigma_\ell + \bm{1}_N (\ell) \tilde \sigma_\ell^2)-  ( \frac{a_\ell}{\ell} - \frac{\bm{1}_N(\ell)}{\ell}\tilde \sigma_\ell ) ( \ell -1 )  \\
    & \equiv \frac{1}{\ell} ( \ell - a_\ell  + \bm{1}_N (\ell) ) - ( \frac{a_\ell}{\ell} - \frac{\bm{1}_N(\ell)}{\ell} )(\ell - 1) \pmod{I_{\cD^{(\ell)}_{m}}} \\
    & \equiv 1 - a_\ell + \bm{1}_N (\ell) \pmod{I_{\cD^{(\ell)}_{m}}} .
\end{align*}
It follows that $\overline{\lambda}^{(\ell)}_{m^\prime, 1} \in I_{\cD^{(\ell)}_{m}}$ if either of the following hold:
\begin{itemize}
        \item $a_\ell =2$, $\ell \nmid N$, 
        \item $a_\ell =1$ and $\ell \mid N$.
\end{itemize}
If $n=1$ then $\lambda^{(\ell)}_{1} (\tilde \sigma_\ell) \equiv \overline{\lambda}^{(\ell)}_{m^\prime, 1} \pmod{I_{\cD^{(\ell)}_{m}}}$, therefore the first two points in claim (a) hold in this case. We now consider $\overline{\lambda}^{(\ell)}_{m^\prime, 2}$. Using (\ref{e: c_i relations}) we calculate
\begin{equation}\label{e: cong of lambda when n=2}
    \overline{\lambda}^{(\ell)}_{m^\prime, 2} = \mathrm{Eul}_\ell (\tilde \sigma_\ell)\frac{a_\ell}{\ell} - ( \frac{a_\ell}{\ell} \frac{a_\ell}{\ell} - \frac{\bm{1}_N(\ell)}{\ell}  - \frac{\bm{1}_N(\ell)}{\ell}\frac{a_\ell}{\ell} \tilde \sigma_\ell )(\ell - 1)
    = \frac{a_\ell}{\ell}  \overline{\lambda}^{(\ell)}_{m^\prime, 1} + \frac{\bm{1}_N(\ell)}{\ell}(\ell -1) .
\end{equation}
Therefore, if either
\begin{itemize}
    \item $\overline{\lambda}^{(\ell)}_{m^\prime, 1} \in I_{\cD^{(\ell)}_{m}}$ and $\ell \mid N$, or
    \item $a_\ell =0$ and $\ell \mid N$.
\end{itemize}
Then $\overline{\lambda}^{(\ell)}_{m^\prime, 2} \in I_{\cD^{(\ell)}_{m}}$. By the calculations done for $\overline{\lambda}^{(\ell)}_{m^\prime, 1}$ we have that claim (a) holds if $n \in \{ 1,2\}$. Suppose $n \geq 3$ and let $i=3, \dots, n$. By applying (\ref{e: c_i relations}) to $\overline{\lambda}^{(\ell)}_{m^\prime, i}$ one calculates
\[ \overline{\lambda}^{(\ell)}_{m^\prime, i} = \frac{a_\ell}{\ell}\overline{\lambda}^{(\ell)}_{m^\prime, i-1} -  \frac{\bm{1}_N(\ell)}{\ell} \overline{\lambda}^{(\ell)}_{m^\prime, i-2}. \]
As with $\overline{\lambda}^{(\ell)}_{m^\prime, 2}$, we observe that if either
\begin{itemize}
    \item $\overline{\lambda}^{(\ell)}_{m^\prime, i-1} \in I_{\cD^{(\ell)}_{m}}$ and $\ell \mid N$, or
    \item $a_\ell =0$ and $\ell \mid N$,
\end{itemize}
then $\overline{\lambda}^{(\ell)}_{m^\prime, i} \in I_{\cD^{(\ell)}_{m}}$. Considering this argument inductively completes the proof of (a).\\
    To prove claim (b), let $j \in \{1, \dots , n\}$ and calculate
    \begin{align}
        e_{n,j}^{(\ell)} &= \frac{1}{\ell^{n-j}} \NN_{F_{\ell^n}/F_{\ell^j}} = \frac{1}{\ell^{n-j}} \big ( \ell^{n-j} + {\sum}_{\sigma \in \gal{F_{\ell^n}}{F_{\ell^j}}} ( \sigma -1 ) \big ) . \label{e: e_j calc}
    \end{align}
Under the isomorphism $I_{\cD^{(\ell)}_{m}} / I_{\cD^{(\ell)}_{m}}^2 \cong \cD^{(\ell)}_{m} \otimes \Z_p$ induced by $\sigma -1  \mapsto \sigma$, the sum in (\ref{e: e_j calc}) is mapped to $\prod_{\sigma \in \gal{F_{\ell^n}}{F_{\ell^j}}} \sigma$.
If $\sigma$ has odd order, then $\sigma \neq \sigma^{-1}$ and so both (element and its inverse) appear in the above product. It follows that only the (unique) element of order 2 does not cancel. Since $p$ is odd, this element is trivial in $\cD^{(\ell)}_{m} \otimes \Z_p$. Therefore,
\[ e_{n,j}^{(\ell)} \equiv 1 \pmod{I_{\cD^{(\ell)}_{m}}^2}. \]
Similarly, we have $ \tilde \sigma_\ell e_{n,0}^{(\ell)} \equiv \tilde \sigma_\ell (\ell - 1) \pmod{I_{\cD^{(\ell)}_{m}}^2}$. It therefore follows from the definition that
\begin{align*}
    \nu_m^{(\ell)} = \lambda_{n}^{(\ell)} (\tilde \sigma_\ell) &\equiv \Eul_\ell(\tilde \sigma_\ell ) c_n^{(\ell)} - \tilde F_{\ell}^{(n)}(\tilde \sigma_\ell ) \tilde \sigma_\ell (\ell -1) \pmod{I_{\cD^{(\ell)}_{m}}^2} \\
    &\equiv \frac{1}{\ell}( \ell - a_\ell \tilde \sigma_\ell + \bm{1}_N(\ell) \tilde \sigma_\ell^2 ) c_n^{(\ell)} - ( c_{n+1}^{(\ell) }- \frac{\bm{1}_N(\ell)}{\ell}c_{n}^{(\ell)} \tilde \sigma_\ell ) \tilde \sigma_\ell (\ell -1) \pmod{I_{\cD^{(\ell)}_{m}}^2} 
\end{align*}
The result follows from rearranging and the fact that $c_{n+1}^{(\ell)} = \frac{a_\ell}{\ell}c_n^{(\ell)} - \frac{\bm{1}_N(\ell)}{\ell}c_{n-1}^{(\ell)}$.
\end{proof}

\subsubsection{Norm relations}

We end this section by proving two norm relations for the elements $\nu_{mp^{n}}^{(\ell)}$ which are required to establish the claimed congruence of Mazur--Tate elements in Theorem \ref{mazur--tate main result 2}.

For natural numbers $a,b \in \N$ with $a \mid b$ we denote by $\pi_{b/a} \: \Z_p[ G_b] \rightarrow \Z_p [ G_a ]$ the natural map induced by the restriction map $G_b \rightarrow G_a$.

\begin{lem} \label{nus are compatible}
    Let $m > 1$ be a natural number and let $\ell \neq p$ be a prime divisor of $m$. If $d \in \N$ is a divisor of $m$ with $\ord_\ell (m) = \ord_\ell (d)$, then one has
    \[
    \pi_{m / d} (\nu_m^{(\ell)}) = \nu_{d}^{(\ell)}.
    \]
\end{lem}

\begin{proof}
We set $n \coloneqq \ord_\ell (m)$ and $m' \coloneqq m \ell^{-n}$ so that $m = m' \ell^n$. 
    For clarity, let us write $\tilde \sigma_\ell^{(m)}$ for the element previously denoted $\tilde \sigma_\ell$ if it is regarded as an element of $G_m$. That is, $\tilde \sigma_\ell^{(m)}$ is the image of $(\sigma_\ell, \id)$ under the isomorphism $G_{m'} \times G_{\ell^n} \cong G_{m}$. Similarly, the element $\tilde \sigma_\ell^{(d)}$ is the image of $(\sigma_\ell, \id)$ under $G_{d'} \times G_{\ell} \cong G_d$ with $d' \coloneqq d \ell^{-n}$. The commutative diagram
    \begin{cdiagram}
        G_{m^\prime} \arrow{d}{\pi_{m^\prime/d^\prime}} \arrow[hookrightarrow]{r} & G_{m^\prime} \times G_{\ell^n}  \arrow{r}{\cong} &  G_{m} \arrow{d}{\pi_{m /d }}  \\ 
         G_{d'} \arrow[hookrightarrow]{r} & G_{d^\prime} \times G_{\ell^n}  \arrow{r}{\cong} &  G_{d} , 
    \end{cdiagram}%
    then shows that we have $\pi_{m / d} ( \tilde \sigma_\ell^{(m)}) = \tilde \sigma_\ell^{(d)}$. Using this, we may compute that
    \[
    \pi_{m / d} (\nu_m^{(\ell)}) = \pi_{m / d} ( \lambda_{n}^{(\ell)} (\tilde \sigma_\ell^{(m)})) = 
    \lambda_{n}^{(\ell)} ( \pi_{m / d} ( \tilde \sigma_\ell^{(m)})) = \lambda_{n}^{(\ell)} ( \tilde \sigma_\ell^{(d)}) = \nu_d^{(\ell)},
    \]
    as claimed. 
\end{proof}

\subsection{A construction of local points}

In this section we use the theory of formal groups to construct certain local points, and this will allow us to prove Theorem \ref{local points main result} in \S\,\ref{proof local points main result section}.

\subsubsection{Review of Honda theory}

For the convenience of the reader, we review relevant aspects of Honda's article \cite{Honda70} (see also \cite[\S\,8.1]{Kobayashi03}).\\
Suppose $\cQ$ is a finite extension of $\Q_p$ with ring of integers $\cR$, maximal ideal $\cM = (\pi)$, and an automorphism $\phi \: \cQ \to \cQ$ that satisfies $\phi (a) \equiv a^p \mod \cM$ for every $a \in \cR$.
We then define the `Frobenius operator'
\[
\widehat \phi \: \cQ \llbracket X \rrbracket \to \cQ \llbracket X \rrbracket,
\quad \sum_{i = 0}^\infty \beta_i X^i \mapsto \sum_{i = 0}^\infty \phi (\beta_i) X^{ip}.
\]

Given $f, g \in \cQ \llbracket X \rrbracket$ and an ideal $\a$ of $\cR \llbracket X \rrbracket$, we write $f \equiv g \mod \a$ if $f - g \in \a$.

\begin{lem} \label{useful congruence}
    Suppose $f = \sum_{i = 0}^\infty \beta_i X^i$ is a power series in $\cQ \llbracket X \rrbracket$ with the property that $i \beta_i \in \cR$ for every $i \geq 0$, and that the power series $g \coloneqq \sum_{i = 0}^\infty \phi (\beta_i) ((X + 1)^{ip} - 1)$ exists in $\cQ \llbracket X \rrbracket$. 
    Then one has 
    \[
    \widehat \phi ( f) \equiv g \mod \pi \cR \llbracket X \rrbracket.
    \]
\end{lem}

\begin{proof}
    This follows from the congruence $i^{-1} (X + \pi Y)^i \equiv i^{-1} X^i \mod \pi$ in \cite[Lem.\@ 2.1]{Honda70}.
\end{proof}

\begin{definition}
Let $u \in \cR [X]$ be a polynomial with $u (0) = \pi$. A power series $f = \sum_{i = 0}^\infty \beta_i X^i$ in $\cQ \llbracket X \rrbracket$ is of `Honda type $u$' if $\beta_0 = 0$, $\beta_1 = 1$, and
\[
( u ( \widehat \phi)) (f) \equiv 0 \mod \pi \cR \llbracket X \rrbracket.
\]
\end{definition}

\begin{thm}[Honda] \label{Honda's thm}
The following claims are valid.
\begin{liste}
    \item Suppose $f \in \cQ \llbracket X \rrbracket$ is a power series of type $u \in \cR [X]$. Then there exists a one-dimensional commutative formal group $\cF$ over $\cR$ such that $\log_\cF = f$.
    \item Suppose that $\cF$ and $\mathcal{G}$ are two one-dimensional commutative formal groups such that $\log_\cF$ and $\log_{\mathcal{G}}$ have the same type. Then $\exp_\mathcal{G} \circ \log_\cF$ belongs to $\cR \llbracket X \rrbracket$ and defines an isomorphism $\cF \stackrel{\simeq}{\to} \mathcal{G}$ over $\cR$.
\end{liste}
\end{thm}

\begin{proof}
    Claim (a) is proved in \cite[Thm.\@ 2]{Honda70}.
    To prove claim (b), we note that $\cF (X, Y) = \exp_\cF ( \log_\cF (X) + \log_\cF (Y))$ and $\mathcal{G} (X, Y) = \exp_\mathcal{G} ( \log_\mathcal{G} (X) + \log_\mathcal{G} (Y))$ by \cite[Thm.\@ 1]{Honda70}.
    If these have the same type, then they are isomorphic over $\cR$ by \cite[Thm.\@ 2]{Honda70}. 
    By \cite[Prop.\@ 1.6]{Honda70} the latter holds if and only if  $\exp_\mathcal{G} \circ \log_\cF$ belongs to $\cR \llbracket X \rrbracket$, as required to prove claim (b).
\end{proof}

We conclude our review of Honda theory with the following result.

\begin{prop} \label{elliptic curve Honda type}
Let $\widehat E$ denote the formal group of the minimal model (over $\Z_p$) of an elliptic curve $E$ defined over $\Q$. Then $\log_{\widehat{E}}$ is of Honda type $p - a_p X + \bm{1}_N (p) X^2$.
\end{prop}

\begin{proof}
    This is proved in \cite[Thm.\@ 9, (6.6)]{Honda70} in the case of good reduction and in \cite[Thm.~5]{Honda68} for the case of bad reduction.
\end{proof}

\subsubsection{An `Artin--Hasse type' exponential}

In this section we will apply the results from Honda theory reviewed in the last section to a certain explicit power series. To define this power series, we fix $m\in \N$ with $p \nmid m$ and also let $q$ denote an auxiliary prime number coprime to $m (p - 1) p$. Take $\cQ$ to be the unramified extension of $\Q_p$ obtained as the completion of $F_{mq}$ at a $p$-adic place, $\phi \coloneqq \sigma_p$ its Frobenius automorphism, and write $(p) \subseteq \cR$ for the maximal ideal and ring of integers of $\cQ$, respectively.\\ 
Fix an embedding $\overline{\Q} \hookrightarrow \overline{\Q_p}$ that allows to view $\zeta_a$, for every $a \in \N$, as an element of $\overline{\Q_p}$.\\ 
In the following, we define, for an element $\delta \in \Z_p$, the power series
\[
(1 + X)^\delta \coloneqq \exp ( \delta \log (1 + X)) = \sum_{i = 0}^\infty \binom{\delta}{i} X^i
\]
with $\binom{\delta}{i} \coloneqq i!^{-1} \prod_{j = 0}^{i - 1} (\delta - i) \in \Z_p$.\\ 
Given  $F = \sum_{i = 0}^\infty \beta_i X^i \in \cQ \llbracket X \rrbracket$ and $\sigma \in \gal{\cQ}{\Q_p}$, we also write $\sigma(F) \coloneqq \sum_{i = 0}^\infty \sigma (\beta_i) X^i$. 

\begin{prop} \label{definition g_chi proposition}
Write $\tau \: (\Z / p\Z)^\times \hookrightarrow \Z_p^\times$ for the $p$-adic Teichm\"uller character. For every $m_0 \mid d \mid m$ (viewed as an element of $\Z_p^\times$) and Dirichlet character $\chi \: (\Z / p\Z)^\times \to \Z_p^\times$ with $\chi \neq \tau$, we define
\begin{equation} \label{definition g_chi}
g_{\chi,d} (X) \coloneqq \log_{\widehat{E}} (X) + (p - 1)^{-1} \sum_{i = 0}^\infty c_{i + 1}^{(p)} \zeta_d^{p^i} \sum_{j = 1}^{p - 1} \chi^{-1} (j) ( (X + 1)^{d^{-1} \tau (j) p^i} - 1).
\end{equation}
Then the following claims are valid.
\begin{liste}
\item $g_{\chi , d} (X)$ is a well-defined element of $\cQ \llbracket X \rrbracket$ and of Honda type $p - a_p X + \bm{1}_N (p) X^2$.
    \item $(\sigma_p^{-n} g_{\chi , d}) (X)$ converges at $X = \zeta_{p^n} - 1$ for every $n \in \N$, and one has
    \[
    (\sigma_p^{-n} g_{\chi , d}) ( \zeta_{p^n} - 1) = \log_{\widehat{E}} ( \zeta_{p^n} - 1) + e_\chi \cdot \sum_{i = 0}^{n - 1} c_{i + 1}^{(p)} \sigma_p^{i - n} (\zeta_d) ( \zeta_{p^n}^{d^{-1} p^i} -1)
    \]
    with the usual idempotent $e_\chi \coloneqq (p - 1)^{-1} \sum_{a = 1}^{p- 1} \chi (a)^{-1} \sigma_a$.
    \item Set $h_{\chi , d} \coloneqq \exp_{\widehat E} \circ g_{\chi , d}$. Then $h_{\chi , d}$ is a well-defined element of $\cR \llbracket X \rrbracket$,
    the power series $(\sigma_p^{-n}g_{\chi , d}) (X)$ converges in $\zeta_{p^n} - 1$, and one has
    \[
    \log_{\widehat{E}} ( (\sigma_p^{-n} h_{\chi , d}) ( \zeta_{p^n} - 1)) = (\sigma_p^{-n}g_{\chi , d}) (\zeta_{p^n} - 1).
    \]
\end{liste}
\end{prop}

\begin{proof}
To prove the first part of claim (a), we note that when expanding out the expression $\sum_{i = 0}^\infty c_{i + 1}^{(p)} \zeta_d^{p^i} \sum_{j = 1}^{p - 1} \chi^{-1} (j) ( (X + 1)^{d^{-1} \tau (j) p^i} - 1)$, one obtains a power series of the form $\sum_{l = 0}^\infty \beta_l X^l$ where $\beta_l = \sum_{i = 0}^{\infty} c_{i + 1}^{(p)} \zeta_d^{p^i} \sum_{j = 1}^{p - 1} \chi (j)^{-1}  \binom{p^i d^{-1} \tau (j)}{l}$. Now,
\begin{align} \nonumber
\sum_{j = 1}^{p - 1} \chi (j)  \binom{p^i d^{-1} \tau (j)}{l} & = 
 \frac{1}{l !}  \sum_{j = 1}^{p - 1} \chi^{-1} (j) \prod_{\alpha = 0}^{l - 1} ( p^i d^{-1} \tau (j) - \alpha)\\
& =  \frac{1}{l !} \sum_{j = 1}^{p - 1} \chi^{-1} (j) \sum_{a = 1}^l B_a (p^i d^{-1} \tau (j))^a
\label{expanding binomial quotients out}
.\end{align}
with suitable integers $B_a \in \Z$. Since $\tau \neq \chi$, we have $\sum_{j = 1}^{p - 1} \chi (j)^{-1} \tau (j) = 0$ and so the $p$-adic valuation of (\ref{expanding binomial quotients out}) is at least $2i - l$. Here we are using Legendre's formula to observe the bound $\ord_p(l!) \leq l$. This then combines with the observations that $\ord_p (c_{i + 1}^{(p)}) \geq - i$ and $\ord_p (\zeta_d^{p^i}) = 0$ to imply that $\beta_l$ is a sum of terms that have $p$-adic valuation at least $i - l$
, and hence converges.\\
To show that $g_{\chi , d} (X)$ is of type $p - a_p X + \bm{1}_N (p) X^2$ we begin by noting that it is immediate from the definition of $g_{\chi , d} (X)$ that its constant term is $\log_{\widehat{E}} (0) + \beta_0 = 0$. Similarly, we see that $g_{\chi , d} ' (0) = 1$ because $\beta_1 = \sum_{i\geq 0} c_{i + 1}^{(p)} \zeta_d^{p^i} p^i d^{-1} \sum_{j = 1}^{p - 1} \chi^{-1} (j) \tau (j) = 0$.\\
We next verify that $(p - a_p \widehat \phi + \bm{1}_N (p)\widehat \phi^2) g_{\chi , d} (X)$ belongs to $p \cR \llbracket X \rrbracket$. In light of Proposition~\ref{elliptic curve Honda type}, we need only verify that $(p - a_p \widehat \phi + \bm{1}_N (p)\widehat \phi^2) \sum_{l = 0}^\infty \beta_l X^l$ belongs to $p \cR \llbracket X \rrbracket$. \\
Let us consider the $p$-order of the binomial expression in the definition of $\beta_l$. We have
\[ \binom{p^i d^{-1} \tau (j)}{l} = \frac{1}{l !} (p^i d^{-1} \tau (j)) \prod_{a=1 }^{l-1} ( p^i d^{-1} \tau (j) -a ) . \]
Write $l-1$ in its $p$-adic expansion, $l-1= x_n p^n + \dots + x_1 p +x_0$, i.e.\@ $x_i \in \{ 0, 1, \dots , p-1 \}$ and $x_n \neq 0$. We then note the following equality of sets, for $\gamma \in \{ 1, \dots , x_n \}$,
\begin{multline*}
    \{ p^i d^{-1}\tau(j) - a \pmod{p^n} \mid (\gamma-1)p^n +1 \leq a \leq \gamma p^n  \} \\ = \{ a \pmod{p^n} \mid (\gamma-1)p^n +1 \leq a \leq \gamma p^n \} .
\end{multline*}
Therefore, 
\begin{equation}\label{e: step 1 ord calc}
    \ord_p\bigg( \frac{\prod_{a=1 }^{x_n p^n} ( p^i d^{-1} \tau (j) -a )}{(x_np^n)!} \bigg) \geq 0.
\end{equation}
Similarly, we have for $\gamma=1, \dots , x_{n-1}$,
\begin{multline*}
    \{ p^i d^{-1}\tau(j) - a \pmod{p^{n-1}} \mid p^n + (\gamma-1)p^{n-1} +1 \leq  a \leq p^n + \gamma p^{n-1}  \} \\ = \{ a \pmod{p^{n-1}} \mid p^n + (\gamma-1)p^{n-1} +1 \leq a \leq  p^n + \gamma p^{n-1} \} .
\end{multline*}
We note that for any $a$ in the range used in the set above the highest power of $p$ that can divide it is $p^{n-1}$. Combining this calculation with (\ref{e: step 1 ord calc}) we have
\[ \ord_p\bigg( \frac{\prod_{a=1 }^{x_n p^n+ x_{n-1}p^{n-1}} ( p^i d^{-1} \tau (j) -a )}{(x_np^n+ x_{n-1}p^{n-1})!} \bigg) \geq 0. \]
Repeating this process one observes that 
\[ \ord_p\bigg( \frac{\prod_{a=1 }^{l} ( p^i d^{-1} \tau (j) -a )}{(l-1)!} \bigg) \geq 0. \]
In particular, $\ord_p ( \binom{p^i d^{-1} \tau (j)}{l}) \geq i - \ord_p( l)$. Therefore, by the observations made above we have $ l c_{i + 1}^{(p)}\binom{p^i d^{-1} \tau (j)}{l} \in \Z_p$. Thus, we can deduce from Lemma \ref{useful congruence} that
\begin{align*}
\widehat \phi \big( c_{i + 1}^{(p)} ( ( X + 1)^{d^{-1} \tau (j)p^i} - 1) \big)
& = c_{i + 1}^{(p)} ( ( X^p + 1)^{d^{-1} \tau (j)p^i} - 1) \\
& \equiv c_{i + 1}^{(p)} ( ( X + 1)^{d^{-1} \tau (j)p^{i + 1}} - 1) \mod p \Z_p \llbracket X \rrbracket
\end{align*}
for every $i \geq 0$ and $j \in (\Z / p \Z)^\times$. For the first equality we note that $\phi$ acts trivially on $c_{i+1}^{(p)}\binom{p^i d^{-1} \tau (j)}{l} \in \Q_p$. We may therefore use the relation $p c_{i + 1}^{(p)} - a_p c_i^{(p)} + \bm{1}_N (p) c_{i - 1}^{(p)} = 0$ to calculate that
\begin{align*}
 & \qquad (p - a_p \widehat \phi + \bm{1}_N (p)\widehat \phi^2) \sum_{i = 0}^\infty c_{i + 1}^{(p)} \zeta_d^{p^i} ( ( X + 1)^{d^{-1} \tau (j) p^i} - 1) \\
 & \equiv 
\sum_{i = 0}^\infty c_{i + 1}^{(p)}  \big( p \zeta_d^{p^i} ( ( X + 1)^{d^{-1}\tau (j) p^i} - 1) 
 - a_p \zeta_d^{p^{i + 1}} ( ( X + 1)^{d^{-1} \tau (j) p^{i + 1}} - 1) \\
 & \qquad + \bm{1}_N (p) \zeta_d^{p^{i + 2}} ( ( X + 1)^{d^{-1} \tau (j) p^{i + 2}} - 1) \big) \mod p \Z_p \llbracket X \rrbracket \\
& = p \zeta_d^{p} c_1^{(p)} X + \sum_{i = 1}^\infty \zeta_d^{p^i} \big( p c_{i + 1}^{(p)} - a_p c_{i}^{(p)} + \bm{1}_N (p) c_{i - 1}^{(p)} \big) ( X + 1)^{d^{-1} \tau (j) p^i} - 1) \\
& = p \zeta_d^{p} X.
\end{align*}
This implies the congruence required to establish that $g_{\chi , d} (X)$ is of the claimed Honda type. \\
As for claim (b), it follows immediately from the calculations made in the proof of (a) that $l \beta_l \in \cR$ for all $l \geq 0 $. 
Given this, it follows from \cite[Ch.\@ IV, Lem.\@ 6.3\,(a)]{Silverman} that $(\sigma_p^{-n}g_{\chi , d}) (X)$ converges when evaluated at $\zeta_{p^n} - 1$ and we may compute $(\sigma_p^{-n} g_{\chi , d}) (\zeta_{p^n} - 1)$ by evaluating (\ref{definition g_chi}) after applying $\sigma_p^{-n}$ (cf.\@ \cite[Lem.\@ 3.17]{Kataoka1}).\\
Since $g_{\chi , d} (X)$ and $\log_{\widehat{E}} (X)$ have the same Honda type (by claim (a) and Proposition \ref{elliptic curve Honda type}), it follows from Theorem \ref{Honda's thm} that $h_{\chi , d}$ belongs to $\cR \llbracket X \rrbracket$ and, in particular, converges when evaluated at $\zeta_{p^n} - 1$. To prove the remaining part of claim (c), we note that we have the formal identity $(\log_{\widehat{E}} \circ \exp_{\widehat{E}}) (X) = X$, which implies that $(\log_{\widehat{E}} \circ h_{\chi , d}) (X) = g_{\chi , d} (X)$ as power series. 
Since $\log_{\widehat{E}}$ has coefficients in $\Q_p$, this also implies that $(\log_{\widehat{E}} \circ  \sigma_p^{-n} (h_{\chi , d})) (X) = \sigma_p^{-n} (g_{\chi , d}) (X)$.
The latter equality still holds true when evaluated at $\zeta_{p^n} - 1$ because the coefficients of $\sigma_p^{-n} (h_{\chi , d})$ are in $\cR$ (see, for example, \cite[\S\,6.1.5]{Robert00}).
\end{proof}

\begin{rk}
The purpose of the sum $\sum_{j = 1}^{p - 1}$ in (\ref{definition g_chi}) is to avoid delicate convergence issues, and this extends an idea of Kobayashi. That is,
our definition of $g_{\chi, d}$ is directly inspired by a definition of Kobayashi in \cite[\S\,2]{Kobayashi06} that can be seen as a special case of (\ref{definition g_chi}) for $\chi = \bm{1}$. 
\end{rk}

We can now define useful local points by appropriately evaluating the power series $h_{\chi, d}$ defined in Proposition \ref{definition g_chi proposition}.

\begin{definition} \label{def x tilde}
Take $\epsilon_d$ to be a preimage of $\zeta_d p$ under $\log_{\widehat E} \: \widehat E ( p \cR) \stackrel{\simeq}{\to} \mathbb{G}_a (p \cR)$. We define an element of $\Q_p \otimes_{\Z_p} \widehat E ( \cM_{\cQ (\zeta_{p^n})})$ as
\[
\widetilde x_{mp^n} \coloneqq 
\sum_{m_0 \mid d \mid m}  \big( \sum_{\chi \neq \tau} \big( \sigma_p^{- n} (h_{\chi , d}) (\zeta_{p^n} - 1)  -_{\widehat{E}} (\zeta_{p^n} - 1) \big)  +_{\widehat E} (p\Eul_p (\sigma_p))^{-1} \sigma_p^{-n}(\epsilon_d) \big),
\]
where the sum ranges over all characters $\chi \: (\Z / p \Z)^\times \to \Z_p^\times$ that are not equal to the Teichm\"uller character $\tau$ and $\cM_{\cQ (\zeta_{p^n})}$ is the maximal ideal of $\cQ (\zeta_{p^n})$.
\end{definition}

The following result establishes the connection of these local points with Otsuki's elements from Definition \ref{otsuki element}.

\begin{lem} \label{key property x tilde}
In $\Q_p (\zeta_{mp^n})$, one has the equality
\[
\log_{\widehat{E}} ( \widetilde x_{mp^n}) = (1 - e_\tau) \cdot \sum_{m_0 \mid d \mid m} \big( \Eul_p (\hat \sigma_p)^{-1} (X) \big) ( \zeta_{dp^n}).
\]
\end{lem}

\begin{proof}
At the outset we recall that the idempotent $1 - e_\tau$ acts as the identity on $\cQ$.
Using Proposition \ref{definition g_chi proposition} and Otsuki's relation (\ref{Otsuki relation}) (for the equalities ($\ast$)), we may then calculate that
\begin{align*}
\log_{\widehat{E}} ( \widetilde x_{mp^n}) & = \sum_{m_0 \mid d \mid m} \big(  \sum_{\chi \neq \tau} \big( \sigma_p^{- n} (g_{\chi , d}) (\zeta_{p^n} - 1)  - \log_{\widehat E} (\zeta_{p^n} - 1) \big) +  (p \Eul_p (\sigma_p))^{-1} p \sigma_p^{-n} (\zeta_d) \big) \\
& \stackrel{(\ast)}{=} \sum_{m_0 \mid d \mid m} \big(  (1 - e_\tau) \sum_{i = 0}^{n-1} c_{i + 1} \sigma_p^{i - n} (\zeta_d) (\zeta_{p^{n}}^{ d^{-1} p^{i}} - 1) \\
& \qquad \qquad +  \big( \sum_{i = 0}^{n - 1} c_{i + 1} \sigma_p^i + \widetilde{F}^{(n)}_p (\sigma_p) \cdot \Eul_p (\sigma_p)^{-1} \big) \sigma_p^{-n} (\zeta_d) \big) \\
& = \sum_{m_0 \mid d \mid m} \big( (1 - e_\tau) \big( \sum_{i = 0}^{n - 1} c_{i + 1} \sigma_p^{i - n} (\zeta_d) \zeta_{p^{n}}^{d^{-1} p^{i}} + \widetilde{F}^{(n)}_p (\sigma_p) \cdot \Eul_p (\sigma_p)^{-1} \cdot \sigma_p^{-1}(\zeta_d) \big) \big) \\
& =  (1 - e_\tau) \cdot \sum_{m_0 \mid d \mid m} \big( \big( \sum_{i = 0}^{n - 1} c_{i + 1} \hat \sigma_p^i + \widetilde{F}^{(n)}_p (\sigma_p) \cdot \Eul_p (\sigma_p)^{-1} \hat \sigma_p^n \big) (X) \big) ( \sigma_p^{-n} (\zeta_d) \zeta_{p^n}^{d^{-1}}) \\
& \stackrel{(\ast)}{=} (1 - e_\tau) \cdot \sum_{m_0 \mid d \mid m} \big( \mathrm{Eul}_p (\hat \sigma_p)^{-1} (X) \big)( \zeta_{dp^n}),
\end{align*}
where the last equality uses that $\sigma_p^{-n} (\zeta_d) \zeta_{p^n}^{d^{-1}} = \zeta_{dp^n}$. 
\end{proof}

\subsubsection{An explicit reciprocity law and the proof of Theorem \ref{local points main result}}
\label{proof local points main result section}

We require the following modification of Kato's explicit reciprocity law (as stated in Theorem~\ref{kato euler system}~(c)).

\begin{prop}[Otsuki] \label{Otsuki's reciprocity law}
    For every $m \in \N$ coprime to $p$ and $n \in \Z_{\geq 0}$ one has
    \[    
    {\sum}_{\sigma \in G}  \mathrm{Tr}_{(\Q_p \otimes_\Q F_{mp^n}) / \Q_p} (  \sigma (x_{mp^n}) \cdot (\exp^\ast_{\omega_E} \circ \mathrm{loc}^{(p)}_{/f}) ( y^\mathrm{Kato}_{mp^n}) ) \sigma  = \theta_{mp^n}^\mathrm{MT} . \]
\end{prop}

\begin{proof}
This is proved in \cite[Thm.\@ 3.6]{Otsuki}. Note that $p$ is assumed to be a prime of good reduction for $E$ in loc.\@ cit.\@ but this assumption is not needed for the proof of \cite[Thm.~3.6]{Otsuki}.
 Furthermore, the difference between $x_{m}$ and $x_{m}^\text{Otsuki}$ noted in Remark \ref{r: difference in x_m} is taken into account by our definition of $y_m^\mathrm{Kato}$ (which differs from the element used by Otsuki in \cite[Prop.\@ 3.2]{Otsuki} by $\Eul_p (\sigma_p^{-1})^{-1}$). 
\end{proof}

We now give the proof of Theorem \ref{local points main result}. \medskip \\
\textit{Proof (of Theorem \ref{local points main result}):}
Let $\mathfrak{k}_{mp^n}$ be the preimage of $\sum_{m_0 \mid d \mid m} (\mathrm{Eul}_p (\hat \sigma_p)^{-1} (X))(\zeta_{dp^n}) \in \Q_p \otimes_\Z F_{mp^n}$ under the isomorphism
\[
 H^1_f (\Q_p, \mathrm{V}_p E_{F_{mp^n} / \Q}) =
{\bigoplus}_{v \mid p} (\Q_p \otimes_{\Z_p} \widehat{E} (\cM_{F_{mp^n, v}}))
\stackrel{\simeq}{\longrightarrow} {\bigoplus}_{v \mid p} (\Q_p \otimes_{\Z_p} 
\cO_{F_{mp^n, v}})
= \Q_p \otimes_\Z F_{mp^n} 
\]
defined as the sum $\oplus_{v \mid p} \log_{\widehat{E}}$ of the formal logarithm of $\widehat{E}$. 
Claim (a) of Theorem \ref{local points main result} follows by combining Proposition \ref{Prop Otsuki l-adic euler factors} with Proposition \ref{Otsuki's reciprocity law} and Lemma \ref{pairing properties}\,(a, b). 
To prove claim (b) of Theorem \ref{local points main result}, we 
recall that Otsuki has proved, in \cite[Prop.\@ 4.5]{Otsuki}, that 
\[
\Eul_p (\hat \sigma_p) \cdot \oplus_{v \mid p} (\log_{\widehat{E}} ( \widehat{E} ( p \cO_{F_m, v} \llbracket X \rrbracket + X \cO_{F_m, v} \llbracket X \rrbracket))) = 
(\Z_p \otimes_\Z \cO_{F_{m}}) \llbracket X \rrbracket
\]
if $p$ is a good prime with $a_p \not \equiv 1 \mod p$, and the same proof works if $p$ is an additive prime. In this case, therefore, $\mathfrak{k}_{mp^n}$ belongs to $\widehat{E} (F_{mp^n, p})$, as claimed. \\
We observe that Lemma \ref{key property x tilde} implies that $(1 - e_\tau) \mathfrak{k}_{mp^n}$ is equal to the family $(\tilde x_{mp^n, v})_{v \mid p}$ with $\tilde x_{mp^n, v}$ the element from Definition \ref{def x tilde} with $\cQ$ taken to be the completion of $F_m$ at the place lying below $v$ (so that $\cQ (\zeta_{p^n})$ is the completion of $F_{mp^n}$ at $v$). By construction, $p \Eul_p (\tilde \sigma_p) \tilde x_{mp^n, v}$ belongs to $\widehat{E} ( \cM_{F_{mp^n, v}})$, with $\cM_{F_{mp^n, v}}$, hence so does $(1 - e_\tau) p \mathrm{Eul}_p (\tilde \sigma_p) \mathfrak{k}_{mp^n}$, as claimed. \\
We now turn to the proof of claim (c). At the outset we note that it suffices to prove the claimed relations for the elements $ (\oplus_{v \mid p} \log_{\widehat{E}}) ( \mathfrak{k}_{mp^n}) = \Eul_p (\hat \sigma_p)^{-1} \cdot \sum_{m_0 \mid d \mid m} \zeta_{dp^n}$. 
    Equality (i) then follows from the fact that by \cite[Prop.\@ 2.5]{Otsuki} one has
    \begin{align*}
    \mathrm{Tr}_{mp^{n + 1}/mp^n} \big ( (\Eul_p (\hat \sigma_p)^{-1} (X)) (\zeta_{dp^n}) \big) & = 
    \mathrm{Tr}_{d p^{n + 1}/ d p^n} \big ( (\Eul_p (\hat \sigma_p)^{-1} (X)) (\zeta_{dp^n}) \big ) \\
    & = \begin{cases}
        a_p \zeta_{dp^n} - \bm{1}_N (p) \zeta_{dp^{n - 1}} \quad & \text{ if } n > 2, \\
        a_p \zeta_{dp^n} - \bm{1}_N (p) \Eul_p (\sigma_p)^{-1} \zeta_d & \text{ if } n = 1, \\
        (a_p - \bm{1}_N (p) \sigma_p - \sigma_p^{-1}) \Eul_p (\sigma_p)^{-1} \zeta_d & \text{ if } n = 0.
        \end{cases}
    \end{align*}
    As for the equality in (ii), we let $d$ be a natural number with $m_0 \mid d \mid m$ and
    first recall that $\mathrm{Tr}_{\ell d p^n / d p^n} (\zeta_{\ell dp^n})$ vanishes if $\ell \mid d$, so that
    \[
     \mathrm{Tr}_{\ell m p^n/mp^n} \big (  \sum_{(\ell m )_0 \mid d' \mid \ell m} \zeta_{d'p^n} \big) = 
     \begin{cases}
     \mathrm{Tr}_{\ell m p^n/mp^n} (  \sum_{m_0 \mid d \mid m} \zeta_{dp^n}) \quad & \text{ if } \ell \mid m, \\
      \mathrm{Tr}_{\ell m p^n/mp^n} (  \sum_{m_0 \mid d \mid m}  \zeta_{\ell d p^n}) \quad & \text{ if } \ell \nmid m.
     \end{cases}
    \]
Claim (ii) therefore follows from the observation that, for a natural number $d' \mid  \ell m$, one has 
    \[
    \mathrm{Tr}_{\ell m p^n/mp^n} ( \zeta_{d' p^n} ) = \begin{cases} 
    \ell \zeta_{d' p^n} \quad & \text{ if } \ell \nmid d', \\
    - \sigma_{\ell}^{-1} \zeta_{d' / \ell} & \text{ if } \ell \mid d'.
    \end{cases}
    \]

We show claim (d) of Theorem \ref{local points main result} working case by case. \\
If $a_\ell =2$, $\ell \nmid N$ and $\ell^2 \nmid m$, then $\ord_\ell(m)=1$ and thus from Proposition \ref{l: lambda mod I^2}\,(b) we have
\[ \nu_{mp^n}^{(\ell)} \equiv c_{1}^{(\ell)} ( 1 - a_{\ell}\tilde \sigma_{\ell} )  + \frac{\bm{1}_N(\ell)}{\ell} ( \ell_j c_{1}^{(\ell)} \tilde \sigma_{\ell}^2  + c_{0}^{(\ell)}\tilde\sigma_{\ell}(\ell - 1 ) ) \pmod{I_{\cD^{(\ell)}_{mp^n}}^2} . \]
Since $c_0^{(\ell)}=0$ and $c_1^{(\ell)}=1$ we can calculate
\[ \nu_{mp^n}^{(\ell)} \equiv 1-2 \tilde\sigma_{\ell} + \tilde\sigma_{\ell}^2 = ( 1-\tilde\sigma_{\ell} ) \equiv 0 \pmod{I_{\cD^{(\ell)}_{mp^n}}^2} . \]

If $\bm{1}_N(\ell) =0 $ then one can inductively show that $c_j^{(\ell)} = (\frac{a_\ell}{\ell})^{j-1}$ for $j \geq 1$. Proposition \ref{l: lambda mod I^2}\,(b) now gives that, in this case,
\[ \nu_{mp^n}^{(\ell)} \equiv c_{1}^{(\ell)} ( 1 - a_{\ell}\tilde \sigma_{\ell} )   \pmod{I_{\cD^{(\ell)}_{mp^n}}^2} . \]
In both the remaining cases of claim (d) of Theorem \ref{local points main result} we have $\bm{1}_N(\ell) =0 $, hence the result follows immediately from the above equation and definition of $c_j^{(\ell)}$ in these cases.  \\
Finally, claim (e) is immediate from the definition of $\lambda_{\ord_\ell (m)}^{(\ell)} (X)$ upon noting that the element $\omega_{\ord_\ell (m), 0}^{(\ell)} = - \tilde \sigma_\ell \ell^{- \ord_\ell (m)} \NN_{F_{\ell^{\ord_\ell (m)}} / \Q}$ differs from $\NN_{\mathcal{I}^{(\ell)}_{mp^n}}$ only by a unit in $\Z_p [G_{mp^n}]$.
\qed 

\section{Nekov\'a\v{r}--Selmer complexes and the equivariant Tamagawa Number Conjecture}

In this section we recall some general aspects of the theory of Nekov\'a\v{r}--Selmer complexes 
and discuss examples of such complexes that will be important to us. Even though most material contained in this section extends to general $p$-adic representations, we prefer to restrict attention to the setting most relevant to us -- namely that of the representation given by the $p$-adic Tate module of $E$. For a more in-depth treatment of Nekov\'a\v{r}--Selmer complexes an interested reader is invited to refer to \cite{NekovarSelmerComplexes} or \cite[\S\,3]{BB}.\\
We freely use some of the notation and conventions introduced in Appendix \ref{algebra appendix}. In particular, we write $D( \cR)$ for the derived category of $\cR$-modules of a ring $\cR$ and $D^\mathrm{perf} (\cR)$ for its full triangulated subcategory of complexes that are perfect.

\subsection{Nekov\'a\v{r}--Selmer structures}

We begin by introducing a suitable variant of the notion of `Selmer structure' used by Mazur and Rubin in \cite[Def.\@ 2.1.1]{MazurRubin04}. Recall that we have defined $S(K)= S_{pNm\infty}$, where $m$ denotes conductor of $K$ and $N$ is the conductor of $E$, and that Shapiro's lemma gives natural isomorphisms $\mathrm{R}\Gamma (\cO_{\Q, U}, T_{K / \Q}) \cong \mathrm{R}\Gamma (\cO_{K, U}, \mathrm{T}_p E)$ and $\mathrm{R}\Gamma (\Q_v, T_{K / \Q}) \cong \bigoplus_{w \mid v}\mathrm{R}\Gamma (K_w, \mathrm{T}_p E)$ for every 
finite set $U \subseteq S(K)$ and place $v$ of $\Q$.

\begin{definition}
A `Nekov\'a\v{r}--Selmer structure' $\cF$ for $T_{K / \Q}$ consists of the following data:
\begin{itemize}
    \item A finite set $S(\cF)$ of places of $\Q$ that contains $S(K)$. 
    \item For every place $v \in S(\cF)$ a complex $\mathrm{R}\Gamma_\cF (\Q_v, T_{K / \Q})$ of $\Z_p [G]$-modules together with a morphism $i_{\cF, v} \: \mathrm{R}\Gamma_\cF (\Q_v, T_{K / \Q}) \to \mathrm{R}\Gamma (\Q_v, T_{K / \Q})$ in $D(\Z_p [G])$.
\end{itemize}
One then defines the `Nekov\'a\v{r}--Selmer complex' $\mathrm{R}\Gamma_\cF (K, \mathrm{T}_p E)$ as
\[
\mathrm{cone} \big\{ 
\mathrm{R}\Gamma (\cO_{K, S(\cF)}, \mathrm{T}_p E) \oplus 
\bigoplus_{v \in S(\cF)} \mathrm{R}\Gamma_\cF (\Q_v, T_{K / \Q})
\xrightarrow{\oplus_{v \in S(\cF)} ( \mathrm{loc}_v - i_{\cF, v})}
\bigoplus_{v \in S(\cF)} \mathrm{R}\Gamma (\Q_v, T_{K / \Q})
\big\} [-1],
\]
where $\mathrm{loc}_v$ denotes the natural localisation map 
\[
\mathrm{R}\Gamma (\cO_{K, S(\cF)}, \mathrm{T}_p E) \to {\bigoplus}_{w \mid v} \mathrm{R}\Gamma (K_w, \mathrm{T}_p E) = \mathrm{R}\Gamma (\Q_v, T_{K / \Q})
\]
for every $v \in S(\cF)$, and sets $H^i_\cF (K, \mathrm{T}_p E) \coloneqq H^i ( \mathrm{R}\Gamma_\cF (K, \mathrm{T}_p E))$ for all $i \in \Z$. 
\end{definition}

\begin{rk} \label{Selmer complex triangle}
For every Selmer structure $\cF$, the octrahedral axiom implies the existence of an exact triangle
\begin{cdiagram}[column sep=small]
    \mathrm{R}\Gamma_\cF (K, \mathrm{T}_p E) \arrow{r} & \mathrm{R}\Gamma (K, \mathrm{T}_p E) \arrow{r} & 
    {\bigoplus}_{v \in S (\cF)} \mathrm{R}\Gamma_{ / \cF} (K_v, \mathrm{T}_p E) \arrow{r} & \mathrm{R}\Gamma_\cF (K, \mathrm{T}_p E) [1]
\end{cdiagram}
    with $\mathrm{R}\Gamma_{ / \cF} (K_v, \mathrm{T}_p E) \coloneqq \mathrm{cone} (i_{\cF, v}) [-1]$. 
\end{rk}

\begin{bspe}
\item (Relaxed structures)
    For every finite set $\Sigma$ of places of $\Q$ that contains $S (K)$ we define a `relaxed' Nekov\'a\v{r}--Selmer structure $\cF_{\Sigma, \mathrm{rel}}$ by taking 
    \begin{itemize}
        \item $S ( \cF_{\Sigma, \mathrm{rel}}) \coloneqq \Sigma$,
        \item $\mathrm{R}\Gamma_{\cF_{\Sigma, \mathrm{rel}}} (\Q_v, T_{K / \Q}) \coloneqq \mathrm{R} \Gamma ( (\Q_v, T_{K / \Q})$ for all $v \in S ( \cF_{\Sigma, \mathrm{rel}})$,
        \item $i_{\cF_{\Sigma, \mathrm{rel}}}$ to be the identity map for all $v \in S ( \cF_{\Sigma, \mathrm{rel}})$.
    \end{itemize}
    The associated Nekov\'a\v{r}--Selmer complex $\mathrm{R}\Gamma_{\cF_{\Sigma, \mathrm{rel}}} (K, \mathrm{T}_p E)$ then coincides with $\mathrm{R}\Gamma (\cO_{K, \Sigma}, \mathrm{T}_p E)$.

\item (Strict structures)
    For every finite set of places of $\Q$ that contains $S (K)$ we define a `strict' Nekov\'a\v{r}--Selmer structure $\cF_{\Sigma, \mathrm{str}}$ by means of
    \begin{itemize}
        \item $S( \cF_{\Sigma, \mathrm{str}}) \coloneqq \Sigma$,
        \item $\mathrm{R}\Gamma_{\cF_{\Sigma, \mathrm{str}}} (\Q_v, T_{K / \Q}) = 0$ for all $v \in S( \cF_{\Sigma, \mathrm{str}})$,
        \item $i_{\cF_{\Sigma, \mathrm{str}}, v} \coloneqq 0$ for all $v \in S( \cF_{\Sigma, \mathrm{str}})$.
    \end{itemize}
    The associated Nekov\'a\v{r}--Selmer complex $\mathrm{R}\Gamma_{\cF_{\Sigma, \mathrm{rel}}} (K, \mathrm{T}_p E)$ then coincides with the `compact-support' cohomology complex $\mathrm{R}\Gamma_\mathrm{c} (\cO_{K, \Sigma}, \mathrm{T}_p E)$.
\end{bspe}

\begin{rk}[Dual structures] \label{dual structures}
For every finite place $v$ of $\Q$ local Tate duality induces an isomorphism (cf.\@ \cite[Th.~5.2.6]{NekovarSelmerComplexes})
\begin{equation} \label{derived local Tate duality}
\RHom_{\Z_p} (\mathrm{R}\Gamma (\Q_v, T_{K / \Q}), \Z_p)[-2] \xrightarrow{\simeq} \mathrm{R} \Gamma (\Q_v, (T_{K / \Q})^\ast (1)) \cong \mathrm{R} \Gamma (\Q_v, T_{K / \Q}).
\end{equation}
(Here the second isomorphism is induced by the Weil pairing.) Any Nekov\'a\v{r}--Selmer structure $\cF$ on $T_{K / \Q}$ therefore naturally induces a dual structure $\cF^\ast$ by taking 
\begin{itemize}
    \item $S (\cF^\ast) \coloneqq S (\cF)$,
    \item $\mathrm{R}\Gamma_{\cF^\ast} (\Q_v, T_{K / \Q}) \coloneqq \RHom_{\Z_p} ( \mathrm{R}\Gamma_{/ \cF} (\Q_v, T_{K / \Q}), \Z_p) [-2]$ if $v$ is finite and zero otherwise,
    \item $i_{\cF^\ast, v}$ is the composite of $\RHom_{\Z_p} ( i_{\cF, v}, \Z_p) [-2]$ and (\ref{derived local Tate duality}) if $v$ is finite (and the zero map otherwise).
\end{itemize}
\end{rk}

\subsection{Compactly supported \'etale cohomology} \label{s: cpctly supported etale cohom}

Fix a finite set $\Sigma$ of places of $\Q$ that contains $S(K)$.
In the next section, we often use the complex
\[
C^\bullet_{K, \Sigma} \coloneqq \RHom_{\Z_p} ( \mathrm{R}\Gamma_\mathrm{c} (\cO_{K, \Sigma}, \mathrm{T}_p E), \Z_p) [-3].
\]
The following result records useful properties of the complex $C^\bullet_{K, \Sigma}$. 

\begin{lem} \label{complex lemma}
Let $\Sigma$ be a finite set of places of $\Q$ that contains $S (K)$, and assume that $E(K)$ has no point of order $p$. The following then hold.
\begin{liste}
\item $C^\bullet_{K, \Sigma}$ is a perfect object of $D (\Z_p [G])$ that has vanishing Euler characteristic in $K_0 (\Z_p [G])$ and is acyclic outside degrees one and two. There is a canonical isomorphism $H^1 (C^\bullet_{K, \Sigma}) \cong H^1 (\cO_{K, \Sigma}, \mathrm{T}_p E)$ and a split-exact sequence
\begin{cdiagram}
    0 \arrow{r} & H^2 ( \cO_{K, \Sigma}, \mathrm{T}_p E) \arrow{r} & H^2 (C^\bullet_{K, \Sigma}) \arrow{r}{\pi} & 
    (T_{K / \Q}^+)^\ast \arrow{r} & 0
\end{cdiagram}%
with $T_{K / \Q}^+ \coloneqq H^0 (\R, T_{K/ \Q})$. 
\item $C^\bullet_{K, \Sigma}$ admits a standard representative in the sense of Definition \ref{standard representative def} with $F^0 = 0$.
\end{liste}
\end{lem}

\begin{proof}
    The assumption that $E (K) [p]$ vanishes implies that $H^1 (\cO_{K, \Sigma}, \mathrm{T}_pE )$ is $\Z_p$-torsion free (cf.\@ \cite[Ex.\@ 3.3\,(b)]{BullachDaoud}). That is, \cite[Hyp.\@ 2.16]{sbA} is satisfied and so \cite[Prop.\@ 2.22\,(ii)]{sbA} combines with \cite[Prop.\@ A.11]{sbA} to imply claims (a) and (b). We only note that the description of the cohomology given in claim (a) is a consequence of the fact that Artin--Verdier duality (as in \cite[Lem.\@ 12\,(b)]{BurnsFlach01}) induces an exact triangle
    \begin{cdiagram}[column sep=small]
       \mathrm{R}\Gamma (\cO_{K, \Sigma}, \mathrm{T}_p E) \arrow{r} & C^\bullet_{K, \Sigma} \arrow{r} & (T_{K / \Q}^+)^\ast [-2] \arrow{r} &  \mathrm{R}\Gamma (\cO_{K, \Sigma}, \mathrm{T}_p E) [1]
    \end{cdiagram}%
    in $D (\Z_p [G])$. 
\end{proof}

We next make a convenient choice of $\Z_p [G]$-basis for $T_{K / \Q}^+$.
To do this, write $w_\iota$ for the place of $K$ corresponding with $\iota_K$, the restriction to $K$ of the embedding $\iota \: \overline{\Q} \hookrightarrow \C$ fixed in \S\,\ref{notation section}, and denote by $\mathcal{D}^{(\iota)}_K \subseteq G$ the associated decomposition group. Viewing $w_\iota$ as an equivalence class of embeddings $\sigma \: K \hookrightarrow \C$, one has that $\Z_p \otimes_\Z (\bigoplus_{\sigma \in w_\iota} H^1 (E^\sigma_K (\C),  \Z))^+
= \big( ( \bigoplus_{\sigma \in w_\iota} \Z_p \sigma) \otimes_\Z H^1 (E (\C),  \Z)) \big)^+$ is a free $\Z_p [\mathcal{D}_K^{(\iota)}]$-module of rank one with basis 
\[
\gamma_K \coloneqq (\tfrac{1}{2} (1 + c) \iota_K  \otimes \gamma_+) +  (\tfrac{1}{2} (1 - c) \iota_K \otimes \gamma_- ).
\]
We then define $b \coloneqq b_K$ to be the image of $\gamma_K$ under the comparison isomorphism 
\[
\Z_p \otimes_\Z ({\bigoplus}_{\sigma \in w_\iota} H^1 (E^\sigma_K (\C),  \Z))^+\cong H^0 (K_{w_\iota}, \mathrm{T}_p E).
\]
By construction, the element $b$ is indeed a $\Z_p [G]$-basis of $T_{K / \Q}^+ \cong \bigoplus_{w \mid \infty} H^0 (K_w, \mathrm{T}_p E)$.
\\
Let $\Sigma$ be a finite set of places of $\Q$ that contains $S (K)$, and write 
\[
\Theta_{K, \Sigma}  \coloneqq \vartheta_{C^\bullet_{K, \Sigma}, \{ b^\ast \}} \: 
\Det_{\Z_p [G]} ( C^\bullet_{K, \Sigma})^{-1} \to H^1 ( \cO_{K, \Sigma}, \mathrm{V}_p E)
\]
for the map from Definition \ref{def projection map from det} with $b^\ast$ the $\Z_p [G]$-linear dual of the basis element $b \in T_{K / \Q}^+ $
defined above.

\begin{thm}
    \label{etnc result 2}
    Assume $p > 3$ is a prime number such that the image of $\rho_{E, p}$ contains $\mathrm{SL}_2 (\Z_p)$.
    If $E$ does not have potentially good reduction at $p$, then we also assume that $K$ does not contain a primitive $p$-th root of unity.
 Then $z^\mathrm{Kato}_{K}$ is contained in the image of $\Theta_{K, S(K)}$.
\end{thm}

\begin{proof}
    This is proved by Burns and the first author in \cite[Cor.\@ 9.6\,(i)]{BB}.
\end{proof}

\begin{rk}
The result of Theorem \ref{etnc result 2} is directly related to the `equivariant Tamagawa Number Conjecture' for the pair $(h^1 ( E / K) (1), \Z_p [G])$, see \cite[Rk.\@ 9.7]{BB} for more details.
\end{rk}

\subsection{Bloch--Kato Selmer complexes} \label{RGamma_f section}

Following Bloch and Kato \cite{BlochKato}, we define the local `finite-support' cohomology complex
\[
\mathrm{R}\Gamma_f (\Q_v, T_{K / \Q}) \coloneqq  \begin{cases}
    \bigoplus_{w \mid v} E (K_v)^\wedge [-1] \quad & \text{ if } v \nmid \infty, \\ 
     \mathrm{R}\Gamma (\Q_v, T_{K / \Q}) & \text{ if } v = \infty,
\end{cases}
\]
for every place $v$ of $\Q$. For every finite set of places $\Sigma$ containing $S(K)$, one then defines the Bloch--Kato Selmer structure $\cF_{\Sigma, \mathrm{BK}}$ for $T_{K / \Q}$ as follows. 
\begin{itemize}
    \item $S(\cF_{\Sigma, \mathrm{BK}}) = \Sigma$,
    \item $\mathrm{R}\Gamma_{\cF_{\Sigma, \mathrm{BK}}} (\Q_v, T_{K / \Q}) \coloneqq \mathrm{R}\Gamma_f (\Q_v, T_{K / \Q})$ for all $v \in \Sigma$,
    \item $i_{\cF_{\Sigma, \mathrm{BK}}, v}$ is taken to be the map induced by the Kummer map if $v \in \Sigma$ is a finite place and the identity map if $v = \infty$.
\end{itemize}  
As is customary, we then replace any adornments $\cF_{\Sigma, \mathrm{BK}}$ by simply $f$ and, for example, write
\[
\mathrm{R}\Gamma_f (K, \mathrm{T}_p E) \coloneqq \mathrm{R} \Gamma_{\cF_{\Sigma, \mathrm{BK}}} (K, \mathrm{T}_p E)
\]
and $H^i_f (K, \mathrm{T}_p E) \coloneqq H^i_{\cF_{\Sigma, \mathrm{BK}}} (K, \mathrm{T}_p E)$ for every $i \in \Z$ (this is consistent with the definition for $\ell$-adic fields given in \S\,\ref{local duality section}). Note that this definition of $\mathrm{R}\Gamma_f (K, \mathrm{T}_p E)$ does not depend on the choice of $\Sigma$, see \cite[Lem.\@ 2.5]{BurnsMaciasCastillo}.

\begin{lem} \label{description finite support cohomology}
    One has canonical identifications
    \[
    H^i_f (K, \mathrm{T}_p E) = \begin{cases}
        \Sel_{p, E / K}^\wedge \quad & \text{ if } i = 1,\\
        \Sel_{p, E / K}^\vee & \text{ if } i = 2, \\
        E (K) [p^\infty]^\vee & \text{ if } i = 3,\\
        0 & \text{ otherwise}.
    \end{cases}
    \]
\end{lem}

\begin{proof}
    See, for example, \cite[Lem.\@ 1]{BurungaleFlach}.
\end{proof}

\begin{rk}
Local Tate duality induces  an isomorphism
\begin{equation} \label{Tate duality RGamma_f}
\mathrm{R}\Gamma_{/f} (\Q_\ell, T_{K / \Q}) \cong \RHom_{\Z_p} ( \mathrm{R}\Gamma_f (\Q_\ell, T_{K / \Q}) , \Z_p) [-2] 
\end{equation}
that is valid for every prime number $\ell$.
This is a special case of the general result of \cite[Lem.~19]{BurnsFlach01} and, in the case at hand, can also be checked explicitly as in \cite[Lem.~2]{BurungaleFlach}.
\end{rk}

\section{Nekov\'a\v{r}--Selmer complexes and weak main conjectures} \label{approximation complexes section}

The fundamental difficulty in deriving explicit statements involving classical Selmer groups such as Conjecture \ref{Mazur--Tate conj 1} from the equivariant Tamagawa Number Conjecture is that the finite-support cohomology complex $\mathrm{R}\Gamma_f (K, \mathrm{T}_p E)$ is rarely perfect as an object in $D (\Z_p [G])$ (cf.~\cite[Lem.~5.1]{BurnsMaciasCastillo} or \cite[\S\,4.1]{NickelEllerbrock}). In this subsection we therefore construct auxiliary Selmer complexes that are perfect and approximate $\mathrm{R}\Gamma_f (K, \mathrm{T}_p E)$.\\
Throughout this subsection we fix an abelian number field $K$ with Galois group $G = G_K \coloneqq \gal{K}{\Q}$. For every place $v$ of $K$, we denote the residue field at $v$ by $\mathbb{F}_v$.

\subsection{Unramified cohomology}

Write $\widetilde E$ for the reduction of $E$ modulo $p$, denote the group of non-singular $\mathbb{F}_v$-rational points on $\widetilde E$ by $\widetilde E^\mathrm{ns} ( \mathbb{F}_v)$, and set
\[
(E_1 / E_0) ( K_\ell) \coloneqq {\bigoplus}_{v \mid \ell} ( \widetilde E^\mathrm{ns} (\mathbb{F}_v) \otimes_\Z \Z_p ).
\]

\begin{lem} \label{unr cohom lem}
    For every prime number $\ell$ there exists a natural number $n_\ell$ such that, if $\ell$ is unramified in $K$, there exists an $(n_\ell \times n_\ell)$-matrix $A_{\ell, K}$ with the following properties.
    \begin{liste}
    \item $\det_{\Z_p [G]} ( A_{\ell, K}) \cdot \Z_p [G] =  \ell \cdot \Eul_\ell ( \sigma_\ell^{-1}) \cdot \Z_p [G]$.
        \item There is an exact sequence 
        \begin{cdiagram}
        0 \arrow{r} & \Z_p [G]^{\oplus n_\ell} \arrow{r}{\cdot A_{\ell, K}} & \Z_p [G]^{\oplus n_\ell} \arrow{r} & (E_1 / E_0) ( K_\ell) \arrow{r} & 0.
    \end{cdiagram}%
    In particular, we have the equality
    $
    \Fitt^0_{\Z_p [G]} \big ( (E_1 / E_0) ( K_\ell) \big)
    = \ell \cdot \Eul_\ell (\sigma_\ell^{-1}) \cdot \Z_p [G]$.
        \item If $L$ is a subfield of $K$, then the functor $(-) \otimes_{\Z_p [G_K]} \Z_p [G_L]$ takes the exact sequence in (b) for $K$ to that for $L$.
    \end{liste}
\end{lem}

\begin{proof}
We abbreviate (a choice of) inertia subgroup of $\gal{\overline{\Q}}{\Q}$ at $\ell$ to $\mathcal{I}_\ell \coloneqq \mathcal{I}^{(\ell)}_{\overline{\Q}}$.
If $\ell \nmid p$, then we have an exact sequence (see, for example, the argument of \cite[Lecture 5, Lem.\@ 5]{kings-bsd})
\begin{equation} \label{resolution unr cohom}
\begin{tikzcd}
    0 \arrow{r} & (T_{K / \Q})^{\mathcal{I}_\ell} \arrow{r}{1 - \sigma_\ell^{-1}} & (T_{K / \Q})^{\mathcal{I}_\ell} \arrow{r} & 
    (E_1 / E_0) ( K_\ell) \arrow{r} & 0.
\end{tikzcd}%
\end{equation}%
This sequence may be used for the computation of the Fitting ideal because 
\[
(T_{K / \Q})^{\mathcal{I}_\ell} = ( \mathrm{T}_p E \otimes_{\Z_p} \Z_p [G])^{\mathcal{I}_\ell} = (\mathrm{T}_p E)^{\mathcal{I}_\ell} \otimes_{\Z_p} \Z_p [G]
\]
is a free $\Z_p [G]$-module of rank $n_\ell \coloneqq \mathrm{rk}_{\Z_p} ( (\mathrm{T}_p E)^{\mathcal{I}_\ell})$. (We have used here that $\ell$ is unramified in $K$.) It follows that
the Fitting ideal of $(E_1 / E_0) ( K_\ell)$ is generated by 
\begin{align*}
 {\det}_{\Z_p[G]} ( 1 - \sigma_\ell^{-1} \mid  (T_{K / \Q})^{\mathcal{I}_\ell}) 
 = {\det}_{\Z_p} ( 1 - \sigma_\ell^{-1} X \mid  (\mathrm{T}_p E)^{\mathcal{I}_\ell})_{\mid_{X = \sigma_\ell^{-1}}}
 = \Eul_\ell (\sigma_\ell^{-1}),
\end{align*}
which generates the same ideal of $\Z_p [G]$ as $\ell \cdot \Eul_\ell (\sigma_\ell^{-1})$ because $\ell \neq p$. Upon fixing a $\Z_p$-basis of $(\mathrm{T}_p E)^{\mathcal{I}_\ell}$, we therefore obtain an $(n_\ell \times n_\ell)$-matrix $A_{\ell, K}$ that represents multiplication by $1- \sigma_\ell^{-1}$ on $(\mathrm{T}_p E)^{\mathcal{I}_\ell} \otimes_{\Z_p} \Z_p [G]$ and has all properties listed in the lemma.\\
Let us now turn to the case $\ell = p$. We first assume that $E$ has bad reduction and note that the reduction type of $E$ (over $K$) at $v$ agrees with the reduction type of $E$ (over $\Q$) at $p$ because $p$ is assumed to be unramified in $K$. If $p$ is a prime of multiplicative reduction, then 
$(E_1 / E_0) ( K_\ell) \cong \bigoplus_{v \mid p} (\mathbb{F}_v^\times \otimes_\Z \Z_p)$ therefore vanishes
and so in this case the claim follows upon taking $n_p = 1$ and $A_{p, K} = 1$, and noting that $p \Eul_p (\sigma_p^{-1}) = p \pm \sigma_p^{-1}$ is a unit in $\Z_p [G]$. If $p$ is a prime of additive reduction, on the other hand, then 
$(E_1 / E_0) ( K_\ell) \cong \bigoplus_{v \mid p} \mathbb{F}_v$.
Now, the latter is isomorphic to $\mathbb{F}_p [G]$ by the normal basis theorem and so we deduce that its Fitting ideal over $\Z_p [G]$ is generated by $p = p \Eul_p (\sigma_p^{-1})$, and that we may take $n_p = 1$ and $A_{p, K} = p$.\\
We finally assume that $p$ has good reduction at $p$. 
In this case (the proof of) \cite[Lem.\@ 4.4]{BurnsCastilloWuthrich} (see also \cite[Lem.\@ 6.12]{BurnsMaciasCastillo}) shows that $\Fitt^0_{\Z_p [G]} ( (E_1 / E_0) (K_p)) = p \Eul_p (\sigma_p^{-1}) \Z_p [G]$. Moreover, $(E_1 / E_0) (K_p)$ is a cyclic $\Z_p [G]$-module (cf.\@ \cite[Ch.~III, Cor.~6.4]{Silverman}) so that its Fitting ideal agrees with its annihilator. Upon choosing a $\Z_p [G]$-generator of $(E_1 / E_0) (K_p)$ we therefore 
obtain an isomorphism $\Z_p [G] /(p \Eul_p (\sigma_p^{-1}) \Z_p [G]) \cong (E_1 / E_0) (K_p)$. It follows that conditions (a) and (b) are satisfied with $n_p = 1$ and $A_{p, K} = p \Eul_p (\sigma_p^{-1})$, and that condition (c) is satisfied if we (as we may) choose the $\Z_p [G_L]$-generators of $(E_1 / E_0) (L_p)$ compatibly as $L$ varies over all finite abelian extensions of $\Q$ unramified at $p$.
\end{proof}

\subsection{Local finite-support cohomology}

We define, for $i \in \{0, 1\}$,
\[
(E / E_i) (K_\ell) \coloneqq {\prod}_{v \mid \ell} ( E ( K_v) / E_i (K_v))^\wedge.
\] 
We will use without comment that if $\ell \neq p$, then $(E / E_1) (K_\ell)$ agrees with $E (K_\ell) \coloneqq \bigoplus_{v \mid \ell}  (K_v)^\wedge$ because $E_1 (K_\ell)$ is pro-$\ell$.
Note that the Tamagawa number $| E (K_v) / E_0 (K_v)|$ is independent of the choice of place $v \mid \ell$ because $K$ is an abelian field. For simplicity we therefore denote this number as $\mathrm{Tam}_{K, \ell}$.

\begin{lem} \label{approximation complexes}
Assume that $p$ is unramified in $K$ if $E$ has additive reduction at $p$. 
Then the following claims are valid.
\begin{liste}
\item For every prime number $\ell$, there exists a morphism 
\[
\rho_{\ell} \: D_{\ell}^\bullet \to (E / E_1) ( K_\ell) [-1]
\]
in $D (\Z_p [G])$ with $D_{\ell}^\bullet$ a perfect complex of $\Z_p [G]$-modules that has the following properties: 
\begin{romanliste}
\item $D_{\ell}^\bullet$ is acyclic outside degree one.
    \item $D_{\ell}^\bullet$ admits a representative of the form 
$[F \to F]$ for a free $\Z_p [G]$-module of finite rank (here the first term is placed in degree 0).
\item One has $\vartheta_{D_\ell^\bullet, \emptyset} ( \Det_{\Z_p [G]} (D^\bullet_{\ell})^{-1}) =  \ell \cdot \Eul_\ell (\tilde \sigma_\ell^{-1}) \cdot \Z_p [G]$.
\item If $\ell$ is unramified in $K$ and $p \nmid \Tam_{K, \ell}$, then $\rho_\ell$ is a quasi-isomorphism.
\end{romanliste}   
\item For every point $Q$ in $E_1 (K_p)$, there exists a morphism 
\[
\rho_{p, Q} \: D_{p, Q}^\bullet \coloneqq D_{p}^\bullet \oplus (\Z_p [G] [-1]) \to \mathrm{R}\Gamma_f (\Q_p, T_{K / \Q})
\]
in $D (\Z_p [G])$ such that the restriction of $H^1 (\rho_{p, Q})$ to $\Z_p [G]$ maps $1$ to $Q$.
\end{liste}
\end{lem}

\begin{proof}
Write $m$ for the conductor of $K$ and $m' \coloneqq m \ell^{- \ord_\ell (m)}$ for the prime-to-$\ell$ part of $m$. We can then define a $\Z_p [G_m]$-algebra map $j_\ell \: \Z_p [G_{m'}] \to \Z_p [G_m]$ by sending $\sigma_a$ to $\tilde \sigma_a$ (with $\tilde \sigma_a$ as defined before Theorem \ref{local points main result} in \S\,\ref{section statement local points main result}).
Write $j_\ell ( A_{\ell, F_{m'}})$ for the matrix obtained by applying $j_\ell$ to each of entry of the matrix $A_{\ell, m'}$ from Lemma \ref{unr cohom lem} and set $\widetilde{ A_\ell} \coloneqq \pi_{F_m / K} (j_\ell ( A_{\ell, F_{m'}}))$.\\
Let $\cI^{(\ell)} \subseteq G$ be the inertia subgroup at $\ell$, and $\pi_\ell \: \Z_p [G] \to \Z_p [G / \cI^{(\ell)}]$ the natural restriction map. It then follows Lemma \ref{unr cohom lem}\,(c) that $\pi_\ell ( \widetilde{A_\ell})$ agrees with $A_\ell \coloneqq A_{\ell, K}$. The exact sequence from Lemma \ref{unr cohom lem}\,(b) therefore fits into a commutative diagram of the form
\begin{cdiagram}
0 \arrow{r} &  \Z_p [G]^{\oplus n_\ell} \arrow{r}{\cdot \widetilde{A_\ell}} \arrow[twoheadrightarrow]{d}{\pi_\ell^{\oplus n_\ell}} &
    \Z_p [G]^{\oplus n_\ell} \arrow{r} \arrow[twoheadrightarrow]{d}{\pi_\ell^{\oplus n_\ell}} & \coker ( j_\ell (A_\ell)) \arrow{r} \arrow[dashed, twoheadrightarrow]{d} & 0 \\
    0 \arrow{r} & \Z_p [G / \cI^{(\ell)}]^{\oplus n_\ell} \arrow{r}{\cdot A_\ell} & \Z_p [G / \cI^{(\ell)}]^{\oplus n_\ell} \arrow{r} & (E_0 / E_1) ((K^{\mathcal{I}_\ell})_\ell) \arrow{r} & 0.
\end{cdiagram}%
The surjectivity of the dashed arrow follows from the snake lemma, and multiplication by $j_\ell (A_\ell)$ is injective by \cite[Ch.\@ III, \S\,8.2, Prop.\@ 3]{bourbaki} because
\[
\det ( j_\ell (A_{\ell, F_{m'}})) = j_\ell (\det (A_{\ell, F_{m'}})) = j_\ell ( \Eul_\ell ( \sigma_\ell^{-1})) = \ell \cdot \Eul_\ell ( \tilde \sigma_\ell^{-1})
\]
is a nonzero divisor. Indeed, for this it is enough to prove that $\chi ( \Eul_\ell ( \tilde \sigma_\ell^{-1})) = \Eul_\ell ( \chi (\tilde \sigma_\ell^{-1}))$ is nonzero for ever character $\chi$ of $G$, which is true because the polynomial $\Eul_\ell(X)$ has no root of complex absolute value 1.\\
Note that $(E_0 / E_1) ((K^{\mathcal{I}_\ell})_\ell)$ and $(E_0 / E_1) (K_\ell)_\ell)$ agree if $\ell$ is not an additive ramified prime because in any such case $K$ and $K^{\mathcal{I}_\ell}$ have the same residue field at $\ell$. If $\ell$ is additive and ramified in $K$, then $\ell \neq p$ by the assumed validity of Hypothesis \ref{hyp}\,(iii) and so $(E_0 / E_1) ((K^{\mathcal{I}_\ell})_\ell) = 0$.
Writing $D_{\ell}^\bullet$ for the perfect complex $\big [  \Z_p [G]^{\oplus n_\ell} \xrightarrow{\cdot \widetilde{A_\ell}} \Z_p [G]^{\oplus n_\ell} \big ]$, where the first term is placed in degree zero, in all cases we have therefore constructed a composite morphism
\[
\rho_{\ell} \: D_{\ell}^\bullet = \coker ( \widetilde{A_\ell})) [-1] \to  (E_0 / E_1) ((K^{\mathcal{I}_\ell})_\ell) [-1] \to (E_0 / E_1) (K_\ell) [-1]
\]
in $D (\Z_p [G])$. We also note that $D_{\ell}^\bullet$ is acyclic outside degree one and that one has
\begin{align*}
\Det_{\Z_p [G]} ( D_{\ell}^\bullet) & = \det ( \widetilde{A_\ell}) \cdot \Z_p [G] = \ell \cdot \Eul_\ell (\tilde \sigma_\ell^{-1}) \cdot \Z_p [G].
\end{align*}
It is then clear from the construction that $D^\bullet_\ell$ has all the properties claimed in (i) -- (iii). Moreover, it follows from Lemma \ref{unr cohom lem}\,(b) that $\rho_\ell$ is a quasi-isomorphism if $\ell$ is unramified in $K$ and $p \nmid \Tam_{K, \ell}$.
This proves claim (a)\,(iv), thereby concluding the proof of (a).\\ 
To prove claim (b), we take $\rho_{p, Q}$ to be any map that makes the diagram
\begin{cdiagram}[column sep=small]
\Z_p [G] [-1] \arrow{d}{1 \mapsto Q}  \arrow{r} & \Z_p [G][-1] \oplus D_{p}^\bullet \arrow[dashed]{d}{\rho_{p, Q}} \arrow{r} & 
    D_{p}^\bullet \arrow{d}{\rho_{p}} \arrow{r} & \phantom{X} \\ 
    E_1 ( K_p) [-1] \arrow{r} & \mathrm{R}\Gamma_f (\Q_\ell, T_{K / \Q}) \arrow{r} & (E / E_1) (K_p) [-1] \arrow{r} & \phantom{X}
\end{cdiagram}%
   commute. 
\end{proof}

\begin{rk} \label{Tamagawa number modified local complex rk}
    If one wants to account for Tamagawa numbers, on can modify $D^\bullet_\ell$ as follows. Since $p$ is odd, each quotient $E ( K_v) / E_0 (K_v)$ is a cyclic group by the Kodaira--N\'eron Theorem and hence
    $(E / E_0) ( K_p)$ is a cyclic $\Z_p [G]$-module generated by $t_\ell$, say.
    We therefore have a well-defined map $\Z_p [G] / \Tam_{K, \ell} \Z_p [G] \to (E / E_0) (K_\ell)$ induced by sending 1 to $t_\ell$. We then obtain a morphism
    \[
    \rho_{\ell, \mathrm{Tam}, 0} \: D_{\ell, \mathrm{Tam}}^\bullet \coloneqq \big [ 
    \Z_p [G] \xrightarrow{\cdot \Tam_{K, \ell}} \Z_p [G] \big] \to (E / E_0) ( K_\ell) [-1]
    \]
in $D (\Z_p [G])$ by taking $\rho_{\ell, \mathrm{Tam}}^0$ to be the zero map and $\rho_{\ell, \mathrm{Tam}}^1$ the map that sends $1$ to $t_\ell$. If we define $D^\bullet_{\ell, \mathrm{Tam}} \coloneqq D^\bullet_{\ell} \oplus D_{\ell, \mathrm{Tam}}^\bullet$ and $\rho_{\ell, \mathrm{Tam}} \coloneqq \rho_{\ell} \oplus \rho_{\ell, \mathrm{Tam}}$, then the modified complex $D^\bullet_{\ell, \mathrm{Tam}}$ also has all of the properties listed in Lemma \ref{approximation complexes}\,(a) as long as one replaces (iii) by
\begin{align*}
\vartheta_{D_{\ell, \mathrm{Tam}}^\bullet, \emptyset} (\Det_{\Z_p [G]} ( D^\bullet_{\ell, \mathrm{Tam}})) & = \vartheta_{D_{\ell}^\bullet, \emptyset} (\Det_{\Z_p [G]} ( D^\bullet_{\ell} )) \cdot \vartheta_{D_{\ell, \mathrm{Tam},0}^\bullet, \emptyset} ( \Det_{\Z_p [G]} ( D_{\ell, \mathrm{Tam},0}^\bullet)) \\
& = \ell \cdot \Eul_\ell (\tilde \sigma_\ell^{-1}) \cdot \Tam_{K, \ell} \cdot \Z_p [G].
\end{align*}
By defining $D_{p, Q, \mathrm{Tam}} \coloneqq D^\bullet_{p, \mathrm{Tam}} \oplus (\Z_p [G] [-1])$ one can then also adapt the construction made in Lemma \ref{approximation complexes}\,(b).
\end{rk}

\subsection{Global finite-support cohomology}
 In this section we use the complexes from Lemma ~\ref{approximation complexes} to define an auxiliary Nekov\'a\v{r}--Selmer complex. To do this,
let $\Pi$ be a finite set of finite places of $K$, fix $Q \in E_1 (K_p)$, and 
define a 
Selmer structure $\cF_{\Sigma, \Pi, Q}$ on $T_{K / \Q}$ as follows. We take $S (\cF_{\Sigma, \Pi, Q}) \coloneqq \Sigma \cup \Pi$, set
\[
\mathrm{R}\Gamma_{\cF_{\Sigma, \Pi, Q}} (\Q_\ell, T_{K / \Q}) \coloneqq
\begin{cases}
    D^\bullet_{\ell}\quad & \text{ if }  \ell \in \Pi \setminus \{ p \}, \\
    D^\bullet_{p, Q} & \text{ if } \ell = p \in \Pi, \\
    \Z_p [G] [-1] & \text{ if } \ell = p \not \in \Pi, \\
    0 & \text{ otherwise,}
\end{cases}
\]
and define $i_{\cF_{\Sigma, \Pi, Q}, v}$ to be the composite map 
\[
\mathrm{R}\Gamma_{\cF_{\Sigma, \Pi, Q}} (\Q_v, T_{K / \Q}) \to 
\mathrm{R} \Gamma_f ( \Q_v, T_{K / \Q}) \to \mathrm{R} \Gamma ( \Q_v, T_{K / \Q}), 
\]
where the first arrow is $\rho_\ell$, $\rho_{p, Q}$ (both constructed in Lemma \ref{approximation complexes}), or the zero map as appropriate. If $p$ is a good prime that ramifies in $K$, then by $\rho_{p, Q}$ we mean the restriction of $\rho_{p, Q}$ to $\Z_p [G][-1]$, so the map induced by sending $1$ to $Q$.\\
To lighten notation, we then introduce the abbreviations
\[
\widetilde D_{K, \Sigma, \Pi, Q}^\bullet \coloneqq \mathrm{R} \Gamma_{\cF_{\Sigma, \Pi, Q}} (K, \mathrm{T}_p E)
\quad \text{ and } \quad 
D_{K, \Sigma, \Pi, Q} \coloneqq \RHom_{\Z_p} ( \widetilde D_{K, \Sigma, \Pi, Q}^\bullet, \Z_p) [-3].
\]

The next result records the basic properties of these complexes.

\begin{lem} \label{approximation complex properties lemma}
Assume that $E (K) [p] = 0$ and that $p$ is unramified in $K$ if $p$ is an additive prime for $E$.
For every $Q \in E_1 (K_p)$, finite set $\Sigma \supseteq S(K)$ of places of $\Q$, and finite set $\Pi$ of finite places of $\Q$, the following claims are valid.
\begin{liste}
\item One has an inclusion
\[
\Fitt^0_{\Z_p [G]} ( H^2 ( \widetilde D^\bullet_{K, \Sigma, \Pi, Q})) \subseteq \Fitt^0_{\Z_p [G]} (H^2_f (K, \mathrm{T}_p E)).
\]
\item Let $F$ be a subfield of $K$ and define $Q' \coloneqq \mathrm{Tr}_{K / F} (Q)$. Then one has
\[
\pi_{K / F} ( \Fitt^0_{\Z_p [G_K]} ( H^2 ( \widetilde D^\bullet_{K, \Sigma, \Pi, Q}))) = \Fitt^0_{\Z_p [G_F]} ( H^2 ( \widetilde D^\bullet_{F, \Sigma, \Pi, Q'})).
\]
\item Let $\Sigma'$ be a finite set of places of $K$ that contains $\Sigma$. Then one has
an exact triangle
\begin{cdiagram}[column sep=small]
    \bigoplus_{\ell \in (\Sigma' \setminus \Sigma)} \mathrm{R} \Gamma_f (\Q_\ell, T_{K / \Q}) \arrow{r} & \widetilde D^\bullet_{K, \Sigma', \Pi, Q} \arrow{r} & \widetilde D_{K, \Sigma, \Pi, Q}^\bullet \arrow{r} & \bigoplus_{\ell \in (\Sigma' \setminus \Sigma)} \mathrm{R} \Gamma_f (\Q_\ell, T_{K / \Q})[1].
\end{cdiagram}
\item If $\ell \in \Pi \setminus S(K)$, then $\widetilde D^\bullet_{K, \Sigma, \Pi, Q} = \widetilde D^\bullet_{K, \Sigma \setminus \{ \ell\}, \Pi \setminus \{ \ell \}, Q}$.
\end{liste}
\end{lem}

\begin{proof}
At the outset we define 
\[
\widetilde D^\bullet_{\mathrm{loc}, \Pi, Q} \coloneqq 
\bigoplus_{\ell \in \Pi \cup \{ p \}} \mathrm{R}\Gamma_{\cF_{\Sigma, \Pi, Q}} (\Q_\ell, T_{K / \Q})
\]
and $\rho_\Pi \coloneqq \rho_{p, Q} \oplus \bigoplus_{\ell \in \Pi} \rho_\ell$, 
and note that the definition of $D^\bullet_{K, \Sigma, \Pi, Q}$ implies that there is an exact triangle
\begin{equation} \label{triangle for tilde D}
    \begin{tikzcd}[column sep=small]
          \mathrm{R} \Gamma_c (\cO_{K, \Sigma}, \mathrm{T}_p E) 
        \arrow{r} & \widetilde D^\bullet_{K, \Sigma, \Pi, Q} \arrow{r}
        & \widetilde D^\bullet_{\mathrm{loc}, \Pi, Q} \arrow{r}  & \mathrm{R} \Gamma_c (\cO_{K, \Sigma}, \mathrm{T}_p E) [1].
    \end{tikzcd}
\end{equation}
This combines with the definition of the complex $\mathrm{R} \Gamma_f (K, \mathrm{T}_p E)$ and the octahedral axiom to yield an exact triangle of the form\begin{equation} \label{diagram comparing D and RGamma_f}
\begin{tikzcd}[column sep=small]
    \widetilde D^\bullet_{K, \Sigma, \Pi, Q} \arrow{r} & \mathrm{R} \Gamma_f (K, \mathrm{T}_p E) \arrow{r} & \mathrm{cone} (\rho_{\Pi}) \oplus \bigoplus_{\ell \in \Sigma \setminus (\Pi \cup \{ p \}} \mathrm{R}\Gamma_f (\Q_\ell, T_{K / \Q}) \arrow{r} & \widetilde D^\bullet_{K, \Sigma, \Pi, Q} [1].
\end{tikzcd}%
\end{equation}
Since the cone of the morphism $\rho_\Pi$ and $\bigoplus_{\ell \in \Sigma \setminus \Pi} \mathrm{R}\Gamma_f (\Q_\ell, T_{K / \Q})$ both have no nonzero cohomology in degrees greater than one, this exact triangle induces a surjection
\[
H^2 ( \widetilde D^\bullet_{K, \Sigma, \Pi, Q}) \twoheadrightarrow H^2_f (K, \mathrm{T}_p E).
\]
A standard property of Fitting ideals then combines with this surjection to imply the inclusion claimed in (a). Similarly, the properties of Fitting ideals reduce claim (b) to the claim that one has an isomorphism
\[
H^2 ( \widetilde D^\bullet_{K, \Sigma, \Pi, Q}) \otimes_{\Z_p [G_K]} \Z_p [G_F] \cong H^2 ( \widetilde D^\bullet_{K, \Sigma, \Pi, Q'}).
\]
Now, the triangle (\ref{triangle for tilde D}) combines with the assumption $E (K) [p] = 0$ to imply $\widetilde D^\bullet_{K, \Sigma, \Pi, Q}$ is acyclic in degree greater than two and so the desired isomorphism follows from Lemma~\ref{properties top and bottom cohomology}~(a) if we can demonstrate that one has an isomorphism
\[
\widetilde D^\bullet_{K, \Sigma, \Pi, Q} \otimes_{\Z_p [G_K]}^\mathbb{L} \Z_p [G_F] 
\cong 
\widetilde D^\bullet_{F, \Sigma, \Pi, Q'}
\]
in $D (\Z_p [G_F])$. To do this, we first note that one has 
$D^\bullet_{p, Q} \otimes_{\Z_p [G_K]}^\mathbb{L} \Z_p [G_F]  \cong  D^\bullet_{p, Q'}$ and 
$D^\bullet_{\ell} \otimes_{\Z_p [G_K]}^\mathbb{L} \Z_p [G_F]  \cong  D^\bullet_{\ell}$,
as can be checked from Lemma \ref{unr cohom lem}\,(c) and the explicit definitions of these complexes.
In addition, one has commutative diagrams
\begin{cdiagram}
D^\bullet_{K, p, Q} \otimes_{\Z_p [G_K]}^\mathbb{L} \Z_p [G_F] \arrow{d}{\simeq} \arrow{r}{\overline{\rho_{K, p, Q}}} & E (K_p) [-1] \arrow{d}{\mathrm{Tr}_{K / F}} 
&  D^\bullet_{K, \ell} \otimes_{\Z_p [G_K]}^\mathbb{L} \Z_p [G_F] \arrow{d}{\simeq} \arrow{r}{\overline{\rho_{K, \ell}}} & 
(E / E_0)(K_\ell) [-1] \arrow{d}{\mathrm{Tr}_{K / F}}
\\
D_{F, p, Q'} \arrow{r}{\rho_{F, p, Q}} & E (F_p) [-1] &  D_{F, \ell}^\bullet \arrow{r}{\rho_{F, \ell}} & (E / E_0) (F_\ell),
\end{cdiagram}%
where we have introduced subscripts $K$ and $F$ to emphasise the field of definition, and $\overline{\rho_{K, p, Q}}$ and $\overline{\rho_{K, \ell}}$ denote the maps induced by $\rho_{K, p, Q}$ and $\rho_{K, \ell}$, respectively.
The claimed property of $\widetilde D^\bullet_{K, \Sigma, \Pi, Q}$ then follows by combining this with the triangle (\ref{triangle for tilde D}) and the well-known isomorphism
\[
\mathrm{R} \Gamma_c (\cO_{K, \Sigma}, \mathrm{T}_p E)  \otimes_{\Z_p [G_K]}^\mathbb{L} \Z_p [G_F]  \cong \mathrm{R} \Gamma_c (\cO_{F, \Sigma}, \mathrm{T}_p E),
\]
(see, for example, \cite[Prop.\@ 1.6.5\,(3)]{FukayaKato}). Claim (c) follows from (\ref{triangle for tilde D}) and the triangle
\[
    \bigoplus_{\ell \in (\Sigma' \setminus \Sigma)} \mathrm{R} \Gamma_f (\Q_\ell, T_{K / \Q}) \to \mathrm{R} \Gamma_c (\cO_{K, \Sigma'}, \mathrm{T}_p E) \to \mathrm{R}\Gamma_c (\cO_{K, \Sigma}, \mathrm{T}_p E) \to  \bigoplus_{\ell \in (\Sigma' \setminus \Sigma)} \mathrm{R} \Gamma_f (\Q_\ell, T_{K / \Q}) [1]
\]
via the octrahedral axiom. Finally, claim (d) is a consequence of the fact that, if $\ell \in \Pi \setminus S (K)$ (and so $\ell \neq p$ is a prime of good reduction that is unramified in $K$), then the map $\rho_\ell \: D^\bullet_\ell \to \mathrm{R}\Gamma_f (\Q_\ell, T_{K / \Q})$ is a quasi-isomomorphism by Lemma \ref{approximation complexes}\,(a)\,(iv) (cf.\@ also the argument of \cite[Lem.\@ 2.5]{BurnsMaciasCastillo}). 
\end{proof}

\begin{lem} \label{propeties D complex}
Assume $E (K) [p] = 0$ and that $p$ is unramified in $K$ if $p$ is an additive prime for $E$. For given $Q \in E_1 (K_p)$, consider the composite morphism
\begin{align*}
\rho'_{p, Q} \:  C^\bullet_{K, \Sigma} & \longrightarrow \mathrm{R}\Gamma_{/f} (\Q_p, T_{K / \Q}) \cong \RHom_{\Z_p} ( \mathrm{R}\Gamma_f (\Q_p, T_{K / \Q}), \Z_p) [-2] \\
& \xrightarrow{\rho_{p, Q}^\ast} \RHom_{\Z_p} (D^\bullet_{p, Q}, \Z_p)[-2] = \RHom_{\Z_p} (D^\bullet_{p}, \Z_p)[-2] \oplus \Z_p [G][-1] \\
& \longrightarrow \Z_p [G][-1].
\end{align*}
Here the first arrow is the natural localisation map composed with the morphism $\mathrm{R}\Gamma (\Q_p, T_{K / \Q}) \to \mathrm{R}\Gamma_{/f} (\Q_p, T_{K / \Q}) $, the isomorphism is by local Tate duality (\ref{Tate duality RGamma_f}), $\rho^\ast_{p, Q} \coloneqq \RHom_{\Z_p} (\rho_{p, Q}, \Z_p) [-2]$ is the dual of the map $\rho_{p, Q}$ defined in Lemma \ref{approximation complexes}\,(b), and the last arrow is projection onto the direct summand $\Z_p [G][-1]$.
Then the map $H^1 (\rho'_{p, Q})$ induced by $\rho'_{p, Q}$ on cohomology in degree one 
coincides with the composite map
\[
 H^1 (\cO_{K, \Sigma}, \mathrm{T}_p E) \to H^1_{/ f} (\Q_p, T_{K / \Q}) \xrightarrow{\cP_K (\cdot, Q)} \Z_p [G].
\]
\end{lem}

\begin{proof}
Using the (dual of the) representative of $C^\bullet_{K, \Sigma}$ constructed in Lemma \ref{complex lemma}, we have a commutative diagram
\begin{cdiagram}
0 \arrow{d} \arrow{r} & \Z_p [G] \arrow{d}{\rho^{2}_{p, Q}} \arrow[equals]{r} & \Z_p [G] \arrow{d}{H^2 (\rho_{p, Q})} \arrow{r} & 0\\ 
    P \arrow{r}{\phi^\ast} & P \arrow{r} & H^2_c (\cO_{K, \Sigma}, \mathrm{T}_p E) \arrow{r} & 0, 
\end{cdiagram}%
where the first commutative square represents the morphism $\rho_{p, Q}$ 
composed with the morphism $\mathrm{R}\Gamma_f (\Q_p, T_{K / \Q}) [-1] \to \mathrm{R}\Gamma_c ( \cO_{K, \Sigma}, \mathrm{T}_p E) \cong \RHom_{\Z_p} ( C_{K,\Sigma}^\bullet, \Z_p) [-3]$
(and $\rho^{2}_{p, Q}$ is the map induced by $\rho_{p, Q}$ in degree two). By dualising, we therefore obtain the commutative diagram
\begin{cdiagram}
    0 \arrow{r} & H^1 (\cO_{K, \Sigma}, \mathrm{T}_p E) \arrow{d}{H^0 (\rho'_{p, Q})} \arrow{r} & P \arrow{d}{(\rho_{p, Q}')^0} \arrow{r}{\phi} & P \arrow{d} \\ 
    0 \arrow{r} & \Z_p [G] \arrow[equals]{r} & \Z_p [G] \arrow{r} & 0,
\end{cdiagram}%
    where now the second commutative square represents the morphism $\rho'_{p, Q}$ (and $(\rho'_{p, Q})^0$ is the map induced by $\rho'_{p, Q}$ in degree zero). In particular, $H^0 (\rho'_{p, Q})$ coincides with the dual of $H^2 (\rho_{p, Q})$. The claim therefore follows from the observation that $H^2 (\varrho_{p, Q})$ is the composite map
    \[
    \Z_p [G] \xrightarrow{1 \mapsto Q}  H^1_f (\Q_p, T_{K / \Q}) \to H^2_c (\cO_{K, \Sigma}, \mathrm{T}_p E)
    \]
    so that its dual map is given by the composite
    \[
    H^1 (\cO_{K, \Sigma}, \mathrm{T}_p E) \to H^1_{/ f} (\Q_p, T_{K / \Q}) \cong H^1_f (\Q_p, T_{K / \Q})^\ast \to \Z_p [G].
    \]
    Here the first arrow is the localisation map (followed by projection), the isomorphism is induced by the pairing $(\cdot, \cdot)_{E / K}$, and the last arrow is evaluation at $Q$. 
\end{proof}

\subsection{Determinantal ideals}

In this section we establish a connection between elements in the image of the map $\Theta_{K, \Sigma}$ (defined before the statement of Theorem \ref{etnc result 2}) and the Selmer complexes $\widetilde D_{K, \Sigma, \Pi, Q}$ introduced in the last section. Our main result in this direction is as follows. 

\begin{prop} \label{etnc explicit cor}
Assume that $E (K) [p] = 0$ and that $p$ is unramified in $K$ if $E$ has additive reduction at $p$. Fix $Q \in E_1 (K_p)$, and let $\Sigma$ be a finite set of places of $K$ that contains $S(K)$. If $\Pi \subseteq \Sigma$ is a subset of rational primes that are either unramified in $K$ or at which $E$ has multiplicative reduction, then for every 
\begin{itemize}
    \item $a \in \Det_{\Z_p [G]} (C^\bullet_{K, \Sigma})^{-1}$,
    \item $t \in \prod_{\ell \in \Pi} \Ann_{\Z_p [G]} ( ( E/ E_0) (K_\ell))^{-1, \#}$.
    \item $\nu \in  \big( \prod_{\ell \in \Pi} (\ell^{-1} \Eul_\ell (\widetilde \sigma_\ell^{-1})^{-1} \NN_{\cI_K^{(\ell)}} \Z_p [G] + \Z_p [G]  \big)$, 
\end{itemize} 
one has the containments
\begin{align} \label{prop 4.6 first inclusion}
 \nu \cdot \cP_K
(\Theta_{K, \Sigma}(a), Q) & \in 
\Fitt^0_{\Z_p [G]} (H^2 (\widetilde D^\bullet_{K, \Sigma, \Pi, Q}))^{\#}, \\
 t \cdot \nu \cdot \cP_K
(\Theta_{K, \Sigma}(a), Q) & \in 
\Z_p [G].
\label{prop 4.6 second inclusion}
\end{align}
\end{prop}

Before proving Proposition \ref{etnc explicit cor}, we first establish an auxiliary result that concerns
the maps 
\[
\vartheta_{K, \Sigma, \Pi, Q} \coloneqq \vartheta_{D^\bullet_{K, \Sigma, \Pi, Q}, \{ b^\ast \}}
\quad \text{ and } \quad
\widetilde \vartheta_{K, \Sigma, \Pi, Q} \coloneqq \vartheta_{\widetilde D^\bullet_{K, \Sigma, \Pi, Q}, \emptyset}
\]
(with $b^\ast$ the dual of the $\Z_p [G]$-basis $b$ of $T_{K / \Q}^+$ from \S\,\ref{s: cpctly supported etale cohom}) defined as the relevant instances of Definition \ref{def projection map from det}.\\
Setting $D^\bullet_{\mathrm{loc}, \Pi, Q} \coloneqq \RHom_{\Z_p} ( \widetilde D^\bullet_{\mathrm{loc}, \Pi, Q}, \Z_p) [-2]$ and dualising the triangle (\ref{triangle for tilde D}) gives an exact triangle
\begin{equation} \label{triangle comparing C and D}
\begin{tikzcd}[column sep=small]
    D^\bullet_{K, \Sigma, \Pi, Q} \arrow{r} & C^\bullet_{K, \Sigma} \arrow{rrr}{\rho_{\mathrm{loc, \Pi, Q}}} & & & D^\bullet_{\mathrm{loc}, \Pi, Q} \arrow{r} & D^\bullet_{K, \Sigma, \Pi, Q} [1]. 
\end{tikzcd}
\end{equation}
To state the next lemma below, we will make use of the isomorphism
\begin{align}
\nonumber
\Det_{\Z_p [G]} (C^\bullet_{K, \Sigma})^{-1}  
& \cong \Det_{\Z_p [G]} ( D^\bullet_{K, \Sigma, \Pi, Q})^{-1} \otimes \Det_{\Z_p [G]} ( D^\bullet_{\mathrm{loc}, \Pi, Q})^{-1}  \\ \label{isomorphisms of dets}
& \cong \Det_{\Z_p [G]} ( D^\bullet_{K, \Sigma, \Pi, Q})^{-1}.
\end{align}
Here the first isomorphism is induced by the the triangle (\ref{triangle comparing C and D}) and the second isomorphism is induced by the isomorphism
\[
(\otimes_{\ell \in \Pi} \mathrm{Ev} (a_\ell)) \: 
\Det_{\Z_p [G]} ( D^\bullet_{\mathrm{loc}, \Pi, Q})^{-1} \xrightarrow{\simeq} \Det_{\Z_p [G]} ( \Z_p [G] [-1])^{-1} = \Z_p [G]
\]
with $a_\ell$ the canonical element of $\Det_{\Z_p [G]} ( \RHom_{\Z_p} ( D^\bullet_\ell, \Z_p)) \cong \Det_{\Z_p [G]} (D^\bullet_\ell)^{-1, \#}$ that satisfies $\vartheta_{D^\bullet_{\ell}, \emptyset} (a_\ell) = \ell \cdot \Eul_\ell (\sigma_\ell)$ (which exists by Lemma \ref{approximation complexes}\,(a)\,(iii)).

\begin{lem} \label{etnc explicit lem}
Let  $Q \in E_1 (K_p)$. If $E(K) [p] = 0$, then the following claims are valid.
\begin{liste}
\item Put $\eta_\Pi \coloneqq \prod_{\ell \in \Pi } (\ell \cdot \Eul_\ell (\tilde \sigma_\ell))$ for brevity. Then the following diagram commutes.
\begin{cdiagram}
    \Det_{\Z_p [G]} (C^\bullet_{K, \Sigma})^{-1} \arrow{r}{\Theta_{K, \Sigma}} \arrow{d}[right]{\simeq}[left]{(\ref{isomorphisms of dets})} & H^1 (\cO_{K, \Sigma}, \mathrm{T}_p E) \arrow{d}{\pm \eta_\Pi^{-1} \cdot \cP_K (\cdot, Q)} \\
    \Det_{\Z_p [G]} ( D^\bullet_{K, \Sigma, \Pi, Q})^{-1} \arrow{r}{\vartheta_{K, \Sigma, \Pi, Q}} & \Q_p [G].
\end{cdiagram}%
    \item One has the equality $\vartheta_{K, \Sigma, \Pi, Q} ( \Det_{\Z_p [G]} ( D^\bullet_{K, \Sigma, \Pi, Q})^{-1}) = \widetilde \vartheta_{K, \Sigma, \Pi, Q} ( \Det_{\Z_p [G]} ( \widetilde D^\bullet_{K, \Sigma, \Pi, Q})^{-1})^\#$ and an inclusion
    \[
    \Fitt^0_{\Z_p [G]} ( H^1 (\widetilde D_{K, \Sigma, \Pi, Q})_\tor^\vee ) \cdot 
     \widetilde \vartheta_{K, \Sigma, \Pi, Q} ( \Det_{\Z_p [G]} ( \widetilde D^\bullet_{K, \Sigma, \Pi, Q})^{-1}) \subseteq \Fitt^0_{\Z_p [G]} ( H^2 ( \widetilde D_{K, \Sigma, \Pi, Q}^\bullet)).
    \]
    \end{liste}
\end{lem}

\begin{proof}
Claim (a) follows from Proposition \ref{functoriality properties of det projection map}\,(a)\,(i), applied for every $\mathrm{Ev}_{a_\ell}$, and  Proposition~\ref{functoriality properties of det projection map}~(b) (combined with Lemma \ref{propeties D complex}). \\
To justify claim (b), we first note that, by Lemma \ref{approximation complexes}\,(a)\,(ii), the complex $D^\bullet_{\mathrm{loc}, \Pi, Q}$ admits a representative of the form $[F \oplus \Z_p [G] \xrightarrow{\partial \oplus 0} F]$ for some free $\Z_p [G]$-module $F$ of finite rank and endomorphism $\partial$ of $F$. In the following we set $F' \coloneqq F \oplus \Z_p [G]$ and fix a $\Z_p [G]$-basis $y_1, \dots, y_s$ of $F'$ such that $F$ is identified with $\bigoplus_{i = 2}^s \Z_p [G] y_i$. \\
Write $[P \stackrel{\phi}{\to} P]$ for the representative of $C^\bullet_{K, \Sigma}$ provided by Lemma \ref{complex lemma}. That is, $P$ is a free $\Z_p [G]$-module of rank $d$, say, with basis $x_1, \dots, x_d$, and $P = P' \oplus (T_{K / \Q})^+$ with $P' \coloneqq \bigoplus_{i = 2}^d \Z_p [G] x_i$.
From this choice of bases we see that, by (\ref{triangle comparing C and D}) and the definition of the mapping cone, that $D^\bullet_{K, \Sigma, \Pi, Q}$ is represented by $[P \stackrel{\phi'}{\to} P \oplus F' \stackrel{\partial'}{\to} F ]$ with the first term placed in degree one. 
Here $\phi' = - (\phi \oplus \rho^1)$ and $\partial' = - \rho^2 + \partial$ with $\rho^i$ the component in degree $i$ of the morphism $\rho \: C^\bullet_{K, S} \to D^\bullet_{\mathrm{loc}, \Pi, Q}$.\\
We next recall from (\ref{dual sharp isomorphism}) that for every $\Z_p [G]$-module $M$ we have a canonical isomorphism $\Hom_{\Z_p} (M, \Z_p) \xrightarrow{\simeq} \Hom_{\Z_p [G]} (M, \Z_p [G])^\#$.
It follows that the complex $\widetilde D^\bullet_{K, \Sigma, \Pi, Q}$,
which is isomorphic to the (shifted) $\Z_p$-linear dual of $D^\bullet_{K, \Sigma, \Pi, Q}$, can be represented by $[F \xrightarrow{{\partial}^{\#, \mathrm{tr}}} P \oplus F' \xrightarrow{{\phi'}^{\#, \mathrm{tr}}} P]$ with the first term placed in degree zero. Here the maps ${\partial'}^{\#, \mathrm{tr}}$ and ${\phi'}^{\#, \mathrm{tr}}$ are defined by means of the matrices that are obtained by applying the involution $\#$ to the transpose of the matrices representing $\partial'$ and $\phi'$, respectively. 
This shows that $\widetilde{D}^\bullet_{K, \Sigma, \Pi, Q}$ admits a standard representative in the sense of Definition \ref{standard representative def} and so the inclusion claimed in (b) is a special case of Proposition~\ref{integrality properties of det projection map}~(b). 
\\
To prove the remainder of claim (b), we note that the complex $[F \xrightarrow{{\partial'}^{\#, \mathrm{tr}}} P' \oplus F' \xrightarrow{{\phi'}^{\#, \mathrm{tr}}} P]$ is a perfect complex of $\Z_p [G]$-modules that is both acyclic outside degrees one and two (this uses that $\partial^{\#, \mathrm{tr}}$, and hence also ${\partial'}^{\#, \mathrm{tr}}$, is injective) and has vanishing Euler characteristic in $K_0 (\Z_p [G])$. 
We may then apply Proposition \ref{functoriality properties of det projection map}\,(c) to deduce that
\begin{align*}
    \widetilde \vartheta_{K, \Sigma, \Pi, Q} ( \Det_{\Z_p [G]} ( \widetilde D^\bullet_{K, \Sigma, \Pi, Q})^{-1})^\# & \stackrel{(\ref{dual sharp isomorphism})}{=}
    \widetilde \vartheta_{K, \Sigma, \Pi, Q} ( \Det_{\Z_p [G]} ( \RHom_{\Z_p [G]} (D^\bullet_{K, \Sigma, \Pi, Q}, \Z_p [G]))^{-1}) \\
    & = \vartheta_{K, \Sigma, \Pi, Q} ( \Det_{\Z_p [G]} ( D^\bullet_{K, \Sigma, \Pi, Q})^{-1}),
\end{align*}
as claimed.
\end{proof}

We can now give the proof of Proposition \ref{etnc explicit cor}. 
\begin{proof}[Proof (of Proposition \ref{etnc explicit cor}):]
We first explain how to deduce the inclusion (\ref{prop 4.6 first inclusion}) from Lemma \ref{etnc explicit lem}. To do this, we note that, as $H^1_f (K, \mathrm{T}_p E) = E (K) \otimes_{\Z} \Z_p$ is assumed to be $\Z_p$-torsion free, it follows from the triangle (\ref{diagram comparing D and RGamma_f}) that we have an identification
\begin{align*}
 H^1 (\widetilde D_{K, \Sigma, \Pi, Q} )_\tor  & = H^0 \big( \mathrm{cone} (\widetilde \rho_{\mathrm{loc}, \Pi, Q} ) \big)_\mathrm{tor} \\
& = \ker \big \{ {\bigoplus}_{\ell \in \Pi} \big( \Z_p [G]^{n_\ell} / (j_\ell (A_\ell)) \big) \longrightarrow {\bigoplus}_{\ell \in \Pi} (E / E_0) (K_\ell) \big \} \\ 
& = \ker \big \{ {\bigoplus}_{\ell \in \Pi_\mathrm{bad}} \big( \Z_p [G] / (\ell \Eul_\ell (\tilde \sigma_\ell))  \big) \longrightarrow {\bigoplus}_{\ell \in \Pi_\mathrm{bad}} (E / E_0) (K_\ell) \big \} \\
& = {\bigoplus}_{\ell \in \Pi_\mathrm{bad}} \big( \Ann_{\Z_p [G]} ( (E_0 / E_1) (K_\ell)) / (\ell \Eul_\ell (\tilde \sigma_\ell)) \big).
\end{align*}
Here $\Pi_\mathrm{bad} \subseteq \Pi$ denotes the subset of places at which $E$ has bad reduction (in which case $n_\ell = 1$ and $j_\ell (A_\ell) = \ell \Eul_\ell (\tilde \sigma_\ell)$), and we have used that any prime $\ell \in \Pi \setminus \Pi_\mathrm{bad}$ is by assumption unramified in $K$ and so does not contribute to the above kernel because $\rho_\ell$ is an quasi-isomorphism for any such $\ell$ by Lemma \ref{approximation complexes}.\\
Before proceeding, it is convenient to first make a general observation concerning an ideal $\a$ of $\Z_p [G]$ that contains a nonzero divisor $x$. In any such situation one has 
\begin{equation} \label{fun isomorphism}
\big ( \a / \Z_p [G] x \big)^\vee \cong \big( (\Z_p [G] x)^\ast / \a^\ast \big)^\# \cong \big( ( x^{-1} \Z_p [G]) / \a^{-1} \big)^\# \cong \Z_p [G] / (x \a^{-1})^\#,  
\end{equation}
where we have used the isomorphism $\a^\ast \cong \a^{-1} \coloneqq \{ y \in \Q_p [G] \mid y \a \subseteq \Z_p [G] \}$ that is valid because $\a$ contains a nonzero divisor (cf.\@ \cite[Prop.\@ 6.1\,(4)]{bass}). From (\ref{fun isomorphism}) we then conclude that 
$\Fitt^0_{\Z_p [G]} ( ( \a / \Z_p [G] x)^\vee ) = (x \a^{-1})^{\#}$. This general observation applied with $\a$ and $x$ taken to be $\Ann_{\Z_p [G]} ( (E_0 / E_1) (K_\ell))$ and $ \ell \Eul_\ell (\tilde \sigma_\ell^{-1})$, respectively, shows that $\Fitt^0_{\Z_p [G]} ( H^1 (\widetilde D_{K, \Sigma, \Pi, Q} )_\tor^\vee) $ is equal to the product ideal
\[
{\prod}_{\ell \in \Pi_\mathrm{bad}}  \ell \Eul_\ell (\tilde \sigma_\ell^{-1})  \cdot \Ann_{\Z_p [G]} ( (E_0 / E_1) (K_\ell))^{-1, \#}.
\]
Now, the factors $ \ell \Eul_\ell (\tilde \sigma_\ell^{-1})$ cancel with the corresponding factors in the definition of the element $\eta_\Pi^{-1}$ from Lemma \ref{etnc explicit lem}, and so Lemma \ref{etnc explicit lem}\,(a) implies that
\[
\big( {\prod}_{\ell \in \Pi_\mathrm{bad}} \Ann_{\Z_p [G]} ( (E_0 / E_1) (K_\ell))^{-1, \#} \big) \cdot
\cP_K (\Theta_{K, \Sigma} (a), Q)
\]
is contained in
\[
\big( {\prod}_{\ell \in \Pi \setminus \Pi_\mathrm{bad}} \ell \Eul_\ell (\tilde \sigma_\ell^{-1}) \big) \cdot \Fitt^0_{\Z_p [G]} (H^2 (\widetilde D^\bullet_{K, \Sigma, \Pi, Q}))^{\#}.
\]
It now only remains to observe that, for every $\ell \in \Pi_\mathrm{bad}$, one has that $\Ann_{\Z_p [G]} ( (E_0 / E_1) (K_\ell))^{-1, \#}$ contains $\ell^{-1} \Eul_\ell (\tilde \sigma_\ell)^{-1} \NN_{\cI_K^{(\ell)}} \Z_p [G] + \Z_p [G]$ because 
$\Ann_{\Z_p [G]} ( (E_0 / E_1) (K_\ell))$ is generated by $\Eul_\ell (\tilde \sigma_\ell^{-1})$ and the augmentation ideal $I (\cI_K^{(\ell)})$ as a consequence of Lemma \ref{unr cohom lem} and the assumption that $\Pi_\mathrm{bad}$ contains no additive primes. 
Since $\cI_K^{(\ell)}$ is trivial for any $\ell \in \Pi \setminus \Pi_\mathrm{bad}$ by assumption (and so $\ell^{-1} \Eul_\ell (\tilde \sigma_\ell)^{-1} \NN_{\cI_K^{(\ell)}} \Z_p [G] + \Z_p [G]$ simplifies to $\ell^{-1}\Eul_\ell (\tilde \sigma_\ell)^{-1} \Z_p [G]$ for such $\ell$), this concludes the proof of (\ref{prop 4.6 first inclusion}). \\
As for the proof of (\ref{prop 4.6 second inclusion}), this follows from a very similar argument and so we only provide a sketch.
Using the complexes defined in Remark \ref{Tamagawa number modified local complex rk}, we may define a modified Nekov\'a\v{r}--Selmer structure $\cF^\mathrm{Tam}_{\Sigma, \Pi, Q}$ by taking $S (\cF^\Tam_{\Sigma, \Pi, Q}) \coloneqq \Sigma \cup \Pi$ and 
\[
\mathrm{R}\Gamma_{\cF^\Tam_{\Sigma, \Pi, Q}} (\Q_\ell, T_{K / \Q}) \coloneqq
\begin{cases}
    D^\bullet_{\ell, \Tam}\quad & \text{ if }  \ell \in \Pi \setminus \{ p \}, \\
    D^\bullet_{p, Q, \Tam} & \text{ if } \ell = p \in \Pi, \\
    \Z_p [G] [-1] & \text{ if } \ell = p \not \in \Pi, \\
    0 & \text{ otherwise.}
\end{cases}
\]
Now, if we set $\eta_\Pi^\Tam \coloneqq \prod_{\ell \in \Pi} ( \Tam_{K, \ell} \cdot \ell \cdot \Eul_\ell (\sigma_\ell^{-1}))$
and $\widetilde D_{\Tam} \coloneqq \mathrm{R}\Gamma_{\cF^\Tam_{\Sigma, \Pi, Q}} (\cO_{K, \Sigma}, \mathrm{T}_p E)$, then the argument of Lemma \ref{etnc explicit lem}\,(a) shows that we have a commutative diagram of the form 
\begin{cdiagram}
    \Det_{\Z_p [G]} (C^\bullet_{K, \Sigma})^{-1} \arrow{r}{\Theta_{K, \Sigma}} \arrow{d}[right]{\simeq} & H^1 (\cO_{K, \Sigma}, \mathrm{T}_p E) \arrow{d}{\pm (\eta_\Pi^\Tam)^{-1} \cdot \cP_K (\cdot, Q)} \\
    \Det_{\Z_p [G]} ( \RHom_{\Z_p} ( \widetilde D_{\Tam}, \Z_p) [-3])^{-1} \arrow{r}{\vartheta^\Tam} & \Q_p [G],
\end{cdiagram}%
where the map $\vartheta^\Tam$ is defined as the relevant special case of Definition \ref{def projection map from det}. In addition, the argument of Lemma \ref{etnc explicit lem}\,(b) shows that 
\begin{align*}
    \Fitt^0_{\Z_p [G]} ( H^1 (\widetilde D_{\Tam})_\tor^\vee ) \cdot 
     (\im \vartheta^\Tam)^\#  & \subseteq \Fitt^0_{\Z_p [G]} ( H^2 ( \widetilde D_{\Tam}^\bullet)) \subseteq \Z_p [G].
    \end{align*}
The inclusion (\ref{prop 4.6 second inclusion}) can therefore be proved in exactly the same way as the inclusion (\ref{prop 4.6 first inclusion}) once we have observed that $H^1 (\widetilde D_{\Tam} )_\tor$ identifies with
\[
{\bigoplus}_{\ell \in \Pi_\mathrm{bad}} \big( \Ann_{\Z_p [G]} ( (E / E_0) (K_\ell)) / (\Tam_{K, \ell}) \big) \oplus \big( \Ann_{\Z_p [G]} ( (E_0 / E_1) (K_\ell)) / (\ell \Eul_\ell (\tilde \sigma_\ell)) \big).
\qedhere
\]
\end{proof}

\subsection{The proof of Theorem \ref{mazur--tate main result 1}\,(b)}

In this section we prove Theorem \ref{mazur--tate main result 1}\,(b), our main result towards the `weak main conjecture' of Mazur and Tate. If the Euler factors $\Eul_\ell (\tilde \sigma_\ell)^{-1}$ are invertible in $\Z_p [G_{mp^n}]$ for all prime divisors $\ell$ of $m$ and
the element $\mathfrak{k}_{mp^n}$ from Theorem \ref{local points main result} is in $E_1 (K_p)$, then this is a straightforward consequence of Theorem \ref{etnc result 2} and Proposition \ref{etnc explicit cor}. However, this will not be the case in general and so a more detailed analysis is required in order to prove Theorem~\ref{mazur--tate main result 1}~(b). This will
be done in the rather technical Lemma \ref{inductive argument} below, where we will use the Euler system norm relations in order to prove that
the denominators arising from factors of the form $\Eul_\ell (\tilde \sigma_\ell)^{-1}$ can be `absorbed' by $z^\mathrm{Kato}_{mp^n}$. \\
Before stating this result, we first give a different characterisation of the set of prime numbers $C^{(p)}_\times (K)$ that was defined in Remark \ref{mazur--tate main result 1 rk}\,(c).

\begin{lem} \label{invertible Euler factors lemma}
    A prime number $\ell$ belongs to $C_\times^{(p)} (K)$ if and only if $\ell \cdot \Eul_\ell (\tilde \sigma_\ell) \in \Z_p [G]^\times$.
\end{lem}

\begin{proof}
     Let $G^{(p)} \subseteq G$ denote the $p$-Sylow subgroup of $G$ and set $\Delta \coloneqq G / G^{(p)}$. We further write $\widehat{\Delta}$ for the group of all 
     characters $\chi \: \Delta \to \overline{\Q_p}^\times$ and let $\cO$ be the unramified extension of $\Z_p$ that is generated by the values of all $\chi \in \widehat{\Delta}$.
     An element of $\Z_p [G]$ is then a unit if and only if it is a unit in the integral extension $\cO [G]$ of $\Z_p [G]$. Now, one has a decomposition $\cO [G] \cong \bigoplus_{\chi \in \widehat{\Delta}} \cO [G^{(p)}]$ induced by the isomorphism 
     $ 
     \cO [\Delta] \xrightarrow{\simeq} {\bigoplus}_{\chi \in \widehat{\Delta}} \cO,  x \mapsto (\chi (x))_\chi$.
     Each $\cO [G^{(p)}]$ is a local ring with maximal ideal $(p, I_{\cO, G^{(p)}})$, so this argument proves that an element $x$ of $\Z_p [G]$ is a unit if and only if $\chi (x) \not \equiv 0 \mod p$ for all $\chi \in \widehat{\Delta}$. To prove that $\ell \cdot \Eul_\ell (\tilde \sigma_\ell) \in \Z_p [G]^\times$ if and only if $\ell$ satisfies the explicit condition given in the definition of $C^{(p)}_\times$ in Remark \ref{mazur--tate main result 1}\,(c) it now suffices to prove that the set $\{ \chi (\tilde \sigma_\ell) : \chi \in \widehat{\Delta}\}$ coincides with the group of $f^{(p)}_{K / \Q}$-th roots of unity of $\overline{\Q_p}^\times$. For this it is in turn enough to prove that the order of $(\tilde \sigma_\ell)_{\mid_K}$ in $G$ is equal to the residue degree of $\ell$ in $K / \Q$ or, equivalently, the order of $(\sigma_\ell)_{\mid_{K'}}$ with $K'$ the maximal subextension of $K$ unramified at $\ell$. To do this, we write $m$ for the conductor of $K$ and set $M \coloneqq m \ell^{-\ord_\ell (m)}$. Then $K' = K \cap F_M$ so that the restriction map induces  an isomorphism $\gal{F_m}{K} \cong \gal{F_M}{K'}$. Given this, it follows directly from the definition of $\tilde \sigma_\ell$ that a power $\tilde \sigma_\ell^n$ belongs to $\gal{F_m}{K}$ if and only if $\sigma_\ell^n$ belongs to $\gal{F_M}{K'}$, as required to prove the claim.
\end{proof}

To prepare for the statement of the next result, we write the conductor of $K$ as $mp^n$ with $m\in \N$ coprime to $p$ and  $n \in \Z_{\geq 0}$. Given a subset $\mathscr{L} \subseteq S_{mp^n}$, we set $(mp^n)_\mathscr{L} \coloneqq mp^n \cdot (\prod_{\ell \in \mathscr{L}} \ell^{\ord_\ell (mp^n)} )^{-1}$ and $y^\mathrm{Kato}_{\mathscr{L}} \coloneqq y^\mathrm{Kato}_{(mp^n)_\mathscr{L}}$. Write $F_\mathscr{L} \coloneqq F_{(mp^n)_\mathscr{L}}$ and $G_\mathscr{L} \coloneqq G_{(mp^n)_\mathscr{L}}$. We also define the following sets of primes,
\begin{align*}
\mathscr{Z} (K)  \coloneqq \{ \ell \mid mp^n : \ell \not \in C^{(p)}_\times (K) \}
\quad \text{ and } \quad 
\mathscr{Y} (K)  \coloneqq 
\{ \ell \mid m :  \ell\in C^{(p)}_\times (K) \}.
\end{align*}
\begin{lem} \label{inductive argument}
Suppose $(\a_\mathscr{L})_{\mathscr{L} \subseteq \mathscr{Z} (K)}$ is a collection of ideals $\a_\mathscr{L} \subseteq \Z_p [G_{\mathscr{L}}]$ with the property that
    \begin{equation} \label{assumption}
(p\Eul_p (\sigma_p^{-1}))^{-\bm{1}_\mathscr{L} (p)} \cdot 
    \cP_{F_\mathscr{L}} ( y^\mathrm{Kato}_{\mathscr{L}}, \mathrm{Tr}_{F_{mp^n} / F_\mathscr{L}} (Q))
    \in \a_\mathscr{L}
    \end{equation}
    for all $Q \in \widehat{E} (F_{m, p})$. Then
    \begin{equation}\label{e: required inclusion}
        (1 - e_\tau) \theta^\mathrm{MT}_{K} \in ( {\prod}_{\ell \in \mathscr{Y} (K)} \nu^{(\ell)}_{mp^n} )^{\#} \cdot \pi_{F_{mp^n} / K } \big ( {\sum}_{\mathscr{L} \subseteq \mathscr{Z} (K)} \a_\mathscr{L} \NN_{F_{mp^n} / F_\mathscr{L}} \big) 
        \quad \subseteq \Z_p [G].
    \end{equation}
    If $a_p \not \equiv 1 \mod p$ and $p$ is either of good reduction for $E$ a prime of potentially good reduction for $E$, then the same conclusion holds with $(1 - e_\tau) \theta^\mathrm{MT}_{K}$ replaced by $\theta_{K}^\mathrm{MT}$. 
\end{lem}

\begin{proof}
This is an extension of the argument used by Otsuki in \cite[Lem.\@ 4.2]{Otsuki}.\\
At the outset we note that, by Lemma \ref{invertible Euler factors lemma} the Euler factor $\Eul_\ell (\tilde \sigma_\ell)$ is a unit in $\Z_p [G]$ for a prime number $\ell \mid m$ if and only if $\ell \in \mathscr{Y} (K)$. In order to simplify some statements later on it is convenient to set $\nu^{(\ell)}_{mp^n} \coloneqq 1$ if $\ell \mid N$ and $\ell \nmid m$. From the equation 
\begin{align*} 
 (1 - e_\tau) \theta_{mp^n}^\mathrm{MT} & = \pi_{F_{mp^n} / K} \Big( \big ( {\prod}_{\ell \in S_m } \Eul_\ell (\tilde \sigma_\ell)^{-1} \cdot \nu^{(\ell)}_{mp^n} \big)^{\#} \cdot \cP_{F_{mp^n}} (y^\mathrm{Kato}_{mp^n}, (1 - e_\tau) \mathfrak{k}_{mp^n}) \Big)
\end{align*}
proved in Theorem \ref{local points main result}\,(a), one sees that it is enough to prove that 
\[
\big ( \prod_{\ell \in S_{m} \setminus \mathscr{Y} (K)} \Eul_\ell (\tilde \sigma_\ell^{-1})^{-1} \cdot \nu^{(\ell), {\#}}_{mp^n} \big) \cdot \cP_{F_{mp^n}} ( y^\mathrm{Kato}_{mp^n},(1 - e_\tau) \mathfrak{k}_{mp^n})
\in \sum_{\mathscr{L} \subseteq \mathscr{Z} (K)} \a_\mathscr{L} \NN_{F_{mp^n} / F_\mathscr{L}}.
\]
By Theorem \ref{local points main result}\,(e) we can write, for each prime $\ell \mid m$, 
\[
\Eul_\ell (\widetilde \sigma_\ell)^{-1} \nu_{mp^n}^{(\ell)} = \alpha_\ell + \beta_\ell \Eul_\ell (\widetilde \sigma_\ell)^{-1} \NN_{\cI^{(\ell)}_{mp^n}} 
\]
with suitable $\alpha_\ell, \beta_\ell \in \Z_p [G_{mp^n}]$. For any subset $\mathscr{L}$ of $\mathscr{Z} (K)$, we then define $\cI_\mathscr{L}$ to be the subgroup of $G_{mp^n}$ generated by $\cI_{mp^n}^{(\ell)}$ for $\ell \in \mathscr{L}$ and note that $\cI_\mathscr{L} = \prod_{\ell \in \mathscr{L}} \cI_{mp^n}^{(\ell)}$ because the $\cI_{mp^n}^{(\ell)}$ are disjoint. 
As a consequence, we also have $\NN_{F_{mp^n} / F_\mathscr{L}} = \prod_{\ell \in \mathscr{L}} \NN_{\cI_{mp^n}^{(\ell)}}$.\\
We can then write
\begin{align}
\big ( {\prod}_{\ell \in S_{m} \setminus \mathscr{Y} (K)} \Eul_\ell (\tilde \sigma_\ell^{-1})^{-1} \cdot \nu^{(\ell), {\#}}_{mp^n} \big) & = \big ( {\prod}_{\ell \in \mathscr{Z} (K) } \Eul_\ell (\tilde \sigma_\ell^{-1})^{-1} \cdot \nu^{(\ell), {\#}}_{mp^n} \big)  \notag \\ 
& = \big ( {\prod}_{\ell  \in \mathscr{Z} (K)} (\alpha_\ell^{\#} + \beta_\ell^{\#} \Eul_\ell (\widetilde \sigma_\ell^{-1})^{-1} \NN_{\cI^{(\ell)}_{mp^n}}) \big)  \cdot  \notag
 \\
& =  {\sum}_{\mathscr{L} \subseteq \mathscr{Z} (K) } A_\mathscr{L} ({\prod}_{\ell \in \mathscr{L}} \Eul_\ell (\widetilde \sigma_\ell^{-1} )^{-1})\NN_{F_{mp^n} / F_\mathscr{L}}  \label{e: prod of euler factors}
\end{align}
with suitable $A_\mathscr{L} \in \Z_p [G_{mp^n}]$ and the sum ranging over all subsets of $\mathscr{Z} (K)$ (including the empty set). We then calculate, by Lemma~\ref{pairing properties}~(b, c), that for every subset $\mathscr{L}$ of $\mathscr{Z} (K)$
\begin{align}
    &\NN_{F_{mp^n} / F_\mathscr{L}} \cdot \cP_{F_{mp^n}} ( y^\mathrm{Kato}_{mp^n}, (1 - e_\tau)\mathfrak{k}_{mp^n}) \notag \\ &= \cP_{F_{mp^n}} ( \cores_{F_{mp^n} / F_\mathscr{L}} (y^\mathrm{Kato}_{mp^n}), (1 - e_\tau)\mathfrak{k}_{mp^n}) \notag \\
    &= \cP_{F_\mathscr{L}} ( \cores_{F_{mp^n} / F_\mathscr{L}} (y^\mathrm{Kato}_{mp^n}), \mathrm{Tr}_{F_{mp^n} / F_\mathscr{L}}( (1 - e_\tau)\mathfrak{k}_{mp^n})) \cdot \NN_{F_{mp^n} / F_\mathscr{L}} . \label{e: norm and pairing}
\end{align}

In addition, by Theorem \ref{local points main result}\,(b), one has
\begin{equation} \label{decomposition frak k}
(1 - e_\tau) \mathfrak{k}_{mp^n} = Q + (p \Eul_p (\sigma_p))^{-1} P
\end{equation}
with suitable $Q \in \widehat{E} (F_{mp^n, p})$ and $P \in \widehat{E} (F_{m,p})$. For every subset $\mathscr{L}$ of $\mathscr{Z} (K)$, we then have
\begin{align}
 &\cP_{F_\mathscr{L}} ( \ \cdot \ , \mathrm{Tr}_{F_{mp^n} / F_\mathscr{L}} ((1 - e_\tau)\mathfrak{k}_{mp^n})) \notag \\ & = 
 \cP_{F_\mathscr{L}} (\ \cdot \ , \mathrm{Tr}_{F_{mp^n} / F_\mathscr{L}} ( Q + (p \Eul_p (\sigma_p))^{-1} P)) \notag \\
 & = \cP_{F_\mathscr{L}} (\ \cdot \ , \mathrm{Tr}_{F_{mp^n} / F_\mathscr{L}} (Q)) \notag \\
 & \quad + (p \Eul_p (\sigma_p^{-1}))^{-1} \cdot \cP_{F_{\mathscr{L} \cup \{ p \}}} (\NN_{F_\mathscr{L} / F_{\mathscr{L} \cup \{ p \}}} \cdot, \mathrm{Tr}_{F_{mp^n} / F_\mathscr{L}} ( P)) \NN_{\cI_{mp^n}^{(p)}} \label{e: pairing after idempotent}
\end{align}
because $\cP_{F_{mp^n}} (\cdot, \cdot)$ is $\#$-semilinear in the second component by Lemma \ref{pairing properties}\,(b). Combining (\ref{e: prod of euler factors}), (\ref{e: norm and pairing}) and (\ref{e: pairing after idempotent}) we have thereby proved that
\begin{align*}
& \big ( {\prod}_{\ell \in S_{m} \setminus \mathscr{Y} (K) } \Eul_\ell (\tilde \sigma_\ell^{-1})^{-1} \cdot \nu^{(\ell), {\#}}_{mp^n} \big) \cdot \cP_{F_{mp^n}} ( y^\mathrm{Kato}_{mp^n}, (1 - e_\tau) \mathfrak{k}_{mp^n}) \notag \\ 
= & {\sum}_{\mathscr{L} \subseteq \mathscr{Z} (K)} A'_\mathscr{L} ({\prod}_{\ell \in \mathscr{L} \setminus \{ p \}} \Eul_\ell (\widetilde \sigma_\ell^{-1} )^{-1}) \cdot 
    \cP_{F_\mathscr{L}} (  \cores_{F_{mp^n} / F_\mathscr{L}} (z^\mathrm{Kato}_{mp^n}), Q_\mathscr{L} ) \NN_{F_{mp^n} / F_\mathscr{L} } 
\end{align*}
with
\[
A'_\mathscr{L} \coloneqq 
(p\Eul_p ( \sigma_p^{-1}))^{-\bm{1}_\mathscr{L} (p)} \cdot A_{\mathscr{L} \setminus \{ p \}}
\quad \text{ and } \quad 
Q_\mathscr{L} \coloneqq \begin{cases}  Q & \text{ if } \mathscr{L} = \emptyset, \\
\mathrm{Tr}_{F_{mp^n} / F_\mathscr{L}} (Q)  & \text{ if } p \notin \mathscr{L} , \\
     \mathrm{Tr}_{F_{mp^n} / F_\mathscr{L}} (P) & \text{ if } p \in \mathscr{L}.
\end{cases}
\]
In light of the Euler system relation $\cores_{F_{mp^n} / F_\mathscr{L}} (y^\mathrm{Kato}_{mp^n}) = (\prod_{\ell \in \mathscr{L} \setminus \{ p \}} \Eul_\ell (\sigma_\ell^{-1}) \big) \cdot y_{\mathscr{L}}^\mathrm{Kato}$ we have therefore proved that
\begin{equation}
\label{e: putting together}
(1 - e_\tau) \theta^\MT_{mp^n} = {\sum}_{\mathscr{L} \subseteq \mathscr{Z} (K)} A'_\mathscr{L}  \cdot 
    \cP_{F_\mathscr{L}} (  y_{\mathscr{L}}^\mathrm{Kato}, Q_\mathscr{L} ) \NN_{F_{mp^n} / F_\mathscr{L} },
\end{equation}
and this combines with the assumption (\ref{assumption}) to imply the claimed inclusion (\ref{e: required inclusion}) after applying $\pi_{F_{mp^n} / K}$ to both sides. \\
If $a_p \not \equiv 1 \mod p$, then $\mathfrak{k}_{mp^n}$ belongs to $\widehat{E} (F_{mp^n, p})$ by Theorem \ref{local points main result}\,(b) 
so that (\ref{decomposition frak k}) holds without the factor $(1 - e_\tau)$ and with $P = 0$. With this changed definition of $P$, (\ref{e: putting together}) then holds without the factor $(1 - e_\tau)$. Again considering the assumption (\ref{assumption}) one sees that the second claim of the theorem holds as well.
\end{proof}

\begin{rk} \label{inductive argument variant}
    The proof of Lemma \ref{inductive argument} shows that if $p$ is a prime number with the property that $p \Eul_p (\sigma_p)$ belongs to $\Z_p [G_{mp^n}]^\times$, then (\ref{assumption}) can be replaced by the simpler condition
    \[
    \eta_{m,N_p} \cdot 
     \cP_{F_\mathscr{L}} ( y^\mathrm{Kato}_{\mathscr{L}}, (1 - e_\tau) \cdot \mathrm{Tr}_{F_{mp^n} / F_\mathscr{L}} (\mathfrak{k}_{mp^n}))
    \in \a_\mathscr{L}
    \]
    in order for (\ref{e: required inclusion}) to hold.
\end{rk} 

We now give the proof of Theorem \ref{mazur--tate main result 1}\,(b).
\medskip \\ 
\textit{Proof of Theorem \ref{mazur--tate main result 1}\,(b)}:
We set
\[
\Pi_\mathscr{L} \coloneqq \begin{cases}
    \{ p \} \cup (S_N \setminus \mathscr{Y} (K)) \quad & \text{ if } p \in \mathscr{L},\\
    S_N \setminus \mathscr{Y} (K) & \text{ if } p \notin \mathscr{L},
\end{cases}
\]
and define an element of $\Z_p [G_\mathscr{L}]$ as
$
\eta_\mathscr{L} \coloneqq \prod_{\ell \in \Pi_\mathscr{L}} (\ell \Eul_\ell (\sigma_\ell^{-1}))^{-1}$.
Now, assuming condition (ii) of Theorem \ref{mazur--tate main result 1}, it follows from Theorem \ref{etnc result 2} that there is $a_\mathscr{L} \in \Det_{\Z_p [G_{\mathscr{L}}]} (C^\bullet_{F_\mathscr{L}, S(F_\mathscr{L})})^{-1}$ with $\Theta_{F_\mathscr{L}, S(F_\mathscr{L})} (a_\mathscr{L}) = z^\mathrm{Kato}_{\mathscr{L}}$. 
Write the conductor of $K$ as $mp^n$ with $p \nmid m$ and $n \in \Z_{\geq 0}$.
Then, dnoting by $\sim$ equality up to a unit in $\Z_p [G_\mathscr{L}]$, we deduce from the first inclusion in Proposition \ref{etnc explicit cor} that
\begin{align*}
\eta_{\mathscr{L}} \cdot \mathscr{P}_{F_\mathscr{L}} ( z^\mathrm{Kato}_{\mathscr{L}}, Q_\mathscr{L}) 
&  \sim (p \Eul_p (\sigma_p^{-1}))^{\bm{1}_\mathscr{L} (p)} \cdot \mathscr{P}_{F_\mathscr{L}} ( y^\mathrm{Kato}_{\mathscr{L}}, Q_\mathscr{L}) \\
&
\in \Fitt^0_{\Z_p [G_{F_\mathscr{L}}]} ( H^2 ( \widetilde D^\bullet_{F_\mathscr{L}, S(F_\mathscr{L}), \Pi_\mathscr{L}, Q_\mathscr{L}} ))^\#
\end{align*}
for every $Q_\mathscr{L} \coloneqq \mathrm{Tr}_{F_{mp^n} / F_\mathscr{L}} (Q)$ with
$Q \in E_1 (F_{mp^n, p})$.  Setting $\Pi^\prime_\mathscr{L} \coloneqq \Pi_\mathscr{L} \cup ( S(F_{mp^n}) \setminus S(F_\mathscr{L}) )$, Lemma \ref{approximation complex properties lemma}\,(d) gives
\[ 
H^2 ( \widetilde D^\bullet_{F_\mathscr{L}, S(F_\mathscr{L}), \Pi_\mathscr{L}, Q_\mathscr{L}} )) = H^2 ( \widetilde D^\bullet_{F_\mathscr{L}, S(F_{mp^n}), \Pi^\prime_\mathscr{L}, Q_\mathscr{L}} ). 
\]
This combines with Lemma \ref{inductive argument} to imply that
\begin{align} \nonumber
    (1 - e_\tau) \theta^\mathrm{MT}_{K} & \in \pi_{F_{mp^n} / K} \big ( {\sum}_{\mathscr{L} \subseteq \mathscr{Z} (K)} 
    \Fitt^0_{\Z_p [G_{F_\mathscr{L}}]} ( H^2 ( \widetilde D^\bullet_{F_\mathscr{L}, S(F_{mp^n}), \Pi^\prime_\mathscr{L}, Q_\mathscr{L}} ))^\# \NN_{\cI_\mathscr{L}} \big) \\ 
    &\subseteq  \pi_{F_{mp^n} / K} \big ({\sum}_{\mathscr{L} \subseteq \mathscr{Z} (K)} \Fitt^0_{\Z_p [G_{mp^n}]} ( H^2 ( \widetilde D^\bullet_{F_{mp^n}, S(F_{mp^n}), \Pi^\prime_\mathscr{L}, Q} ))^\# \big)
 \nonumber \\ \label{deduction}
 &\subseteq  \pi_{F_{mp^n} / K} \big ( \Fitt^0_{\Z_p [G_{mp^n}]} ( \Sel_{p, E / F_{mp^n}}^\vee)^\# \big).
\end{align}
Here the first inclusion follows from Lemma \ref{approximation complex properties lemma}\,(b), in particular from the fact that
\[
H^2 ( \widetilde D^\bullet_{F_{mp^n}, S(F_{mp^n}), \Pi_\mathscr{L}, Q} ) \otimes_{\Z_p [G_{mp^n}]} \Z_p [G_{\mathscr{L}}] \cong
H^2 ( \widetilde D^\bullet_{F_\mathscr{L}, S(F_{mp^n}), \Pi_\mathscr{L}, Q_\mathscr{L}} ),
\]
by Lemma \ref{properties top and bottom cohomology}\,(a),
and the properties of Fitting ideals. The
second inclusion follows from Lemma \ref{approximation complex properties lemma}\,(a) and Lemma \ref{description finite support cohomology}. \\
We next note that, since $E(F_{mp^n})$ has trivial $p$-torsion, the natural map $H^1 (\cO_{K, S(K)}, E[p^\infty]) \to H^1 (\cO_{F_{mp^n}, S(K)}, E[p^\infty])$ is injective and so restricts to an injection $\Sel_{p, E / K } \hookrightarrow \Sel_{p, E / F_{mp^n}}$. Upon taking Pontryagin duals, we therefore deduce a surjection $\Sel_{p, E / F_{mp^n}}^\vee \twoheadrightarrow \Sel_{p, E / K}^\vee$. By a standard property of Fitting ideals, the existence of this surjection implies an inclusion
\begin{align*}
\pi_{F_{mp^n} / K} ( \Fitt^0_{\Z_p [G_{mp^n}]} ( \Sel_{p, E / F_{mp^n}}^\vee)) & = 
\Fitt^0_{\Z_p [G_K]} ( \Sel_{p, E / F_{mp^n}}^\vee \otimes_{\Z_p} \Z_p [G_K])\\
& \subseteq \Fitt^0_{\Z_p [G_K]} ( \Sel_{p, E / K}^\vee), 
\end{align*}
which combines with (\ref{deduction}) to prove that $\theta^{\mathrm{MT}, \#}_{K}$ is contained in $\Fitt^0_{\Z_p [G_K]} (\Sel_{p, E / K}^\vee)$ if $K$ contains no primitive $p$-th root of unity.\\ 
If $a_p \not \equiv 1 \mod p$ and $E$ has potentially good reduction at $p$, then one can again use Theorem~\ref{etnc result 2} and Lemma~\ref{inductive argument} to deduce (\ref{deduction}), (without the factor of $(1 - e_\tau)$), and so we conclude that $\theta^{\mathrm{MT}, \#}_{K}$ is contained in $\Fitt^0_{\Z_p [G_K]} (\Sel_{p, E / K}^\vee)$.\\
This concludes the proof of Theorem \ref{mazur--tate main result 1}\,(b). \qed

\section{The multiplicative group} \label{multiplicative group section}

In this section we prove a number of auxiliary results that are concerned 
with the multiplicative group $\mathbb{G}_m$ and that will be key in the proofs of Theorem \ref{mazur--tate main result 1} (a) and Theorem \ref{mazur--tate main result 2}.

 \subsection{A cohomological interpretation of Otsuki's points}

Fix a finite abelian extension $K$ of $\Q$ with Galois group $G \coloneqq \gal{K}{\Q}$ and consider the $\mathscr{G}_\Q$-module $\Z_p (1)_{K / \Q} \coloneqq \mathrm{Ind}_{\mathscr{G}_\Q}^{\mathscr{G}_K} (\Z_p (1))$. Upon fixing a prime number $\ell$ and an embedding $\iota_\ell \: \overline{\Q} \hookrightarrow \overline{\Q_\ell}$, we may regard $\mathscr{G}_{\Q_\ell}$ as a subgroup of $\mathscr{G}_\Q$, and hence $\Z_p (1)_{K / \Q}$ as a $\mathscr{G}_{\Q_\ell}$-module.
Consequently, we have the complex
\[
A^\bullet_{K, \ell} \coloneqq \mathrm{R}\Gamma (\Q_\ell, \Z_p (1)_{K / \Q})
\cong {\bigoplus}_{v \mid \ell} \mathrm{R}\Gamma (K_v, \Z_p (1))
, 
\]
which is perfect as an object of $D (\Z_p [G])$ and acyclic outside degrees 1 and 2. Moreover, one has canonical isomorphisms
\[
H^1 ( A^\bullet_{K, \ell}) \cong \widehat{K_\ell^{\times}} \coloneqq {\bigotimes}_{v \mid \ell} \widehat{K_v^{\times}}
\quad \text{ and } \quad
H^2 ( A^\bullet_{K, \ell}) \cong {\bigotimes}_{v \mid \ell} \Z_pv
\]
induced by the Kummer map and the invariant map of local class field theory, respectively. 
The Euler characteristic of $A^\bullet_{K, \ell}$ in $K_0 (\Z_p [G])$ is equal to $- [\Z_p [G]^{\oplus (1 - \bm{1}_p (\ell))}]$ (cf.\@ \cite[\S\,5]{Flach00}), and so
Definition \ref{def projection map from det} provides us with a map 
\[
\vartheta_{K, \ell}^0 \coloneqq 
\vartheta_{A^\bullet_{K, \ell}, \emptyset}
\: \Det_{\Z_p [G]} ( A^\bullet_{K, \ell})^{-1} \to \Q_p \otimes_{\Z_p} \exprod^{1 - \bm{1}_p (\ell)}_{\Z_p [G]} \widehat{K_\ell^{\times}}.
\]
Write $v_0$ for the place of $K$ that corresponds to the restriction of $\iota_\ell$ to $K$. 
If $\ell$ splits completely in $K$, then $v_0$ defines a 
$\Z_p [G]$-basis of $\bigoplus_{v \mid \ell} \Z_p v$
and hence Definition \ref{def projection map from det} also gives a map
\[
\vartheta_{K, \ell, v_0} \coloneqq  
\vartheta_{A^\bullet_{K, \ell}, \{ v_0 \}}
\: \Det_{\Z_p [G]} ( A^\bullet_{K, \ell})^{-1} \to 
\Q_p \otimes_{\Z_p} \exprod^{2 - \bm{1}_p (\ell)}_{\Z_p [G]} \widehat{K_\ell^{\times}}.
\]
We moreover recall that, after letting $L / \Q$ denote another finite abelian extension that contains $K$ and setting $\cG \coloneqq \gal{L}{\Q}$, by \cite[Prop.\@ 1.6.5]{FukayaKato} one has an isomorphism
\begin{equation} \label{base change isom for A complex}
A^\bullet_{L, \ell} \otimes_{\Z_p [\cG]}^\mathbb{L} \Z_p [G] \cong A^\bullet_{K, \ell}
\end{equation}
in $D (\Z_p [G])$ that induces a map
\[
\mathrm{pr}_{L / K} \: \Det_{\Z_p [\cG]} ( A^\bullet_{L, \ell})^{-1} \to \Det_{\Z_p [\cG]} ( A^\bullet_{L, \ell})^{-1} \otimes_{\Z_p [\cG]} \Z_p [G] \cong \Det_{\Z_p [G]} ( A^\bullet_{K, \ell})^{-1}.
\]

If $E$ has split-multiplicative reduction at $p$, Tate uniformisation induces an isomorphism $F \: \widehat{E} \stackrel{\simeq}{\to} \mathbb{G}_m$ given by $(\exp_{\mathbb{G}_m} \circ \log_{\widehat{E}} (1 + X)) - 1 \in \Z_p \llbracket X \rrbracket$ (for details see, for example, \cite[\S\,3]{Kobayashi06}).

\begin{definition}
    Suppose that $E$ has split-multiplicative reduction at $p$. For every natural number $m$ coprime with $p$ and integer $n \geq 0$, we define
    \[
    \mathfrak{l}_{mp^n} \coloneqq F ( \widetilde x_{mp^n}) \in (F_{mp^n, v_0}^\times)^\wedge \subseteq H^1 (A^\bullet_{F_{mp^n}, p})
    \]
    with the element $\widetilde x_{mp^n}$ from Definition \ref{def x tilde}. 
    (Here we have used that $p \Eul_p (\tilde \sigma_p) = p - \tilde \sigma_p$ is invertible in $\Z_p [G]$ and hence that $\widetilde x_{mp^n}$ belongs to $\widehat{E} (\cM_{F_{mp^n, v_0}})$.)
\end{definition}

The following is the main result of this section.

\begin{thm} \label{multiplicative group thm}
Fix a natural number $m$ coprime with $p$ and an integer $n\geq 0$. For every prime divisor $\ell$ of $mp^n$ at which $E$ has split-multiplicative reduction the following claims are valid.
\begin{liste}
    \item If $\ell \neq p$, then there exists a unique family $(t_{mp^n}^{(\ell)})_{n \in \N} \in \varprojlim_{n \in \N} \Det_{\Z_p [G_{mp^n}]} ( A_{F_{mp^n}, \ell}^\bullet)^{-1}$, where the limit is taken with respect to the maps $\mathrm{pr}_{F_{mp^{n + 1} / F_{mp^n}}}$, such that
    \[
    \vartheta_{F_{mp^n}, \ell}^0 ( t_{mp^n}^{(\ell)}) = \Eul_\ell (\widetilde \sigma_\ell)^{-1} \cdot \nu_{mp^n}^{(\ell)}
    \]
    for all $n \in \N$. If $K$ is a subfield of $F_{mp^n}$ in which $\ell$ splits completely, then moreover
    \[
    (\vartheta_{K, \ell, v_0} \circ \mathrm{pr}_{F_{mp^n} / K}) ( t_{mp^n}^{(\ell)}) = - \ell^{- (\ord_\ell (m) - 1)} \Eul_\ell (1)^{-1} \otimes \ell 
    \]
    as an equality in $\Q_p \otimes_{\Z_p} \widehat{K_{v_0}^\times} = \Q_p \otimes_{\Z_p} \widehat{\Q_\ell^\times} \subseteq \Q_p \otimes_{\Z_p} \widehat{K_\ell^\times}$.
    \item If $\ell = p$, then there exists a unique family $(t_{mp^n}^{(p)})_{n \in \N} \in \varprojlim_{n \in \N} \Det_{\Z_p [G_{mp^n}]} ( A_{F_{mp^n}, p})^{-1}$ such that
    \[
    \vartheta_{F_{mp^n}, p}^0 ( t_{mp^n}^{(p)}) = \mathfrak{l}_{mp^n}
    \]
    for all $n \in \N$. If $K$ is a subfield of $F_{mp^n}$ in which $p$ splits completely, then moreover
    \[
    (\vartheta_{K, p, v_0} \circ \mathrm{pr}_{F_{mp^n} / K}) ( t_{mp^n}^{(p)}) =  \NN_{F_m / K }(\mathfrak{l}_m) \wedge p \quad \in  \exprod^2_{\Z_p} \widehat{K_{v_0}^\times}. 
    \]
    (Here the right hand side is viewed as an element of $\exprod^2_{\Z_p [G]} \widehat{K^\times_\ell}$ via the isomorphism $\exprod^2_{\Z_p [G]} \widehat{K^\times_\ell} \cong \Z_p [G] \otimes_{\Z_p} \exprod^2_{\Z_p} \widehat{K_{v_0}^\times}$.)
\end{liste}
\end{thm}

The proof of this result will occupy the remainder of this section.

\begin{rk} \label{epsilon constant rk}
    The element $\mathfrak{l}_1 \wedge p = (p \Eul_p (1))^{-1} \exp_{\mathbb{G_m}} (p) \wedge p$ is a $\Z_p$-basis of $\exprod^2_{\Z_p} \widehat{\Q_p^\times}$ and so Theorem \ref{multiplicative group thm}\,(b) combines with Nakayama's lemma to imply that $t^{(p)}_{p^n}$ is a $\Z_p [G_{p^n}]$-basis of $\Det_{\Z_p [G_{p^n}]} (A^\bullet_{F_{p^n}, p})^{-1}$ for all $n \geq 0$. This is perhaps reason to more generally expect a direct relation between $t^{(p)}_{m p^n}$ and the canonical basis of $\Det_{\Z_p [G_{p^n}]} (A^\bullet_{F_{p^n}, p})^{-1}$ given by Kato's local $\epsilon$-constant \cite{Kato-local-epsilon}. 
\end{rk}

\subsection{Iwasawa theory}

Fix a natural number $m$ coprime with $p$ and define the complex
\[
A_{F_{mp^\infty}, \ell}^\bullet \coloneqq \underset{\overleftarrow{n \in \N}}{\operatorname{Rlim}} \ A^\bullet_{F_{mp^n}, \ell},
\]
where the (homotopy) limit is taken with respect to the maps induced by the relevant instances of the isomorphisms (\ref{base change isom for A complex}). Setting $\Lambda_m \coloneqq \varprojlim_{n \in \N} \Z_p [G_{mp^n}]$, the complex $A_{F_{mp^\infty}}$ is then perfect as a complex of $\Lambda_m$-modules that is acyclic outside degrees 1 and 2, and one has canonical isomorphisms
\[
H^1 ( A^\bullet_{mp^\infty, \ell}) \cong {\varprojlim}_{n \in \N} \widehat{F_{mp^n, \ell}^{\times}} 
\quad \text{ and } \quad
H^2 ( A^\bullet_{mp^\infty, \ell}) \cong {\bigoplus}_{v \mid \ell} \Z_p.
\]
In particular, since no finite prime splits completely in $F_{mp^\infty}$, it follows that $H^2 (A^\bullet_{mp^\infty, \ell})$ is a $\Lambda_m$-torsion module and Definition \ref{def projection map from det} provides us with an injective map
\[
\vartheta_{mp^\infty, \ell} \coloneqq \vartheta_{A^\bullet_{mp^\infty, \ell}, \emptyset} 
\: \Det_{\Lambda_m} ( A^\bullet_{mp^\infty, \ell})^{-1} \hookrightarrow Q (\Lambda_m) \otimes_{\Lambda_m} \exprod^{1 - \bm{1}_p (\ell)}_{\Lambda_m} H^1 ( A^\bullet_{mp^\infty, \ell})
\]
where $Q (\Lambda_m)$ is the total ring of fractions of $\Lambda_m$.\\
In the following result we write $I (U) \coloneqq \ker \{ \Z_p \llbracket U \rrbracket \to \Z_p \}$ for the ($p$-adic) augmentation ideal of an abelian group $U$, and we use the notation $\bidual^r_R M$ for the $r$-th `exterior bidual' of an $R$-module $M$ (see \S\,\ref{apendix section Fitting ideals biduals} for details). In addition, we recall that for any ideal $\mathfrak{a} \subseteq \Lambda_m$ we can naturally regard $\a^{\ast \ast}$ as an ideal of $\Lambda_m$ via the injective map $\a^{\ast \ast} \to (\Lambda_m)^{\ast \ast} \cong \Lambda_m$. 

\begin{lem} \label{local Iwasawa theory description image of projection map}
Fix a natural number $m$ coprime with $p$ and write 
$\cD_{mp^\infty}^{(\ell)} \subseteq G_{mp^\infty}$ for the decomposition group at $\ell$. Then one has 
    \[
    \im (\vartheta_{mp^\infty, \ell}) = \big( \Ann_{\Lambda_m} ( {\bigoplus}_{v \mid \ell} \Z_p (1))^{- 1} \cdot I (\cD_{mp^\infty}^{(\ell)}) \big)^{\ast \ast} \cdot \bidual^{1 - \bm{1}_p (\ell)}_{\Lambda_m} H^1 (A^\bullet_{mp^\infty, \ell}).
    \]
\end{lem}

\begin{proof}
Since $\Det_{\Lambda_m} ( A^\bullet_{mp^\infty, \ell})^{-1}$ is a free rank-one module and $\vartheta_{mp^\infty, \ell}$ is injective, the image of $\vartheta_{mp^\infty, \ell}$ is reflexive and so, by the argument of \cite[Lem.\@ C.11]{Sakamoto20}, it suffices to verify the claimed equality after localising at an arbitrary prime ideal $\p \subseteq \Lambda_m$ of height at most one.\\
To do this, we first note that $A^\bullet_{mp^\infty, \ell}$ is acyclic outside degrees 1 and 2 so that the argument of \cite[Prop.\@ A.11\,(i)]{sbA} shows that $A^\bullet_{mp^\infty, \ell}$ admits a representative of the form $Q \to P$, where $P$ is a finitely generated free $\Lambda_m$-module and $Q$ is a finitely generated $\Lambda_m$-module of finite projective dimension (that is placed in degree 1). The Auslander--Buchsbaum formula implies that the localisation $Q_\p$ of $Q$ at $\p$ is of projective dimension at most one, and so there is an exact sequence $0 \to F_1 \to F_0 \to Q \to 0$ with finitely generated free $\Lambda_{m, \p}$-modules $F_1$ and $F_0$. It follows that $F_1 \to F_0 \to P_\p$ is a standard representative in the sense of Definition \ref{standard representative def} (with respect to $(0, \emptyset)$) for the complex $A^\bullet_{mp^\infty, \ell} \otimes_{\Lambda_m}^\mathbb{L} \Lambda_{m, \p}$. \\  
To proceed, we first verify that the Matlis dual of the $\Lambda_{m, \p}$-torsion submodule of $H^1 (A^\bullet_{mp^\infty, \ell})_\p$ coincides with $(\bigoplus_{v \mid \ell} \Z_p (1))_\p$. 
By Lemma \ref{ext and dual of tor} it suffices to compute $\Ext^1_{\Lambda_m} ( H^1 (A^\bullet_{mp^\infty, \ell}), \Lambda_m)_\p$ for this. To do this, we will use the convergent spectral sequence
\[
E_2^{i, j} = \Ext^i_{\Lambda_m} ( H^{-j} ( A^\bullet_{mp^\infty, \ell}), \Lambda_m) \, \Rightarrow \, 
E^{i + j} = H^{i + j} ( \mathrm{RHom}_{\Lambda_m} ( A^\bullet_{F_{mp^\infty, \ell}}, \Lambda_m)).
\]
Since $\Lambda_{m, \p}$ is a Gorenstein ring of dimension one, one has that $\Ext^i_{\Lambda_m} ( -, \Lambda_m)_\p = 0$ if $i > 1$, and so this spectral sequence gives an isomorphism 
\begin{align*}
    \Ext^1_{\Lambda_m} ( H^1 (A^\bullet_{mp^\infty, \ell}), \Lambda_m)_\p  & \cong H^0 ( \mathrm{RHom}_{\Lambda_m} ( A^\bullet_{F_{mp^\infty, \ell}}, \Lambda_m))_\p\\
     & \cong  \big( {\varprojlim}_{n \in \N} H^2 ( \Q_\ell, (\Z_p)_{F_{mp^n} / \Q}) \big)_\p\\
     & \cong \big( {\bigoplus}_{v \mid \ell} \Z_p (1) \big)_\p.
\end{align*}
Here the second isomorphism is by (derived) local Tate duality \cite[Thm.~5.2.6]{NekovarSelmerComplexes} and the final isomorphism follows easily from (classical) local Tate duality.\\
If $\p$ does not contain $p$, then the localisation of $\Lambda_m$ at $\p$ is a regular local ring (see, for example, the discussion in \cite[\S\,3C1]{bks2}) and hence, in this case, the claim follows from Proposition~\ref{integrality properties of det projection map}~(c) applied to the complex $A^\bullet_{mp^\infty, \ell} \otimes_{\Lambda_m}^\mathbb{L} \Lambda_{m, \p}$ and the computation
of $H^1 (A^\bullet_{mp^\infty, \ell})_{\tor, \p}^\vee$ above.
\\ 
Furthermore, if $\p$ contains $p$, then the localisation of $H^2 (A^\bullet_{mp^\infty, \ell})$ at $\p$ vanishes by the general result of \cite[Lem.\@ 5.6]{Flach04} because $H^2 (A^\bullet_{mp^\infty, \ell})$ is a finitely generated $\Z_p$-module. Similarly, $\Ext^1_{\Lambda_m} (H^1 (A^\bullet_{mp^\infty, \ell}), \Lambda_m) = \Z_p (1)$ is a finitely generated $\Z_p$-module and so also vanishes when localised at $\p$. It follows from Lemma \ref{ext and dual of tor} that
\[
\Ext^1_{\Lambda_m} ( H^1 (A^\bullet_{mp^\infty, \ell}), \Lambda_m)_\p = \Hom_{\Lambda_m} ( H^1 (A^\bullet_{mp^\infty, \ell})_\tor, Q (\Lambda_m) / \Lambda_m)_\p = 0
\]
so that $H^1 (A^\bullet_{mp^\infty, \ell})_\p$ has depth 1 as a $\Lambda_{m, \p}$-module. By the Auslander--Buchsbaum formula, this implies that it is in fact a free, possibly zero, $\Lambda_{m, \p}$-module. Given this, the same proof as Proposition \ref{integrality properties of det projection map}\,(c) works and shows the claim.
\end{proof}

Suppose that $\ell \neq p$.
In this case, Lemma \ref{nus are compatible} shows that the family $\nu_{mp^\infty}^{(\ell)} \coloneqq (\nu_{mp^n}^{(\ell)})_{n \in \N}$ defines an element of $\Lambda_m$. If we can prove that $\Eul_\ell (\tilde \sigma_\ell)^{-1} \cdot \nu_{mp^\infty}^{(\ell)}$ belongs to the image of $\vartheta_{mp^\infty, \ell}$, then it follows from the diagram 
\begin{equation} \label{commutative diagram for det projection maps}
\begin{tikzcd}[row sep=small]
    \Det_{\Lambda_m} ( A^\bullet_{mp^\infty, \ell})^{-1} \arrow{r}{\vartheta_{mp^\infty, \ell}} \arrow{d}[left]{\mathrm{pr}_{F_{mp^\infty} / F_{mp^n}}} &   \big( \Ann_{\Lambda_m} (  {\bigoplus}_{v \mid \ell} \Z_p (1))^{- 1} \cdot I (\cD_{mp^\infty}^{(\ell)}) \big)^{\ast \ast} \arrow{d}{\pi_{F_{mp^\infty} / F_{mp^n}}} \\ 
    \Det_{\Z_p [G_{mp^n}]} ( A_{F_{mp^n}, \ell}) \arrow{r}{\vartheta_{F_{mp^n}, \ell}^0} & 
    \Q_p [G_{mp^n}],
\end{tikzcd}%
\end{equation}
which commutes for every $n \in \N$ as can be checked using the explicit description of the maps $\vartheta_{mp^\infty, \ell}$ and $\vartheta_{F_{mp^n}, \ell}^0$ given in Lemma~\ref{explicit description projection map lem}, that  the first claim in Theorem \ref{multiplicative group thm}\,(a) is valid. \\
To prove that $\Eul_\ell (\tilde \sigma_\ell)^{-1} \cdot \nu_{mp^\infty}^{(\ell)}$ is indeed in the image of $\vartheta_{mp^\infty, \ell}$, \cite[Lem.~C.11]{Sakamoto20} allows us to check the claimed containment locally at a height-one prime $\p$ of $\Lambda_m$.
To do this, let us first assume that $\p$ contains $p$. In this case, then, Lemma \ref{local Iwasawa theory description image of projection map} combines with \cite[Lem.~5.6]{Flach04} to imply that $(\im \vartheta_{mp^\infty, \ell})_\p = \Lambda_{m, \p}$. It therefore suffices to prove that $\Eul_\ell (\tilde \sigma_\ell)^{-1} \cdot \nu_{mp^\infty}^{(\ell)}$ belongs to $\Lambda_{m, \p}$ and this will follow if we can show that $\Eul_\ell (\tilde \sigma_\ell) = \ell^{-1} (\ell - \tilde \sigma_\ell)$ is a unit in $\Lambda_{m, \p}$. Fix a decomposition $G_{mp^\infty} \cong \Delta \times \Gamma$ with a finite group $\Delta$ and $\Gamma \cong \Z_p$.
Note that, writing $\Q_\infty$ for the cyclotomic $\Z_p$-extension of $\Q$, the restriction map then also gives an isomorphism $\Gamma \cong \gal{\Q_\infty}{\Q}$.
If we denote by $\Delta^{(p)}$ the $p$-Sylow subgroup of $\Delta$, then the discussion in \cite[\S\,3C1]{bks2} implies that there exists a prime-to-$p$-order character $\chi_\p \: \Delta \to \overline{\Q_p}^\times$ such that $\Lambda_{m, \p}$ naturally identifies with the localisation of $\Z_p [\im \chi_\p] [\Delta^{(p)}] \llbracket \Gamma \rrbracket$ at the prime ideal $(p, I_{\Z_p [\im \chi_\p], \Delta^{(p)}})$. Since $\ell - \chi_\p (\tilde \sigma_\ell) \sigma_\ell$ is not divisible by $p$ in $\Lambda [\im \chi_\p] \llbracket \Gamma \rrbracket$ (because $\sigma_\ell$ generates an open subgroup of $\gal{\Q_\infty}{\Q}$), we conclude that $\Eul_\ell (\tilde \sigma_\ell)$ is a unit in $\Lambda_{m, \p}$, as claimed. \\
We next assume $\p \subseteq \Lambda_m$ is a height-one prime that does not contain $p$. By 
\cite[\S\,3C1]{bks2} one then has that  there exists a character $\chi_\p \: \Delta \to \overline{\Q_p}^\times$ such that $\Lambda_{m, \p}$ identifies with the localisation of $\Z_p [\im \chi_\p]$ at one of its height-one primes, and \cite[Lem.~8.23~(i)]{BB} moreover shows that
    \[
    \Ann_{\Lambda_m} \big ( {\bigoplus}_{v \mid \ell} \Z_p (1) \big)_\p =
    \begin{cases}
\Eul_\ell (\sigma_\ell) \Lambda_{m, \p} \quad & \text{ if } \chi_\p ( \cI^{(\ell)}_{mp^\infty}) = 1, \\
\Lambda_{m, \p} & \text{ if } \chi_\p ( \cI^{(\ell)}_{mp^\infty}) \neq 1.
    \end{cases}
    \]
Recall from Theorem \ref{local points main result}\,(e) that the element $\Eul_\ell (\tilde \sigma_\ell)^{-1} \nu_{mp^\infty}^{(\ell)}$ belongs to the submodule $\Lambda_m + \Eul_\ell (\sigma_\ell)^{-1} \NN_{\cI^{(\ell)}_{mp^\infty}} \Lambda_m$ of $Q(\Lambda_m)$, and hence also to
$\Ann_{\Lambda_m} (\Z_p (1))_\p^{-1}$ by the explicit description given above. In addition, we know from Theorem \ref{local points main result}\,(d) that $\nu^{(\ell)}_{mp^\infty}$ is contained in $I_{\cD^{(\ell)}_{mp^\infty}} = \Lambda_m \cdot I ( \cD^{(\ell)}_{mp^\infty})$. However, one can explicitly check that the supports of $\Lambda_m / I_{\cD^{(\ell)}_{mp^\infty}}$ and $\Lambda_m / \Ann_{\Lambda_m} ( \bigoplus_{v \mid \ell} \Z_p (1))$ are disjoint. Indeed, if $\p$ is in the support of $\Lambda_m / I_{\cD^{(\ell)}_{mp^\infty}}$, then $\p = \ker \{ \Lambda_m \to \Z_p [ G_{mp^\infty} / \cD^{(\ell)}_{mp^\infty}] \stackrel{\chi}{\to} \Q_p (\chi) \}$ for some character $\chi$. Since $\chi ( \sigma_\ell) - \ell \neq 1$, we see that $\Ann_{\Lambda_m} ( \bigoplus_{v \mid \ell} \Z_p (1))$ contains an element that is a unit in $\Lambda_{m, \p}$, as required to prove the claim. We have thereby proved that
\begin{align*}
\Eul_\ell (\tilde \sigma_\ell)^{-1} \nu_{mp^\infty}^{(\ell)} & \in \big( \Ann_{\Lambda_m} ( {\bigoplus}_{v \mid \ell} \Z_p (1))^{-1} \cap I_{\cD^{(\ell)}_{mp^\infty}} \big)_\p \\
& = \big( \Ann_{\Lambda_m} ( {\bigoplus}_{v \mid \ell} \Z_p (1))^{-1} \cdot I_{\cD^{(\ell)}_{mp^\infty}} \big)_\p 
\end{align*}
for every height-one prime ideal $\p$ of $\Lambda_m$ that does not contain $p$.
Together with the earlier argument for primes $\p$ that contain $p$, this shows that $\Eul_\ell (\tilde \sigma_\ell)^{-1} \nu_{mp^\infty}^{(\ell)}$ is in $( \Ann_{\Lambda_m} ( {\bigoplus}_{v \mid \ell} \Z_p (1))^{-1} \cdot I ( \cD^{(\ell)}_{mp^\infty}))^{\ast \ast}$, and hence in $\im (\vartheta_{mp^\infty, \ell})$ by Lemma \ref{local Iwasawa theory description image of projection map}. This proves the first claim in Theorem~\ref{multiplicative group thm}\,(a).\\
Let us now turn to the the case $\ell = p$
and the proof of the first claim in Theorem \ref{multiplicative group thm}\,(b).
By construction, the elements $\mathfrak{l}_{mp^n}$ belong to the $(1 - e_\tau)$-isotypic component of $H^1 (A^\bullet_{F_{mp^n}, p})$ which, in particular, is torsion free. Since $p$ is assumed to be a split-multiplicative prime, Theorem~\ref{local points main result}\,(c)\,(i) moreover shows that $\mathfrak{l}_{mp^\infty} \coloneqq ( \mathfrak{l}_{mp^n})_{n \in \N}$ is a norm-coherent family with trivial bottom value. By the argument of \cite[Thm.\@ 3.8\,(b)]{BullachDaoud}, applied to the complex $A^\bullet_{F_{mp^\infty, p}} \otimes_{\Lambda_m}^\mathbb{L} (1 - e_\tau) \Lambda_m$, this implies that 
$\mathfrak{l}_{mp^\infty} $ is in $I_{\cD^{(\ell)}_{mp^\infty}}^{\ast \ast} \cdot (1 - e_\tau) H^1 ( A_{mp^\infty, p}^\bullet)$. We therefore deduce from Lemma \ref{local Iwasawa theory description image of projection map} that $\mathfrak{l}_{mp^\infty} $ belongs to $(1 - e_\tau) \im (\vartheta_{mp^\infty, p})$.
Now, one has a commutative diagram comparing $\vartheta_{mp^\infty, p}$ and $\vartheta_{mp^n, p}$ as in (\ref{commutative diagram for det projection maps}), and so it follows that $\mathfrak{l}_{mp^n} $ is in the image of $\vartheta_{mp^n, p}$ for all $n \in \N$, as claimed in Theorem \ref{multiplicative group thm}\,(b).\\
This proves the first two claims in parts (a) and (b) of Theorem \ref{multiplicative group thm}. 

\subsection{Descent calculations} \label{local descent calculations section}

To prove the remaining claims of Theorem \ref{multiplicative group thm}, we will perform descent calculations similar to those in \cite[\S\,5]{bks2}. In doing so we will, in particular, prove local analogues of the `Mazur--Rubin--Sano' conjecture from \cite{bks2}.
We remark that, as our setting is entirely local, we do not need to assume the validity of `global-to-local' hypotheses used in \cite{bks2} such as the conjectures of Leopoldt or Gross--Kuz'min.\\
Let $L$ be a finite abelian extension of $\Q$ and $K$ a subfield of $L$ in which a prime $\ell$ splits completely. We also fix a place $w_0$ of $L$ above $v_0$. Setting $\cG \coloneqq \gal{L}{\Q}$, $H \coloneqq \gal{L}{K}$, and $I_H \coloneqq I(H) \cdot \Z_p [\cG]$, we then write
\[
\beta_{L / K}^{(\ell)}  \: 
H^1 (A^\bullet_{K, \ell}) \cong H^1 (A^\bullet_{L, \ell} \otimes_{\Z_p [\cG]}^\mathbb{L} \Z_p [G])
\xrightarrow{\beta_{A^\bullet_{K,\ell}, I_H}} H^2 ( A^\bullet_{L, \ell}) \otimes_{\Z_p [\cG]} I_H 
\xrightarrow{w_0^\ast} I_H / I (D^{(\ell)}_L)I_H
\]
for the relevant instance of the Bockstein map from Definition \ref{bockstein map def}.
This map has the following explicit description.

\begin{lem} \label{bockstein maps and reciprocity maps}
Suppose $\ell$ splits completely in $K$. For every $a = (a_v)_{v \mid \ell} \in \bigoplus_{v \mid \ell} \widehat{K_v^\times} = H^1 (A^\bullet_{K, \ell})$ one then has
    \[
    \beta_{L / K}^{(\ell)} (a) = \sum_{\sigma \in G_K} (\mathrm{rec}_{\ell} ( a_{\sigma v_0}) - 1) \widetilde \sigma^{-1} \in I_H / I (D^{(\ell)}_L)I_H
    \]
    with the local reciprocity map $\rec_{\ell}\: \Q_\ell \to \gal{\Q_\ell^\mathrm{ab}}{\Q_\ell} \twoheadrightarrow \cD^{(\ell)}_L \subseteq H$ and a choice of lift $\widetilde \sigma \in \cG$ of $\sigma \in G$. 
\end{lem}

\begin{proof}
    See \cite[Lem.\@ 10.3]{burns07} or \cite[Lem.\@ 5.21]{bks}.
\end{proof}

Throughout the remainder of this section let $K_\infty$ denote the cyclotomic $\Z_p$-extension of $K$ and, for every $n \in \N$, write $K_n$ for its $n$-th layer. We will write $\Gamma \coloneqq \gal{K_\infty}{K}$ and $\Gamma_n \coloneqq \gal{K_n}{K}$ for the relevant Galois groups. If $K_n$ has conductor $mp^{n + t}$ for some $t \geq 0$, we define 
\[ t^{(p)}_{K_n} \coloneqq \mathrm{pr}_{F_{mp^{n + t} / K}} (t_{mp^{n + t}}^{(p)}).\]
As the remaining calculations are subtantially different according to whether the prime $\ell$ is equal to $p$ or not, we now consider these two cases separately. 

\subsubsection{The case $\ell \neq p$}
Let us first assume that $\ell \neq p$.  By enlarging $K$ is necessary, we may assume that the decomposition group $D^{(\ell)}_{K_n}$ of $\ell$ in $K_n / \Q$ is equal to $\Gamma_n$ for all $n \in \N$. In particular, $\beta_{K_n / K}^{(\ell)}$ defines a map $H^1 (A^\bullet_{K_n, \ell}) \to I_{\Gamma_n} / I_{\Gamma_n}^2$.
In order to apply Proposition \ref{bockstein proposition}\,(c) in this situation, we need to verify
that $\mathrm{pd}_{\Z_p [G_{K_n}]} ((\widehat{K_{n, \ell}^\times})_\tor) \leq 1$ and $\mathrm{pd}_{\Z_p} ((\widehat{K_{\ell}^\times})_\tor) \leq 1$. (This will also verify condition (\ref{pd assumption}) by Remark \ref{base change for dual of Fitt 1-dim case}.)
The second inequality is clear because $\Z_p$ is a discrete valuation ring, and for the first inequality it is sufficient to prove that 
$(K_{n, v_0}^\times)_\tor$ is $\Gamma_n$-cohomologically trivial. The required cohomological triviality is however true because the extension $K_{n, v_0} / \Q_\ell$ is unramified (cf.\@ \cite[Prop~9.1.4]{NSW}).\\
Proposition \ref{bockstein proposition}\,(c) now shows that
\begin{equation} \label{half 1}
(\beta^{(\ell)}_{K_n / K} \circ \vartheta_{K, \ell, v_0} \circ \mathrm{pr}_{K_n / K}) (t^{(\ell)}_{mp^n}) 
\equiv \vartheta_{K_n, \ell}^0 ( t^{(\ell)}_{K_n})  \mod I_{\Gamma_n}^2
\end{equation}
as an equality in $\Ann_{\Z_p} ((\Q_\ell^\times)_\tor)^{-1} \otimes_{\Z_p} ( I_{\Gamma_n} / I_{\Gamma_n}^2 )$. 
The right hand side of this congruence we can compute, using Theorem \ref{local points main result}\,(d), to be
\begin{align} \nonumber 
\vartheta_{K_n, \ell}^0 ( t^{(\ell)}_{K_n}) & = \pi_{F_{mp^{n + t}} / K_n} (\Eul_\ell (\tilde \sigma_\ell)^{-1} \nu_{mp^{n + t}}^{(\ell)}) \\
& \equiv \ell (\ell - 1)^{-1} \otimes \ell^{- (\ord_\ell (m) - 1)} (1 - \tilde \sigma_\ell) \mod I_{\Gamma_n}^2. \label{half 2}
\end{align}
(Note that $\Ann_{\Z_p} ((\Q_\ell^\times)_\tor)^{-1}$ is generated by $(\ell - 1)^{-1}$.)
Combining (\ref{half 1}) and (\ref{half 2}) we obtain
\begin{equation} \label{we will pass to the limit over this}
    (\beta^{(\ell)}_{K_n / K} \circ \vartheta_{K, \ell, v_0} \circ \mathrm{pr}_{K_n / K}) (t^{(\ell)}_{mp^n}) 
\equiv \ell (\ell - 1)^{-1} \otimes \ell^{- (\ord_\ell (m) - 1)} (1 - \tilde \sigma_\ell) \mod I_{\Gamma_n}^2,
\end{equation}
which, by using the isomorphism
\[
I_{\Gamma_n} / I_{\Gamma_n}^2 \cong \Z_p [G] \otimes_{\Z_p} ( I (\Gamma_n) / I (\Gamma_n)^2),
\]
we can regard as an equality in $(\ell - 1)^{-1} \Z_p [G] \otimes_{\Z_p} ( I (\Gamma_n) / I (\Gamma_n)^2)$.
By taking the limit of the maps $\beta_{K_n / K}^{(\ell)}$ we may define a limit map
\[
H^1 (A^\bullet_{K, \ell}) \to \Z_p [G] \otimes_{\Z_p} \big( {\varprojlim}_{n \in \N} ( I (\Gamma_n) / I (\Gamma_n)^2 )  \big) \cong \Z_p [G] \otimes_{\Z_p} \Gamma,
\]
where the isomorphism is induced by sending $(g - 1) \mapsto g$ for every $g \in \Gamma$.
Passing to the limit (over $n$) in (\ref{we will pass to the limit over this}), we then obtain
\begin{align} \nonumber 
(\beta^{(\ell)}_{K_{\infty} / K} \circ \vartheta_{K, \ell, v_0} \circ \mathrm{pr}_{K_{\infty} / K}) (t^{(\ell)}_{K_\infty}) 
& = \ell (\ell - 1)^{-1} \otimes \ell^{- (\ord_\ell (m) - 1)} (1 - \tilde \sigma_\ell) \\
\label{bockstein and nus} 
& = - \ell (\ell - 1)^{-1} \otimes \ell^{- (\ord_\ell (m) - 1)} (\tilde \sigma_\ell - 1)
\end{align}
in $(\ell - 1)^{-1} \Z_p [G] \otimes_{\Z_p} \Gamma$.
Since $\ell$ is unramified in $K_\infty / K$, the composite map
\[
\widehat{K_{v_0}^\times} \xrightarrow{\rec_{\ell}} \Z_p \otimes_{\Z_p} \Gamma \stackrel{\sigma_\ell \mapsto 1}{\cong} \Z_p
\]
agrees with $\ord_\ell$. This map restricts to an isomorphism on the torsion-free quotient of $\widehat{K_{v_0}^\times}$, and so Lemma \ref{bockstein maps and reciprocity maps} combines with (\ref{bockstein and nus}) to imply that
\[
(\vartheta_{K, \ell, v_0} \circ \mathrm{pr}_{K_\infty / K}) (t^{(\ell)}_{K_\infty}) = - \ell^{- \ord_\ell (m) + 1} (1 - \ell^{-1})^{-1} \otimes \ell , 
\]
as required to conclude the proof of Theorem \ref{multiplicative group thm}\,(a).

\subsubsection{The case $\ell = p$}

In the remainder of this section we consider the case $\ell = p$. It then suffices to prove that 
\[
(\vartheta_{K, p, v_0} \circ \mathrm{pr}_{K_n / K}) (t^{(p)}_{K_n}) =  \NN_{F_m / K} (\mathfrak{l}_m) \wedge p. 
\]
As a first step in this direction, we combine Proposition \ref{bockstein proposition}\,(c) with the argument of \cite[Thm.\@ 5.10]{bks} to deduce that
\begin{equation} \label{one half of MRS}
{\sum}_{\sigma \in \Gamma_n} \sigma \vartheta^0_{K_n, p} ( t_{K_n}^{(p)}) \otimes \sigma^{-1}
= - (\beta^{(p)}_{K_n / K} \circ \vartheta_{K, p, v_0} \circ \mathrm{pr}_{K_n / K}) (t^{(p)}_{K_n}) 
\end{equation}
in $\widehat{K_{n,p}^\times} \otimes_{\Z_p [G_{K_n}]} (I_{\Gamma_n} / I_{\Gamma_n}^2)$ with $\Gamma_n \coloneqq \gal{K_n}{K}$.\\
Recall that for any norm-coherent sequence $u = (u_n)_{n \geq 0} \in \varprojlim_{n \geq 0} (\Z_p \otimes_\Z \cO_{F_{mp^n}})^\times$ there exists a unique power series $\mathrm{Col} (X) \in (\Z_p \otimes_\Z \cO_{F_m}) \llbracket X \rrbracket^\times$, called its `Coleman power series', with the property that
\[
(\sigma_p^{-n} \mathrm{Col}) ( \zeta_{p^n} - 1) = u_n \quad 
\text{ for all } n \geq 1. 
\]
(See \cite[Ch.\@ I]{deShalit} for details.)\\
In the following, we will use the maps
\begin{align*}
    \Ord_p \: & {\bigoplus}_{v \mid p} \widehat{K_v^\times} \to \Z_p [G], \quad 
    a \mapsto {\sum}_{\sigma \in G} \ord_p  (a_{\sigma v_0}) \sigma^{-1},\\
    \mathrm{Rec}_{p, K_n / K} \: & {\bigoplus}_{v \mid p} \widehat{K_v^\times} \to I_{\Gamma_n} / I_{\Gamma_n}^2, 
    \quad a \mapsto {\sum}_{\sigma \in G} (\mathrm{rec}_{\ell} ( a_{\sigma v_0}) - 1) \tilde \sigma^{-1},
\end{align*}
and the induced isomorphisms 
\[
\Ord_p \: \exprod^2_{\Z_p [G]} \widehat{K^\times_p} \stackrel{\simeq}{\to} \widehat{\cO_{K, p}^\times}
\quad \text{ and } \quad 
\mathrm{Rec}_{p, K_n / K} \:\exprod^2_{\Z_p [G]} \widehat{K^\times_p} \stackrel{\simeq}{\to} p^{\Z_p [G]}.
\]

\begin{lem} \label{bley hofer lemma}
Suppose $u = (u_n)_{n \geq 0} \in \varprojlim_{n \geq 0} (\Z_p \otimes_\Z \cO_{F_{mp^n}})^\times$ is a norm-coherent sequence with $\NN_{F_m / K} (u_0) = 1$ and $\mathrm{Col} (0) \in (\Z_p \otimes_\Z \cO_{F_m})^\times$. For big enough $n$, one then has 
\[
{\sum}_{\sigma \in \Gamma_n} \sigma \NN_{F_{mp^n} / K_n} (u_n) \otimes \sigma^{-1} =  (\mathrm{Rec}_{p, K_n / K} \circ \Ord_p^{-1}) ( \NN_{F_m / K}( \mathrm{Col} (0)))
\]
in $\widehat{K_{n, p}^\times} \otimes_{\Z_p [G_{K_n}]} (I_{\Gamma_n} / I_{\Gamma_n}^2)$.
\end{lem}

\begin{proof}
    This follows from the main result of Bley and Hofer in \cite{BlHo20} via the argument of \cite[Thm.\@ 5.1]{bullachhofer}. For the convenience of the reader, we sketch this argument.\\  
    Observe that $K^\times$ contains no primitive $p$-th root of unity because $p$ is assumed to split completely in $K$, and hence that also each $\widehat{K_n^\times}$ is $\Z_p$-torsion free by the general result of \cite[Prop.\@ 1.6.12]{NSW}. 
    If we fix a topological generator $\gamma$ of $\Gamma$, then we may therefore use the argument 
    of \cite[Thm.\@ 3.8\,(b)]{BullachDaoud}, applied to the complex $A^\bullet_{K_\infty, p}$, to deduce from the assumption $\NN_{F_m / K} (u_0) = 1$ that the sequence $u' \coloneqq ( \NN_{F_{mp^n} / K_n} (u_n))_{n \in \N}$ is divisible by $\gamma - 1$ in $\varprojlim_{n \in \N} \widehat{K_{n, p}^\times}$.
     Writing $\kappa = (\kappa_n)_{n \geq 0}$ for the unique element of $\varprojlim_{n \geq 0} \widehat{K_{n, p}^\times}$ such that $u' = (\gamma - 1) \kappa$, one then has
    \[
    {\sum}_{\sigma \in \Gamma_n} \sigma \NN_{F_{mp^n} / K_n} (u_n) \otimes \sigma^{-1} = \kappa_0 \otimes (\gamma - 1)
    \]
    by \cite[Lem.\@ 3.14]{bullachhofer}. To proceed, define the characters $s_\gamma \: \Gamma \to \Z_p, \gamma^a \mapsto a$ and $s_{\gamma, n} \: \Gamma_n \to \Z / p^n \Z, \gamma^a \mapsto a \mod p^n$. Applying \cite[Cor.\@ 3.17]{BlHo20} (see also \cite[Prop.\@ 5.2]{bullachhofer}) to the character $\rho_{\gamma, n} \: G_{mp^\infty} \to \Gamma \xrightarrow{s_{\gamma, n}} \Z / p^n \Z$, we obtain the equality
    \[
    \frac{\ord_{K_n} (\kappa_n)}{p^n} = - \frac{(s_{\gamma, n} \circ \rec_p \circ \NN_{F_m / K}) (\mathrm{Col} (0)))}{p^n} \quad \text{ in } {\bigoplus}_{v \mid p} (\Q / \Z). 
    \]
    Now, $\ord_{K_n} (\kappa_n) = \ord_K (\kappa_0)$ because $K_n / K$ is totally ramified at $p$, and so taking the limit (over $n$) shows that
    \[
    \ord_p (\kappa_0) = - (s_\gamma \circ \rec_p \circ \NN_{F_m / K}) (\mathrm{Col} (0)))
    \quad \text{ in } {\bigoplus}_{v \mid p} \Z_p. 
    \]
    In addition, local class field theory implies that the group of universal norms $\bigcap_{n \in \N} \NN_{K_n / K} (\widehat{K_n^\times})$ is equal to $\bigoplus_{v \mid p} p^{\Z_p}$, on which $\Ord_p$ is injective. Since $\kappa_0$ clearly belongs to this group and the same is true for the image of the isomorphism $\exprod^2_{\Z_p [G]} \widehat{K^\times_p} \xrightarrow{\simeq} p^{\Z_p [G_K]}$ induced by $\Rec_{p, K_n / K}$ (cf.\@ the general result of \cite[Lem.\@ 2.17\,(ii)]{BB}), the last displayed equality in fact implies that 
    \[
    \kappa_0 \otimes (\gamma - 1) = - (\Ord_p^{-1} \circ \mathrm{Rec}_{p, K_n /K}) ( \NN_{F_m / K}) (\mathrm{Col} (0))) = (\mathrm{Rec}_{p, K_n / K} \circ \Ord_p^{-1}) ( \NN_{F_m / K}( \mathrm{Col} (0))),
    \]
    as claimed. 
\end{proof}

Write $\mathrm{Col} (X) \in \Z_p \llbracket X \rrbracket$ for the Coleman power series associated to the norm-coherent sequence $(\mathfrak{l}_{mp^n})_n$. 
Then Lemma \ref{bley hofer lemma} combines with (\ref{one half of MRS}) and Lemma \ref{bockstein maps and reciprocity maps} to imply that
\begin{align*}
(\mathrm{Rec}_{p, K_n / K} \circ \Ord_p^{-1}) ( \NN_{F_m / K}( \mathrm{Col} (0)))
& = {\sum}_{\sigma \in \Gamma_n} \sigma \vartheta^0_{K_{n}, p} ( t_{K_n}^{(p)}) \otimes \sigma^{-1}\\
& = 
- (\mathrm{Rec}_{p, K_n / K} \circ \vartheta_{K, p, v_0} \circ \mathrm{pr}_{K_n / K}) (t^{(p)}_{K_n}) 
\end{align*}
in $\widehat{K_{n, p}^\times} \otimes_{\Z_p [G_{K_n}]} (I_{\Gamma_n} / I_{\Gamma_n}^2)
$. Note that the outer terms in this equality belong to the image of $
\widehat{K^\times} \otimes_{\Z_p [G_{K_n}]} (I_{\Gamma_n} / I_{\Gamma_n}^2) \cong 
\widehat{K^\times} \otimes_{\Z_p} ( I (\Gamma_n) /  I (\Gamma_n)^2)$ so that we can regard this as where the equality takes places.
Taking the limit over $n$, the map $\mathrm{Rec}_{p, K_n / K}$ defines an isomorphism $\exprod^2_{\Z_p [G]} \widehat{K^\times_p} \stackrel{\simeq}{\to} p^{\Z_p [G]} \otimes_{\Z_p} (I (\Gamma) / I (\Gamma)^2)$
so that from the last displayed equality we can deduce that
\[
\Ord_p^{-1} ( \NN_{F_m / K}( \mathrm{Col} (0))) = - (\vartheta_{K, p, v_0} \circ \mathrm{pr}_{K_n / K}) (t^{(p)}_{K_n}).
\]

Now, by definition of $\tilde x_{mp^n}$ (Definition \ref{def x tilde}) one has
\begin{align*}
\mathrm{Col} (X) & = \exp_{\mathbb{G}_m} \circ \log_{\widehat{E}} \circ \big( \sum_{m_0 \mid d \mid m} \big( ( \sum_{\chi \neq \tau} h_{\chi, d} (X) -_{\widehat{E}} X) +_{\widehat{E}} (p \Eul_p (\sigma_p))^{-1} \epsilon_d \big) \big)\\ 
& = \sum_{m_0 \mid d \mid m} \big( ( \sum_{\chi \neq \tau} (\exp_{\mathbb{G}_m} \circ g_{\chi, d} (X)) -_{\mathbb{G}_m} (\exp_{\mathbb{G}_m} \circ \log_{\widehat{E}})(X)) \\
& \qquad \qquad +_{\mathbb{G}_m} (p \Eul_p (\sigma_p))^{-1} \exp_{\mathbb{G}_m} (\zeta_d p) \big)
\end{align*}
and so we deduce that
\[
\mathrm{Col} (0) = (p \mathrm{Eul}_p (\sigma_p))^{-1} \exp_{\mathbb{G}_m} \big( {\sum}_{m_0 \mid d \mid m}\zeta_d p \big) =  \mathfrak{l}_m. 
\]
The claim now follows from the equality $\Ord_p (  \NN_{F_m / K} (\mathfrak{l}_m) \wedge p) = - \NN_{F_m / K} (\mathfrak{l}_m)$
because the map $\Ord_p \:\exprod^2_{\Z_p [G]} \widehat{K_p^\times} \to (\Z_p \otimes_\Z \cO_K)^\times$ is an isomorphism.\\
This concludes the proof of Theorem \ref{multiplicative group thm}. \qed 

\section{Nekov\'a\v{r}--Selmer complexes and exceptional zeros}

In this section we establish Theorems \ref{mazur--tate main result 1}\,(a) and \ref{mazur--tate main result 2}. The proofs involve an auxiliary Nekov\'a\v{r}--Selmer structure, introduced in \S\,\ref{definition split selmer structure section} and studied further in \S\,\ref{computation of split selmer structure cohomology section}, where we describe the associated Nekov\'a\v{r}--Selmer complexes and their cohomology. After some preliminary constructions of \S\,\ref{split selmr structure det section} concerning the determinants of these complexes, we first prove Theorem~\ref{mazur--tate main result 1}\,(a) in \S\,\ref{order of vanishing proof section} before turning to the proof of Theorem~\ref{mazur--tate main result 2} in \S\,\ref{congruences proof section}.

\subsection{Consequences of Tate uniformisation}

If $E$ has split-multiplicative reduction at the prime $\ell$, then Tate uniformisation gives rise to an exact sequence (cf.\@ \cite[Prop.\@ 6.1]{SilvermanII})
\begin{equation} \label{tate uniformisation exact sequence}
\begin{tikzcd}
    0 \arrow{r} & \Z_p (1) \arrow{r} & \mathrm{T}_p E \arrow{r} & \Z_p \arrow{r} & 0
\end{tikzcd}%
\end{equation}
of $\gal{\overline{\Q_\ell}}{\Q_\ell}$-modules, which induces an exact triangle
\begin{equation} \label{tate uniformisation triangle}
\mathrm{R}\Gamma (\Q_\ell, \Z_p (1)_{K / \Q}) \xrightarrow{g_\ell} 
\mathrm{R}\Gamma (\Q_\ell, (\mathrm{T}_p E)_{K / \Q}) \xrightarrow{h_\ell}   \mathrm{R}\Gamma (\Q_\ell, \Z_{p, K / \Q}) 
\to \mathrm{R}\Gamma (\Q_\ell, \Z_p (1)_{K / \Q})[1].
\end{equation}%

\begin{rk} \label{tate uniformisation connecting map rk}
    For later use we note that the connecting homomorphism $\Z_p \to H^1 (\Q_\ell, \Z_p (1))$ arising from the exact sequence (\ref{tate uniformisation exact sequence}) sends $1$ to the image of the $\ell$-adic Tate period $q_{E, \ell}$ of $E$ under the Kummer map $\widehat{\Q_\ell^\times} \xrightarrow{\simeq} H^1 (\Q_\ell, \Z_p (1))$. 
\end{rk}

\subsubsection{Definition of a useful Nekov\'a\v{r}--Selmer structure} \label{definition split selmer structure section}

We now use the triangle (\ref{tate uniformisation triangle}) to define a useful Nekov\'a\v{r}--Selmer structure.
To do this, we fix $Q \in \widehat{E} (K_p)$ and a finite set $\Pi$ of primes at which $E$ has split-multiplicative reduction. We further fix a finite set $\Phi$ of primes that is disjoint from $S_{mp}$.
Finally, we let $\Sigma$ be a finite set of places of $\Q$ that contains $S (K) \cup \Phi$.
\\ 
We define a Nekov\'a\v{r}--Selmer structure $\cF \coloneqq \cF^\mathrm{sp}_{\Sigma, \Pi, \Phi, Q}$ with $S (\cF) = \Sigma$ as follows.
\begin{itemize}
    \item If $v \in \Pi$, then we let $\mathrm{R}\Gamma_{\cF} (\Q_v, T_{K / \Q}) \coloneqq \mathrm{R}\Gamma (\Q_v, \Z_{p} ( 1)_{K / \Q})$ and take $i_{\cF_{v}}$ to be the map $g_v$ from (\ref{tate uniformisation triangle}).
    \item If $v \in \Phi$, then we define $\mathrm{R}\Gamma_{\cF} (\Q_v, T_{K / \Q})$ and $i_{\cF, v}$ by the triangle
    \[
        \mathrm{R}\Gamma_{\cF} (\Q_v, T_{K / \Q}) \xrightarrow{i_{\cF, v}}  \mathrm{R}\Gamma (\Q_v, T_{K / \Q}) \to \RHom_{\Z_p} ( D_{v}^\bullet, \Z_p) [-2] \to \phantom{X}, 
    \]
    where $D_{v}^\bullet$ is the complex defined in Lemma \ref{approximation complexes}\,(a) and the second arrow is the composite of the natural map $\mathrm{R}\Gamma (\Q_v, T_{K / \Q}) \to \mathrm{R}\Gamma_{/ f} ( \Q_v, T_{K / \Q})$ and
    the map $\rho_{v}^\ast \coloneqq \RHom_{\Z_p} ( \rho_{v}, \Z_p ) [-2]$ dual to the morphism $\rho_{v}$ from Lemma \ref{approximation complexes}\,(b). 
    \item If $v = p$ is unramified in $K$, then we define $\mathrm{R}\Gamma_{\cF} (\Q_v, T_{K / \Q})$ and $i_{\cF, v}$ by the triangle
    \[
        \mathrm{R}\Gamma_{\cF} (\Q_v, T_{K / \Q}) \xrightarrow{i_{\cF, v}}  \mathrm{R}\Gamma (\Q_v, T_{K / \Q}) \to \RHom_{\Z_p} ( D_{p, Q}^\bullet, \Z_p) [-2] \to \phantom{X}, 
    \]
    where $D_{p, Q}^\bullet$ is the complex defined in Lemma \ref{approximation complexes}\,(b) and the second arrow is the composite of the natural map $\mathrm{R}\Gamma (\Q_v, T_{K / \Q}) \to \mathrm{R}\Gamma_{/ f} ( \Q_v, T_{K / \Q})$ and
    the map $\rho_{p, Q}^\ast \coloneqq \RHom_{\Z_p} ( \rho_{p, Q}, \Z_p ) [-2]$ dual to the morphism $\rho_{p, Q}$ from Lemma \ref{approximation complexes}\,(b). 
    \item If $v = p$ is ramified in $K$ and 
    does not belong to $\Pi$, then 
    we define $\mathrm{R}\Gamma_{\cF} (\Q_v, T_{K / \Q})$ and $i_{\cF, v}$ by means of the triangle
    \begin{cdiagram}[column sep=small]
        \mathrm{R}\Gamma_{\cF} (\Q_v, T_{K / \Q}) \arrow{rr}{i_{\cF, v}} & & \mathrm{R}\Gamma (\Q_v, T_{K / \Q}) \arrow{r} & \Z_p [G][-1] \arrow{r} & \mathrm{R}\Gamma_{\cF} (\Q_v, T_{K / \Q}) [1]. 
    \end{cdiagram}%
    Here the second arrow is induced by $\cP_K ( \cdot, Q)$ (equivalently, the composite of the natural morphism $\mathrm{R}\Gamma (\Q_v, T_{K / \Q}) \to \mathrm{R}\Gamma_{/ f} ( \Q_v, T_{K / \Q})$ with the 
  dual of the map $\Z_p [G] [-1] \to \mathrm{R}\Gamma_f ( \Q_v, T_{K / \Q})$ that sends $1$ to $Q$).
    \item In all other cases we take $\mathrm{R}\Gamma_{\cF} (\Q_v, T_{K / \Q}) \coloneqq \mathrm{R}\Gamma (\Q_v, T_{K / \Q})$ and  $i_{\cF, v}$ to be the identity map. 
\end{itemize}

\begin{lem} \label{surprisingly subtle calculation}
Let $(\Sigma, \Pi, \Phi, Q)$ be as above. 
    For every place $v$ of $\Q$, the complex $\mathrm{R}\Gamma_{\cF} (\Q_v, T_{K / \Q})$ is perfect in $D (\Z_p [G])$ and acyclic outside degrees one and two. 
\end{lem}

    \begin{proof}
    In all cases perfectness is a consequence of the definitions and the general result of Flach \cite[\S\,5]{Flach00}. The claim regarding acyclicity is clear in all cases apart from when $v \in \Phi$ or $v = p$ is unramified in $K$. We justify the claim in the latter case and leave the case $v \in \Phi$, which can be treated in exactly the same fashion, to the reader.
    To do this, we write $\rho$ for the composite of the map $\rho_{p, Q}$ and the map $\mathrm{R}\Gamma_f (\Q_p, T_{K / \Q}) \to \mathrm{R}\Gamma (\Q_p, T_{K / \Q})$. By definition of $\cF$, one then has $\mathrm{R}\Gamma_{\cF} (\Q_v, T_{K / \Q}) \cong \RHom_{\Z_p} ( \mathrm{cone} ( \rho), \Z_p) [-2]$ and this isomorphism gives rise to a convergent spectral sequence
\[
E^{i, j}_2 = \Ext^i_{\Z_p} ( H^{-j} (\mathrm{cone} ( \rho)), \Z_p)
\quad \Rightarrow \quad E^{i + j} = H^{i + j + 2}_{\cF} (\Q_p, T_{K / \Q}).
\]
Since $\Ext^i_{\Z_p} ( -, \Z_p) = 0$ if $i \not \in \{ 0, 1\}$, this spectral sequence reduces us to proving that 
\begin{romanliste}
    \item $H^i (\mathrm{cone} (\rho)) = 0$ if $i \not \in \{0, 1, 2\}$,
    \item $H^2 ( \mathrm{cone} (\rho))$ is $\Z_p$-torsion,
    \item $H^0 ( \mathrm{cone} (\rho))$ is $\Z_p$-torsion free.
\end{romanliste}
By its definition, the complex $D^\bullet_{p, Q} = \Z_p [G] [-1] \oplus D^\bullet_{p}$ is acyclic outside degree one and so claim (i) follows from the long exact sequence in cohomology induced by the defining triangle for $\mathrm{cone} (\rho)$. Moreover, this long exact sequence gives an isomorphism $H^2 ( \mathrm{cone} (\rho)) \cong H^2 (\Q_p, T_{K / \Q})$ and hence proves (ii) because $H^2 (\Q_p, T_{K / \Q}) \cong \bigoplus_{v \mid p} ( E(K_v)_\tor \otimes_\Z \Z_p)^\vee$ is finite.\\ 
To prove claim (iii), we recall that the map $H^1 (\rho)$ is defined to be the composite of $H^1 (\rho_{p, Q})$ and the Kummer map $E (K_p) \hookrightarrow H^1 (\Q_p, T_{K / \Q})$. Now, by definition of $\rho_{p, Q}$ one has a commutative diagram
\begin{cdiagram}
    0 \arrow{r} & \Z_p [G] \arrow{r} \arrow{d}{1 \mapsto Q} & H^1 (D^\bullet_{p, Q}) \arrow{r} \arrow{d}{H^1 (\rho_{p, Q})} & H^1 ( D^\bullet_p) \arrow{r} \arrow{d} & 0\\ 
    0 \arrow{r} & E_1 (K_p) \arrow{r} & E (K_p) \arrow{r} & (E / E_1) (K_p) \arrow{r} & 0
\end{cdiagram}%
in which the the rightmost vertical arrow is an isomorphism by Lemma \ref{unr cohom lem} because $p$ is assumed to be unramified in $K$. The snake lemma therefore implies that $H^0 (\mathrm{cone} (\rho)) = \ker H^1 (\rho_{p, Q})$ identifies with an ideal of $\Z_p [G]$, and hence that it is $\Z_p$-torsion free, as required to verify (iii).  
\end{proof}

\subsubsection{A computation of Nekov\'a\v{r}--Selmer groups} \label{computation of split selmer structure cohomology section}

The cohomology of the Nekov\'a\v{r}--Selmer complex associated to $\cF \coloneqq \cF^\mathrm{sp}_{\Sigma, \Pi, \Phi, Q}$ has the following important properties.

\begin{lem} \label{split multiplicative selmer complex properties lemma}
For every triple $(\Sigma, \Pi, \Phi, Q)$ as in \S\,\ref{definition split selmer structure section} the following claims are valid.
\begin{liste}
    \item One has $\mathrm{rk}_{\Z_p} ( H^1_{\cF} (K, \mathrm{T}_p E)^G) \geq |\Pi | + \mathrm{rk}_{\Z_p} (\Sel_{p, E / \Q}^\vee )$. 
    \item If $E (K) [p] = 0$, then there is a surjection
    \[
    H^2_{\cF} (K, \mathrm{T}_p E) \twoheadrightarrow  {\bigoplus}_{\ell \in \Pi} H^2 (\Q_\ell, \Z_p (1)_{K / \Q}) \cong  Y_{K, \Pi} \coloneqq {\bigoplus}_{v \mid \ell} {\bigoplus}_{\ell \in \Pi} \Z_p v. 
    \]
    \item  If $E (K) [p] = 0$, then $H^1_{\cF} (K, \mathrm{T}_p E)$ is $\Z_p$-torsion free.
\end{liste}
\end{lem}

\begin{proof}
To prove claim (a), we will show that $\dim_{\Q_p} ( \Q_p \otimes_{\Z_p} H^1_{\cF} (K, \mathrm{T}_p E)^G)$ is at least $|\Pi| + \dim_{\Q_p} H^2_f (\Q, \mathrm{V}_p E)$.
To do this, we write $H^1_{h(\cF)} (\Q_p, T_{K / \Q})$ for the image of $H^1 (\mathrm{R}\Gamma_\cF (\Q_p, T_{K / \Q}))$ in $H^1 (\Q_p, T_{K / \Q})$, and define $H^1_{/ h(\cF)} (\Q_p, T_{K / \Q}) \coloneqq H^1 (\Q_p, T_{K / \Q}) /H^1_{h(\cF)} (\Q_p, T_{K / \Q})$.\\
Note that $H^1 (\Q_\ell, T_{K / \Q})$ is finite for every $\ell \neq p$, and hence that the relevant instance of the triangle in Remark \ref{Selmer complex triangle} gives the exact sequence 
\begin{equation} \label{exact sequence for H1 SC}
      \Q_p \otimes_{\Z_p} Y_{K, \Pi_\mathrm{ram}} \hookrightarrow  \Q_p \otimes_{\Z_p} H^1_{\cF} (K, \mathrm{T}_p E) \to H^1 (\cO_{K, \Sigma}, \mathrm{V}_p E) \to  \Q_p \otimes_{\Z_p} H^1_{/h(\cF) } (\Q_p, T_{K / \Q}).
\end{equation}
If we can prove that $H^1_{h(\cF) } (\Q_p, T_{K / \Q})$ contains $H^1_f (\Q_p, T_{K / \Q})$, then 
exactness of (\ref{exact sequence for H1 SC}) will show that 
the image of the second arrow in (\ref{exact sequence for H1 SC}) contains $H^1_f (K, \mathrm{V}_p E)$. Since $H^1_f (K, \mathrm{V}_p E)^G \cong H^1_f (\Q, \mathrm{V}_p E)$ has the same $\Q_p$-dimension as $H^2_f (\Q, \mathrm{V}_p E)$ by duality (cf.\@ \cite[Lem.\@ 19]{BurnsFlach01}), taking $G$-invariants will then imply claim (a). \\
If $p \in \Pi$, then we have defined $H^1_{h(\cF) } (\Q_p, T_{K / \Q})$ to be the kernel of the map $H^1 (h_p)$ induced by $h_p$. 
From the long exact sequence associated with the triangle (\ref{tate uniformisation triangle})
we deduce that the kernel of $H^1 (h_p)$ is the image of $H^1 (g_p)$. Now, one has a commutative diagram
\begin{cdiagram}[row sep=small]
\textstyle \bigoplus_{v \mid p} \widehat{K_v^\times} \arrow[twoheadrightarrow]{r} \arrow{d}{\simeq} & \textstyle \bigoplus_{v \mid p} E (K_v)^\wedge \arrow[hookrightarrow]{d}{\kappa} \\  
    H^1 (\Q_p, \Z_p (1)_{K/ \Q}) \arrow{r}{H^1 (g_p)} &  H^1 (\Q_p, T_{K / \Q}),
\end{cdiagram}%
where the vertical arrows are the respective Kummer maps and the top arrow is induced by Tate uniformisation. This shows that  $H^1_{h(\cF) } (\Q_p, T_{K / \Q}) = H^1_f (\Q_p, T_{K / \Q})$ in this case, as desired. \\
If $p \not \in \Pi$, then $\Q_p \otimes_{\Z_p} H^1_{h(\cF) } (\Q_p, T_{K / \Q})$ is by definition the orthogonal complement of $\Q_p [G] \cdot Q$ with respect to local Tate duality. Since $Q$ was chosen to be an element of $H^1_f (\Q_p, V_{K / \Q})$, which is its own orthogonal complement, it follows that   $\Q_p \otimes_{\Z_p} H^1_{h(\cF) } (\Q_p, T_{K / \Q})$ contains $H^1_f (\Q_p, V_{K / \Q})$. This concludes the proof of claim (a).\\
As for claim (b), by definition of the Nekov\'a\v{r}--Selmer structure $\cF$ and the octahedral axiom we have an exact triangle
\[
\mathrm{R}\Gamma_\mathrm{c} (\cO_{K, \Sigma}, \mathrm{T}_p E) \to \mathrm{R}\Gamma_{\cF} (K, \mathrm{T}_p E) \to   {\bigoplus}_{\ell \in \Sigma} \mathrm{R} \Gamma_{\cF} (\Q_\ell, T_{K / \Q}) \to, 
\]
which, since $H^3_\mathrm{c} ( \cO_{K, \Sigma}, \mathrm{T}_p E) \cong (\Z_p \otimes_\Z E (K)_\tor)^\vee$ is assumed to vanish, induces a surjection 
\[
H^2_{\cF} (K, \mathrm{T}_p E) \twoheadrightarrow {\bigoplus}_{\ell \in \Sigma} H^2_{\cF} (\Q_\ell, T_{K / \Q}) \twoheadrightarrow {\bigoplus}_{\ell \in \Pi_\mathrm{ram}} H^2 (\Q_\ell, \Z_p (1)_{K / \Q}),
\]
as claimed.
Finally, claim (c) follows from the exact sequence (obtained from the relevant instance of the triangle in Remark \ref{Selmer complex triangle})
\[
0 \to {\bigoplus}_{v \in \Pi_\mathrm{ram}} \Z_p v \to H^1_{\cF} (K, \mathrm{T}_p E) \to H^1 (\cO_{K, \Sigma}, \mathrm{T}_p E)
\]
and the fact that $H^1 (\cO_{K, \Sigma}, \mathrm{T}_p E)$ is $\Z_p$-torsion free if $E (K)$ has no point of order $p$.
\end{proof}

We now write $\cF^\ast \coloneqq (\cF^\mathrm{sp}_{\Sigma, \Pi, \Phi, Q})^\ast$ for the dual Nekov\'a\v{r}--Selmer structure of $\cF$ (as defined in Remark~\ref{dual structures}) and define a complex as
\[
C^\bullet_{K, \cF} \coloneqq \RHom_{\Z_p} ( \mathrm{R}\Gamma_{\cF^\ast} (K, \mathrm{T}_p E), \Z_p) [-3].
\]
This complex is described in a little more detail in the following result which is an analogue of Lemma \ref{complex lemma}. 

\begin{lem} \label{split multiplicative selmer complex properties lemma 2}
  For every triple $(\Sigma, \Pi, \Phi, Q)$ as in \S\,\ref{definition split selmer structure section} the following claims are valid.
    \begin{liste}
        \item $C^\bullet_{K, \cF} $ is a perfect object of $D (\Z_p [G])$ with Euler characteristic \[ \chi_{\Z_p [G]} ( C^\bullet_{K, \cF} ) =  [\Z_p [G]].\]
        \item There is a canonical isomorphism $H^1 (C^\bullet_{K, \cF}) \cong H^1_{\cF} (K, \mathrm{T}_p E)$.
        If $E (K) [p] = 0$, then there is also a split-exact sequence of $\Z_p [G]$-modules
        \begin{equation} \label{exact sequence for H2}
        \begin{tikzcd}
            0 \arrow{r} & H^2_{\cF} (K, \mathrm{T}_p E) \arrow{r} & H^2 (C^\bullet_{K, \cF}) \arrow{r} & T_{K / \Q}^+ \arrow{r} & 0,
        \end{tikzcd}%
        \end{equation}
        and $C^\bullet_{K, \cF}$ is acyclic outside of degrees zero and one.
        \item In $D(\Z_p [G])$ there is an exact triangle
\begin{align} \label{triangle for sp Selmer complex}
   C^\bullet_{K, \cF} \to C^\bullet_{K, \Sigma} \to  {\bigoplus}_{v \in \Sigma} \mathrm{R}\Gamma_{ / \cF} (\Q_v, T_{K / \Q}) \to  C^\bullet_{K, \cF} [1].
\end{align}
    \end{liste}
\end{lem}

\begin{proof}
    The triangle (\ref{triangle for sp Selmer complex}) in claim (c) is obtained by dualising the triangle 
    \begin{align} \label{triangle for sp Selmer complex and compact support complex}
    \mathrm{R}\Gamma_c (\cO_{K, \Sigma}, \mathrm{T}_p E)
    & \to \mathrm{R}\Gamma_{\cF^\ast} (K, \mathrm{T}_p E)
     \to {\bigoplus}_{v \in \Sigma} \mathrm{R}\Gamma_{\cF^\ast} (\Q_v, T_{K / \Q}) 
\to \mathrm{R}\Gamma_c (\cO_{K, \Sigma}, \mathrm{T}_p E) [1]
\end{align}
that exists by the definition of $\cF$ and the octahedral axiom. As a consequence of (\ref{triangle for sp Selmer complex}), we deduce that $C^\bullet_{K, \cF}$ is perfect and has Euler characteristic $[\Z_p [G]]$. This is because the complex  $C^\bullet_{K, \Sigma}$ is perfect with vanishing Euler characteristic (cf.\@ Lemma \ref{complex lemma}\,(a)) and each complex ${\bigoplus}_{v \in \Sigma} \mathrm{R}\Gamma_{ / \cF} (\Q_v, T_{K / \Q})$ is by construction perfect and has vanishing Euler characteristic if $v \neq p$ resp.\@ Euler characteristic $-[\Z_p [G]]$ if $v = p$ (as follows from the computations of Euler characteristics in \cite{Flach00}). Having thereby proved claim (a), we now turn to claim (b). Firstly, Artin--Verdier induces a canonical exact triangle
\begin{equation} \label{artin--verdier triangle}
\mathrm{R}\Gamma_{\cF} (K, \mathrm{T}_p E) \to C^\bullet_{K, \cF} \to (T_{K / \Q})^+ [-2] \to \mathrm{R}\Gamma_{\cF} (K, \mathrm{T}_p E) [1],
\end{equation}
and this directly implies the first part of claim (b). 
Assuming now that $E (K) [p] = 0$, the isomorphism $H^3_\mathrm{c} (\cO_{K, \Sigma}, \mathrm{T}_p E) \cong (E (K)_\tor \otimes_\Z \Z_p)^\vee$ shows that $\mathrm{R}\Gamma_c (\cO_{K, \Sigma}, \mathrm{T}_p E)$ is acyclic outside degree one and two. From the triangle for $\cF$ analogous to (\ref{triangle for sp Selmer complex and compact support complex}) and Remark \ref{surprisingly subtle calculation} we therefore see that $\mathrm{R}\Gamma_{\cF} (K, \mathrm{T}_p E)$ is acyclic outside degrees one and two. Given this, the triangle (\ref{artin--verdier triangle}) implies both that $C^\bullet_{K, \cF}$ is acyclic outside degree one and two and that one has the exact sequence (\ref{exact sequence for H2}), which is split-exact because $(T_{K / \Q})^+$ is a free $\Z_p [G]$-module. This concludes the proof of claim (b).
\end{proof}

\subsubsection{Determinants and passage to cohomology} \label{split selmr structure det section}

Set $\cF \coloneqq \cF^\mathrm{sp}_{\Sigma, \Pi, \Phi, Q}$
for data $(\Sigma, \Pi, \Phi, Q)$ as in \S\,\ref{definition split selmer structure section} with the additional assumption that all places in $\Pi$ are ramified in $K$. Recall the basis $b_K$ of $T_{K / \Q}^+$ defined in \S\,\ref{s: cpctly supported etale cohom}. In light of Lemma~\ref{split multiplicative selmer complex properties lemma 2}, 
Definition \ref{def projection map from det} provides us with a map
\begin{align*}
F_{K, \Sigma, \Pi, \Phi, Q} \coloneqq \vartheta_{C^\bullet_{K, \cF}, \{ b_K^\ast \}} \: \Det_{\Z_p [G]} ( C^\bullet_{K, \cF})^{-1} \to 
\Q_p [G].
\end{align*}
We next define a composite isomorphism $f_{K, \Sigma, \Pi, \Phi, Q}$ as
\begin{align*}
\Det_{\Z_p [G]} ( C^\bullet_{K, \Sigma})^{-1} & \stackrel{\simeq}{\longrightarrow} \Det_{\Z_p [G]} ( C^\bullet_{K, \cF})^{-1} \otimes_{\Z_p [G]} {\bigotimes}_{\ell \in \Pi \cup \Phi \cup \{ p \}} \Det_{\Z_p [G]} ( \mathrm{R}\Gamma_{/ \cF} (\Q_\ell, T_{K / \Q}))^{-1} \\
& \stackrel{\simeq}{\longrightarrow} \Det_{\Z_p [G]} ( C^\bullet_{K, \cF})^{-1},
\end{align*}
where the first arrow is the isomorphism induced by the triangle (\ref{triangle for sp Selmer complex}) and the second arrow is $\id \otimes (\otimes_{\ell \in \Pi \cup \Phi \cup \{ p \}} \mathrm{Ev}_{x_\ell})$ with $x_\ell$ defined as follows.
\begin{itemize}
\item If $\ell  \in \Pi$, then we we use the element $t_K^{(\ell)}$ from Theorem \ref{multiplicative group thm} to define $x_\ell$ as the image of $(-1)^{\bm{1}_\ell (p)} \cdot t^{(\ell)}_K$ under the isomorphism
\begin{align*}
\Det_{\Z_p [G]} ( \mathrm{R}\Gamma_{/ \cF} (\Q_\ell, T_{K / \Q})) & \cong 
\Det_{\Z_p [G]} ( \mathrm{R}\Gamma ( \Q_\ell, \Z_{p, K / \Q})) \\
& \cong 
\Det_{\Z_p [G]} ( \mathrm{R}\Gamma ( \Q_\ell, \Z_{p} (1)_{K / \Q})^{-1, \#}.
\end{align*}
Here the first isomorphism is induced by the triangle (\ref{tate uniformisation triangle}) and the second isomorphism by local Tate duality (\ref{derived local Tate duality}) and the isomorphism (\ref{dual sharp isomorphism}).
\item If $\ell \in \Phi$, then we take $x_\ell$ to be the unique element of 
\[
\Det_{\Z_p [G]} ( \mathrm{R}\Gamma_{ / \cF} ( \Q_\ell, T_{K / \Q}))
\cong \Det_{\Z_p [G]} ( \RHom_{\Z_p} ( D^\bullet_\ell, \Z_p) [-2]) \cong \Det_{\Z_p [G]} (D^\bullet_\ell)^{-1, \#}
\]
with $\vartheta_{D^\bullet_\ell, \emptyset} (x_\ell) = \Eul_\ell (\sigma_\ell)$ that exists by Lemma \ref{approximation complexes}\,(a).
\item If $p$ is ramified in $K$ but not in $\Pi$, then we take $x_p$ to be the element corresponding to $\id_{\Z_p [G]}$ in 
$
 \Det_{\Z_p [G]} ( \mathrm{R}\Gamma_{/ \cF} (\Q_\ell, T_{K / \Q})) \cong  \Det_{\Z_p [G]} ( \Z_p [G] [-1]) \cong \Z_p [G]^\ast$.
    \item If $p$ is unramified in $K$, then we let $x_p$ be the unique element of 
    \[
    \Det_{\Z_p [G]} ( \mathrm{R}\Gamma_{/ \cF} (\Q_\ell, T_{K / \Q})) \cong 
    \Det_{\Z_p [G]} ( \RHom_{\Z_p} ( D_{p, Q}^\bullet, \Z_p) [-2]) \cong \Det_{\Z_p [G]} ( F_{p, Q}^\bullet)^{-1, \#}
    \]
    with $\vartheta_{D^\bullet_{p, Q}, \{ \id_{\Z_p [G]} \}} (x_p^\#) = p \Eul_p (\sigma_p)$ (which exists by Lemma~\ref{approximation complexes}~(a)~(iii)).
\end{itemize}

The following result gives an explicit description of the map $F_{K, \Sigma, \Pi, Q} \circ f_{K, \Sigma, \Pi, Q}$.

\begin{lem} \label{description of F circ f lem}
Let $K$ be a finite abelian extension of $\Q$ of conductor $mp^n$ with $m \in \N$ coprime to $p$ and $n \geq 0$.
For every $a \in \Det_{\Z_p [G]} (C_{K, \Sigma}^\bullet)^{-1}$ one has the following equality in $\Q_p [G]$.
    \begin{align} \nonumber
(F_{K, \Sigma, \Pi, \Phi, Q} \circ f_{K, \Sigma, \Pi, \Phi, Q}) (a) = \, & 
({\prod}_{\ell \in \Pi \setminus \{ p \}} \nu_{mp^n}^{(\ell)})^\# \cdot \big ( {\prod}_{\ell \in (\Pi \cup \Phi)\setminus \{ p \}} \Eul_\ell (\widetilde \sigma_\ell^{-1}) \big)^{-1} \\
& \quad \cdot (p\Eul_p (\sigma_p^{-1}))^{- \bm{1}_{p^n} (p)} \cdot 
    \mathcal{P}_K (\Theta_{K, \Sigma} (a), Q).
\end{align}
\end{lem}

\begin{proof}
    In the case that $p \not \in \Pi_\mathrm{ram}$, this is a consequence of Proposition \ref{functoriality properties of det projection map}\,(a). In the case that $p \in \Pi_\mathrm{ram}$ we note that $1 - e_{\bm{1}}$ acts as the identity on both sides of the claimed equality, and hence that we may verify this equality over $(1 - e_{\bm{1}}) \Q_p [G]$. Since in $D (\Q_p [G])$ we have the isomorphism
\begin{align*}
(1 - e_{\bm{1}}) \Q_p [G] \otimes_{\Z_p [G]}^\mathbb{L} \mathrm{R} \Gamma (\Q_p, \Z_p (1)_{K / \Q})
& = ( (1 - e_{\bm{1}}) \Q_p [G] \otimes_{\Z_p [G]} H^1 ( \Q_p, \Z_p (1)_{K / \Q})) [-1] \\
& \cong (1 - e_{\bm{1}}) \Q_p [G] [-1], 
\end{align*}
the claim again follows from Proposition \ref{functoriality properties of det projection map}\,(a).
\end{proof}

\subsection{Bounds on the order of vanishing of Mazur--Tate elements} \label{order of vanishing proof section}

In this section we prove the following result, thereby also establishing Theorem \ref{mazur--tate main result 1}\,(a).

\begin{thm} \label{order of vanishing main result}
    Fix an abelian number field $K$ and write the conductor of $K$ as $mp^n$ with $m, n \geq 0$ integers such that $p \nmid m$.
    Suppose that the pair $(K, p)$ satisfies Hypothesis \ref{hyp}. Then one has the inclusions
    \[
    \theta^\mathrm{MT}_K \in I (G)^{r_p + \mathrm{sp} (mp^n) + 2 c^{(p)} (K)}
    \]
    and
    \[
    \theta^\mathrm{MT}_K \in \big( {\prod}_{\ell \in \mathrm{Sp} (mp^n) } I ( \cD^{(\ell)}_K) \big) \cdot \big({\prod}_{\ell \in C_0^{(p)} (m)\cup C_2^{(p)} (K)} I ( \cD^{(\ell)}_K) \big)^2.
    \]
\end{thm}

\begin{proof}
Our approach in this proof is to apply Lemma \ref{inductive argument}, so we begin by verifying the conditions required for this.\\
We will use the Nekov\'a\v{r}--Selmer structure $\cF \coloneqq \cF^\mathrm{sp}_{\Sigma, \Pi, \Phi, Q}$ for the following data:
Let $\Pi \coloneqq \mathrm{Sp} (mp^n)$ be the set of split-multiplicative primes which ramify in $K$, and $\Phi \coloneqq S_N \setminus S_{mp}$ the set of non-$p$-adic bad reduction primes that are unramified in $K$. 
If $p \in \Pi$, then we take $Q \coloneqq (1 - e_\tau) \mathrm{Tr}_{F_{mp^n} / K}(\mathfrak{k}_{mp^n})$. If $p \not \in \Pi$, then we let $Q$ be an arbitrary element of $\widehat{E} (K_p)$.\\
By Theorem \ref{etnc result 2}, there exists an element $\fz_{K}$ of $\Det_{\Z_p [G]} ( C^\bullet_{K, S(K)})^{-1}$ with the property that $\Theta_{K, S(K)} (\fz_{K}) = z^\mathrm{Kato}_K$. Lemma \ref{description of F circ f lem} combines with the definition of $y_K^\mathrm{Kato}$ to imply that one has
\begin{align} \nonumber 
(F_{K, \Sigma, \Pi, \Phi, Q} \circ f_{K, \Sigma, \Pi, \Phi, Q}) (\fz_{K}) = \,  
\big({\prod}_{\ell \in \Pi \setminus \{ p \}} \nu_{mp^n}^{(\ell)} \Eul_\ell (\widetilde \sigma_\ell)^{-1} \big)^\# & \cdot (p\Eul_p (\sigma_p^{-1}))^{- \bm{1}_{p^n} (p)} \\
&  \cdot 
    \mathcal{P}_K (y^\mathrm{Kato}, Q).
    \label{description of F circ f}
\end{align}
Now, from Lemma \ref{split multiplicative selmer complex properties lemma}\,(c) we know that $H^1 (C^\bullet_{K, \cF})$ is $\Z_p$-torsion free if $H^1 (\cO_{K, S(K)}, \mathrm{T}_p E)$ is. The computation of the Euler characteristic of $C^\bullet_{K, \cF}$ in Lemma \ref{split multiplicative selmer complex properties lemma 2}\,(a) combines with the exact sequence (\ref{exact sequence for H2}) to show that $H^2 (C^\bullet_{K, \cF}) \otimes_{\Z_p [G]} \Z_p$ has the same rank as $H^1 (C^\bullet_{K, \cF})^G$. From Lemma \ref{split multiplicative selmer complex properties lemma}\,(a) we therefore deduce that 
    \begin{equation} \label{Fitting ideal augmentation ideal 1}
    \Fitt^0_{\Z_p [G]} ( H^2 (C^\bullet_{K, \cF})) \subseteq \Fitt^0_{\Z_p [G]} ( H^2 (C^\bullet_{K, \cF}) \otimes_{\Z_p [G]}\Z_p) \subseteq I (G)^{|\Pi| + r_p}
    \end{equation}
    with $r_p \coloneqq \mathrm{rk}_{\Z_p} (\Sel_{p, E / \Q}^\vee )$. 
    In addition, the surjection in Lemma \ref{split multiplicative selmer complex properties lemma}\,(b) implies that
    \begin{equation} \label{Fitting ideal augmentation ideal 2}
    \Fitt^0_{\Z_p [G]} ( H^2 (C^\bullet_{K, \cF})) \subseteq \Fitt^0_{\Z_p [G]} ( Y_{K, \Pi}) = {\prod}_{\ell \in \Pi} I (\cD^{(\ell)}_K).
      \end{equation}
We next combine these observations with Proposition \ref{integrality properties of det projection map}.
By Lemma \ref{split multiplicative selmer complex properties lemma 2} we may apply \cite[Prop.\@ A.11\,(i)]{sbA} to the complex $C^\bullet_{K, \cF}$ in order to deduce that it admits a standard representative in the sense of Definition \ref{standard representative def}. Given this, it follows from Proposition~\ref{integrality properties of det projection map}\,(b) and the inclusions (\ref{Fitting ideal augmentation ideal 1}) and (\ref{Fitting ideal augmentation ideal 2}) that  
\[
(F_{K, \Sigma, \Pi, \Phi, Q} \circ f_{K, \Sigma, \Pi, \Phi, Q}) (\fz_{K}) \in \Fitt^0_{\Z_p [G]} ( H^2 ( C^\bullet_{K, \cF})) \subseteq I (G)^{r_p + |\Pi|}
\cap {\prod}_{\ell \in \Pi} I (\cD^{(\ell)}_K).
\]
Since this containment is true for every field $K$ satisfying Hypothesis \ref{hyp}, it combines with (\ref{description of F circ f}) to imply that we have verified the conditions of Lemma \ref{inductive argument} (in the version of Remark \ref{inductive argument variant} if $p \in \Pi$). 
The inclusions claimed in Theorem~\ref{order of vanishing main result} therefore now follow from Lemma~\ref{inductive argument} upon noting that Theorem~\ref{local points main result}\,(d) shows 
$\nu_{mp^n}^{(\ell)} \in I ( \cD^{(\ell)}_{mp^n})^2$ for every prime number $\ell$ that belongs to the subset $ C_0^{(p)} (m)\cup C_2^{(p)} (K)$ of the set $\mathscr{Y} (K)$ appearing in Lemma \ref{inductive argument}.
\end{proof}

\subsection{Congruences for Mazur--Tate elements} \label{congruences proof section}

In this section we will prove Theorem \ref{mazur--tate main result 2} as an application of the formalism of Bockstein morphisms from \S\,\ref{bockstein section}. We begin by introducing some general notation that will be in place throughout the section. We will assume the conditions of Theorem \ref{mazur--tate main result 2} to be valid in this section.\\
 To do this, we fix a finite abelian extension $L$ of $\Q$ and let $K$ be a subfield of $L$. 
 The relevant Galois groups will be denoted as $\cG \coloneqq \gal{L}{\Q}$, $G \coloneqq \gal{K}{\Q}$, and $H \coloneqq \gal{L}{K}$.
 We denote the conductors of $K$ and $L$ as $mp^n$ and $m' p^{n'}$, respectively, where $m, m' \in \N$ are coprime with $p$ and $n,n ' \geq 0$ are integers.
We write $\Pi \subseteq \mathrm{Sp} (m'p^{n'})$ for a set of split-multiplicative primes of $E$ that ramify in $L$, and $\Pi' \subseteq \Pi$ for the subset of primes that split completely in $K$. We also fix an ordering $\Pi' = \{ \ell_1, \dots, \ell_{s} \}$, where we adopt the convention that $\ell_1 = p$ if $p \in \Pi'$.
We define $M' \coloneqq m' p^{n'} \prod_{\ell \in \Pi'} \ell^{- \ord_\ell (m'p^{n'})}$ and take 
\[
Q \coloneqq  (p \Eul_p (\sigma_p))^{\bm{1}_{M'} (p)} \cdot  
\mathrm{Tr}_{F_{m' p^{n'}} / L}(\mathfrak{k}_{m' p^{n'}})
\]
with $\mathfrak{k}_{m' p^{n'}}$ the element constructed in Theorem \ref{local points main result}.
Note that $Q$ belongs to $\widehat{E} (L_p)$ under the assumptions of Theorem \ref{mazur--tate main result 2}. 
This uses Theorem~\ref{local points main result}~(b) and that 
$p \Eul_p (\sigma_p)$ belongs to $\Z_p [\cG]^\times$ if $p$ ramifies in $L$ by Lemma \ref{invertible Euler factors lemma} and the condition on $M'$ assumed in Theorem~\ref{mazur--tate main result 2}.
Given an additional subset $\Phi$ of $S_{N} \setminus S_{mp}$, we will use the 
Nekov\'a\v{r}--Selmer structure $\cF^\mathrm{sp}_{S(L), \Pi, \Phi, Q}$ defined in \S\,\ref{definition split selmer structure section}
as well as the associated complexes
\[
\overline{\SC}^\bullet_{\Pi, \Phi} \coloneqq C^\bullet_{L, \cF^\mathrm{sp}_{S(L), \Pi, \Phi, Q}} 
\quad \text{ and } \quad 
\overline{\SC}^\bullet_{\Pi, \Phi} \coloneqq \SC^\bullet_{\Pi, \Phi} \otimes_{\Z_p [\cG]}^\mathbb{L} \Z_p [G].
\]
The complexes $\overline{\SC}^\bullet_{\Pi, \Phi}$ inherit from $\SC^\bullet_{\Pi, \Phi}$ the properties of being perfect, acyclic outside degrees one and two, and that $H^1 (\overline{\SC}^\bullet_{\Pi, \Phi})$ is $\Z_p$-torsion free if $E (K) [p] = 0$.

\subsubsection{Definition of Bockstein morphisms}

By Lemmas \ref{split multiplicative selmer complex properties lemma}\,(b) and \ref{split multiplicative selmer complex properties lemma 2}\,(b) the module $Y_{L, \Pi'}$ is a quotient of $H^2 (\SC^\bullet_{\Pi, \Phi})$.
Write $\mathfrak{X} (\Pi')$ for the ordered set of $\Z_p [\cG]$-generators of the module $Y_{L, \Pi'}$
that is induced by our choices of places $\{w_1, \dots, w_s \}$ of $L$ above $\{ \ell_1, \dots, \ell_s\}$.
As a special case of Definition~\ref{bockstein map def}, we then have Bockstein maps
\begin{align*}
\beta_i \: H^1 (\overline{\mathrm{SC}}^\bullet_{\Pi, \Phi}) \xrightarrow{\beta_{\mathrm{SC}^\bullet_{\Pi, \Phi},  I_H}} H^2 ( \mathrm{SC}^\bullet_{\Pi, \Phi}) \otimes_{\Z_p [\cG]} I_H / I_H^2 & \longrightarrow Y_{L, \Pi'} \otimes_{\Z_p [\cG]} I_H \stackrel{w_i^\ast}{\longrightarrow} 
I_H / I_i I_H
\end{align*}
with $I_i \coloneqq I (\cD^{(\ell_i)}_L)$ the augmentation ideal associated with the decomposition group $\cD^{(\ell_i)}_L$ of $\ell_i$.
On the other hand, applying Definition~\ref{bockstein map def} to the complexes $A^\bullet_{K, \ell_i}$ also provides us with `local' Bockstein maps
\[
\beta_i^\mathrm{loc} \: H^1 (A^\bullet_{K, \ell_i}) 
\xrightarrow{\beta_{A^\bullet_{K, \ell_i}, I_H}} H^2 (A^\bullet_{L, \ell_i}) \otimes_{\Z_p [\cG]} I_H / I_H^2 \xrightarrow{w_i^\ast}  I_H / I_i I_H
\]
that have already appeared in \S\,\ref{local descent calculations section}. To state the relation between $\beta_i$ and $\beta_i^\mathrm{loc}$, we write $\mathrm{loc}_i \: \overline{\mathrm{SC}}^\bullet_{\Pi, \Phi} \to A^\bullet_{K, \ell_i}$ for the natural `localisation morphism'. Since $H^2 (\mathrm{loc}_i)$ agrees with the composite map $H^2 (\overline{\mathrm{SC}}^\bullet_{\Pi, \Phi}) \to Y_{K, \Pi'} \to Y_{K, \{ \ell_i \}} = H^2 (A^\bullet_{L, \ell_i})$, the naturality of the definition of Bockstein maps implies that we have a commutative diagram
\begin{equation} \label{local global Bockstein} \begin{tikzcd}
    H^1 (\overline{\mathrm{SC}}^\bullet_{\Pi, \Phi})  \arrow{d}[left]{H^1 (\mathrm{loc}_i)} \arrow{r}{\beta_i} & I_H / I_i I_H \arrow[equals]{d} \\
    H^1 (A^\bullet_{K, \ell_i})  \arrow{r}{\beta_i^\mathrm{loc}} & I_H / I_i I_H.
\end{tikzcd} \end{equation}
To proceed, it is convenient to set $v_i \coloneqq (w_i)_{\mid_K}$ and to recall that we have isomorphisms
\[
I_H / I_i I_H \cong \Z_p [G] \otimes_{\Z_p} (I (H) / I_i I (H))
\quad \text{ and } \quad 
H^1 (A^\bullet_{K, \ell_i}) \cong \Z_p [G] \otimes_{\Z_p} H^1 (K_{v_i}, \Z_p (1)),
\]
for the first of which we refer to \cite[(3)]{Sano14} for details.
Since $H^1 ( \overline{\mathrm{SC}}^\bullet_{\Pi, \Phi})$ is $\Z_p$-free, we therefore obtain isomorphisms (cf.\@ \cite[Prop.~A.6]{MazurRubin16} or \cite[Lem.\@ 2.5]{Sano14} for more details)
\begin{align*}
\Hom_{\Z_p [G]} ( H^1 ( \overline{\mathrm{SC}}^\bullet_{\Pi, \Phi}), I_H / I_i I_H) & \stackrel{\simeq}{\longrightarrow} H^1 ( \overline{\mathrm{SC}}^\bullet_{\Pi, \Phi})^\ast \otimes_{\Z_p} (I (H) / I_i I (H)) \\
\Hom_{\Z_p [G]} ( H^1 ( \overline{\mathrm{SC}}^\bullet_{\Pi, \Phi}), H^1 (A^\bullet_{K, \ell_i}))
& \stackrel{\simeq}{\longrightarrow} H^1 ( \overline{\mathrm{SC}}^\bullet_{\Pi, \Phi})^\ast \otimes_{\Z_p} H^1 (A^\bullet_{K, \ell_i}).
\end{align*}
As in Proposition \ref{bockstein proposition}\,(b), it follows that the maps $\beta_i$ and $H^1 (\mathrm{loc}_i)$ induce maps
\begin{align*}
(\exprod_{1  \leq i \leq s} \beta_i) \: \bidual^{s}_{\Z_p [G]} H^1 ( \overline{\mathrm{SC}}^\bullet_{\Pi, \Phi}) & \longrightarrow I (H)^s / \mathfrak{A} I (H) \\
(\exprod_{1  \leq i \leq s} H^1 (\mathrm{loc}_i)) \: \bidual^{s}_{\Z_p [G]} H^1 ( \overline{\mathrm{SC}}^\bullet_{\Pi, \Phi}) & \longrightarrow \bigotimes_{i = 1}^s H^1 (A^\bullet_{K, \ell_i})
\end{align*}
with $\mathfrak{A} \coloneqq \prod_{i = 1}^s I_i$,
and the diagram (\ref{local global Bockstein}) implies that we have
\begin{equation} \label{local global wedge of bockstein}
(\exprod_{1  \leq i \leq s} \beta_i) = (\otimes_{i = 1}^s \beta^\mathrm{loc}_i) \circ (\exprod_{1  \leq i \leq s} H^1 (\mathrm{loc}_i)).
\end{equation}

\subsubsection{Congruences for Bockstein morphisms}

Recall that in \S\,\ref{s: cpctly supported etale cohom} we have defined a $\Z_p [\cG]$-basis $b_L$ of $T^+_{L / \Q}$.
Let $\mathfrak{X}' (\Pi')$ denote the ordered set of generators of $T_{K / \Q}^+ \oplus Y_{K, \Pi'}$ induced by $\{ b_L \} \cup \mathfrak{X} (\Pi')$ and define
\[
\overline{F}_{\Pi, \Phi} \coloneqq \vartheta_{\overline{\SC}^\bullet_{\Pi, \Phi}, \mathfrak{X}' (\Pi')} \: \Det_{\Z_p [G]} ( \overline{\SC}^\bullet_{\Pi, \Phi})^{-1} \to 
\bidual^{s}_{\Z_p [G]} H^1 (\overline{\SC}^\bullet_{\Pi, \Phi})
\]
as the relevant instance of Definition \ref{def projection map from det}.
We also write $\mathrm{pr}$ for the `projection map' 
\[
\Det_{\Z_p [\cG]} ( \mathrm{SC}^\bullet_{\Pi, \Phi})^{-1} \to 
\Det_{\Z_p [\cG]} ( \mathrm{SC}^\bullet_{\Pi, \Phi})^{-1} \otimes
_{\Z_p [\cG]} \Z_p [G] \cong 
\Det_{\Z_p [G]} ( \overline{\SC}^\bullet_{\Pi, \Phi})^{-1}
\]
and $\fz_{L}$ for the element of $\Det_{\Z_p [\cG]} (C^\bullet_{L, S(L)})^{-1}$ with $\Theta_{L, S(L)} (\fz_{L}) = z^\mathrm{Kato}_L$ that exists by Theorem~\ref{etnc result 2}.
With this notation in place, 
Proposition \ref{bockstein proposition}\,(c) shows that
\begin{equation} \label{congruence coming from appendix case 1}
(F_{\Pi, \Phi} \circ f_{\Pi, \Phi}) (\fz_{L}) \equiv \big( (\exprod_{1 \leq i \leq s} \beta_i) \circ \overline{F}_{\Pi, \Phi} \circ \mathrm{pr} \circ f_{\Pi, \Phi} \big) (\fz_{L}) \mod  I_H \mathfrak{A}
\end{equation}
with $F_{\Pi, \Phi} \coloneqq F_{L, S(L), \Pi, \Phi, Q}$ and $f_{\Pi, \Phi} \coloneqq f_{L, S(L), \Pi, \Phi, Q}$ as in Lemma \ref{description of F circ f lem}. \\
Having computed the left-hand side of (\ref{congruence coming from appendix case 1}) in Lemma \ref{description of F circ f lem}, we will explicitly calculate its right-hand side in the next section \S\,\ref{fun computation section}. Theorem \ref{mazur--tate main result 2} is then obtained as a consequence of (\ref{congruence coming from appendix case 1}) and these calculations in \S\,\ref{proof of main result 2 section}.

\subsubsection{A computation of Bockstein morphisms} \label{fun computation section}

To state the next result, it is convenient to introduce the notation
\begin{align*}
\eta_{\Pi, \Phi} \coloneqq \big( {\prod}_{\ell \in \Pi'} \Tam_\ell^{-1} \Eul_\ell (1)^{-1} \big) \cdot 
({\prod}_{\ell \in (\Pi \setminus \Pi') \cup \Phi} \Eul_\ell (\sigma_\ell)^{-1} \nu^{(\ell)}_{mp^n} )^\#
\quad \in \Q_p [G_{mp^n}],
\end{align*}
where we have set $\nu^{(l)}_{mp^n} = 1$ if $\ell \nmid m$.
The following lemma then computes the right-hand side of (\ref{congruence coming from appendix case 1}).

\begin{lem} \label{this computes the RHS}
With notation as above, one has
    \begin{align*}
    \big( (\bigwedge_{1 \leq i \leq s} \beta_i) \circ \overline{F}_{\Pi, \Phi} \circ \mathrm{pr} \circ f_{\Pi, \Phi} \big) (\fz_{L})  = \, & \eta_{\Pi, \Phi} \cdot \pi_{F_{M'} / K} (\mathcal{P}_{F_{M'}} (z^\mathrm{Kato}_{M'},  \mathfrak{k}_{M'})) \cdot  \prod_{i = 1}^s  (\rec_{\ell_i} (q_{E, \ell_i}) - 1)
\end{align*}
as an equality in $\Z_p [\cG] /  I_H \mathfrak{A}$.
\end{lem}

In light of the relation (\ref{local global wedge of bockstein}), the first step towards proving this result will be the calculation of the image under the map $\exprod_{1 \leq  i \leq s} H^1 (\loc_i)$ of the element 
\[
a_{\Pi, \Phi} \coloneqq (\overline{F}_{\Pi, \Phi} \circ \mathrm{pr} \circ f_{\Pi, \Phi} \big) (\fz_{L}).
\]

\begin{lem} \label{calculation part 1}
    In $\bigotimes_{i = 1}^s ( \Q_p \otimes_{\Z_p} H^1 (A^\bullet_{K, \ell_i}))$, one has the equality
 \begin{equation} \label{restatement of calculation part 1}
(\exprod_{1 \leq  i \leq s} H^1 (\loc_i)) (a_{\Pi, \Phi}) = \eta_{\Pi, \Phi} \cdot \pi_{F_{M'} / K} (\mathcal{P}_{F_{M'}} (z^\mathrm{Kato}_{M'}, \mathfrak{k}_{M'})) \cdot \otimes_{i = 1}^s q_{E, \ell_i}.
\end{equation}
\end{lem}

\begin{proof}
Write $\overline{f}_{\Pi, \Phi}$ for the map obtained from $f_{\Pi, \Phi}$ via base-changing to $\Z_p [G]$.
We also let $\Pi^\ast \coloneqq \Pi \cap \{ p \}$ and define $U \subseteq \{1, \dots, s \}$ be the subset with $\{ \ell_i \mid i \in U \} = \Pi' \setminus \{ p \}$ (so $U = \{1, \dots, s\}$ if $p \not \in \Pi'$ and $U = \{2, \dots, s \}$ otherwise).
\\
We may then compute, using 
    Proposition \ref{functoriality properties of det projection map} and Theorem \ref{multiplicative group thm}, that
\begin{align*} 
(\exprod_{i \in U} H^1 (\loc_i)) (a_{\Pi, \Phi})
& = (\exprod_{i \in U} H^1 (\loc_i)) \big( (\overline{F}_{\Pi, \Phi}  \circ \overline{f}_{\Pi, \Phi} \circ \mathrm{pr} \big) (\fz_{L}) \big) \\ \nonumber
 & = \big (\bigotimes_{i \in U} \vartheta_{K, \ell, v_i} (t^{(\ell_i)}_K) \big) \cdot  \big (\prod_{\ell \in (\Pi \setminus (\Pi' \cup \{ p \}) \cup \Phi} \vartheta_{K, \ell}^0 (x_\ell) \big)^\# \cdot 
(\overline{F}_{\Pi^\ast, \emptyset} \circ \overline{f}_{\Pi^\ast, \emptyset}) (\mathrm{pr} (\fz_L)) \\ 
& = \big (\bigotimes_{i \in U}  (- \ell_i^{- (\ord_{\ell_i} (m') - 1)} \Eul_{\ell_i} (1)^{-1} \otimes \ell) \big)\\
& \qquad \cdot 
(\prod_{\ell \in (\Pi \setminus (\Pi' \cup \{ p \}) \cup \Phi} \Eul_\ell (\sigma_\ell)^{-1} \nu^{(\ell)}_m )^\#
 \cdot (\overline{F}_{\Pi^\ast, \emptyset} \circ \overline{f}_{\Pi^\ast, \emptyset}) (\mathrm{pr} (\fz_L)).
\end{align*}
If $p \not \in \Pi'$, then we can use Proposition \ref{functoriality properties of det projection map}\,(d) and
Lemma \ref{description of F circ f lem} to compute that 
\begin{align*}
    (\overline{F}_{\Pi^\ast, \emptyset} \circ \overline{f}_{\Pi^\ast, \emptyset}) (\mathrm{pr} (\fz_L)) & = 
    (\pi_{L / K} \circ F_{\Pi^\ast, \emptyset} \circ f_{\Pi^\ast, \emptyset}) ( \fz_{L}) \\
    & = (p\Eul_p (\sigma_p^{-1}))^{-\bm{1}_{M'} (p)} \cdot \pi_{L / K} (\cP_L (z_L^\mathrm{Kato}, Q )) \\
    & = \pi_{F_{m'p^{n'}} / K} (\cP_L (z_L^\mathrm{Kato}, \mathfrak{k}_{m'p^{n'}})\\
    & = ( {\prod}_{\ell \in \Pi' \setminus \{ p \}} (- 1) \cdot \ell^{\ord_\ell (m')  - 1} \big) \cdot 
\pi_{F_{M'} / K} (\mathcal{P}_{F_{M'}} (z^\mathrm{Kato}_{M'}, \mathfrak{k}_{M'})),
\end{align*}
where the last equality is a consequence of Theorem \ref{local points main result}\,(c),(ii). Since $\ord_\ell (q_{E, \ell}) = \Tam_\ell$ for all $\ell \in \Pi'$, this proves the claimed equality in the case $p \not \in \Pi'$.\\
It remains to consider the case that $p$ belongs to $\Pi'$. In this case, one has
\[
( {\prod}_{\ell \in \Pi' \setminus \{ p \}} (- 1) \cdot \ell^{\ord_\ell (m')  - 1} \big) \cdot 
\pi_{F_{M'} / K} (\mathcal{P}_{F_{M'}} (z^\mathrm{Kato}_{M'}, \mathfrak{k}_{M'}))
= \pi_{F_{m'} / K} (\cP_{F_{m'}} ( h^1_p ( z^\mathrm{Kato}_{m'}), \mathfrak{k}_{F_{m'}})) 
\]
and so, writing $h^1_p \coloneqq H^1 ( h_p)$ for the map induced on cohomology by the map $h_p$ in (\ref{tate uniformisation triangle}), it suffices to prove that
\begin{equation} \label{claim at p}
(\loc_p \circ  \overline{F}_{\{ p \}, \emptyset} \circ \overline{f}_{\{ p \}, \emptyset}) (\mathrm{pr} (\fz_L)) = \pi_{F_{m'} / K} (\cP_{F_{m'}} ( h^1_p ( z^\mathrm{Kato}_{m'}), \mathfrak{k}_{F_{m'}})) \cdot q_{E, p}.
\end{equation}
The equation (\ref{claim at p}) can be proved via an argument very similar to Proposition \ref{functoriality properties of det projection map}\,(b) which we now briefly sketch. \\
    Since $p$ is assumed to split completely, $\mathrm{R}\Gamma (\Q_p, \Z_p (1)_{K / \Q})$ and $\mathrm{R}\Gamma (\Q_p, \Z_{p, K / \Q})$ are represented by $[\widehat{K_p^\times} \xrightarrow{0} \Z_p [G]]$ (first term in degree $1$) and $[\Z_p [G] \xrightarrow{0}  (\widehat{K_p^\times})^\ast]$ (first term in degree $0$), respectively. In particular, it follows from $(\vartheta_{K, p, v_0} \circ \mathrm{pr}_{L / K})(t^{(p)}_L) = \NN_{F_{m'p^n} / K}(\mathfrak{l}_{F_{m'}}) \wedge p$, as is proved in Theorem \ref{multiplicative group thm}\,(b), that $t^{(p)}_K$ can be described as the element $(\NN_{F_{m'} / K}(\mathfrak{l}_{F_{m'}}) \wedge p) \otimes \id_{\Z_p [G]}$ of $(\exprod^2_{\Z_p [G]} \widehat{K^\times_p}) \otimes_{\Z_p [G]} \Z_p [G]^\ast = \Det_{\Z_p [G]} ( \mathrm{R}\Gamma (\Q_p, \Z_p (1)_{K / \Q}))^{-1}$.\\
    We now write $[P \xrightarrow{\partial} P]$ for the representative of $C^\bullet_{K, S(L)}$ constructed in Lemma \ref{complex lemma}\,(b). Here $P$ is a finitely generated free $\Z_p [G]$-module of rank $n$, say, and we choose a $\Z_p [G]$-basis $x_1, \dots, x_n$ such that the composite map $P \to H^2 (C^\bullet_{K, S(L)}) \xrightarrow{\pi} (T^+_{K / \Q})$ sends $x_1$ to the element $b^\ast$ defined in \S\,\ref{s: cpctly supported etale cohom}. 
    It follows from the triangle (\ref{triangle for sp Selmer complex}) and a standard mapping cone construction that 
    the complex $\overline{\SC}^\bullet_{\{ p \}, \emptyset}$ admits a representative of the form $[P \oplus \Z_p [G] \xrightarrow{(\partial \oplus \varphi, 0)} P \oplus (\widehat{K_p^\times})^\ast]$ with the map
    \[
    \varphi \: P \to (\widehat{K_p^\times})^\ast, \quad y \mapsto \{ z \mapsto {\sum}_{\sigma \in G} ( h_p (\sigma z), z)_{\mathbb{G}_m / K} \sigma^{-1} \},
    \]
    where $(\cdot, \cdot)_{\mathbb{G}_m / K}$ denotes the cup product pairing $H^1 (\Q_p, \Z_{p, K / \Q}) \times H^1 (\Q_p, \Z_p (1)_{K / \Q}) \to \Z_p$ (analogous to (\ref{semilocal Tate pairing})).
    Given $a = a_1 \otimes (\bigwedge_{1 \leq i \leq n} x_i^\ast)$ in $ \Det_{\Z_p [G]} (C^\bullet_{K, S(L)})^{-1} = (\exprod^n_{\Z_p [G]} P) \otimes_{\Z_p [G]} (\exprod^n_{\Z_p [G]} P^\ast )$, one then has
    \[
    \mathrm{Ev}_{t_K^{(p)}} (a) = (a_1 \wedge 1) \otimes \big (({\bigwedge}_{1 \leq i \leq n} x_i^\ast) \wedge (\NN_{F_{m'} / K}(\mathfrak{l}_{F_{m'}}) \wedge p) \big)
    \]
when regarded as an element of 
\[
\Det_{\Z_p [G]} (\overline{\SC^\bullet_{\{ p \}, \emptyset}})^{-1} = \big( \exprod^{n + 1}_{\Z_p [G]} (P \oplus \Z_p [G] \big) \otimes_{\Z_p [G]} \big( \exprod^{n + 2}_{\Z_p [G]} (P^\ast \oplus \widehat{K_p^\times}) \big).
\]
Now, the module $Y_{K, \{ p \}}$ is identified with the $\Z_p [G]$-submodule of $(\widehat{K_p^\times})^\ast$ generated by the map $\Ord_p$ (which is the $\Z_p [G]$-dual basis element of $p$) so that we have $\mathfrak{X}' ( \{ p \}) = (x_1, p)$.
The explicit description of $\vartheta_{\overline{\SC}^\bullet_{\{ p \}, \emptyset}, \mathfrak{X}' (\{ p \})}$ given in Lemma \ref{explicit description projection map lem} therefore allows us to calculate that
\begin{align*}
(\overline{F}_{\{ p \}, \emptyset} \circ \overline{f}_{\{ p \}, \emptyset}) (\mathrm{pr} (\fz_L))
& = {\sum}_{\sigma \in G} \big ( (h^1_p \circ \exprod_{2 \leq i \leq n} (x_i^\ast \circ \partial)) (\sigma a_1), \NN_{F_{m'} / K} (\mathfrak{l}_{F_{m'}}) \big )_{\mathbb{G}_m / K} \sigma^{-1} \\
& =  {\sum}_{\sigma \in G} ( (h^1_p  \circ \Theta_{K, S(L)}) (\sigma a), \NN_{F_{m'} / K} (\mathfrak{l}_{F_{m'}}) )_{\mathbb{G}_m / K} \sigma^{-1}
\end{align*}
in $\Z_p [G] = H^0 (\Q_p, \Z_{p, K / \Q})$. The connecting homomorphism $H^0 (\Q_p, \Z_{p}) \to H^1 (\Q_p, \Z_{p} (1))$ sends $1$ to $q_{E, p}$ by Remark \ref{tate uniformisation connecting map rk}, hence we conclude that
\begin{align*}
(H^1 (\loc_1) \circ  \overline{F}_{\{ p \}, \emptyset} \circ \overline{f}_{\{ p \}, \emptyset}) (\mathrm{pr} (\fz_L)) & = \pi_{F_{m'} / K} \big ( {\sum}_{\sigma \in G_{m'}}  ( h^1_p ( \sigma z^\mathrm{Kato}_{m'}), \mathfrak{l}_{F_{m'}})_{\mathbb{G}_m / F_{m'}}\sigma^{-1} \big) \cdot q_{E, p} \\
& = \pi_{F_{m'} / K} \big (\cP_{F_{m'}} (  z^\mathrm{Kato}_{m'}, \mathfrak{k}_{F_{m'}}) \big) \cdot q_{E, p}, 
\end{align*}
as claimed in (\ref{claim at p}). 
This concludes the proof of the lemma. 
\end{proof}

We can now give the proof of Lemma \ref{this computes the RHS}.
\begin{proof}[Proof (of Lemma \ref{this computes the RHS})]
As a first step, we use Lemma \ref{calculation part 1} to calculate $(\exprod_{1 \leq  i \leq s} H^1 (\loc_i)) (a_{\Pi, \Phi})$. An extra argument is necessary because Lemma \ref{calculation part 1} only gives an equality in a $\Q_p$-vector space and so does not contain information about the torsion component.\\ 
To do this, we note that we have the commutative diagram
\begin{cdiagram}[row sep=small]
 & \bigoplus_{i = 1}^s H^0 ( \Q_{\ell_i}, \Z_{p, K / \Q}) \arrow{r} \arrow{d}{\simeq} & H^1 (\overline{\mathrm{SC}}^\bullet_{\Pi, \Phi}) \arrow{r} \arrow{d}{\oplus_{i = 1}^s H^1 (\mathrm{loc}_i)} & H^1 (\cO_{K, S(L)}, \mathrm{T}_p E) \arrow{d} \\
 0 \arrow{r} & \bigoplus_{i = 1}^s \Z_p [G] \arrow{r}{\oplus_{i = 1}^s\delta_i} & \bigoplus_{i = 1}^s H^1 (A^\bullet_{K, \ell_i}) \arrow{r} & 
 \bigoplus_{i = 1}^s H^1 (\Q_{\ell_i}, T_{K / \Q}),
\end{cdiagram}%
where the bottom line is induced by the exact sequence (\ref{tate uniformisation exact sequence}). 
By Remark \ref{tate uniformisation connecting map rk}, the connecting homomorphism arising from (\ref{tate uniformisation exact sequence}), labelled $\delta_i$ in the diagram above, sends $1$ to $q_{E, \ell_i} \in \widehat{K^\times_{v_i}}$ with $v_i \coloneqq (w_{i})_{\mid_K}$ our fixed choice of place of $K$ above $\ell_i$ (that also induces the isomorphism $H^0 ( \Q_{\ell_i}, \Z_{p, K / \Q}) \cong \Z_p [G]$). This shows that $\otimes_{i = 1}^s q_{E, \ell_i}$ is contained in the image of the 
composite map 
\[
\bigotimes_{i = 1}^s H^0 ( \Q_{\ell_i}, \Z_{p, K / \Q}) \to \bidual^s_{\Z_p [G]} 
H^1 (\overline{\mathrm{SC}}^\bullet_{\Pi, \Phi}) \xrightarrow{\exprod_{1 \leq  i \leq s} H^1 (\loc_i)} \bigotimes_{i = 1}^s H^1 (A^\bullet_{K, \ell_i}).
\]
In particular, $\otimes_{i = 1}^s q_{E, \ell_i}$ is contained in the image of $\exprod_{1 \leq  i \leq s} H^1 (\loc_i)$.
Moreover, we know that  
\[
\eta_{L / K} \cdot \pi_{F_{M'} / K} (\mathcal{P}_{F_{M'}} (z^\mathrm{Kato}_{M'}, \mathfrak{k}_{M'})) = \eta_{L / K} \cdot \pi_{F_{M'} / K} (\mathcal{P}_{F_{M'}}  (\Theta_{F_{M'}, S (F_{M'})} (\fz_{F_{M'}}), \mathfrak{k}_{M'})) 
\]
belongs to $\Z_p [G]$ by (\ref{prop 4.6 second inclusion}) in Proposition \ref{etnc explicit cor} (here we have used that each $\Eul_\ell (\sigma_\ell)^{-1} \nu^{(\ell)}_m$ belongs to the ideal $\Eul_\ell (\tilde \sigma_\ell)^{-1} \NN_{\cI^{(\ell)}} \Z_p [G] + \Z_p [G]$ of $\Q_p [G]$ for every $\ell \in \Pi$ by Theorem \ref{local points main result}\,(d),
and $\Tam_\ell^{-1}$ generates $\Ann_{\Z_p [G]} ( ( E / E_0) (K_\ell))$ if $\ell \in \Pi'$).
We have therefore proved that
\begin{align} \label{the expression}
 \eta_{L / K} 
\cdot \pi_{F_{M'} / K} (\mathcal{P}_{F_{M'}} (z^\mathrm{Kato}_{M'}, \mathfrak{k}_{M'}))
\cdot 
( \otimes_{i = 1}^s  q_{E, \ell_i})
\in (\exprod_{1 \leq  i \leq s} H^1 (\loc_i)) \big( \bidual^s_{\Z_p [G]} H^1 (\overline{\mathrm{SC}}^\bullet_{\Pi, \Phi}) \big).
\end{align}
Now, the composite map
\[
(\exprod_{1 \leq  i \leq s} H^1 (\loc_i)) \big( \bidual^s_{\Z_p [G]} H^1 (\overline{\mathrm{SC}}^\bullet_{\Pi, \Phi}) \big)
\stackrel{\subseteq}{\longrightarrow}
\bigotimes_{i = 1}^s H^1 (A^\bullet_{K, \ell_i}) 
\to \bigotimes_{i = 1}^s (\Q_p \otimes_{\Z_p} H^1 (A^\bullet_{K, \ell_i}))
\]
is injective (because $\bidual^s_{\Z_p [G]} H^1 (\overline{\mathrm{SC}}^\bullet_{\Pi, \Phi})$ is $\Z_p$-torsion free) and so 
we conclude from (\ref{the expression}) that the equality (\ref{restatement of calculation part 1}) of Lemma \ref{calculation part 1}) in fact already holds in $\bigotimes_{i = 1}^s  H^1 (A^\bullet_{K, \ell_i})$.
By (\ref{local global wedge of bockstein}) and Lemma \ref{bockstein maps and reciprocity maps}, this argument shows that
\begin{align*}
     (\exprod_{1 \leq i \leq s} \beta_i) (a_{\Pi})   = &  \eta_{L / K} 
\cdot \pi_{F_{M'} / K} (\mathcal{P}_{F_{M'}} (z^\mathrm{Kato}_{M'}, \mathfrak{k}_{M'}))
 \cdot 
 {\prod}_{i = 1}^s (\rec_{\ell_i} (q_{E, \ell_i}) - 1), 
\end{align*}
as required to prove Lemma \ref{this computes the RHS}.
\end{proof}

\subsubsection{The proof of Theorem \ref{mazur--tate main result 2}}
\label{proof of main result 2 section}

We now explain how the congruence (\ref{congruence coming from appendix case 1}) implies Theorem \ref{mazur--tate main result 2}.
To do this we take $\Pi \coloneqq \mathrm{Sp} (m'p^{n'})$ and $\Phi \coloneqq S_N \setminus S_{m'p}$. 
By assumption, $\Eul_\ell (\tilde \sigma_\ell)$ belongs to $\Z_p [\cG]^\times$ for every $\ell \in S_{m'} \setminus \Pi'$. In addition, Lemma \ref{nus are compatible} shows that
\begin{equation} \label{multiply by this}
\pi_{F_{m' p^{n'}}  / L} \big( \prod_{\ell \mid m', \ell \not \in \Pi'} ( \nu_{m'p^{n'}}^{(\ell)} \Eul_\ell (\tilde \sigma_\ell)^{-1})^\# \big)
\equiv \pi_{F_{M'} / L} \big( \prod_{\ell \mid m', \ell \not \in \Pi'} ( \nu_{M'}^{(\ell)} \Eul_\ell (\tilde \sigma_\ell)^{-1})^\# \big) \mod I_H.
\end{equation}
If we multiply (\ref{congruence coming from appendix case 1}) by (\ref{multiply by this}) we then obtain a new congruence valid in $\mathfrak{A} / I_H \mathfrak{A}$. 
Using  Lemma~\ref{description of F circ f lem} and Theorem \ref{local points main result}\,(a) for the left-hand side of this new congruence, and Lemma \ref{this computes the RHS} for its right-hand side, we obtain (notice the change from $z_{M'}^\mathrm{Kato}$ to $y_{M'}^\mathrm{Kato}$)
\begin{align*}
    \theta^\MT_L & \equiv \pi_{F_{M'} / K}
\big ( (\prod_{\ell \mid M', \ell \neq p} \nu_{M'}^{(\ell)} \Eul_\ell (\tilde \sigma_\ell)^{-1} )^\#  \cdot  \mathcal{P}_{F_{M'}} (y^\mathrm{Kato}_{M'},  \mathfrak{k}_{M'}) \big) \cdot \big( \prod_{i = 1}^s \Tam_\ell^{-1} (\rec_\ell (q_{E, \ell_i}) - 1) \big)
\\
& \equiv \pi_{F_{M'} / K} (\theta^\MT_{M'}) \cdot \big( \prod_{i = 1}^s \Tam_\ell^{-1} (\rec_\ell (q_{E, \ell_i}) - 1) \big)
\mod \mathfrak{A}I_H.
\end{align*}
 For this congruence we are also (again) using that we have assumed $\ell \in C^{(p)}_\times (L)$, and therefore that $\Eul_\ell(\tilde \sigma_\ell )$ is invertible in $\Z_p [\cG]$ by Lemma \ref{invertible Euler factors lemma}, for all $\ell \mid m^\prime$ that are not in $\Pi'$. The final congruence is then by Theorem~\ref{local points main result}~(a).  \\
 This concludes the proof of Theorem \ref{mazur--tate main result 2}.
\qed 

\begin{rk}
    It seems possible that the technical condition on $M'$ in Theorem \ref{mazur--tate main result 2} can be removed if one combines the calculations of this section with the strategy of Lemma \ref{inductive argument}. 
\end{rk}

\renewcommand{\thesection}{\Alph{section}}
\renewcommand{\thesubsubsection}{\Alph{section}.\arabic{subsubsection}}
\renewcommand{\thethm}{(\Alph{section}.\arabic{thm})}
\setcounter{section}{0}
\setcounter{thm}{0}

\appendix 

\section{Integrality of Mazur--Tate elements} \label{integrality Mazur Tate appendix}

In this appendix we derive integrality results for Mazur--Tate elements by following ideas of Stevens \cite[\S\,3]{Stevens89} with refinements due to Wiersema and Wuthrich \cite{wiersemawuthrich}.

\paragraph{Statement of the main result}
We write 
\[
\varphi_0 \: X_0 (N) \to E
\quad \text{ and } \quad 
\varphi_1 \: X_1 (N) \to E
\]
for the modular parametrisations of $E$. We then define the Manin and Manin--Stevens constants $c_0$ and $c_1$ by the relations
\[
\varphi_0^\ast \omega_E = c_0 \cdot \omega_f 
\quad \text{ and } \quad 
\varphi_1^\ast \omega_E = c_1 \cdot \omega_f,
\]
where $\omega_f \coloneqq f(q) \frac{\mathrm{d}q}{q}$ is the differential 1-form associated to $f$. It was first observed by Gabber that $c_0$ and $c_1$ are (nonzero) integers (see \cite[Prop.\@ 2]{Edixhoven91} and \cite[Thm.\@ 1.6]{Stevens89}). Manin has conjectured that $c_0 \in \{ \pm 1 \}$ for some curve in the isogeny class in $E$ (namely the strong Weil curve) and Stevens has conjectured that always $c_1 \in \{\pm 1 \}$ (see \cite[Conj.\@ 1]{Stevens89}). 

\begin{rk}
    It is known that if $p$ is a prime number that divides $c_1$, then $p^2$ must divide the conductor $N$ of $E$ (for odd $p$ this was proved by Mazur \cite{Mazur78}, and for $p = 2$ by \v{C}esnavi\v{c}ius \cite{Cesnavicius18}). This shows that $c_1 \in \{ \pm 1 \}$ if $E$ is `semistable' (that is, $N$ is square-free).\\
    In addition, \v{C}esnavi\v{c}ius--Neururer--Saha \cite{cesnavicius2022manin} have proved that $c_1$ divides the degree of $\varphi_1$, and this can often be used to rule out that a given additive prime divides $c_1$.
\end{rk}

Although, to the best of the authors' knowledge, the following result on the integrality of Mazur--Tate elements has not previously appeared in the literature in this exact form, it is probably well-known to experts.\\
To state this result, we write $D(m) \coloneqq \mathrm{gcd} ( m , N)$ and $\delta (m) \coloneqq \mathrm{gcd} (D (m), \frac{N}{D(m)})$.  

\begin{thm} \label{integrality mazur tate}
    For every $m \in \N$, one has
    \[
     c_\infty c_0 \Ann_{\Z [G_m]} ( E (F_{\delta (m)})_\tor) \cdot \theta_m^\mathrm{MT} \subseteq \Z [G_m]
    \]
    and 
    \[
     c_\infty c_1 \Ann_{\Z [G_m]} ( E (F_{D(m)})_\tor) \cdot \theta_m^\mathrm{MT} \subseteq \Z [G_m]. 
    \]
\end{thm}

\paragraph{The proof of Theorem \ref{integrality mazur tate}}

We define the \textit{N\'eron lattice} $\mathscr{L}_E$ of $E$ to be 
\begin{equation} \label{neron lattice}
\mathscr{L}_E \coloneqq \Big \{ \int_\gamma \omega \mid \gamma \in H_1 ( E(\C), \Z ) \Big \} 
= 
\begin{cases} \frac{1}{2} \Z \Omega^+ \oplus \Z \Omega^-  \quad & \text{ if } c_\infty = 2, \\ 
\Z  \Omega^+ \oplus \frac{1}{2} \Z  (\Omega^+ + \Omega^-) & \text{ if } c_\infty = 1.
\end{cases}
\end{equation}

We define the `Stevens element' as 
\[
\theta_m^\mathrm{St} \coloneqq   {\sum}_{a \in (\Z / m \Z)^\times} \lambda_f (\zfrac{a}{m}) \sigma_a 
\quad \in \C [G_m]. 
\]
Recall that we have defined the modular symbol $\lambda_f (\zfrac{a}{m})$ at the start of \S\,\ref{s:Modular symbols and Mazur--Tate elements}.

\begin{lem} \label{integrality stevens lemma}
The lattice $\mathscr{L}_E \otimes_\Z  \Z [G_m]$ contains both $ c_0 \Ann_{\Z [G_m]} ( E (F_{\delta (m)})_\tor)\cdot \theta_m^\mathrm{St}$ and $ c_1 \Ann_{\Z [G_m]} ( E (F_{D(m)})_\tor) \cdot \theta_m^\mathrm{St}$ for every $m \in \N$.
\end{lem}

We now first explain how to deduce Theorem \ref{integrality mazur tate} from Lemma \ref{integrality stevens lemma}. 
Consider the maps
\begin{align*}
\mathrm{Re} & \: \C [G_m] \to \R [G_m], \quad {\sum}_{\sigma \in G_m} x_\sigma \sigma \mapsto {\sum}_{\sigma \in G_m} \mathrm{Re} (x_\sigma) \sigma \\ 
\mathrm{Im} & \: \C [G_m] \to \R [G_m], \quad {\sum}_{\sigma \in G_m} x_\sigma \sigma \mapsto {\sum}_{\sigma \in G_m} \mathrm{Im} (x_\sigma) \sigma.
\end{align*}
Observe that these maps are $\R [G_m]$-linear. Using the explicit description of $\mathscr{L}_E$ given in (\ref{neron lattice}), we obtain that for every element $\alpha \in \mathscr{L}_E \otimes_\Z \Z [G_m]$ one has
\begin{equation} \label{neron lattice consequence}
\frac{\mathrm{Re} (\alpha)}{\Omega^+} + i \frac{\mathrm{Im} (\alpha)}{\Omega^-}  \in \frac{1}{c_\infty} \Z [G_m]. 
\end{equation}
We consider the $c_0$ case, the $c_1$ case is similar. Fix an element $x$ of $\Ann_{\Z [G_m]} ( E (F_{\delta (m)})_\tor)$. 
Lemma \ref{integrality stevens lemma} then 
combines with (\ref{neron lattice consequence}) to imply that
\begin{align*}
  c_\infty c_0 x \cdot \Big( \frac{ \mathrm{Re} ( \theta_m^\mathrm{St})}{\Omega^+} + i \frac{ \mathrm{Im} ( \theta_m^\mathrm{St})}{\Omega^-} \Big) 
  & =      \frac{c_\infty \mathrm{Re} (  c_0 x\cdot \theta_m^\mathrm{St})}{\Omega^+} + i \frac{c_\infty \mathrm{Im} (  c_0 x \cdot \theta_m^\mathrm{St})}{\Omega^-}  \subseteq   \Z [G_m]. 
\end{align*}
To prove Theorem \ref{integrality mazur tate}, it therefore suffices to note that
\begin{align*}
\theta^\mathrm{MT}_m & = {\sum}_{a \in (\Z / m\Z)^\times} (\msym{a}{m}^+ + \msym{a}{m}^-) \sigma_a  =  {\sum}_{a \in (\Z / m\Z)^\times} \Big ( \frac{\mathrm{Re} ( \lambda_f ( \tfrac{a}{m}))}{\Omega^+} + i \frac{\mathrm{Im} ( \lambda_f ( \tfrac{a}{m}))}{\Omega^-} \Big) \sigma_a \\ 
& =   \frac{\mathrm{Re} ( \theta_m^\mathrm{St})}{\Omega^+} + i \frac{\mathrm{Im} ( \theta_m^\mathrm{St})}{\Omega^-}  . 
\end{align*}
We now turn to the proof of Lemma \ref{integrality stevens lemma}, again in the $c_0$ case. At the outset we note that
\[
\lambda_f (a) = 2\pi i \int_{i \infty}^a f (\tau) \mathrm{d} \tau = c_0^{-1} \int_{\gamma (a)} \omega, 
\]
where $\gamma (a)$ is the image in $E (\C)$ of the path from $i\infty$ to $a$ given by the vertical line in the upper half-plane from $i\infty$ to $a$. 
Writing $\mathcal{H} \subseteq \C$ for the complex upper half-plane, one therefore has the commutative diagram (see, for example, \cite[Prop.\@ 2.11]{Dar04})
\begin{cdiagram}
(\mathbb{P}^1 (\Q) \cup \mathcal{H}) / \sim  \arrow{d}{c_0 \lambda_f} \arrow[equals]{r} & 
X_0 (N) (\C) \arrow{d}{\varphi_0} \\ 
\C / \mathscr{L}_E \arrow{r}{\simeq} & E (\C).
\end{cdiagram}%
For any cusp $a \in \mathbb{P}^1 (\Q)$ we then write $P_a \coloneqq \varphi_0 (a)$ for the point in $E (\overline{\Q})$ corresponding to $c_0^{-1} \lambda_f (a) + \mathscr{L}_E$, and we note that $P_a$ is a torsion point by the Manin--Drinfeld theorem.\\ 
From this it is clear that $c_0 \lambda_f (a)$ belongs to $\mathscr{L}_E$ if and only if $P_a$ is trivial.\\
It is proved in \cite[Thm.\@ 1.3.1]{Stevens82} that any cusp $\frac{r}{s}$ of $X_0 (N)$ (with $\mathrm{gcd}(r,s) = 1$) is defined over $\Q (\zeta_N)$ and that the action of $G_N$ on $\frac{r}{s}$ is given by
\[
\sigma_a \cdot \frac{r}{s} = \frac{a^{-1} r}{s},
\]
where $a^{-1}$ denotes the inverse of $a$ mod $N$. We note that in \cite[Thm.\@ 1.3.1]{Stevens82} the description of Galois action has the inverted element in the denominator rather than the numerator. However, these two cusps are equivalent in $X_0(n)$, see \cite[Prop.~3.8.3]{DiamondandShurman}. Take $s  = m$ and suppose that $a \equiv 1 \mod \delta (m)$. Then also $a^{-1} \equiv 1 \mod \delta (m)$, 
hence $\frac{a^{-1} r}{m}$ is $\Gamma_0 (N)$-equivalent to $\frac{r}{m}$ ( this can be deduced from \cite[Prop.~3.8.3]{DiamondandShurman}, see the argument on page 103 of loc.\@ cit.).
As a consequence, the subgroup $\{ \sigma_a \mid (a, N) = 1, a \equiv 1 \mod \delta (m) \} = \gal{F_N}{F_{\delta (m)}}$ fixes the cusp $\frac{r}{m}$, and so any such cusp is defined over $F_{\delta (m)}$.\\
We next observe that the map
\[
\mathbb{P}^1 ( \Q) \to  E( \overline{\Q})_\tor, \quad a \mapsto P_a = \varphi_0 (a)
\]
is $\gal{\overline{\Q}}{\Q}$-equivariant because the modular parametrisation $\varphi_0$ is defined over $\Q$ (hence $\gal{\overline{\Q}}{\Q}$-equivariant). It follows that the point $P_a$ belongs to $E ( F_{\delta (m)})_\tor$ for every cusp $a \in \mathbb{P}^1 (\Q)$ of the form $a  = \frac{r}{m}$. 
Now, we extend the earlier diagram $\Z [G_m]$-linearly to obtain the commutative diagram of $\Z [G_m]$-modules
\begin{cdiagram}
    \mathbb{P}^1 (\Q) \arrow{r} \arrow{d} & \C [G] \arrow{d} \\ 
    E ( F_m)_\tor \otimes_\Z \Z [G_m] \arrow[hookrightarrow]{r} & 
    (\C / \mathscr{L}_E) [G_m].
\end{cdiagram}%
Here the top map sends the class of $\frac{1}{m}$ to $\sum_{\sigma \in G_m} \lambda_f ( \sigma \cdot \frac{1}{m}) \sigma^{-1}
= \sum_{a \in ( \Z / m \Z)^\times} \lambda_f ( \zfrac{a}{m}) \sigma_a = \theta^\mathrm{St}_m$,
and the vertical map on the left sends $\frac{1}{m}$ to $\sum_{\sigma \in G_m} (\sigma \cdot P_{1 / m}) \sigma^{-1}
= \sum_{a \in ( \Z / m \Z)^\times} P_{a / m} \sigma_a$.\\
We have seen above that each of the points $P_{a / m}$ belongs to $E (F_{\delta (m)})_\tor$, hence the latter element is annihilated by $\Ann_{\Z [G_m]} ( E (F_{\delta(m)})_\tor)$. Commutativity of the above diagram therefore shows that the class of $\theta^\mathrm{St}_m$ in the quotient $(\C / \mathscr{L}_E) [G_m]$ is annihilated by $\Ann_{\Z [G_m]} ( E (F_{\delta(m)})_\tor)$, as required to prove that $\Ann_{\Z [G_m]} ( E (F_{\delta(m)})_\tor) \cdot \theta_m^\mathrm{St}$ belongs to $\mathscr{L} [G_m]$.

When considering the $c_1$ case we note that \cite[Thm. \@ 3.11]{Stevens89} shows that the cusps of $X_1 (N)$ is defined over $\Q (\zeta_N)$ and that the action of $G_N$ is as defined before. Let $a' \equiv 1 \mod m$. It follows from \cite[Prop.\@ 3.8.3]{DiamondandShurman} that 
$\frac{r}{m} \sim \frac{a' r}{m}$.
Hence, the subgroup $\{ \sigma_a \mid (a, m) = 1, a \equiv 1 \mod D (m) \} = \gal{F_N}{F_{D (m)}}$ fixes the cusp $\frac{a}{m}$. The result in this case then follows as before. \qed

\section{Some general algebra} \label{algebra appendix}

In this section we establish useful results of a general algebraic nature. 

\subsection{Perfect complexes and their determinants} \label{algebra appendix section 1}

Let $\cR$ be a Noetherian commutative ring.
For any $\cR$-module $M$, we endow its $\cR$-linear dual $M^\ast \coloneqq \Hom_\cR (M, \cR)$ with the structure of an $\cR$-module by means of 
\[
\cR \times M^\ast \to M^\ast, \quad (x, f) \mapsto \{ y \mapsto x \cdot f (y) \}. 
\] 
We also set $M^\vee \coloneqq \Hom_\cR (M, \cQ / \cR)$ with $\cQ$ the total ring of fractions of $\cR$. 

\begin{rk} \label{sharp and Pontryagin dual rk}
Note that in the main text the notation $M^\vee$ was used for the Pontryagin dual $\Hom_\Z (M, \Q / \Z)$ of $M$, which might differ from $\Hom_\cR (M, \cQ / \cR)$. For example, if $\cR = \Z [G]$ for a finite abelian group $G$, then (\ref{dual sharp isomorphism}) induces an isomorphism $\Hom_\Z (M, \Q / \Z)^\# \cong \Hom_\cR (M, \cQ / \cR)$.
\end{rk}

We write $D (\cR)$ for the derived category of $\cR$-modules, and $D^\mathrm{perf} (\cR)$ for its full subcategory of complexes that are `perfect' (that is, isomorphic in $D(\cR)$ to a bounded complex of projective $\cR$-modules).\\ 
Given $C^\bullet \in D^\mathrm{perf} (\cR)$
represented by a bounded complex $\ldots \to C^{i - 1} \to C^i \to C^{i + 1} \to \dots$ of projective $\cR$-modules $C^i$ that are each placed in degree $i$, we denote its Euler characteristic as $\chi_\cR (C^\bullet) \coloneqq \sum_{i \in \Z} (-1)^i \cdot [C^i] \in K_0 (\cR)$. Here $K_0(\cR)$ denotes the Grothendieck group of the category of finitely generated projective $\cR$-modules. If $\cR$ is semisimple, 
then the rank function induces an isomorphism $K_0 (\cR) \cong H^0 (\Spec \cR, \Z)$ and so we may regard $\chi_\cR (C^\bullet)$ as an element of $H^0 (\Spec \cR, \Z)$. In this case we write $\chi_\cR (C^\bullet) \leq d$ if $\chi_\cR (C^\bullet)$ belongs to $H^0 (\Spec \cR, \Z_{\leq d})$.\\
The determinant of $C^\bullet$ is denoted by $\Det_\cR (C^\bullet) \coloneqq \bigotimes_{i \in \Z} \Det_\cR (C^i)^{(-1)^i}$. In this context we remark that, following Knudsen and Mumford \cite{KnudsenMumford}, $\Det_\cR (C^\bullet)$ must be considered a graded line bundle in order to avoid technical sign issues. However, since the grading is uniquely determined upon fixing a representative for $C^\bullet$, we have chosen to suppress any explicit reference to the grading in order to simplify the exposition.\\
For every element $a \in \Det_\cR (C^\bullet)^{-1} \coloneqq \Det_\cR (C^\bullet)^\ast$ we also obtain a canonical `evaluation map' 
\[
\mathrm{Ev}_a \: \Det_\cR ( C^\bullet) \to \cR. 
\]

The following definition underlies many of the constructions in the main text. 

\begin{definition} \label{def projection map from det}
Let $\cR$ be a Noetherian reduced ring with total ring of fractions $\cQ$. 
Let $C^\bullet \in D^\mathrm{perf} (\cR)$ be a complex, 
and fix a surjection $\kappa \: H^2 (C^\bullet) \to \cY$ with $\cY$ a finitely generated free $\cR$-module of rank $d \geq \chi \coloneqq \chi_\cQ (\cQ \otimes_\cR^\mathbb{L} C^\bullet)$ together with an ordered $\cR$-basis $\mathfrak{B}$ of $\cY$. 
We define a canonical map as the composite\begin{align*}
    \vartheta_{C^\bullet, \mathfrak{B}} \: \Det_\cR (C^\bullet)^{-1} & \xhookrightarrow{\phantom{\cdot e_{C^\bullet, \cY}}} \cQ \otimes_\cR \Det_\cR (C^\bullet)^{-1} \\
    & \xrightarrow{\; \; \; \; \simeq \; \; \; \;} \Det_\cQ ( \cQ \otimes_\cR^\mathbb{L} C^\bullet)^{-1} \\ 
    & \xrightarrow{\; \; \; \; \simeq \; \; \; \;} {\bigotimes}_{i \in \Z} \Det_{\cQ} ( H^i (\cQ \otimes_\cR^\mathbb{L} C^\bullet) )^{(-1)^{i + 1}} \\
    & \xrightarrow{\cdot e_{C^\bullet, \cY}} \Det_{\cQ} ( H^1 (\cQ \otimes_\cR^\mathbb{L} C^\bullet) ) \otimes_\cQ \Det_\cQ ( \cQ \otimes_\cR \cY)^{-1}\\
    & \xrightarrow{\; \; \; \; \simeq \; \; \; \;} \exprod^{d - \chi}_\cQ ( \cQ \otimes_\cR H^1 (C^\bullet)).
\end{align*}
Here the second arrow follows from the base-change property of the determinant functor, the third arrow is the natural `passage-to-cohomology map' (which exists because $\cQ$ is semisimple), the fourth arrow is multiplication by the idempotent $e_{C^\bullet, \cY}$ of $\cQ$ that is defined as the sum of all primitive idempotents of $\cQ$ that annihilate $\cQ \otimes_\cR \big( (\ker \kappa) \oplus \bigoplus_{i \in \Z \setminus \{ 1, 2\}} H^i (C^\bullet) \big)$, and the final isomorphism is induced by $\mathrm{Ev}_x$ with the element $x \coloneqq \bigwedge_{b \in \mathfrak{B}} b$ of $\Det_\cR (\cY)$.
\end{definition}

We often work with complexes that satisfy the following slight variant of \cite[Def.\@ A.6]{sbA}.

\begin{definition} \label{standard representative def}
Let $C^\bullet \in D^\mathrm{perf} (\cR)$ be a perfect complex, and suppose we are given a surjection $\kappa \: H^2 (C^\bullet) \to \cX$ with $\cX$ an $\cR$-module generated by an ordered set $\mathfrak{X} = \{x_1, \dots, x_d\}$ of cardinality $d$.
 A representative for $C^\bullet$ of the form  
\begin{equation} \label{standard representative}
F^0 \stackrel{\partial_0}{\longrightarrow} F^1 \stackrel{\partial_1}{\longrightarrow} F^2
\end{equation}
is called a `standard representative' for $C^\bullet$ with respect to $( \kappa, \mathfrak{X})$ if the following conditions are satisfied:
\begin{romanliste}
    \item For every $i \in \{0, 1, 2\}$, the module $F^i$ is $\cR$-free of finite rank $n_i$ and placed in degree $i$.
    \item $\partial_0$ is injective (so $C^\bullet$ is acyclic outside degrees 1 and 2).
    \item $d + n_1 \geq n_0 + n_2$.
    \item There exists a basis $\{b_1, \dots, b_{n_2}\}$ of $F^2$ such that the composite map $F^2 \to H^2 (C^\bullet) \to \cX$ maps $b_i$ to $x_i$ if $i \in \{1, \dots, d\}$ and to 0 otherwise. 
\end{romanliste}
\end{definition}

\begin{rk}
Standard representatives exist in many contexts; they can often be constructed via the method of \cite[Prop.\@ A.11\,(i)]{sbA} (see also \cite[\S\,5.4]{bks} and \cite[Lem.\@ 2.35]{BB}).
\end{rk}

To give an explicit description of the map $\vartheta_{C^\bullet, \mathfrak{B}}$ for complexes that admit standard representatives in the sense of Definition \ref{standard representative def}, we shall use that for any $\cR$-module $M$ and integers $0 \leq r \leq s$ one has a map
\[
\Phi^{r, s}_M \: \exprod^s_\cR M^\ast  \to \Hom_\cR \big( \exprod^r_\cR M, \exprod^{r - s}_\cR M \big)
\]
defined by means of the rule
\[
f_1 \wedge \dots \wedge f_s  \mapsto \big \{ m_1 \wedge \dots \wedge m_r \mapsto {\sum}_{\sigma \in \mathfrak{S}_{r, s}} \mathrm{sgn} (\sigma) \det ( f_i (m_{\sigma_j}))_{1 \leq i \leq s} \cdot m_{\sigma (s + 1)} \wedge \dots \wedge m_{\sigma (r)} \big \} .
\]
Here $\mathfrak{S}_{r, s} \coloneqq \{ \sigma \in \mathfrak{S}_r \mid \sigma (1) < \dots < \sigma (s) \text{ and } \sigma (s + 1) < \dots < \sigma (r) \}$. To simplify notation, we denote $\Phi^{r, s}_M (f)$ also by $f$.

\begin{lem} \label{explicit description projection map lem}
    Let $\cR$ be a reduced Noetherian ring with total ring of fractions $\cQ$. 
    Assume $C^\bullet \in D^\mathrm{perf} (\cR)$ is a complex that admits a standard representative (\ref{standard representative}) with respect to $(\kappa, \mathfrak{X})$ for some surjection $\kappa \: H^2 (C^\bullet) \to \cX$ and set of generators $\mathfrak{X} = \{ x_1, \dots, x_{d'}\}$ of $\cX$. Suppose $d \leq d'$ is an integer such that $\cY \coloneqq \bigoplus_{i = 1}^{d} \cR x_i$ is a free $\cR$-module of rank $d$,  and set $\mathfrak{B} \coloneqq \{x_1, \dots, x_d\}$ and $n \coloneqq (d + n_1) - (n_0 + n_2)$. Then the map 
    \begin{align*}
    \Det_\cR (C^\bullet)^{-1} = \big( \exprod^{n_0}_\cR (F^0)^\ast \big) \otimes_\cR \big( \exprod^{n_1}_\cR F^1 \big) \otimes_\cR  \big( \exprod^{n_2}_\cR (F^2)^\ast \big) & \to 
    \cQ \otimes_\cR \exprod^{n}_\cR F^1
    \end{align*}
    defined by means of the rule
    \[
    f \otimes g \otimes \exprod_{1 \leq i \leq n_2} b_i^\ast  \mapsto (-1)^{n (n_2 - d)} \cdot \big( \widetilde f \wedge  \exprod_{d + 1 \leq i \leq n_2} (b_i^\ast \circ \partial_1) \big) (g),
    \]
    where $\widetilde f$ denotes a preimage of $f$ under the surjection $\cQ \otimes_\cR \exprod^{n_0}_\cR (F^1)^\ast \xrightarrow{\partial_0^\ast} \cQ \otimes_\cR \exprod^{n_0}_\cR (F^0)^\ast$, 
    is well-defined, has image in $\cQ \otimes \exprod^{n}_\cR \ker (\partial_1)$,
    and coincides with $\vartheta_{C^\bullet, \mathfrak{B}}$ when composed with the projection $\cQ \otimes_\cR \exprod^n_\cR \ker (\partial_1) \to \cQ \otimes_\cR \exprod^n_\cR H^1 (C^\bullet)$. 
\end{lem}

\begin{proof}
First we observe that condition (iv) in Definition \ref{standard representative def} means that 
the map $F^2 \to H^2 (C^\bullet) \stackrel{\kappa'}{\to} \cY$ induces a direct sum composition $F^2 \cong G^2 \oplus \cY$ with $G^2 \coloneqq \bigoplus_{i = {d + 1}}^{n_2} \cR b_i$. From the free presentation $F^1 \xrightarrow{\partial_1} G^2 \to \ker (\kappa') \to 0$ we then deduce that any element of the form
$\big( \bigwedge_{d + 1 \leq i \leq n_2} (b_i^\ast \circ \partial_1) \big) (g)$ belongs to $\Fitt^0 ( \ker (\kappa')) \subseteq \Ann_\cR ( \ker (\kappa'))$. In particular, the idempotent $e \coloneqq e_{C^\bullet, \cY}$ acts as the identity on any such element. Since the definition of $\vartheta_{C^\bullet, \mathfrak{B}}$ involves extending scalars to $\cQ$ and multiplication by $e$, we may prove the lemma after extending scalars to $e \cQ$, thereby reducing it to the case that $\cR  = e \cQ$. In particular, $\cR$ is a semisimple ring and so 
we may fix splittings $l_0$ and $l_1$ of the exact sequences
\begin{cdiagram}[row sep=tiny]
    0 \arrow{r} &  F_0 \arrow{r}{\partial_1} & \ker (\partial_1) \arrow{r}[below]{q} &  H^1(C^\bullet) \arrow[bend right=30, dashed]{l}[above]{l_0} \arrow{r}  & 0 \\ 
    0 \arrow{r} & \ker \partial_1 \arrow{r} & F_1 \arrow{r}{\partial_1} & G^2 
    \arrow[bend left=30, dashed]{l}[below]{l_1}
    \arrow{r} & 0
\end{cdiagram}%
which induce isomorphisms
\begin{align}
\label{isom induced splitting 1}
\big( \exprod^{n}_\cR F^0 \big) \otimes_\cR 
\big( \exprod^{n_0}_\cR H^1 (C^\bullet) \big)
& \stackrel{\simeq}{\to} \exprod^{n + n_0}_\cR \ker (\partial_1), \\  \nonumber
(c_1 \wedge \dots \wedge c_{n_0}) \otimes (h_1 \wedge \dots \wedge h_n) & \mapsto \partial_0 (c_1) \wedge \dots \wedge \partial_0 (c_{n_0}) \wedge l_0 (h_1) \wedge \dots \wedge l_0 (h_n) 
\end{align}
and 
\begin{align}
\label{isom induced splitting 2}
\big(  \exprod^{n + n_0}_\cR \ker (\partial_1) \big) \otimes_\cR \big( \exprod^{n_2 - d}_\cR G^2 \big) & \stackrel{\simeq}{\to} \exprod^{n_1}_\cR F_1 \\  \nonumber
j \otimes (i_1 \wedge \dots i_{n_2 - d}) & \mapsto 
j \wedge (l_1 \circ \partial_1^{-1}) (i_1) \wedge \dots \wedge (l_1 \circ \partial_1^{-1}) (i_{n_2 - d})
\end{align}
that are independent of the choices of splittings $l_0$ and $l_1$. 
Fix an $\cR$-basis basis $c_1, \dots, c_{n_0}$ of $F^0$ for convenience and write $c_1^\ast, \dots, c_{n_0}^\ast$ for the corresponding dual basis. 
By (\ref{isom induced splitting 2}), we may write $g \in \exprod^{n_1}_\cR F^1$ in the form $j \wedge \bigwedge_{d + 1 \leq i \leq n_2 }(l_1 \circ \partial_1^{-1}) (b_i)$ for some $j \in \exprod^{n + n_0}_\cR \ker (\partial_1)$ that, by (\ref{isom induced splitting 1}), can in turn be written as $j = \partial_0(c) \wedge l_0 (h_1) \wedge \dots \wedge l_0 (h_n)$ with $c = c_1 \wedge \dots \wedge c_{n_0}$ and
a suitable element $h_1 \wedge \dots \wedge h_n \in \exprod^n_\cR H^1 (C^\bullet)$. 
Now, the `passage-to-cohomology map' in the definition of $\vartheta_{C^\bullet, \mathfrak{B}}$ is given by the composite of isomorphisms
\begin{align*}
\Det_\cR (C^\bullet) & = 
\big( \exprod^{n_0}_\cR (F^0)^\ast \big) \otimes_\cR \big( \exprod^{n_1}_\cR F^1 \big) \otimes_\cR  \big( \exprod^{n_2}_\cR (F^2)^\ast \big) \\
& \longrightarrow
    \big( \exprod^{n_0}_\cR (F^0)^\ast \big) \otimes_\cR \big( \exprod^{n_1}_\cR F^1 \big)  \otimes_\cR \big( \exprod^{n_2 - d}_\cR \cY^\ast \big) \otimes_\cR  \big( \exprod^{n_2 - d}_\cR (G^2)^\ast \big) \\ 
& 
\stackrel{\simeq}{\longrightarrow}
    \big( \exprod^{n_0}_\cR (F^0)^\ast \big) \otimes_\cR \big( \exprod^{n_1}_\cR F^1 \big) \otimes_\cR  \big( \exprod^{n_2 - d}_\cR (G^2)^\ast \big) \otimes_\cR \big( \exprod^{n_2 - d}_\cR \cY^\ast \big) \\ 
    & \stackrel{(\ref{isom induced splitting 2})}{\longrightarrow} 
    \big( \exprod^{n_0}_\cR (F^0)^\ast \big) \otimes_\cR
    \big( \exprod^{n + n_0}_\cR \ker (\partial_1) \big) \otimes_\cR \big( \exprod^{n_2 - d}_\cR G^2 \big) \\
    & \qquad \qquad \quad \otimes_\cR  \big( \exprod^{n_2 - d}_\cR (G^2)^\ast \big) \otimes_\cR \big( \exprod^{n_2 - d}_\cR \cY^\ast \big)\\
& \stackrel{(\ast)}{\longrightarrow} \big( \exprod^{n_0}_\cR (F^0)^\ast \big) \otimes_\cR
    \big( \exprod^{n + n_0}_\cR \ker (\partial_1) \big)  \otimes_\cR \big( \exprod^{n_2 - d}_\cR \cY^\ast \big) \\
    & \stackrel{(\ref{isom induced splitting 1})}{\longrightarrow} 
    \big( \exprod^{n_0}_\cR (F^0)^\ast \big) \otimes_\cR \big( \exprod^{n_0}_\cR F^0 \big)
    \otimes_\cR \big( \exprod^n_\cR H^1 (C^\bullet) \big)  \otimes_\cR \big( \exprod^{n_2 - d}_\cR \cY^\ast \big) \\
    & \stackrel{(\ast)}{\longrightarrow} \big( \exprod^n_\cR H^1 (C^\bullet) \big) \otimes_\cR \big( \exprod^{n_2 - d}_\cR \cY^\ast \big),
\end{align*}
where the first arrow is induced by the exact sequence $0 \to \cY^\ast \to (F^2)^\ast \to (G^2)^\ast \to 0$ and the
arrows labelled ($\ast$) are induced by the relevant evaluation isomorphisms.
Set $c^\ast \coloneqq c_1^\ast \wedge \dots \wedge c_{n_0}^\ast$.
Since the composite of the first and second arrows sends $\exprod_{1 \leq i \leq n_2} b_i^\ast$ to $ \bigwedge_{d + 1 \leq i \leq n_2} b_i^\ast \otimes \bigwedge_{1 \leq i \leq d} b_i^\ast$, one therefore has
\[
\vartheta_{C^\bullet, \mathfrak{B}} \big( c^\ast \otimes g \otimes \exprod_{1 \leq i \leq n_2} b_i^\ast \big) = h_1 \wedge \dots \wedge h_n \quad \in \exprod_{\cR}^{n} H^1 (C^\bullet).
\]
Note that this does not depend on the choice of $b_1, \dots, b_{n_0}$ or $c_1, \dots, c_{n_0}$. 
On the other hand, we may compute that \begin{align*}
\big( \widetilde{c^\ast} \wedge \exprod_{d + 1 \leq i \leq n_2} (b_i^\ast \circ \partial_1)  \big) (g) & = (-1)^{n_0 \cdot (n_2 - d)} \cdot \widetilde{c^\ast} \big( \big(\exprod_{d + 1 \leq i \leq n_2} (b_i^\ast \circ \partial_1) \big) (g) \big) \\ 
& = (-1)^{n_0 \cdot (n_2 - d)} \cdot \widetilde{c^\ast}  ( 
(-1)^{(n + n_0)(n_2 - d)} \cdot j) \\ 
& = (-1)^{n (n_2 - d)} \cdot c^\ast ( c \wedge  l_0 (h_1) \wedge \dots \wedge l_0 (h_n))  \\
& = (-1)^{n (n_2 - d)} \cdot  l_0 (h_1) \wedge \dots \wedge l_0 (h_n),
\end{align*}
hence comparing with the last displayed formula gives the desired comparison with $\vartheta_{C^\bullet, \mathfrak{B}}$ after applying $\wedge^n q$, and also shows that the considered map is independent of the choice of $\widetilde c^\ast$. 
\end{proof}

\begin{rk}
    If $F^0 = 0$ and $n_1 = n_2$, then Lemma \ref{explicit description projection map lem} recovers \cite[Prop.~A.11]{sbA}.
\end{rk}

To prepare for the statement of the next result, we remark that the isomorphism $\exprod^{\mathrm{rk}_\cR F}_\cR F^\ast \cong (\exprod^{\mathrm{rk}_\cR F}_\cR F)^\ast $ is valid for any finitely generated free $\cR$-module $F$ and induces an identification 
\[
\Det_\cR ( C^\bullet)^{-1} = \Det_\cR (\RHom_\cR (C^\bullet, \cR)). 
\]

The following result records useful functoriality properties of the map $\vartheta_{C^\bullet, \cY}$. 

\begin{prop} \label{functoriality properties of det projection map}
Let $\cR$ be a Noetherian reduced ring with total ring of fractions $\cQ$. 
\begin{liste}
\item  Assume we are given an exact triangle 
\begin{equation} \label{triangle in statement of prop}
\begin{tikzcd}
    A^\bullet \arrow{r}{f} & B^\bullet \arrow{r}{g} & C^\bullet \arrow{r} & A^\bullet [1]
    \end{tikzcd}
\end{equation}%
in $D^\mathrm{perf} (\cR)$ together with an $\cR$-free quotient $\cY$ of $H^2 (B^\bullet)$ of rank $d \geq \chi_\cQ (\cQ \otimes_\cR^\mathbb{L} B^\bullet)$ and with ordered $\cR$-basis $\mathfrak{B}$.
 Set $n \coloneqq d - \chi_\cQ (\cQ \otimes_\cR B^\bullet)$.
\begin{romanliste}
\item Suppose $\chi_\cR (C^\bullet) = 0$ and that the map $H^2 (A^\bullet) \xrightarrow{H^2 (f)} H^2 (B^\bullet) \to \cY$ induced by (\ref{triangle in statement of prop}) is surjective. 
For every element
    $a \in \Det_\cR (C^{\dagger})^{-1} \cong \Det_\cR (C^\bullet)$ and with the notation $C^{\dagger} \coloneqq \RHom_\cR (C^\bullet, \cR) [-2]$, one has the commutative diagram
    \begin{cdiagram}
        \Det_\cR ( B^\bullet)^{-1} \arrow{d}{\vartheta_{B^\bullet, \mathfrak{B}}} 
        \arrow{r}{\simeq}[below]{(\ref{triangle in statement of prop})} & \Det_\cR (A^\bullet)^{-1} \otimes_\cR \Det_\cR (C^\bullet)^{-1} \arrow{r}{\id \otimes \mathrm{Ev}_a} & \Det_\cR (A^\bullet)^{-1} \arrow{d}{\vartheta_{A^\bullet, \mathfrak{B}}} 
        \\ 
        \cQ \otimes_\cR \exprod^n_\cR H^1 (B^\bullet)  \arrow{r}{\cdot
        \vartheta_{C^{\dagger}, \emptyset} (a)
        }
        & \cQ \otimes_\cR \exprod^n_\cR  H^1 (B^\bullet)  
        & \cQ \otimes_\cR \exprod^n_\cR  H^1 (A^\bullet). \arrow{l}[above]{\wedge^n H^1 (f)}
    \end{cdiagram}%
    \item Suppose $\chi_\cR (A^\bullet) = 0$ and that the map $H^2 (B^\bullet) \to \cY$ factors through the map $H^2 (g) \: H^2 (B^\bullet) \to H^2 (C^\bullet)$ induced by (\ref{triangle in statement of prop}). For every
    $a \in \Det_\cR (A^\bullet)^{-1}$ one then has the commutative diagram
    \begin{cdiagram}
        \Det_\cR (C^\bullet)^{-1} \arrow{r}{x \mapsto a \otimes x} \arrow{d}{\vartheta_{C^\bullet, \mathfrak{B}}}
        & \Det_\cR (A^\bullet)^{-1} \otimes_\cR \Det_\cR (C^\bullet)^{-1} \arrow{r}{\simeq}[below]{(\ref{triangle in statement of prop})} & 
        \Det_\cR (B^\bullet) \arrow{d}{\vartheta_{B^\bullet, \mathfrak{B}}} \\
        \cQ \otimes_\cR \exprod^n_\cR H^1 (C^\bullet) \arrow{r}{\cdot \vartheta_{A^\bullet, \emptyset} (a)} & \cQ \otimes_\cR \exprod^n_\cR H^1 (C^\bullet) & 
        \cQ \otimes_\cR \exprod^n_\cR H^1 (B^\bullet) \arrow{l}[above]{\wedge^n H^2 (g)}.
    \end{cdiagram}
    \end{romanliste}
\item Let $C^\bullet \in D^\mathrm{perf} (\cR)$ be a complex and $\cY$ a quotient of $H^2 (B^\bullet)$ that is $\cR$-free of rank $d \geq 1 + \chi_\cQ (\cQ \otimes_\cR^\mathbb{L} C^\bullet)$ and has ordered $\cR$-basis $\mathfrak{B}$. 
Let $f \: \cR [1] \to C^\dagger \coloneqq \RHom_\cR ( C^\bullet, \cR)$ be a morphism in $D (\cR)$ with dual map $f^\ast \: C^\bullet \to R [-1]$. 
\begin{romanliste}
    \item The map $H^1 (C^\bullet) \otimes_\cR \cQ \to \cQ$ induced by $H^1 (f^\ast)$ via extension of scalars coincides with the composite map
    \[
    H^1 (C^\bullet) \otimes_\cR \cQ = H^1 (\cQ \otimes_\cR^\mathbb{L} C^\bullet) \cong 
    H^{-1} ( \cQ \otimes_\cR^\mathbb{L} C^\dagger)^\ast
    = \cQ \otimes_\cR H^{-1} (C^\dagger)^\ast
    \xrightarrow{H^{-1} (f)^\ast} \cQ.
    \]
    \item Setting $D^\bullet \coloneqq \mathrm{cone} (f^\ast) [-1]$, the following diagram is commutative.
\begin{cdiagram}[column sep=small]
\Det_\cR (C^\bullet)^{-1} \arrow{rr}{\simeq} \arrow{d}{\vartheta_{C^\bullet, \mathfrak{B}}} & &\Det_\cR ( D^\bullet)^{-1} 
\otimes_\cR \Det_\cR (\cR)^{-1} \arrow{r}{\id \otimes \mathrm{Ev}_1} & \Det_\cR ( D^\bullet)^{-1} \arrow{d}{\vartheta_{D^\bullet, \mathfrak{B}}} \\ 
\cQ \otimes_\cR \exprod^n_\cR  H^1 (C^\bullet) \arrow{rr}{H^1 (f^\ast)} & &
\cQ \otimes_\cR \exprod^{n - 1}_\cR  H^1 (C^\bullet) &
\cQ \otimes_\cR \exprod^{n - 1}_\cR  H^1 (D^\bullet). \arrow{l}
\end{cdiagram}%
\end{romanliste}
\item Suppose $C^\bullet \in D^\mathrm{perf} (\cR)$ is a complex with $\chi (C^\bullet) = 0$ that is acyclic outside degrees 1 and 2. Setting, $C^\dagger \coloneqq \RHom_\cR (C^\bullet, \cR) [i]$ for a fixed odd integer $i$, one then has an equality of $\cR$-submodules of $\cQ$
\[
\vartheta_{C^\bullet, \emptyset} ( \Det_\cR ( C^\bullet)^{-1}) =  
\vartheta_{C^\dagger, \emptyset} ( \Det_\cR ( C^\dagger)^{-1}). 
\]
    \item Let $C^\bullet \in D^\mathrm{perf} (\cR)$ be a complex and $\cR \to \cR'$ a surjective morphism of reduced Noetherian rings. Write $\rho \: C^\bullet \to C^\bullet \otimes_\cR^\mathbb{L} \cR' $ and $\rho^{\det} \: \Det_\cR (C^\bullet)^{-1} \to \Det_{\cR'} (C^\bullet \otimes_\cR^\mathbb{L} \cR')$ for the induced morphisms. 
    Suppose $\cY$ is a quotient of $H^2 (C^\bullet)$ that is $\cR$-free of rank $d \geq \chi_\cQ ( \cQ \otimes_\cR^\mathbb{L} C^\bullet)$ and has the property that the map $H^2 (C^\bullet) \otimes_\cR \cR' \to \cY \otimes_\cR \cR'$ factors through the map $H^2 (C^\bullet) \otimes_\cR \cR' \to H^2 (C^\bullet \otimes_\cR^\mathbb{L} \cR')$ induced by $H^2 (\rho)$. We also fix an ordered $\cR$-basis $\mathfrak{B}$ of $\cY$ and write $\mathfrak{B}'$ for the induced $\cR'$-basis of $\cY \otimes_\cR \cR'$.
    Then one has the commutative diagram
    \begin{cdiagram}[column sep=large]
        \Det_\cR (C^\bullet)^{-1} \arrow{r}{\rho^{\det}} \arrow{d}{\vartheta_{C^\bullet, \mathfrak{B}}} & 
        \Det_{\cR'} (C^\bullet \otimes_\cR^\mathbb{L} \cR') \arrow{d}{\vartheta_{C^\bullet \otimes_\cR^\mathbb{L} \cR', \mathfrak{B}'}} \\
        \cQ \otimes_\cR \exprod^n_\cR H^1 (C^\bullet) \arrow{r}{\wedge^n H^1 (\rho)} & \cQ \otimes_\cR \exprod^n_{\cR'} H^1 (C^\bullet \otimes_\cR^\mathbb{L} \cR') 
    \end{cdiagram}
\end{liste}

\end{prop}

\begin{proof}
To prove claim (a)\,(i), we note that the determinant functor behaves well under base change so that, replacing all appearing complexes by those obtained from applying the functor $(e_{C^\bullet, \emptyset} \cdot e_{B^\bullet, \cY}\cdot \cQ) \otimes_\cR^\mathbb{L} (-)$ to them and replacing $\cY$ by $(e_{C^\bullet, \emptyset} \cdot e_{B^\bullet, \cY} \cdot \cQ) \otimes_\cR \cY$ if necessary, we may assume $\cR = \cQ$ and $e_{B^\bullet, \cY} = 1 = e_{C^\bullet, \emptyset} = e_{C^\dagger, \emptyset}$. In particular, $B^\bullet$ is acyclic outside degrees 1 and 2 with $H^2 (C^\bullet) = \cY$ and $C^\bullet$ is acyclic. Since $\cR$ is semisimple (even a product of fields), it follows that $B^\bullet$ admits a representative of the form $[H^1 (B^\bullet) \stackrel{0}{\to} \cY]$ and $C^\bullet$ can be represented by the zero complex. 
Our convention then identifies $\Det_\cR( C^{\bullet, \ast})$ with $\Det_\cR (C^\bullet)^{-1} = \cR^\ast$ and under this identification $\mathrm{Ev}_a$ corresponds with the map $\cQ \to \cQ$ that sends $x$ to $a (x)$. On the other hand, the isomorphism $\cR^\ast \cong \cR, f \mapsto 
f(1)
$ combines with Lemma \ref{explicit description projection map lem} to imply that $\vartheta_{C^{\bullet, \ast}, \emptyset} (a) = 
a (1)
$. In addition, for our fixed choices of representatives, the triangle (\ref{triangle in statement of prop}) identifies $A^\bullet$ with $B^\bullet$ and so, using the description of $\vartheta_{A^\bullet, \mathfrak{B}}$ given in Lemma \ref{explicit description projection map lem}, it is straightforward to check that the diagram in claim (a)\,(i) is indeed commutative. The commutativity of the diagram in (a)\,(ii) can be proved in the same fashion.\\
To prove part (i) of claim (b), it suffices to note that $\cQ$ is a finite product of fields and therefore $\Hom_\cQ (-, \cQ)$ an exact functor. 
As for part (ii) of claim (b), we may again reduce to the case that $\cR = \cQ$. To justify a further reduction, we let $e$ be a primitive orthogonal idempotent of $\cQ$ such that $e \cdot (( \wedge^n H^1 (f^\ast)) \circ \vartheta_{C^\bullet, \mathfrak{B}})$ is nonzero. In particular, the map $eH^1 (f^\ast) \: e H^1 (C) \to e \cQ$ is nonzero and hence surjective because $e\cQ$ is a field. From the exact sequence
\begin{cdiagram}
    0 \arrow{r} & H^1 (D^\bullet) \arrow{r} & H^1 (C^\bullet) \arrow{rr}{H^1 (f^\ast)}& & R \arrow{r} & H^2 (D^\bullet) \arrow{r} & H^2 (C^\bullet) \arrow{r} & 0
\end{cdiagram}%
we then conclude that $e H^2 (D^\bullet) \cong e H^2 (C^\bullet)$. On the other hand, if $e \vartheta_{C^\bullet, \mathfrak{B}}$ is nonzero, then $e e_{C^\bullet, \mathfrak{B}} = e$ and this combines with the isomorphism $e H^i (D^\bullet) \cong e H^i (C^\bullet)$ that is valid for every integer $i \neq 1$ to imply that also $e e_{D^\bullet, \mathfrak{B}} = e$. This proves that $e_{D^\bullet, \mathfrak{B}}$ acts as the identity on the image of $( \wedge^n H^1 (f^\ast)) \circ \vartheta_{C^\bullet, \mathfrak{B}}$. Since the same is true for the image of $\vartheta_{D^\bullet, \mathfrak{B}}$ by its very definition, we may replace the complexes $C^\bullet$ and $D^\bullet$ by $e_{D^\bullet, \mathfrak{B}} \otimes_\cQ^\mathbb{L} C^\bullet$ and $e_{D^\bullet, \mathfrak{B}} \otimes_\cQ^\mathbb{L} D^\bullet$, respectively, to reduce to the case that $e_{C^\bullet, \mathfrak{B}} = 1 = e_{D^\bullet, \mathfrak{B}}$.\\  
In this case, then, $C^\bullet$ admits a representative of the form $[H^1 (C^\bullet) \stackrel{0}{\to} \cY]$ and a standard mapping cone construction yields the representative $[H^1 (C^\bullet) \xrightarrow{H^1 (f^\ast) \oplus 0} \cR \oplus \cY]$ for $D^\bullet$. Given these explicit representatives, the commutativity of the diagram in claim (b)\,(ii) is a direct consequence of the descriptions of the maps $\vartheta_{C^\bullet, \mathfrak{B}}$ and $\vartheta_{D^\bullet, \mathfrak{B}}$ obtained in Lemma \ref{explicit description projection map lem}. \\
To prove claim (c), we fix a representative
\begin{equation} \label{a representative in the middle of the proof}
\dots \to 0 \to C^{-n} \xrightarrow{\partial_{-n}} C^{- n + 1} \to \dots \to C^{m - 1} \xrightarrow{\partial_{m - 1}} C^m \to 0\to \dots  
\end{equation}
of $C^\bullet$ with suitable natural numbers $n, m \in \N$ and finitely generated free $\cR$-modules $C^j$ for $j \in \{ -n, \dots, m\}$ that are each placed in degree $j$.\\
The assumption that $C^\bullet$ has vanishing cohomology outside degrees 1 and 2 combines with the vanishing of $\chi (C^\bullet)$ to imply that $C^\bullet \otimes_\cR^\mathbb{L} e\cQ$, with $e \coloneqq e_{C^\bullet, \emptyset}$, is acyclic. In particular, for every $j$
we may fix a splitting 
\[ f_j \: e \cQ \otimes_\cR C^j \cong e \cQ \otimes_\cR \big( \im \partial_{j - 1} \oplus ( C^j / \ker \partial_{j}) \big)
\]
as well as an isomorphism of $\cR$-modules 
\[
\phi \: \cQ \otimes_\cR C^\mathrm{odd} \coloneqq {\bigoplus}_{j \text{ odd}} (\cQ \otimes_\cR C^j ) \to C^\mathrm{even} \coloneqq {\bigoplus}_{j \text{ even}} (\cQ \otimes_\cR C^j)
\]
that restricts to give the isomorphism
\[
e \cQ \otimes_\cR C^\mathrm{odd} \stackrel{\simeq}{\longrightarrow} e \cQ \otimes_\cR C^\mathrm{even}, \quad ( a_j)_j \mapsto ( (\partial_{j - 1} \circ f_{j - 1})^{-1} \oplus \partial_{j}) ( f_j (a_j)).
\]
Writing $M (\phi)$ for the matrix (with entries in $\cQ$) representing $\phi$ in a choice of $\cR$-bases for $C^\mathrm{even}$ and $C^\mathrm{odd}$, one has the equality of $\cR$-submodules of $\cQ$ 
\[
\vartheta_{C^\bullet, \emptyset} ( \Det_\cR ( C^\bullet)^{-1}) = e \cdot  {\det}_\cR ( M (\phi)) \cdot \cR. 
\]
On the other hand, applying $\RHom_\cR ( - , \cR)$ to the representative (\ref{a representative in the middle of the proof}), we see that $C^\dagger$ is represented by
\[
\dots \to 0 \to C^m \xrightarrow{\partial_{m - 1}^{\mathrm{tr}}}
C^{m - 1} \to \dots \to C^{- n + 1} \xrightarrow{\partial_{-n}^{\mathrm{tr}}} C^{-n} \to 0 \to \dots,
\]
where now $C^j$ is placed in degree $ - j + 3$, and the maps $\partial_{j}^{\mathrm{tr}}$ are obtained from $\partial_j$ in the following way: Fix $\cR$-bases of $C^j$ and $C^{j + 1}$, and write $A_j$ for the matrix representing $\partial_j$ with respect to these bases. Then $\partial_{j}^{\mathrm{tr}}$
is the unique $\cR$-linear morphism $C^{j + 1} \to C^j$ that, for our fixed choices of bases, is represented by the transpose $A^{\mathrm{tr}}_j$ of the matrix $A_j$.\\
An analysis similar to the one above shows that 
\[
\vartheta_{C^\dagger, \emptyset} (\Det_\cR (C^\dagger)) =  {\det}_\cR (M (\phi')) \cdot \cR, 
\]
where now $\phi' \: \cQ \otimes_\cR C^{\dagger, \mathrm{odd}} = \cQ \otimes_\cR C^\mathrm{even} \to  \cQ \otimes_\cR C^\mathrm{odd} = \cQ \otimes_\cR C^{\dagger, \mathrm{even}}$ (for the identification we have used that $i$ is odd) is any isomorphism of $\cR$-modules that restricts to
\[
e \cQ \otimes_\cR C^\mathrm{even} \to e\cQ \otimes_\cR C^\mathrm{odd}, \quad 
(a_j)_j \mapsto  (\partial_{j - 1}^{\mathrm{tr}} \oplus ( \partial_{j}^{\mathrm{tr}} \circ f_{j})^{-1} ) (f_j (a_j)).
\]
The equality claimed in (c) therefore follows upon noting that $\det_\cR ( M (\phi)) =  \det_\cR ( M (\phi'))$.\\
Finally, claim (d) is valid because the definitions of $\vartheta_{C^\bullet, \mathfrak{B}}$ and $\vartheta_{C^\bullet \otimes_\cR^\mathbb{L} \cR', \mathfrak{B}'}$ both involve the passage-to-cohomology map. 
\end{proof}

\begin{rk}
    Part (a) of Proposition \ref{functoriality properties of det projection map} generalises the observation of Burns and Flach in \cite[Lem.\@ 1]{BurnsFlach98}.
\end{rk}

\subsection{Fitting ideals and exterior biduals} \label{apendix section Fitting ideals biduals}

In this subsection we study integrality properties of the maps from Definition \ref{def projection map from det}. In doing so, we will be naturally lead to consider $\cR$-modules of the form
\begin{equation} \label{bidual with denominators}
\big( \a \otimes_\cR \exprod^r_\cR M^\ast \big)^\ast 
\end{equation}
for an ideal $\a \subseteq \cR$, an $\cR$-module $M$, and an integer $r \geq 0$. 
Taking $\a = \cR$, this construction specialises to give the `$r$-th exterior bidual' $\bidual^r_\cR M \coloneqq ( \exprod^r_\cR M^\ast)^\ast$ of $M$ that has been studied in detail by Burns and Sano \cite[App.]{sbA} and Sakamoto \cite[App.\@ B]{Sakamoto20} as a generalisation of the lattice utilised by Rubin in \cite{Rub96}. 
The following observation shows that modules of the form (\ref{bidual with denominators}) should be considered a generalisation of Rubin's lattice with `denominators bounded by $\a$' and, in particular, specialise to the lattices studied by Popescu in \cite{Popescu02}.

\begin{lem} \label{biduals and rubin lattices}
    Let $\cR$ be a reduced Noetherian ring with total ring of fractions $\cQ$. For every ideal $\a \subseteq \cR$, finitely generated $\cR$-module $M$, and integer $r \geq 0$, there is a canonical isomorphism
    \[
    \xi^r_{M, \a} \: \big \{ 
    a \in \cQ \otimes_\cR \exprod^r_\cR M \mid \Ann_\cR (\a) \cdot a = 0, f (a) \in \a^{-1} \text{ for all } f \in \exprod^r_\cR M^\ast \big \}  
    \stackrel{\simeq}{\longrightarrow} 
    \big( \a \otimes_\cR \exprod^r_\cR M^\ast \big)^\ast .
    \]
    Here $\a^{-1} \coloneqq \{ q \in \cQ \mid q a \in \cR \text{ for all } a \in \a \}$.
\end{lem}

\begin{proof}
This is a natural generalisation of the result of \cite[Prop.\@ A.8]{sbA}. For the convenience of the reader, we provide the argument. \\
Applying the functor $\Hom_\cR ( \a \otimes_\cR \exprod^r_\cR M^\ast, -)$ to the exact sequence
$0 \to \cR \to \cQ \to \cR / \cQ \to 0$
gives that
\begin{cdiagram}
\big( \a \otimes_\cR \exprod^r_\cR M^\ast \big)^\ast = \ker \Big \{ \Hom_\cR \big ( \a \otimes_\cR \exprod^r_\cR M^\ast, \cQ \big ) \to 
 \Hom_\cR \big ( \a \otimes_\cR \exprod^r_\cR M^\ast, \cQ / \cR \big) \Big \}.
\end{cdiagram}%
As $\cQ$ is a finite product of fields, the ideal $\cQ \cdot \Ann_\cR (\a)$ is generated by an idempotent $e \in \cQ$ (so that $\cQ \cdot \a = (1 - e) \cQ$) and $\cQ \otimes_\cR M$ is a finitely generated projective $\cQ$-module. By \cite[Lem.\@ A.1]{sbA} we therefore have an isomorphism
\[
(1- e) \cQ \otimes_\cR \exprod^r_\cR M \cong (1 - e) \exprod^r_{\cQ} ( \cQ \otimes_R M) 
\stackrel{\simeq}{\longrightarrow} (1 - e) \bidual^r_{\cQ} (\cQ \otimes_R M), \quad 
a \mapsto \{ \varphi \mapsto \varphi (a) \}. 
\]
Furthermore, there is an isomorphism $\Hom_\cQ ( \cQ \otimes_\cR N, \cQ) \cong \Hom_\cR (N, \cQ)$ for any $\cR$-module $N$ by the tensor-hom adjunction. It follows that $(1 - e) \bidual^r_\cQ ( \cQ \otimes_R M) \cong 
 \Hom_\cR ( \a \otimes_\cR \exprod^r_\cR M^\ast, \cQ )$. 
 The lemma now follows upon noting that an element $a$ of $(1 - e) \cQ \otimes_R \exprod^r_R M$
 belongs to the kernel of $\Hom_\cR (\a \otimes_\cR \exprod^r_\cR M^\ast, \cQ) \to \Hom_\cR ( \a \otimes_\cR \exprod^r_\cR M^\ast, \cQ / \cR )$ under these identifications if and only if one has $x f(a) \in \cR$ for all $x \in \a$ and $f \in \exprod^r_\cR M^\ast$. 
\end{proof}

Note that, for every integer $s \leq r$ and $f \in \exprod^s_\cR M^\ast$, we obtain a map
\[
\big( \a \otimes_\cR \exprod^r_\cR M^\ast \big)^\ast \to  \big( \a \otimes_\cR \exprod^{r - s}_\cR M^\ast \big)^\ast, \quad 
\varphi \mapsto \{ a \otimes g \mapsto \varphi ( a \otimes (f \wedge g) ) \} 
\]
that, by abuse of notation, will also be denoted as $f$. This construction gives a commutative diagram
\begin{cdiagram}[row sep=small]
\exprod^r_\cR M \arrow{r}{f} \arrow{d} & \exprod^{r - s}_\cR M \arrow{d}\\ 
\big( \a \otimes_\cR \exprod^r_\cR M^\ast \big)^\ast \arrow{r}{f} &  \big( \a \otimes_\cR \exprod^{r - s}_\cR M^\ast \big)^\ast,
\end{cdiagram}%
where the top arrow is the map $\Phi^{r, s}_M (f)$ defined earlier and the vertical arrows are the maps (with $t \in \{r, r - s\}$)
\[
\exprod^t_\cR M  \to \big( \a \otimes_\cR \exprod^t_\cR M^\ast \big)^\ast, 
\quad m_1 \wedge \dots \wedge m_t \mapsto \big \{ a \otimes f_1 \wedge \dots \wedge f_t \mapsto 
a \cdot \det ( f_i (m_j))_{1 \leq i, j \leq t} \big \}.
\]
Following Sakamoto \cite{Sakamoto20}, the results in the remainder of this section are most naturally stated for rings that are `quasi-normal'. That is, Noetherian rings $\cR$ that satisfy the following two conditions.
\begin{itemize}[align=left]
    \item[($\mathrm{G}_1$)] The localisation $\cR_\p$ of $\cR$, at any prime ideal $\p \subseteq \cR$ of height at most one, is Gorenstein.
    \item[($\mathrm{S}_2$)] The localisation $\cR_\p$ of $\cR$, at any prime ideal $\p \subseteq \cR$  of height $n$, has depth at least $\min \{ 2, n \}$. 
\end{itemize}

\begin{prop} \label{integrality properties of det projection map}
Let $\cR$ be a quasi-normal ring and $C^\bullet \in D^\mathrm{perf} (\cR)$ a complex that admits a standard representative (\ref{standard representative}) with respect to a surjection $\kappa \: H^2 (C^\bullet) \to \cY$ for a free $\cR$-module $\cY$ of rank $d$ and an ordered $\cR$-basis $\mathfrak{B}$ of $\cY$. 
We also set $n \coloneqq d + n_1 - (n_0 + n_2)$. \begin{liste}
    \item There exists a well-defined canonical map
    \[
    \varpi_{C^\bullet, \mathfrak{B}} \: \Det_\cR (C^\bullet)^{-1}
    \to  \big( \Fitt^0_\cR (H^1(C^\bullet)_\tor^\vee) \otimes_\cR \exprod^{n}_\cR H^1 (C^\bullet)^\ast \big)^\ast
    \]
    with the property that, if $\cR$ is reduced, then one has
    \[
    \xi^{-1} \circ \varpi_{C^\bullet, \mathfrak{B}} = \vartheta_{C^\bullet, \mathfrak{B}}
    \]
    with $\xi \coloneqq \xi^{n}_{H^1 (C^\bullet), \Fitt^0_\cR (H^1(C^\bullet)_\tor^\vee)}$ the isomorphism from Lemma \ref{biduals and rubin lattices}.
    \item One has an inclusion of $\cR$-modules
    \[
    \big \{ f ( a) \mid a \in \im ( \varpi_{C^\bullet, \mathfrak{B}}),  f \in \exprod^{n}_\cR H^1 (C^\bullet)^\ast \big \} \subseteq   \Fitt^0_\cR (H^1(C^\bullet)_\tor^\vee)^\ast \otimes_\cR \Fitt^d_\cR (H^2 (C^\bullet))^{\ast \ast}.  
    \]
    Moreover, if $\p \subseteq \cR$ is a height-one prime ideal in the support of the cokernel of this inclusion, then $(H^1(C^\bullet)_\tor^\vee)_\p$ does not have finite projective dimension as an $\cR_\p$-module. 
    \item If $\p \subseteq \cR$ is a prime ideal of height at most one with the property that $\cR_\p$ is a regular local ring, then one has the equality
    \[
    \im ( \vartheta_{C^\bullet, \mathfrak{B}})_\p = \Fitt^d_\cR (H^2 (C^\bullet))_\p \cdot \Fitt^0_\cR (H^1(C^\bullet)_\tor^\vee)^{-1}_\p \cdot 
    \big( \bidual^{n}_\cR H^1 (C^\bullet) \big)_\p.
    \]
\end{liste}
\end{prop}

\begin{proof}
    Dualising the tautological exact sequence $0 \to F^0 \to F^1 \to F^1 / F^0 \to 0$ gives the exact sequence
    \begin{equation} \label{free presentation for Ext}
    \begin{tikzcd}
        0 \arrow{r} & (F^1 / F^0)^\ast \arrow{r} & (F^1)^\ast \arrow{r}{\partial_0^\ast} & (F^0)^\ast \arrow{r} & 
        \Ext^1_\cR (F^1 / F^0, \cR) \arrow{r} & 0,
    \end{tikzcd}%
    \end{equation}
    from which we deduce, by the definition of Fitting ideals, that one has
    \begin{equation} \label{description image in terms of fitt}
    \im \big \{ \exprod^{n_0}_\cR (F^1)^\ast \xrightarrow{\partial_0^\ast} \exprod^{n_0}_\cR (F^0)^\ast \big\} = 
    \Fitt^0_\cR (  \Ext^1_\cR (F^1 / F^0, \cR)) \cdot \exprod^{n_0}_\cR (F^0)^\ast.
    \end{equation}
     In particular, for every $f \in \exprod^{n_0}_\cR (F^0)^\ast$ and $\lambda \in \Fitt^0_\cR (  \Ext^1_\cR (F^1 / F^0, \cR))$ we can find $\widetilde{\lambda f} \in \exprod^{n_0}_\cR (F^1)^\ast$ with $\partial_0^\ast ( \widetilde{\lambda f}) = \lambda \cdot f$. \\
    Write $b_1, \dots, b_{n_2}$ for the basis from condition (iv) in Definition \ref{standard representative def}. Given an element $f \otimes a \otimes \bigwedge_{1 \leq i \leq n_2} b_i^\ast$ of $\Det_\cR ( C^\bullet)^{-1} = \exprod^{n_0}_\cR (F^0)^\ast \otimes_\cR \exprod^{n_1}_\cR F^1
    \otimes_\cR \exprod^{n_2}_\cR (F^2)^\ast$ we now consider
    \[
    a'_{\widetilde{\lambda f}} \coloneqq ( \widetilde{\lambda f} \wedge  {\bigwedge}_{d + 1 \leq i \leq n_2} (b_i^\ast \circ \partial_1)) (a) \in \exprod^{n}_\cR F^1. 
    \]
    Set $F' \coloneqq \bigoplus_{i = d + 1}^{n_2} (\cR \cdot b_i)$ so that we have a direct sum decomposition $F^2 = F' \oplus \cY$. By property (iv) in Definition \ref{standard representative def}, $\partial_1$ maps to $F'$ and therefore it follows from  \cite[Lem.\@ 2.17\,(ii)]{BB} that the map $\bigwedge_{d + 1 \leq i \leq n_2} (b_i^\ast \circ \partial_1) \:  \exprod^{n_1 - n_0}_\cR F_1 \to \exprod^{n}_\cR F_1$ has image in $\bidual^{n}_\cR \ker (\partial_1)$. In particular, $a'_{\widetilde{\lambda f}}$ belongs to $\bidual^{n}_\cR \ker (\partial_1)$. Writing $\phi$ for the map $\bidual^{n}_\cR \ker (\partial_1) \to \bidual^{n}_\cR H^1 (C^\bullet)$ induced by the projection $\ker (\partial_1) \to \ker (\partial_1)/ F_0 = H^1 (C^\bullet)$, we may then define the map
    \begin{align} \nonumber
      \Det_\cR ( C^\bullet)^{-1} & \to \big( \Fitt^0_\cR ( \Ext^1_\cR (F^1 / F^0, \cR)) \otimes_\cR \exprod^{n}_\cR H^1 (C^\bullet)^\ast \big)^\ast,\\
     a & \mapsto \{ \lambda \otimes \varphi \mapsto 
     (-1)^{n(n_2 - d)} \cdot
     \varphi ( \phi ( a'_{\widetilde{\lambda f}})) \}. 
     \label{pre map from det}
    \end{align}
    We now claim that $\varphi ( \phi ( a'_{\widetilde{\lambda f}}))$ is independent of the choice of lift $\widetilde{\lambda f}$ of $\lambda f$ even though $a'_{\widetilde{\lambda f}}$ might depend on it. 
    Let $g_1$ and $g_2$ be two such lifts of $\lambda f$, then we shall show $\varphi ( \phi ( a'_{g_1})) = \varphi ( \phi ( a'_{g_2}))$ by verifying that $\varphi ( \phi (  a'_{g_1} - a'_{g_2}))$ vanishes in the total ring of fractions $\cQ$ of $\cR$. For this it is enough to prove the required vanishing in the localisation $\cR_\p$ at every minimal prime ideal $\p$ of $\cR$. By the validity of condition ($\mathrm{G}_1$) the ring $\cR_\p$ is self-injective and so $\cR_\p \otimes_\cR \Ext^1_\cR ( F^1 / F^0, \cR)$ vanishes. It therefore follows from (\ref{free presentation for Ext}) that $(F^1  / F^0)_\p$ is a free $\cR_\p$-module of rank $n_1 - n_0$ and that we have a canonical isomorphism 
    \begin{equation} \label{isomorphism of wedge powers}
    \exprod^{n_1}_{\cR_\p} (F^1)_\p^\ast \cong ( \exprod^{n_0}_{\cR_\p} (F^0)^\ast_\p ) \otimes_\cR  \exprod^{n_1 - n_0}_{\cR_\p} (F^1 / F^0)^\ast_\p.
    \end{equation}
    In addition, as $\cR_\p$ is self-injective, dualising the injection $H^1 (C^\bullet) = \ker (\partial_1) / F_0 \hookrightarrow F_1 / F_0$ shows that the restriction map $(F_1 / F_0)^\ast_\p \to H^1 (C^\bullet)^\ast_\p$ is surjective so that we may assume $\varphi$ is the restriction of an element $\widetilde \varphi \in \exprod^r_{\cR_\p} (F_1 / F_0)^\ast_\p$. Now, the element 
    \[
    (\lambda f) \otimes \big ( \widetilde \varphi \wedge \exprod_{d + 1 \leq i \leq n_2} (b_i^\ast \circ \partial_1) \big) \in ( \exprod^{n_0}_{\cR_\p} (F^0)^\ast_\p ) \otimes_\cR  \exprod^{n_1 - n_0}_{\cR_\p} (F^1 / F^0)^\ast_\p
    \]
    is a preimage of both $g_1 \wedge \widetilde \varphi \wedge \bigwedge_{d + 1 \leq i \leq n_2} (b_i^\ast \circ \partial_1)$ and $g_2 \wedge \widetilde \varphi \wedge \bigwedge_{d + 1 \leq i \leq n_2} (b_i^\ast \circ \partial_1)$ under the isomorphism (\ref{isomorphism of wedge powers}), and this implies the claimed independence from the lift $\widetilde{\lambda f}$ of $\lambda f$.  \\ 
    Next we note that there is an exact sequence
    \begin{cdiagram}
        \Ext^1_\cR ( J, \cR) \arrow{r}  & \Ext^1_\cR ( F^1 / F^0, \cR) \arrow{r} &  \Ext^1_\cR ( H^1 (C^\bullet), \cR) \arrow{r} & \Ext^2_\cR ( J, \cR), 
    \end{cdiagram}%
    where $J$ denotes the cokernel of $H^1 (C^\bullet) \to F^1 / F^0$. 
    We claim that the middle arrow in this exact sequence is a pseudo-isomorphism.
    Indeed, from the exact sequence $0 \to J \to F^2 \to H^2 (C^\bullet) \to 0$ we obtain isomorphisms $\Ext^i_\cR (J, \cR) \cong \Ext^{i + 1}_\cR ( H^2 (C^\bullet), \cR)$ for every $i \geq 1$, which shows these modules to be pseudo-null because $\cR$ is assumed to satisfy condition ($\mathrm{G}_1$). 
    Since one also has a pseudo-isomorphism $\Ext^1_\cR (H^1 (C^\bullet), \cR) \to H^1 (C^\bullet)_\tor^\vee$ by Lemma \ref{ext and dual of tor} below and reflexive ideals of $\cR$ are uniquely determined by their localisation at primes of height at most one (cf.\@ \cite[Lem.\@ C.11]{Sakamoto20}), one has an equality of reflexive hulls
    \[
    \mathfrak{A} \coloneqq \Fitt^0_\cR ( \Ext^1_\cR ( F^1 / F^0, \cR))^{\ast \ast} = \Fitt^0_\cR ( \Ext^1_\cR (H^1 (C^\bullet), \cR))^{\ast \ast} = \Fitt^0_\cR ( H^1 (C^\bullet)_\tor^\vee)^{\ast \ast}.
    \]
    For any finitely generated $\cR$-module $M$, we therefore obtain pseudo-isomorphisms
    \[
    \Fitt^0_\cR ( \Ext^1_\cR ( F^1 / F^0, \cR)) \otimes_\cR M \longrightarrow \mathfrak{A} \otimes_\cR M \longleftarrow \Fitt^0_\cR ( H^1 (C^\bullet)_\tor^\vee) \otimes_\cR M
    \]
    and this combines with condition ($\mathrm{S}_2$) to imply that, by \cite[Lem.\@ B.5]{Sakamoto20}, taking duals gives isomorphisms
    \[
    \big( \Fitt^0_\cR ( \Ext^1_\cR ( F^1 / F^0, \cR)) \otimes_\cR M \big)^\ast \stackrel{\simeq}{\longleftarrow} \big( \mathfrak{A} \otimes_\cR M \big)^\ast \stackrel{\simeq}{\longrightarrow} \big( \Fitt^0_\cR ( H^1 (C^\bullet)_\tor^\vee) \otimes_\cR M \big)^\ast.
    \]
    We then define the map $\varpi_{C^\bullet, \mathfrak{B}}$ in claim (a) to be the map defined in (\ref{pre map from det}) composed with these isomorphisms for $M = \exprod^{n}_\cR H^1 (C^\bullet)^\ast$. With this definition, it follows from Lemma \ref{explicit description projection map lem} and the definition of the map $\xi$ that one has $\xi^{-1} \circ \varpi_{C^\bullet, \mathfrak{B}} = \vartheta_{C^\bullet, \mathfrak{B}}$, as required to prove claim (a).\\
    As for claim (b), Lemma \cite[Lem.~C.11]{Sakamoto20} reduces to prove the claim locally at primes $\p \subseteq \cR$ of height at most one.
    To do this, we first observe that by \cite[Prop.\@ A.2\,(ii)]{sbA} one has
    \begin{align*}
     \Big \{ h (\exprod_{d + 1 \leq i \leq n_2} (b_i^\ast \circ \partial_1)  (x)) \, \big | \, x \in \exprod^{n_1}_{\cR_\p} F^1_\p, h \in \exprod^{n + n_0}_{\cR_\p} (F^1_\p)^\ast \Big \} 
    & = \Fitt^0_\cR ( \ker \{ H^2 (C^\bullet) \stackrel{\kappa}{\to} \cY \} )_\p \\
    & = \Fitt^d_\cR ( H^2 (C^\bullet))_\p.
    \end{align*}
    Now, the cokernel of the restriction map $(F^1 / F^0)^\ast \to H^1 (C^\bullet)^\ast$ identifies with a submodule of $\Ext^1_\cR (J, \cR)$ and so, as observed above, is pseudo-null. Given this, the inclusion claimed in the first part of (a) follows by taking $h = \widetilde{\lambda f} \wedge \widetilde{\varphi}$ with $\widetilde{\varphi}$ a lift of $\varphi \in \exprod^{n}_{\cR_\p} H^1 (C^\bullet)^\ast_\p$ so that $h (\bigwedge_{d + 1 \leq i \leq n_2} (b_i^\ast \circ \partial_1) (a)) = \varphi ( a_{\widetilde{f\lambda}})$.\\
    In addition, said inclusion is an equality (at $\p$) whenever the subset of $\exprod^{n + n_0}_{\cR_\p} (F^1_\p)^\ast$ comprising all such $\widetilde{\lambda f} \wedge \widetilde \varphi$ is equal to $\exprod^{n + n_0}_{\cR_\p} (F^1_\p)^\ast$. To investigate when this happens, assume that $\Ext^1_\cR ( H^1 (C^\bullet)^\vee_\tor, \cR)_\p = \Ext^1_\cR (F^1 / F^0, \cR)_\p$ has finite projective dimension. From the Auslander--Buchsbaum formula we then see that $\Ext^1_\cR (F^1 / F^0, \cR)_\p$ has projective dimension at most one. From (\ref{free presentation for Ext}) it therefore follows that $\im (\partial_0^\ast)_\p$ is a projective (hence free) $\cR_\p$-module of rank $n_0$ so that we have an isomorphism 
    \begin{align*}
            \exprod^{n_1}_{\cR_\p} (F^1)^\ast_\p & \cong \big( \exprod^{n_1 - n_0}_{\cR_\p} (F^1 / F^0)^\ast_\p \big) \otimes_{\cR} \big( \exprod^{n_1 - n_0}_{\cR_\p} \im (\partial_0^\ast)_\p \big) \\ 
            & \stackrel{(\ref{description image in terms of fitt})}{=}
            \big( \exprod^{n_1 - n_0}_{\cR_\p} (F^1 / F^0)^\ast_\p \big) \otimes_{\cR} \Fitt^0_\cR ( \Ext^1_\cR (F^1 / F^0, \cR)) \cdot \exprod^{n_0}_{\cR_p} (F^0)^\ast_\p.
    \end{align*}
    To finish the proof of claim (b), it therefore suffices to recall that, as already observed earlier, the restriction map $(F^1 / F^0)_\p^\ast \to H^1 (C^\bullet)_\p^\ast$ is surjective.\\ 
    To prove claim (c), assume that $\p \subseteq \cR$ is a prime ideal of height at most one such that $\cR_\p$ is a regular local ring, so either a field or a discrete valuation domain. In this case, one has identifications $\Fitt^0_\cR ( H^1 (C^\bullet)^\vee_\tor)_\p^\ast \cong \Fitt^0_\cR ( H^1 (C^\bullet)^\vee_\tor)_\p^{-1}$ and $\big(\bidual^n_\cR H^1 (C^\bullet)\big)_\p \cong \big( \exprod^n_{\cR} H^1 (C^\bullet)_\mathrm{tf} \big)_\p$ with $H^1 (C^\bullet)_\mathrm{tf} \coloneqq H^1 (C^\bullet) / H^1 (C^\bullet)_\tor$ the torsion-free quotient of $H^1 (C^\bullet)$. Since the latter is a free $\cR_\p$-module, claim (c) follows from claim (b) and the equality
    \[
    A \cdot M = \{ a \in M \mid f(a) \in A \text{ for all } f \in M^\ast \}
    \]
    that holds for every free $\cR_\p$-module $M$ and $\cR_\p$-submodule $A$ of the total ring of fractions of $\cR_\p$.
    \end{proof}

\begin{rk}
    Results similar to Proposition \ref{integrality properties of det projection map} have previously appeared in various places, see for example \cite[Prop.\@ A.11]{sbA}. The main novelty of Proposition \ref{integrality properties of det projection map} is that we do not need to assume $H^1 (C^\bullet)_\tor^\vee$ to vanish, and in this regard our approach is related to \cite[Prop.~3.18]{bst}.
\end{rk}

\begin{lem} \label{ext and dual of tor}
    Let $\cR$ be a ring that satisfies $(\mathrm{G}_1)$. For every finitely generated $\cR$-module $M$ with $\cR$-torsion submodule $M_\tor$ there is a pseudo-isomorphism
    \[
     \Ext^1_\cR (M, \cR) \to  M_\tor^\vee.
    \]
\end{lem}

\begin{proof}
    Since $\cR$ is assumed to be quasi-normal, its total ring of fractions $\cQ$ is a self-injective ring and so $\Hom_\cR ( -, \cQ) \cong \Hom_\cQ ( - \otimes_\cR \cQ, \cQ)$ is an exact functor. As a consequence, $\Ext^1_\cR ( - , \cQ)$ vanishes and we obtain a commutative diagram with exact rows of the form
    \begin{cdiagram}[column sep=small, row sep=small]
        \Hom_\cR ( M / M_\tor, \cQ) \arrow{r} \arrow{d}{\simeq} & \Hom_\cR ( M / M_\tor, \cQ  / \cR) \arrow{r} \arrow{d} & \Ext^1_\cR ( M / M_\tor, \cR) \arrow{r} \arrow{d} & 0 \\ 
        \Hom_\cR (M, \cQ) \arrow{r} & \Hom_\cR (M, \cQ / \cR) \arrow{r} & \Ext^1_\cR (M, \cR) \arrow{r} & 0.
    \end{cdiagram}%
    A diagram chase then shows that one has a composite map
    \begin{align} \label{surjection}
\Ext^1_\cR (M, \cR) & \twoheadrightarrow \coker \{  \Hom_\cR (M / M_\tor, \cR) \to \Ext^1_\cR (M, \cR) \} \\ \nonumber
& \, \cong \, \coker \{ \Hom_\cR ( M / M_\tor, \cQ  / \cR) \to \Hom_\cR (M, \cQ / \cR) \} \\ 
& \hookrightarrow \Hom_\cR (M_\tor, \cQ / \cR). \label{injection}
    \end{align}
    To  show that this map is a pseudo-isomorphism, we may reduce to the case that $\cR$ is a Gorenstein ring of dimension at most one. In this case, then, $\cQ / \cR$ is an injective $\cR$-module (see \cite[Thm.\@ 6.2\,(2)]{bass}) and so the map (\ref{injection}) is an isomorphism. 
    To show that (\ref{surjection}) is an isomorphism it suffices to prove that $\Ext^1_\cR ( M / M_\tor, \cR)$ vanishes. This follows from the fact that $M / M_\tor$ is a reflexive module (as can be seen from combining \cite[Thm.\@ 6.2\,(4)]{bass} and \cite[Thm.\@ A.1]{vasconcelos68}) so that dualising a projective presentation $P_1 \to P_0 \to (M / M_\tor)^\ast \to 0$ gives an exact sequence $0 \to (M / M_\tor) \to P_0^\ast \to P_1^\ast$ from which we conclude that
    \[
    \Ext^1_\cR ( M / M_\tor, \cR) \cong \Ext^3_\cR ( \coker \{ P_0^\ast \to P_1^\ast \}, \cR) = 0
    \]
    vanishes because $\cR$ is Gorenstein of dimension at most one. 
\end{proof}

\subsection{Bockstein morphisms} \label{bockstein section}

In this section we recall the formalism of Bockstein morphisms that is a variant of the theory of algebraic height pairings developed by Nekov{\'a}\v{r} in \cite[\S\,11]{NekovarSelmerComplexes}. We follow the treatment in \cite[\S\,10]{burns07} but generalise it to the setting of Proposition \ref{integrality properties of det projection map}. \\
We begin with an elementary, but important, observation.

\begin{lem} \label{properties top and bottom cohomology}
Let $\cR$ be a commutative Noetherian ring, $C^\bullet \in D^\mathrm{perf} (\cR)$ a complex, and $M$ a finitely generated $\cR$-module.
\begin{liste}
    \item If $i \in \Z$ is an integer such that $H^j (C^\bullet) = 0$ for all $j > i$, then the morphism $C^\bullet \to C^\bullet \otimes_\cR^\mathbb{L} M$ induces a natural isomorphism
    \begin{equation} \label{isomorphism base change top cohomology}
    H^i (C^\bullet) \otimes_\cR M \stackrel{\simeq}{\longrightarrow} H^i (C^\bullet \otimes_\cR^\mathbb{L} M).
    \end{equation}
    \item Let $\a \subseteq \cR$ be an ideal and suppose there is an isomorphism $\nu \: \cR  / \a \xrightarrow{\simeq} \cR [\a]$ of $\cR$-modules. Then $\nu$ induces a morphism $\nu_M \: M \otimes_\cR (\cR  / \a) \to M$ of $\cR$-modules and an isomorphism
    \begin{equation} \label{derived isom induced by j}
\nu_{C^\bullet} \: C^\bullet \otimes_\cR^\mathbb{L} (\cR /\a) \stackrel{\simeq}{\longrightarrow} \RHom_\cR ( \cR / \a, C^\bullet)
\end{equation}
in $D (\cR)$. If $i \in \Z$ is an integer with $H^j (C^\bullet) = 0$ for all $j < i$, then $\nu_{C^\bullet}$ induces an injection
\begin{equation} \label{isom induced by j on cohomology}
H^i (C^\bullet \otimes_\cR^\mathbb{L} (\cR / \a)) \hookrightarrow H^i (C^\bullet) [\a].
\end{equation}
\end{liste}
\end{lem}

\begin{proof}
Assuming $C^\bullet$ is acyclic in degrees greater than $i$, it admits a representative of the form
\begin{equation} \label{long representative}
 \dots \to C^{i - 2} \to C^{i - 1} \to C^i \to 0,
\end{equation}
where each $C^j$ is a finitely generated projective $\cR$-module that is placed in degree $j$. For every finitely generated $\cR$-module $M$, the complex $C^\bullet \otimes_\cR^\mathbb{L} M$ can then be represented by 
\[
\dots \to C^{i - 2} \otimes_\cR M \to C^{i - 1}  \otimes_\cR M \to C^i \otimes_\cR M \to 0
\]
and, as the functor $( - )\otimes_\cR M$ is right exact, this implies claim (a). \\
As for claim (b), we may define $\nu_M \: M \otimes_\cR (\cR / \a) \to M$ by means of the rule $m \otimes x \mapsto \nu (x) \cdot m$. To justify the isomorphism (\ref{derived isom induced by j}), we will identify $\Hom_\cR ( \cR / \a, \cR)$ with the $\a$-torsion submodule $R [\a]$ of $\cR$ via the isomorphism $\varphi \mapsto \varphi (1)$. For any morphism $f \: A \to B$ of finitely generated free $\cR$-modules one has a commutative diagram
\begin{cdiagram}[row sep=small]
    A \otimes_\cR (\cR / \a) \arrow{d}{\simeq} \arrow{r}{f} & B \otimes_\cR ( \cR / \a) \arrow{d}{\simeq} \\ 
    \Hom_\cR ( \cR / \a, A) \arrow{r}{f_\ast} & \Hom_\cR ( \cR /\a, B),
\end{cdiagram}%
where the vertical isomorphisms are given by $\nu_{A}$. Applying the functors $- \otimes_\cR (\cR / \a)$ and $\Hom_\cR ( \cR / \a, - )$ to the representative (\ref{long representative}) we therefore see that the collection of maps $(\nu_{C^j})_{j \in \Z}$ defines an isomorphism of the form (\ref{derived isom induced by j}).\\
Assuming $C^\bullet$ is acyclic in degrees less than $i$, we may set $Q \coloneqq \coker \{ C^{i - 1} \to C^i \}$ to obtain the representative $0 \to Q \stackrel{\partial}{\to} C^{i + 1} \to \dots$ of $C^\bullet$. Consequently, $C^\bullet \otimes_\cR^\mathbb{L} (\cR /\a)$ is represented by $0 \to Q \otimes_\cR (\cR /\a) \to C^{i + 1} \otimes_\cR (\cR /\a) \to \dots$. 
    In particular, we get the vertical dashed arrow on the left hand side in the diagram
\begin{cdiagram}[row sep=small]
    0 \arrow{r} & H^i (C^\bullet) \arrow{r} & Q \arrow{r}{\partial} & C^{i + 1} \\ 
    0 \arrow{r} & H^i (C^\bullet \otimes_\cR^\mathbb{L} (\cR / \a)) \arrow[dashed]{u} \arrow{r} & 
    Q \otimes_\cR (\cR / \a) \arrow{u}{\nu_Q} \arrow{r}{\partial \otimes \id} & C^{i + 1} \otimes_\cR (\cR / \a) \arrow{u}{\nu_{C^{i + 1}}}.
\end{cdiagram}%
If we can prove that $\nu_Q$ is injective, then it will follow that the dashed arrow is injective as well, and this will prove the last part of claim (b). 
To do this, let $l$ be the least integer with $C^l \neq 0$ and set $Q^j \coloneqq \coker ( C^j \to C^{j - 1} \}$ for all $j \in \{ l, \dots, i - 1\}$. We will then prove by induction on $j$ that $\nu_{Q^j} \: Q^j / \a Q^j \to Q^j [\a]$ is an isomorphism and $\Ext^1_\cR (\cR / \a, Q^j) = 0$ for all $j \in \{ l, \dots, i - 1\}$. Taking $j = i - 1$, this then proves the claimed injectivity of $\nu_Q$.\\
For the base case of the induction we note that, because $C^j = 0$ for all $j > l$ and $C^\bullet$ is acylic in degrees less than $i$, we have an exact sequence $0 \to C^l \to C^{l - 1} \to Q^l \to 0$. The long exact sequence obtained from applying the functor $(-) [\a] \cong \Hom_\cR (\cR / \a, -)$ then shows that $\Ext^j_\cR (\cR /\a, Q^l) = 0$ for all $j \geq 1$, and that we have a commutative diagram of the form
\begin{cdiagram}[row sep=small]
     C^l \otimes_\cR (\cR /\a) \arrow{r} \arrow{d}[left]{\nu_{C^l}}[right]{\simeq} & C^{l - 1} \otimes_\cR (\cR /\a) \arrow{r} \arrow{d}[left]{\nu_{C^{l - 1}}}[right]{\simeq} & Q^l \otimes_\cR (\cR /\a) \arrow{r}  \arrow{d}{\nu_{Q^l}} & 0 \\
         C^l [\a] \arrow{r}   & C^{l - 1}[\a] \arrow{r}  & Q^l[\a] \arrow{r}  & 0.
\end{cdiagram}%
It then follows from the Five Lemma that also $\nu_{Q^l}$ is an isomorphism. 
This proves the base case of the induction, and the inductive step follows by the same argument when instead applied to the exact sequence $0 \to Q^{j} \to C^{j - 2} \to Q^{j - 1} \to 0$ for some $j \in \{ l + 1, \dots, i\}$ such that the claim has already been proven for all $j' < j$.
\end{proof}

The following technical consequence of Lemma \ref{properties top and bottom cohomology} will be useful later on.

\begin{lem} \label{base change for dual of Fitt}
    Let $\cR$ be a one-dimensional Gorenstein ring and $\a \subseteq \cR$ an ideal such that also $\cR / \a$ is a one-dimensional Gorenstein ring. We also suppose to be given a complex $C^\bullet \in D^\mathrm{perf} (\cR)$ that admits a representative of the form (\ref{standard representative}) and assume that 
    \begin{equation} \label{pd assumption}
    \mathrm{pd}_\cR (H^1 (C^\bullet)_\tor^\vee) \leq 1
    \quad \text{ and } \quad
    \mathrm{pd}_{(\cR / \a)} (H^1 (C^\bullet)_\tor^\vee \otimes_\cR (\cR / \a)) \leq 1.
    \end{equation}
    Then there is an isomorphism
    \[
    \Fitt^0_\cR ( H^1 (C^\bullet)_\tor^\vee)^\ast \otimes_\cR (\cR / \a) \xrightarrow{\simeq}
    \Fitt^0_{(\cR /\a)} ( H^1 (C^\bullet \otimes^\mathbb{L}_\cR (\cR /\a))_\tor^\vee)^\ast.
    \]
\end{lem}

\begin{proof}
    Consider the complex $D^\bullet \coloneqq \RHom_{\cR} ( C^\bullet, \cR)$. By dualising (\ref{standard representative}) we see that $H^i (D^\bullet) = 0$ for all $i > 0$ and so an application of Lemma \ref{properties top and bottom cohomology}\,(a) shows that
    \begin{equation} \label{isomorphism ingredient 1}
    H^0 ( D^\bullet) \otimes_\cR (\cR / \a) \cong H^0 ( D^\bullet \otimes^\mathbb{L}_\cR (\cR /\a)).
    \end{equation}
    In addition, we may use the spectral sequence 
    \[
    E_2^{i, j} = \Ext^i_\cR ( H^{-j} ( C^\bullet), \cR) \quad \Rightarrow \quad E^{i + j} = H^{i + j} ( D^\bullet)
    \]
    to deduce that there is an isomorphism $H^0 (D^\bullet) \cong  \Ext^1_\cR ( H^1 (C^\bullet), \cR)$. It therefore follows from Lemma \ref{ext and dual of tor} that 
    \begin{equation} \label{isomorphism ingredient 2}
    H^0 (D^\bullet) \cong H^1 (C^\bullet_\tor)^\vee.
    \end{equation}
    Now, the assumption (\ref{pd assumption}) implies that $\Fitt^0_\cR ( H^1 (C^\bullet)^\vee_\tor)$ is a principal ideal $(x)$ generated by an element $x \in \cR$ that is a nonzero divisor both in $\cR$ and $\cR / \a$. By a standard property of Fitting ideals one has the exact sequence
    \[
    \Tor_1^{\cR} ( \cR / (x), \cR / \a) \to \Fitt^0_\cR ( H^1 (C^\bullet)^\vee_\tor) \otimes_\cR ( \cR / \a) \to \Fitt^0_{(\cR /\a)} ( H^1 (C^\bullet)_\tor^\vee \otimes_\cR (\cR / \a)) \to 0
    \]
    in which the first term is isomorphic to $(\cR / \a) [x]$ and so vanishes because $x$ is a nonzero divisor in $\cR / \a$ by (\ref{pd assumption}). Thus, the second arrow is an isomorphism and so we have a composite isomorphism
    \begin{align*}
    \Fitt^0_\cR ( H^1 (C^\bullet)^\vee_\tor)^\ast \otimes_\cR (\cR /\a) & 
    \; \cong \,   
    \Hom_\cR ( 
    \Fitt^0_\cR ( H^1 (C^\bullet)^\vee_\tor), \cR / \a) \\
    & \; \cong \,
    \Hom_\cR ( 
    \Fitt^0_\cR ( H^1 (C^\bullet)^\vee_\tor )\otimes_\cR (\cR / \a), \cR / \a) \\
    & \stackrel{(\ref{isomorphism ingredient 2})}{\cong}  \Hom_{(\cR / \a)} ( 
    \Fitt^0_\cR ( H^0 (D^\bullet)) \otimes_\cR (\cR / \a), \cR / \a) \\
    & \stackrel{(\ref{isomorphism ingredient 1})}{\cong} \Fitt^0_{(\cR / \a)} (H^0 ( D^\bullet \otimes^\mathbb{L}_\cR (\cR /\a)))^\ast \\
    & \; \cong \,
    \Fitt^0_{(\cR /\a)} ( H^1 (C^\bullet \otimes^\mathbb{L}_\cR (\cR /\a))_\tor^\vee)^\ast
    \end{align*}
    where the first isomorphism holds because $\Fitt^0_\cR (H^1 (C^\bullet)^\vee_\tor)$ is a free $\cR$-module of rank one and the last isomorphism
is (\ref{isomorphism ingredient 2}) with $\cR$ and $C^\bullet$ replaced by $\cR / \a$ and $C^\bullet \otimes_{\cR}^\mathbb{L} (\cR /\a)$. (Note that also the complex $D^\bullet \otimes_\cR (\cR /\a) \cong \RHom_{(\cR /\a)} ( C^\bullet \otimes_{\cR}^\mathbb{L} (\cR /\a), \cR /\a)$ admits a representative of the form (\ref{standard representative}).) This concludes the proof of the lemma.
\end{proof}

\begin{rk}\label{base change for dual of Fitt 1-dim case}
    The proof of Lemma \ref{base change for dual of Fitt}
 shows that if, in the setting of said result, $\dim (\cR) = 1$, then there is a canonical isomorphism $H^1 (C^\bullet)_\tor^\vee \otimes_\cR (\cR /\a) \cong H^1 (C^\bullet \otimes_\cR^\mathbb{L} (\cR /\a))_\tor^\vee$.
 \end{rk}

Let $\a \subseteq \cR$ be an ideal and $C^\bullet \in D^\mathrm{perf} (\cR)$ a complex.
By tensoring the exact sequence $0 \to \a \to \cR \to \cR / \a \to 0$, with $C^\bullet$, we then obtain an exact triangle in $D(\cR)$ of the form
\begin{equation} \label{defining bockstein triangle}
\begin{tikzcd}
    C^\bullet \otimes_\cR^\mathbb{L} \a \arrow{r} & 
    C^\bullet \arrow{r} & 
    C^\bullet \otimes_\cR (\cR / \a) \arrow{r} & 
    ( C^\bullet \otimes_\cR^\mathbb{L} \a) [1].
\end{tikzcd}%
\end{equation}
The long exact sequence in cohomology associated with this triangle allows for the following definition.

\begin{definition} \label{bockstein map def}
Let $C^\bullet \in D^\mathrm{perf} (\cR)$ be a complex that is acyclic in degrees greater than $2$.  
The `Bockstein morphism' associated to $C^\bullet$ and an ideal $\a \subseteq \cR$ is defined to be the composite map
\begin{align*}
    \beta_{C^\bullet, \a} \: H^1 (C^\bullet \otimes^\mathbb{L}_\cR (\cR / \a)) 
    & \to H^2 (C^\bullet \otimes_\cR^\mathbb{L} \a) \xrightarrow{\simeq} H^2 (C^\bullet) \otimes_\cR \a 
\end{align*}
Here the first arrow is the connecting homomorphism of the triangle (\ref{defining bockstein triangle}) and the second arrow is the isomorphism (\ref{isomorphism base change top cohomology}) with $M  = \a$.
\end{definition}

We shall apply this construction in the setting of Definition \ref{def projection map from det}. More precisely, we now assume we are given a quotient $\cY'$ of $H^2 (C^\bullet)$ that is of the form
\begin{equation} \label{quotient of the form}
\cY' = \bigoplus_{i = 1}^{d'} (\cR / \a_i)
\end{equation}
for an integer $d' \geq 0$ and ideals $\a_1, \dots, \a_{d'}$ of $\cR$. \\
For every $i \in \{1, \dots, d'\}$, we write  $x_i^\ast \: \cY' \to \cR / \a_i$ for the projection onto the $i$-th component and use this to define the map
\[
\beta_i \: H^1 (C^\bullet \otimes_\cR (\cR / \a)) \xrightarrow{\beta_{C^\bullet, \a}}
H^2 (C^\bullet) \otimes_\cR \a \to \cY' \otimes_\cR \a \stackrel{x_i^\ast}{\longrightarrow} \a / \a_i \a.
\]
To state the main result of this subsection, we now assume that $C^\bullet$ is acyclic outside degrees 1 and 2, and that there is an $\cR$-linear isomorphism $\nu \: \cR / \a \stackrel{\simeq}{\to} \cR [\a]$. 
For every $f \in H^1 (C^\bullet)^\ast$ we may then define an associated map $f^\nu \in H^1 (C^\bullet \otimes_\cR (\cR / \a))^\ast$ as the composite map
\[
H^1 ( C^\bullet \otimes_\cR (\cR / \a )) \xrightarrow{(\ref{isom induced by j on cohomology})} \Hom_\cR ( \cR / \a, H^1 (C^\bullet)) \stackrel{f_\ast}{\longrightarrow} \Hom_\cR ( \cR / \a, \cR) \stackrel{\nu^{-1}}{\longrightarrow} \cR / \a. 
\]
Given $f \in \exprod^s_\cR H^1 (C^\bullet)^\ast$ for some integer $s \geq 0$, we write $f^\nu$ for the image of $f$ under the map $\exprod^s_\cR H^1 (C^\bullet)^\ast \to \exprod^s_{\cR / \a} H^1 (C^\bullet \otimes_\cR (\cR / \a))^\ast$ induced by sending $f_1 \wedge \dots \wedge f_s \mapsto f_1^\nu \wedge \dots \wedge f_s^\nu$. 

\begin{prop} \label{bockstein proposition}
Let $\cR$ be a quasi-normal ring and $C^\bullet \in D^\mathrm{perf} (\cR)$ a complex that admits a standard representative with respect to a surjection $\kappa \: H^2 (C^\bullet) \to \cY'$ onto a module $\cY'$ of the form (\ref{quotient of the form}), and write $\mathfrak{X}$ for its standard ordered set of generators.
We also assume $\a_i = 0$ if $i \in \{1, \dots, d\}$ for some $d \leq d'$. Let $\a$ be an ideal containing $\a_{d + i}$ for all $i \in \{1, \dots, d' - d\}$, and write $\mathfrak{B}$ (resp.\@ $\mathfrak{B}'$) for the canonical ordered $\cR$-basis (resp.\@ $\cR /\a$-basis) of $\cY \coloneqq \bigoplus_{i = 1}^d \cR$ (resp.\@ of $\cY' \otimes_\cR (\cR / \a)$). Setting $n  \coloneqq d + n_1 - (n_0 + n_2)$, $s \coloneqq d' - d$, and $\mathfrak{A} \coloneqq \prod_{i = d + 1}^{d'} \a_i$, the following claims are then valid.
\begin{liste}
    \item For every $a \in \Det_\cR ( C^\bullet)^{-1}$ and $f \in \exprod^{n}_\cR H^1 (C^\bullet)^\ast$, one has a containment
    \[ f ( \varpi_{C^\bullet, \mathfrak{B}} (a)) \in \Fitt^0_\cR (H^1(C^\bullet)_\tor^\vee)^\ast \otimes_\cR \mathfrak{A}^{\ast \ast}. 
\]
\item Assume that $\overline{\cR} \coloneqq \cR / \a$ is a one-dimensional Gorenstein ring, set $\overline{C}^\bullet \coloneqq C^\bullet \otimes_\cR^\mathbb{L} \overline{\cR}$, and suppose that $\mathrm{pd}_\cR (H^1 (\overline{C}^\bullet)_\tor^\vee) \leq 1$. 
For every $i \in \{1, \dots, s \}$ there is a morphism $\widetilde \beta_{d + i} \in \Hom_{\overline{\cR}} ( H^1 (\overline{C}^\bullet), \overline{\cR}) \otimes_{\overline{\cR}} (\a / \a_{d + i} \a)$ such that the natural map
\[
\Hom_{\overline{\cR}} ( H^1 (\overline{C}^\bullet), \overline{\cR}) \otimes_{\overline{\cR}} (\a / \a_{d + i} \a) \to 
\Hom_{\overline{\cR}} ( H^1 (\overline{C}^\bullet), \a / \a_{d + i} \a) 
\]
sends $\widetilde \beta_{d + i}$ to $\beta_{d + i}$, and the maps $( \widetilde \beta_{d + i})_{1 \leq i \leq s}$ induce a map
\[
({\exprod}_{1 \leq i \leq s} \widetilde \beta_{d + i}) \: \big( I \otimes_{\overline{\cR}} \exprod_{\overline{\cR}}^{n + s} H^1 (\overline{C}^\bullet)^\ast \big)^\ast 
\to  \big( I \otimes_{\overline{\cR}} \exprod_{\overline{\cR}}^{n} H^1 (\overline{C}^\bullet)^\ast \big)^\ast \otimes_{\overline{\cR}} (\a^s / \a \mathfrak{A}),
\]
where $I \coloneqq \Fitt^0_{\overline{\cR}} ( H^1 (\overline{C}^\bullet)_\tor^\vee)$.
\item Assume $\cR$ and $\overline{\cR}$ are one-dimensional Gorenstein rings and that (\ref{pd assumption}) is valid so that, by Lemma \ref{base change for dual of Fitt}), one has an isomorphism
\begin{equation} \label{where the congruence takes place}
\Fitt^0_\cR (H^1(C^\bullet)_\tor^\vee)^\ast \otimes_\cR ( \a^s / \a \mathfrak{A}) \cong I^\ast \otimes_{\overline{\cR}} (\a^s / \a \mathfrak{A}).
\end{equation}
Then, for every $a \in \Det_\cR ( C^\bullet)^{-1}$ and $f \in \exprod^{n + d}_\cR H^1 (C^\bullet)^\ast$ one has an equality
\[
f ( \varpi_{C^\bullet, \mathfrak{B}} (a)) = (-1)^{ns} \cdot (f^\nu \circ {\bigwedge}_{1 \leq i \leq s} \widetilde \beta_{d + i} ) (
\varpi_{\overline{C}^\bullet, \mathfrak{B}'} (\overline{a}))
\]
in (\ref{where the congruence takes place}),
where $\overline{a}$ denotes the image of $a$ under the canonical map
\[
\Det_\cR ( C^\bullet)^{-1} \to \Det_\cR ( C^\bullet)^{-1} \otimes_\cR {\overline{\cR}} \cong \Det_{\overline{\cR}} ( \overline{C}^\bullet)^{-1}.
\]
\end{liste}
 
\end{prop}

\begin{proof}
This generalises the argument of \cite[Lem.\@ 10.2]{burns07} (see also \cite[Lem.~5.22]{bks}).
By assumption, one has a surjective map $H^2 (C^\bullet) \twoheadrightarrow \cY'$, from which it follows that $\Fitt^d_\cR (H^2 (C^\bullet))$ is contained in $ \Fitt^d_\cR ( \cY') = \mathfrak{A}$. The first claim therefore follows from Pro\-po\-si\-tion~\ref{integrality properties of det projection map}~(b).\\ 
To prove claim (b), we fix a standard representative $F^0 \xrightarrow{\partial_0} F^1 \xrightarrow{\partial_1} F_2$ of $C^\bullet$ with respect to $\kappa$ and $\mathfrak{X}$. We 
also set $\overline{F}^i \coloneqq F^i \otimes_\cR \overline{\cR}$ for every $i \in \{0, 1, 2\}$,
write $\overline{\partial}_i \: \overline{F}^i \to \overline{F}^{i + 1}$ for the map induced by $\partial_i$,
and note that $\overline{F}^0 \xrightarrow{\overline{\partial_0}} \overline{F}^1 \xrightarrow{\overline{\partial_1}} \overline{F}_2$ is a standard representative for $\overline{C}^\bullet$. 
Define $\overline{Q} \coloneqq \overline{F}^1 / \overline{F}^0$ and note that
taking $\overline{\cR}$-linear duals leads to an exact sequence
\[
0 \to \overline{Q}^\ast \to (\overline{F}^0)^\ast \xrightarrow{\overline{\partial}_0^\ast} (\overline{F}^1)^\ast \to \Ext^1_{\overline{\cR}} ( \overline{Q}, \overline{\cR}) \to 0.
\]
We have seen in the proof of Proposition \ref{integrality properties of det projection map} that the map $\Ext^1_{\overline{\cR}} ( \overline{Q}, \overline{\cR}) \to \Ext^1_{\overline{\cR}} ( H^1 (\overline{C}^\bullet), \overline{\cR})$ is a pseudo-isomorphism, and hence an isomorphism because we are assuming $\overline{\cR}$ to be one-dimensional. Our assumption that $\Ext^1_{\overline{\cR}} ( H^1 (\overline{C}^\bullet), \overline{\cR})$ is of projective dimension at most one therefore implies that the kernel of $\overline{\partial}_0^\ast$ is projective (by Schanuel's lemma). From the above exact sequence we conclude that $\overline{Q}^\ast$, and hence also $\overline{Q}^{\ast \ast}$, is $\overline{\cR}$-projective. In particular, for every finitely generated $\overline{\cR}$-module $M$, the 
central vertical map in the following commutative diagram is an isomorphism.
\begin{equation} \label{diagram used for map lifting} \begin{tikzcd}[row sep=small]
  &  \overline{\partial_1} (\overline{F}^1)^\ast \otimes_{\overline{\cR}} M  \arrow{d} \arrow{r}
& \overline{Q}^\ast \otimes_{\overline{\cR}} M  \arrow{d}{\simeq} \arrow{r} & H^1 (\overline{C}^\bullet)
    \otimes_{\overline{\cR}} M  \arrow{d} \arrow{r} & 0 \\
   0 \arrow{r} & \Hom_{\overline{\cR}} ( \overline{\partial_1} (\overline{F}^1), M) \arrow{r} & 
     \Hom_{\overline{\cR}}  ( \overline{Q}, M) \arrow{r} & \Hom_{\overline{\cR}}  ( H^1 (\overline{C}^\bullet), M) &
     \end{tikzcd}
\end{equation}%
To justify exactness of the top line of this diagram we recall that if $N$ is a finitely generated torsion-free $\overline{\cR}$-module, then $\Ext^1_{\overline{\cR}} ( N, \overline{\cR})$ vanishes because $\overline{\cR}$ is a one-dimensional Gorenstein ring. Dualising the exact sequence 
\begin{equation} \label{ses Q and H1}
    0 \to H^1 (\overline{C}^\bullet) \to \overline{Q} \to \overline{\partial_1} (\overline{F}^1) \to 0
\end{equation}
and tensoring the resulting exact sequence with $M$ then gives the top line of the above diagram. The bottom line, on the other hand, is obtained by applying the functor $\Hom_{\overline{\cR}} ( - , M)$ to (\ref{ses Q and H1}). \\
From the diagram (\ref{diagram used for map lifting}), applied with $M = \a / \a_{d + i} \a$, we see that it suffices to prove that $\beta_{d + i}$ can be lifted to $\Hom_{\overline{\cR}}  ( \overline{Q}, \a / \a_{d + 1} \a)$ in order to define $\widetilde{\beta}_{d + i}$.
To do this, we note that 
by definition $\beta_{d + i}$ coincides with the composition of the snake lemma map arising from the diagram
\begin{cdiagram}[row sep=small]
     & \a \otimes_\cR Q \arrow{r} \arrow{d}{\partial_1} & Q \arrow{d}{\partial_1} \arrow{r} & Q \otimes_\cR \overline{\cR} \arrow{r} \arrow{d}{\partial_1} & 0 \\
    0 \arrow{r} & \a \otimes_\cR F^2 \arrow{r} & F^2 \arrow{r} & F^2 \otimes_\cR \overline{\cR} \arrow{r} & 0
\end{cdiagram}%
composed with $x_{d + i}^\ast$. By condition (iv) in Definition \ref{standard representative def} one therefore has
\begin{equation} \label{lift of Bockstein}
    \beta_{d + i} ( \overline{m}) = (b_{d + i}^\ast \circ \partial_1) ( m) \mod \a_{d + i} \a
\end{equation}
if $\overline{m} \in H^1 (\overline{C}^\bullet)$ is the image of $m \in F^1$ under the surjective map $F^1 \to \overline{F}^1 \to \overline{Q}$. This definition clearly extends to $\overline{Q}$ and so by the above discussion we obtain the desired lift $\widetilde{\beta}_i$ of $\beta_{d + i}$ to $H^1 (\overline{C}^\bullet)^\ast \otimes_{\overline{\cR}} ( \a / \a_{d + i} \a)$.\\
For every $i \in \{1, \dots, s\}$ we may then write $\widetilde{\beta}_{d + i} = \sum_{j = 1}^{n_i} \psi_j \otimes c_j$ for suitable $n_i \in \N$, $\psi_j \in H^1 (\overline{C}^\bullet)^\ast$, and $c_j \in \a / \a_{d + i} \a$. Given this, we may define the map $\exprod_{1 \leq i \leq s} \widetilde{\beta}_{d + i}$ in the statement of claim (b) by means of
\[
(\exprod_{1 \leq i \leq s} \widetilde{\beta}_{d + i})( \lambda \otimes \varphi) \coloneqq \sum_{j_1 = 1}^{n_1} \dots \sum_{j_s = 1}^{n_s} \Phi ( \lambda \otimes \psi_{j_1} \wedge \dots \wedge \psi_{j_s} \wedge \varphi) \otimes \prod_{i = 1}^s c_{j_i}
\]
for all $\Phi \in ( I \otimes_{\overline{\cR}} \exprod_{\overline{\cR}}^{n + s} H^1 (\overline{C}^\bullet)^\ast)^\ast$.\\
To prove claim (c), we fix $a \in \Det_\cR (C^\bullet)^{-1}$, set $J \coloneqq \Fitt^0_\cR ( H^1 (C^\bullet)_\tor^\vee)$, and
and regard $\varpi_{C^\bullet, \mathfrak{B}} (a)$ as an element of $( J \otimes_\cR \exprod^{n + d}_\cR (F^1)^\ast)^\ast$.
If we fix an $\cR$-basis $c_1, \dots, c_{n_0}$ of $F^0$, then we can write $a$ as $(\bigwedge_{1 \leq i \leq n_0} c_i^\ast) \otimes a' \otimes ( \bigwedge_{1 \leq i \leq n_2} b_i^\ast )$ for some $a' \in \exprod^{n_1}_\cR F^1$.
By definition of $\varpi_{C^\bullet, \mathfrak{B}}$, one then has for all $\lambda \otimes \varphi \in J \otimes_\cR \exprod_\cR^{n + s} F_1^\ast$ that
\begin{align*}
\varpi_{C^\bullet, \mathfrak{B}} (a) ( \lambda \otimes \varphi) & = (-1)^{n(n_2 - d)} \cdot 
\varphi \big(\lambda \cdot {\bigwedge}_{1 \leq i \leq n_0} c_i^\ast \wedge  {\bigwedge}_{d + 1 \leq i \leq n_2} (b_i^\ast \circ \partial_1)  \big) (a') \\ 
& = (-1)^{s (n_2 - d')} \cdot (-1)^{n(n_2 - d)} \cdot  
 ( {\bigwedge}_{d + 1 \leq i \leq d'} (b_i^\ast \circ \partial_1) \wedge \varphi  ) (z)
\end{align*}
with the abbreviation
\[
z \coloneqq \big(\lambda \cdot {\bigwedge}_{1 \leq i \leq n_0} c_i^\ast \wedge  {\bigwedge}_{d' + 1 \leq i \leq n_2} (b_i^\ast \circ \partial_1) \big) (a').
\]
On the other hand, for all $\lambda \otimes \varphi \in I \otimes_{\overline{\cR}} \exprod^{n + s}_{\overline{\cR}} (\overline{F}^1)^\ast$ one has 
\begin{align*}
\varpi_{\overline{C}^\bullet, \mathfrak{B}'} (\overline{a}) ( \lambda \otimes \varphi) & = (-1)^{n'(n_2 - d')} \cdot \varphi \big( \lambda \cdot {\bigwedge}_{1 \leq i \leq n_0} c_i^\ast \wedge  {\bigwedge}_{d' + 1 \leq i \leq n_2} (b_i^\ast \circ \partial_1)  \big) (a') \mod \a
\end{align*}
with $n' \coloneqq n + s$.
 Here we have used that $\overline{F}^0  \to \overline{F}^1 \to \overline{F}^2$ is a standard representative for $\overline{C}^\bullet$ with respect to $( \overline{\kappa}, \mathfrak{B}')$, where $\overline{\kappa} \: H^2 (\overline{C}^\bullet) \to \cY'$ is the map induced by $\kappa$.
Writing $\pi$ for the canonical projection $\mathfrak{A} \cdot ( J \otimes_\cR \exprod^{n}_\cR (F^1)^\ast)^\ast \to \mathfrak{A} \otimes_{\overline{\cR}} ( I \otimes_{\overline{\cR}} \exprod^{n}_{\overline{\cR}} (\overline{F}^1)^\ast)^\ast$, we arrive at
\[ 
\pi ( \varpi_{C^\bullet, \mathfrak{B}} (a)) = (-1)^{ns} \cdot (\exprod_{1 \leq i \leq s} (b_{d + i}^\ast \circ \partial_1)) ( \varpi_{\overline{C}^\bullet, \mathfrak{B}'} (\overline{a}))
\mod \a \mathfrak{A}
\]
by comparing with the definition of $z$.
Given this, the equality claimed in (b) follows from (\ref{lift of Bockstein}) and Lemma~\ref{calculation with fj} below.
\end{proof}

\begin{rk}
The maps $\widetilde{\beta}_i$ could in principle be not uniquely specified by their properties described in Proposition \ref{bockstein proposition}\,(b). This ambiguity will however not cause any problems in any of the applications of Proposition \ref{bockstein proposition}\,(c) in the main body of the article.
\end{rk}

\begin{lem} \label{calculation with fj}
    Let $M$ be an $\cR$-module and denote the canonical projection $M \to M \otimes (\cR  / \a)$ by $\pi_M$. 
    Assume there is an $\cR$-linear isomorphism $\nu \: \cR / \a \xrightarrow{\simeq} \cR [\a]$ and, for every $f \in M^\ast$, define $f^\nu \: M \otimes_\cR (\cR /\a) \to \cR / \a$ by $f^\nu (m \otimes \pi_\cR (r)) \coloneqq \nu^{-1} ( f ( \nu ( \pi_\cR (r)) \cdot m))$. 
    Then one has
    \[
    \pi_\cR (f (m)) =  f^\nu ( \pi_M (m))  \quad \text{ in } \cR / \a
    \]
    for every $f \in M^\ast$. 
\end{lem}

\begin{proof}
We may calculate that
\begin{align*}
    f^\nu ( \pi_M (m)) & = f^\nu ( m \otimes \pi_\cR (1)) = \nu^{-1} ( f (  \nu ( \pi_\cR (1)) \cdot m)) 
    = \nu^{-1} ( \nu (\pi_\cR (1)) \cdot f(m)) \\
    & = \nu^{-1} ( \nu ( \pi_\cR (1) \cdot f(m))) = \pi_\cR (1 \cdot f(m)) = \pi_\cR (f (m)),
\end{align*}
    as claimed. 
\end{proof}

\paragraph{Acknowledgments}
The authors would like to extend their gratitude to
Daniel Barrera Salazar,
David Burns, Masato Kurihara, Juan-Pablo Llerena-C\'ordova, Daniel Mac\'ias Castillo, Rei Otsuki, Rob Rockwood, Takamichi Sano, and Christian Wuthrich for helpful comments and discussions.\\
The research for this article was carried out while the first and second authors held, respectively, a Doctoral Prize Fellowship funded by the UK Engineering and Physical Sciences Research Council [EP/W524335/1] at University College London and a Heilbronn Research Fellowship at Imperial College London. During the final editing of the article, the first author held a Juan de la Cierva Fellowship at Instituto Ciencias Matem\'aticas (ICMAT) and participated in the project [PID2022-142024NB-I00], both funded by 
the Spanish Ministry of Science
and Innovation [JDC2023-051626-I, MCIU/AEI/10.13039/501100011033].\\
The authors gratefully acknowledge the support of these institutions.

\addcontentsline{toc}{section}{References}
\tiny
\printbibliography

\small

(Bullach)
\textsc{Instituto de Ciencias Matem\'aticas, c/ Nicol\'as Cabrera 13-15, Campus de Cantoblanco UAM, 28049 Madrid, Spain.}\\
\textit{Email address:} \href{mailto:dominik.bullach@icmat.es}{dominik.bullach@icmat.es} 
\medskip \\ 
(Honnor) \textsc{
St Paul's School, Lonsdale Road,
London SW13 9JT, UK.
}

\end{document}